\documentclass[11pt,twoside]{amsart}
\usepackage[dvips]{graphicx}
\usepackage[usenames,dvipsnames]{xcolor}
\usepackage{colortbl}
\usepackage{multirow}
\usepackage{pgf,tikz}
\usepackage{tkz-graph}
\usepackage{subfig}
\usepackage{amsmath, amsthm, amscd, amsfonts, amssymb, color}
\usepackage[bookmarksnumbered, plainpages]{hyperref}
\usepackage{relsize}
\usepackage{longtable}
\newtheorem{thm}{Theorem}[section]
\newtheorem{cor}[thm]{Corollary}
\newtheorem{lem}[thm]{Lemma}

\newtheorem{defn}[thm]{Definition}
\newtheorem{rem}[thm]{Remark}
\newtheorem{conj}[thm]{Conjecture}

\numberwithin{equation}{section}

\newtheorem{ex}{Example}

\begin{document}

\oddsidemargin 0mm
\evensidemargin 0mm

\thispagestyle{plain}

\vspace{5cc}
\begin{center}

{\large\bf   Zero divisor and unit  elements with supports of size $ 4 $ in
group algebras of torsion-free groups}
\rule{0mm}{6mm}\renewcommand{\thefootnote}{}
\footnotetext{{\scriptsize 2010 Mathematics Subject Classification: 20C07, 16S34. }\\
{\rule{2.4mm}{0mm} \scriptsize Keywords and Phrases: Kaplansky  zero divisor conjecture, Kaplansky  unit conjecture, group ring, torsion-free group.}}

\vspace{1.5cc}
{\large\it Alireza Abdollahi and Fatemeh Jafari}

\vspace{2.5cc}
\parbox{28cc}{{\small \textbf{ABSTRACT. } Kaplansky  Zero Divisor Conjecture states that if $G $ is a 
torsion-free group and $  \mathbb{F} $  is a field, then the group ring $\mathbb{F}[G]$ contains no zero divisor and Kaplansky  Unit Conjecture states that if $G $ is a torsion-free group and $  \mathbb{F} $  is a field, then $\mathbb{F}[G]$ contains no non-trivial units. The support of an  element $ \alpha= \sum_{x\in G}\alpha_xx$ in $\mathbb{F}[G] $, denoted by $supp(\alpha)$, is the set $ \{x \in G|\alpha_x\neq 0\} $.  In this paper we  study  possible zero divisors and units with supports of size $ 4 $ in group algebras of torsion-free groups. We prove that if
 $ \alpha, \beta $ are non-zero elements in $ \mathbb{F}[G] $ for a possible torsion-free group $ G $ and an arbitrary field $ \mathbb{F} $ such that $ |supp(\alpha)|=4 $ and $ \alpha\beta=0 $, then   $|supp(\beta)|\geq 7 $. In [J. Group Theory, $16$ $ (2013),$ no. $5$, $667$-$693$], it is proved that if $ \mathbb{F}=\mathbb{F}_2 $ is the field with two elements, $ G $ is a  torsion-free group and $ \alpha,\beta \in \mathbb{F}_2[G]\setminus \{0\}$ such that  $|supp(\alpha)|=4 $ and $ \alpha\beta =0 $, then   $|supp(\beta)|\geq 8$. We improve the latter result to  $|supp(\beta)|\geq 9$. Also, 
concerning the Unit Conjecture, we prove that if $\mathsf{a}\mathsf{b}=1$ for some $\mathsf{a},\mathsf{b}\in \mathbb{F}[G]$ and $|supp(\mathsf{a})|=4$, then  $|supp(\mathsf{b})|\geq 6$. 
  }}

\end{center}

\vspace{1cc}
\vspace{1cc}
\begin{center}
\section {\bf Introduction and Results}
\end{center}
We call a non-zero element $\alpha$ of a ring  zero divisor if $\alpha \beta=0$  for some non-zero element  $\beta$ in the ring.
Let $G$ be a  group and $\mathbb{F}$  any field. We denote by  $\mathbb{F}[G]$  the group algebra of  $G$ over  $\mathbb{F}$. 
In 1940, Irving Kaplansky \cite{100} stated his well known conjecture as follows:\\
\begin{conj} [Kaplansky Zero Divisor Conjecture] \label{KZDC} Let $  \mathbb{F}$ be a field and $G$ a torsion-free group. Then $\mathbb{F}[G]$ contains no  zero divisor. 
\end{conj}
Another famous conjecture of Kaplansky on group algebras is the following \cite{100}:
\begin{conj}[Kaplansky Unit Conjecture]\label{uk}
Let $  \mathbb{F}$ be a field and $G$ a torsion-free group. Then $\mathbb{F}[G]$ contains no  non-trivial units $($i.e., non-zero scalar multiples of group elements$)$.
\end{conj}
 Conjecture \ref{uk} is actually  stronger than   Conjecture \ref{KZDC} so that  the affirmative solution to Conjecture \ref{uk} implies the positive one for Conjecture \ref{KZDC} \cite[Lemma 13.1.2]{PI}.  Conjecture \ref{uk} is known for any field $ \mathbb{F} $  and any unique product group $ G $  \cite{PI,b1}. 
Conjecture \ref{KZDC} is known to be hold valid for any field $ \mathbb{F} $, where $G$ is a  unique product group \cite[Theorem 26.2]{PI} or  elementary amenable group \cite{a5}. The proof of the former case is elementary however the proof of the latter one needs  advanced concepts and techniques. The latter results are somehow the state-of-art of what one knows on Conjecture \ref{KZDC}; for more details concerning  Conjecture \ref{KZDC} the reader may see  \cite{Brown, a12, a2, a13,  a5, a14, a6, PI, a1}. 
  Conjecture \ref{KZDC} is still open and attempts to confirm it for fixed
fields $  \mathbb{F}$ were even unsuccessful.  It seems that the simplest case to verify Conjecture \ref{KZDC} is for the field with 2 elements which is still far from to have solved.\\
For each  element $ \alpha= \sum_{x\in G}\alpha_xx$ in $\mathbb{F}[G] $, the support of $ \alpha $, denoted by $supp(\alpha)$, is the set $ \{x \in G|\alpha_x\neq 0\}$. \\
 Recently,  zero divisors and  units with small supports have been studied in \cite{BP,PS} and \cite{a55}, respectively. It is proved that the group algebra of a torsion-free group has no zero divisor and no unit with support of size at most $ 2  $  \cite[Theorem $2.1$]{PS} and \cite[Theorem $4.2$]{a55}, respectively. Also,   by using a combinatorial structure in \cite{PS} it is shown that if $ G $ is a torsion-free group and  $ \alpha,\beta \in \mathbb{F}_2[G]\setminus \{0\} $ such that $ |supp(\alpha)|=4$,  then $ |supp(\beta)|>4 $ \cite[Theorem $1.2$]{PS} and with a computer-assisted approach  the latter result improved to $|supp(\beta)|>7$  \cite[Theorem 1.3]{PS}. 
Zero divisors with support of size $3$ are also studied in \cite{PS} and it is proved that  if $ G $ is a torsion-free group and  $ \alpha,\beta \in \mathbb{F}_2[G]\setminus \{0\} $ such that $ |supp(\alpha)|=3$,  then $ |supp(\beta)|>16$. Unit elements with support of size $3$ are studied in \cite{a55}, where it is proved that 
if $G$ is a torsion-free group and $ \mathsf{a},\mathsf{b} \in \mathbb{F}_2[G]$ such that  $\mathsf{a} \mathsf{b}=1$,  then $ |supp(\mathsf{b})|>11$ whenever $ |supp(\mathsf{a})|=3$; and $ |supp(\mathsf{b})|>5$ whenever $ |supp(\mathsf{a})|=5$. 
 In \cite{AT} using two multigraphs $Z(\alpha,\beta)$ and $U(\mathsf{a},\mathsf{b})$ associated with a pair $(\alpha,\beta)$ of zero divisors (i.e. $\alpha \beta=0$) or a pair $(\mathsf{a},\mathsf{b})$ of unit elements (i.e. $\mathsf{a}\mathsf{b}=1$) of any support sizes in a group algebra the lower bound $16$ is improved to $18$ in the latter result; also units with support of size $3$ have been studied. As we are following the same approach introduced in \cite{AT} to study zero divisors and units with support size 3,  we need to  remind some definitions from graph theory that we will use in the sequel. \\

 A graph $ \mathcal{G}=(\mathcal{V}_{\mathcal{G}},\mathcal{E}_{\mathcal{G}},\psi_{\mathcal{G}}) $  consists of a non-empty set $ \mathcal{V}_{\mathcal{G}} $, a possibly empty set $ \mathcal{E}_{\mathcal{G}} $ and if $ \mathcal{E}_{\mathcal{G}}\neq \varnothing $ a function $ \psi_{\mathcal{G}}:\mathcal{E}_{\mathcal{G}}\longrightarrow \mathcal{V}_{\mathcal{G}}^{(1)}\cup \mathcal{V}_{\mathcal{G}}^{(2)}\cup (\mathcal{V}_{\mathcal{G}}\times \mathcal{V}_{\mathcal{G}}) $, where $ \mathcal{V}_{\mathcal{G}}^{(i)} $ denotes the set of all $  i$-element subsets of $ \mathcal{V}_{\mathcal{G}} $ for $i = 1, 2$. The elements of $ \mathcal{V}_{\mathcal{G}} $ and $ \mathcal{E}_{\mathcal{G}} $ are called vertices and edges of the graph $\mathcal{G}$, respectively. An edge
$e \in \mathcal{E}_{\mathcal{G}}$ is called an undirected loop if $ \psi_{\mathcal{G}}(e)\in \mathcal{V}_{\mathcal{G}}^{(1)} $. The edge $e$ is called directed if $ \psi_{\mathcal{G}}(e)\in \mathcal{V}_{\mathcal{G}}\times \mathcal{V}_{\mathcal{G}} $ and it is called a directed loop if $ \psi_{\mathcal{G}}(e)= (v, v)$ for some $v \in \mathcal{V}_{\mathcal{G}}$. The graph $ \mathcal{G} $ is called undirected if
$\psi_{\mathcal{G}}(\mathcal{E}_{\mathcal{G}})\cap (\mathcal{V}_{\mathcal{G}}\times \mathcal{V}_{\mathcal{G}}) = \varnothing$. We say a vertex $u$ is adjacent to a vertex $v$, denoted by $u \sim v$, if $\psi_{\mathcal{G}}(e) = \{u, v\}, (u, v)$ or $(v, u)$ for some $e \in \mathcal{E}_{\mathcal{G}}$; otherwise, we say $u$ is not adjacent to $v$, denoted by $u \nsim v$. In the latter case the vertices $u$ and $v$ are called the endpoints of the edge $e$ and we say $e$ joins its endpoints. If a vertex $v$ is an endpoint of an edge $e$, we say $v$ is adjacent to $e$ or also $e$ is adjacent to $v$. Two edges $e_1$ and
$e_2$ are called adjacent, denoted by $e_1\sim e_2$, if the sets of their endpoints have non-empty intersection.
The graph $ \mathcal{G} $ is called loopless whenever $\psi_{\mathcal{G}}(\mathcal{E}_{\mathcal{G}}) \cap \mathcal{V}_{\mathcal{G}}^{(1)} = \varnothing$ and $\psi_{\mathcal{G}}(\mathcal{E}_{\mathcal{G}}) \cap \{(v, v)|v \in \mathcal{V}_{\mathcal{G}}\}=\varnothing$.  Multi-edges are  two or more edges  that have the same endpoints.  A multigraph is a graph which is permitted to have multi-edges but they have  no loop. A simple graph is an undirected loopless graph having no multi-edge. So, the graph $ \mathcal{G} $ is simple if either $ \mathcal{E}_{\mathcal{G}} $ is empty or $ \psi_{\mathcal{G}}$ is an injective function from
$ \mathcal{E}_{\mathcal{G}} $ to $ \mathcal{V}_{\mathcal{G}}^{(2)} $. The degree of a vertex $v$ of $ \mathcal{G} $, denoted by $deg_{\mathcal{G}}(v)$, is the number of edges adjacent to $v$.
If all vertices of a graph have the same degree $k$, we say that the graph is $k$-regular.  
We say an undirected graph is connected if for any two distinct vertices $v$ and $w$ there are vertices $v_0=v, v_1, \dots , v_n=w$ such that $v_{j-1}\sim v_j$ for $j = 1, \dots , n$. A
subgraph of a graph $\mathcal{G}$ is a graph $H$ such that $ \mathcal{V}_{\mathcal{H}} \subseteq  \mathcal{V}_{\mathcal{G}}, \mathcal{E}_{\mathcal{H}} \subseteq  \mathcal{E}_{\mathcal{G}}$ and $\psi_{\mathcal{H}}$ is the restriction of $ \psi_{\mathcal{G}} $ to
$\mathcal{E}_{\mathcal{H}}  $. In a graph $ \mathcal{G} $, an induced subgraph $\mathcal{I}$ on a set of vertices $W \subseteq \mathcal{V}_{\mathcal{G}}$ is a subgraph in which $\mathcal{V}_{\mathcal{I}} = W$
and $\mathcal{E}_{\mathcal{I}} = \{e \in \mathcal{E}_{\mathcal{G}} \;|\; \text{the endpoints of} \; e \;\text{are in}\; W\}$. An isomorphism between two undirected and loopless
graphs $\mathcal{G}$ and $\mathcal{H}$ is a pair of bijections $\phi_{\mathcal{V}}:\mathcal{V}_{\mathcal{G}}\longrightarrow \mathcal{V}_{\mathcal{H}} $ and $\phi_{\mathcal{E}}:\mathcal{E}_{\mathcal{G}}\longrightarrow \mathcal{E}_{\mathcal{H}} $ preserving adjacency and non-adjacency i.e., for any pair of vertices $u, v \in \mathcal{V}_{\mathcal{G}}$, $\phi_{\mathcal{V}}(u) \sim \phi_{\mathcal{V}} (v) \Leftrightarrow  u \sim v $ and for any pair of edges $e_1, e_2 \in  \mathcal{V}_{\mathcal{G}}$, $\phi_{\mathcal{E}}(e_1) \sim \phi_{\mathcal{E}} (e_2) \Leftrightarrow e_1 \sim e_2$. Two undirected and loopless graphs $\mathcal{G}$ and $\mathcal{H}$ are called
isomorphic (denoted by $\mathcal{G}\cong\mathcal{H}$) if there is an isomorphism between them. We say that an undirected
and loopless graph $\mathcal{H}$ is a forbidden subgraph of a graph $\mathcal{G}$ if there is no subgraph isomorphic to $\mathcal{H}$ in
$\mathcal{G}$. A Cayley graph of a group $G$ with respect to an inverse closed set $S$ with $1\not\in S$, we mean  the graph whose vertex set is $G$ and two vertices $x$ and $y$ are adjacent if $xy^{-1} \in S$.

 In this paper  we  study  $ Z(\alpha,\beta) $ and $ U(\mathsf{a},\mathsf{b}) $ for a pair of zero divisors  $ (\alpha,\beta) $ and a pair of unit elements $ (\mathsf{a},\mathsf{b})$ in $\mathbb{F}[G]$ for a possible torsion-free group $ G $ and arbitrary field $ \mathbb{F} $ such that $ supp(\alpha) $ and $ supp(\mathsf{a}) $ are  of size  4. We specifically study  $ Z(\alpha,\beta) $ for a pair of zero divisors  $ (\alpha,\beta) $ in $\mathbb{F}_2[G]$ for a possible torsion-free group $ G $ such that $ |supp(\alpha)|=4 $. Note that any group algebra over the field $ \mathbb{F}_2 $ has no  unit with support of size  $ 4  $ (Remark \ref{unit4}, below). For a non-zero element $\alpha$ of a group algebra we denote by  $S_{\alpha}$ the set $\lbrace h^{-1}h^\prime \mid h\neq h^\prime, h,h^\prime \in supp(\alpha) \rbrace$. 
Concerning  Conjecture \ref{KZDC}, our main results  are the following:
\begin{thm} \label{a} 
Let $G $ be a torsion-free group, $  \mathbb{F} $  an arbitrary field and $ \alpha,\beta $ be non-zero elements of $ \mathbb{F}[G]$ such that $ |supp(\alpha)|=4 $ and $ \alpha\beta=0 $. Then the following statements hold:
\begin{itemize}
\item[i.]$ | S_{\alpha} | \in \{10,12\} $.
\item[ii.] If  $ | S_{\alpha} | =12 $, then $ Z(\alpha,\beta) $ contains no subgraph isomorphic to one of the  graphs in Figure {\rm \ref{606}}, where each vertex of the subgraph has degree $ 4 $ in $ Z(\alpha,\beta) $, and the graph in Figure {\rm\ref{6066}}, where the degree of white vertices of the subgraph in $ Z(\alpha ,\beta) $ must be 5.
\item[iii.] If  $ | S_{\alpha} | =12 $ and $\mathbb{F}=\mathbb{F}_2$, then $Z(\alpha,\beta)$ contains a $4$-regular subgraph.
\item[iv.]If  $ | S_{\alpha} | =10 $, then $ Z(\alpha,\beta) $  contains no    subgraph isomorphic to   the graph in Figure {\rm \ref{multicycle}} for all   $ n\geq 3 $. 
\item[v.] If $ | S_{\alpha} | =10 $ and $\mathbb{F}=\mathbb{F}_2$, then $ Z(\alpha,\beta) $ contains no  subgraph isomorphic to one of the graphs in Figure {\rm \ref{tt}}, where the degrees of white vertices of any subgraph in $ Z(\alpha,\beta)$ must be $ 4 $.
\end{itemize}
\end{thm}
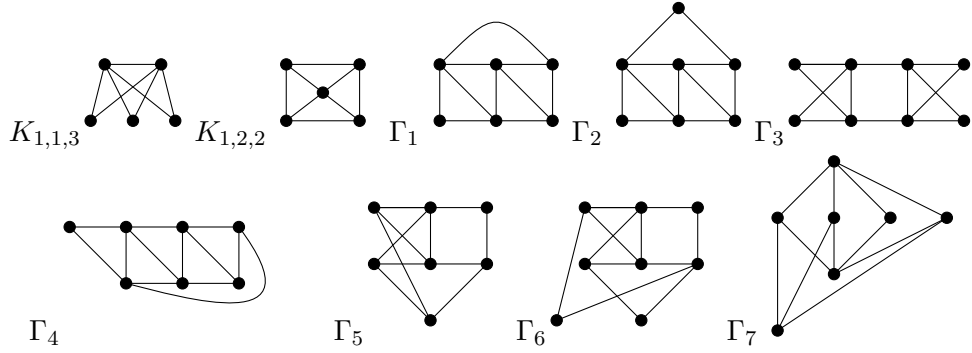
\begin{figure}
\subfloat{$ K_{1,1,3} $}
\begin{tikzpicture}[scale=.75]
\draw [fill] (0,0) circle
[radius=0.1] node  [left]  {};
\draw [fill] (1,0) circle
[radius=0.1] node  [right]  {};
\draw [fill] (-.25,-1) circle
[radius=0.1] node  [below]  {};
\draw [fill] (.5,-1) circle
[radius=0.1] node  [below]  {};
\draw [fill] (1.25,-1) circle
[radius=0.1] node  [below]  {};
\draw (0,0) -- (1,0) -- (-.25,-1) -- (0,0);
\draw  (1,0) -- (.5,-1) -- (0,0);
\draw  (1,0) -- (1.25,-1) -- (0,0);
\end{tikzpicture}
\subfloat{$  K_{1,2,2} $}
\begin{tikzpicture}[scale=.75]
\draw [fill] (0,0) circle
[radius=0.1] node  [left]  {};
\draw [fill] (1.30,0) circle
[radius=0.1] node  [right]  {};
\draw [fill] (0,-1) circle
[radius=0.1] node  [below]  {};
\draw [fill] (1.30,-1) circle
[radius=0.1] node  [below]  {};
\draw [fill] (.65,-.5) circle
[radius=0.1] node  [below]  {};
\draw (0,0) -- (1.30,0) -- (1.30,-1) -- (0,-1)-- (0,0);
\draw  (1.30,-1) -- (.65,-.5) -- (0,0);
\draw  (1.30,0) -- (.65,-.5) -- (0,-1);
\end{tikzpicture}
\subfloat{$ \Gamma_1 $}
\begin{tikzpicture}[scale=.75]
\draw [fill] (0,0) circle
[radius=0.1] node  [left]  {};
\draw [fill] (1,0) circle
[radius=0.1] node  [right]  {};
\draw [fill] (2,0) circle
[radius=0.1] node  [below]  {};
\draw [fill] (0,-1) circle
[radius=0.1] node  [below]  {};
\draw [fill] (1,-1) circle
[radius=0.1] node  [below]  {};
\draw [fill] (2,-1) circle
[radius=0.1] node  [below]  {};
\draw (0,0) -- (1,0) -- (2,0) -- (2,-1)-- (1,-1)-- (0,-1)-- (0,0);
\draw  (1,0) -- (1,-1) -- (0,0);
\draw  (1,0) -- (2,-1) ;
\draw (0,0) .. controls (1,1)  .. (2,0);
\end{tikzpicture}
\subfloat{$  \Gamma_2 $}
\begin{tikzpicture}[scale=.75]
\draw [fill] (0,0) circle
[radius=0.1] node  [left]  {};
\draw [fill] (1,0) circle
[radius=0.1] node  [right]  {};
\draw [fill] (2,0) circle
[radius=0.1] node  [below]  {};
\draw [fill] (0,-1) circle
[radius=0.1] node  [below]  {};
\draw [fill] (1,-1) circle
[radius=0.1] node  [below]  {};
\draw [fill] (2,-1) circle
[radius=0.1] node  [below]  {};
\draw [fill] (1,1) circle
[radius=0.1] node  [below]  {};
\draw (0,0) -- (1,0) -- (2,0) -- (2,-1)-- (1,-1)-- (0,-1)-- (0,0);
\draw  (1,0) -- (1,-1) -- (0,0);
\draw  (1,0) -- (2,-1) ;
\draw  (0,0) -- (1,1)-- (2,0);
\end{tikzpicture}
\subfloat{$  \Gamma_3 $}
\begin{tikzpicture}[scale=.75]
\draw [fill] (0,0) circle
[radius=0.1] node  [left]  {};
\draw [fill] (1,0) circle
[radius=0.1] node  [right]  {};
\draw [fill] (2,0) circle
[radius=0.1] node  [below]  {};
\draw [fill] (0,-1) circle
[radius=0.1] node  [below]  {};
\draw [fill] (1,-1) circle
[radius=0.1] node  [below]  {};
\draw [fill] (2,-1) circle
[radius=0.1] node  [below]  {};
\draw [fill] (-1,0) circle
[radius=0.1] node  [below]  {};
\draw [fill] (-1,-1) circle
[radius=0.1] node  [below]  {};
\draw (-1,0) -- (0,0) -- (1,0) -- (2,0) ;
\draw  (1,0) -- (1,-1) -- (0,-1) -- (0,0);
\draw  (1,0) -- (2,-1) ;
\draw  (0,0) -- (-1,0)-- (0,-1);
\draw  (0,0) -- (-1,-1)-- (0,-1);
\draw  (2,0) -- (1,-1)-- (2,-1);
\end{tikzpicture}\\
\subfloat{$  \Gamma_4 $}  
\begin{tikzpicture}[scale=.75]
\draw [fill] (0,0) circle
[radius=0.1] node  [left]  {};
\draw [fill] (1,0) circle
[radius=0.1] node  [right]  {};
\draw [fill] (2,0) circle
[radius=0.1] node  [below]  {};
\draw [fill] (0,-1) circle
[radius=0.1] node  [below]  {};
\draw [fill] (1,-1) circle
[radius=0.1] node  [below]  {};
\draw [fill] (2,-1) circle
[radius=0.1] node  [below]  {};
\draw [fill] (-1,0) circle
[radius=0.1] node  [below]  {};
\draw (0,0) -- (1,0) -- (2,0) -- (2,-1)-- (1,-1)-- (0,-1)-- (0,0);
\draw  (1,0) -- (1,-1) -- (0,0);
\draw  (1,0) -- (2,-1) ;
\draw  (0,0) -- (-1,0)-- (0,-1);
\draw (0,-1) .. controls (1,-1.25) and (3.5,-2)  .. (2,0);
\end{tikzpicture}
\subfloat{$  \Gamma_5 $}
\begin{tikzpicture}[scale=.75]
\draw [fill] (0,0) circle
[radius=0.1] node  [left]  {};
\draw [fill] (1,0) circle
[radius=0.1] node  [right]  {};
\draw [fill] (0,-1) circle
[radius=0.1] node  [below]  {};
\draw [fill] (1,-1) circle
[radius=0.1] node  [below]  {};
\draw [fill] (-1,0) circle
[radius=0.1] node  [below]  {};
\draw [fill] (-1,-1) circle
[radius=0.1] node  [below]  {};
\draw [fill] (0,-2) circle
[radius=0.1] node  [below]  {};
\draw (-1,0) -- (0,0) -- (1,0) ;
\draw  (1,0) -- (1,-1) -- (0,-1) -- (0,0);
\draw  (0,0) -- (-1,0)-- (0,-1);
\draw  (0,0) -- (-1,-1)-- (0,-1);
\draw  (-1,0) -- (0,-2)-- (1,-1) ;
\draw  (-1,-1) -- (0,-2);
\end{tikzpicture}
\subfloat{$  \Gamma_6 $}
\begin{tikzpicture}[scale=.75]
\draw [fill] (0,0) circle
[radius=0.1] node  [left]  {};
\draw [fill] (1,0) circle
[radius=0.1] node  [right]  {};
\draw [fill] (0,-1) circle
[radius=0.1] node  [below]  {};
\draw [fill] (1,-1) circle
[radius=0.1] node  [below]  {};
\draw [fill] (-1,0) circle
[radius=0.1] node  [below]  {};
\draw [fill] (-1,-1) circle
[radius=0.1] node  [below]  {};
\draw [fill] (0,-2) circle
[radius=0.1] node  [below]  {};
\draw [fill] (-1.5,-2) circle
[radius=0.1] node  [below]  {};
\draw (-1,0) -- (0,0) -- (1,0) ;
\draw  (1,0) -- (1,-1) -- (0,-1) -- (0,0);
\draw  (0,0) -- (-1,0)-- (0,-1);
\draw  (0,0) -- (-1,-1)-- (0,-1);
\draw  (-1,0) -- (-1.5,-2)-- (1,-1) ;
\draw  (-1,-1) -- (0,-2);
\draw  (1,-1) -- (0,-2);
\end{tikzpicture}
\subfloat{$  \Gamma_7 $}
\begin{tikzpicture}[scale=.75]
\draw [fill] (0,0) circle
[radius=0.1] node  [left]  {};
\draw [fill] (1,0) circle
[radius=0.1] node  [left]  {};
\draw [fill] (2,0) circle
[radius=0.1] node  [left]  {};
\draw [fill] (3,0) circle
[radius=0.1] node  [left]  {};
\draw [fill] (1,1) circle
[radius=0.1] node  [left]  {};
\draw [fill] (1,-1) circle
[radius=0.1] node  [left]  {};
\draw [fill] (0,-2) circle
[radius=0.1] node  [left]  {};
\draw  (1,1) -- (1,0) -- (1,-1) -- (0,0)--(1,1);
\draw  (1,1) -- (2,0) -- (1,-1) ;
\draw  (1,1) -- (3,0) -- (1,-1) ;
\draw  (0,0) -- (0,-2) -- (1,0) ;
\draw  (3,0) -- (0,-2)  ;
\end{tikzpicture}
\caption{Some forbidden subgraphs of $Z(\alpha,\beta) $ and $ U(\mathsf{a},\mathsf{b}) $, where  $|S_\alpha|=12$,  $ |S_\mathsf{a}|=12 $ and  each vertex of the subgraph has degree $ 4 $ in  these graphs.}\label{606}
\end{figure}
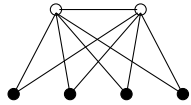
\begin{figure}
\begin{tikzpicture}[scale=.75]
\draw [] (0,0) circle
[radius=0.1] node  [left]  {};
\draw [] (1.5,0) circle
[radius=0.1] node  [right]  {};
\draw [fill] (-.75,-1.5) circle
[radius=0.1] node  [below]  {};
\draw [fill] (.25,-1.5) circle
[radius=0.1] node  [below]  {};
\draw [fill] (1.25,-1.5) circle
[radius=0.1] node  [below]  {};
\draw [fill] (2.25,-1.5) circle
[radius=0.1] node  [below]  {};
\draw (.06,0) -- (1.44,0); 
\draw (1.45,-.06)-- (-.75,-1.5) -- (0,-.06);
\draw  (1.5,-.06) -- (.25,-1.5) -- (0,-.06);
\draw  (1.5,-.06) -- (1.25,-1.5) -- (0,-.06);
\draw  (1.5,-.06) -- (2.25,-1.5) -- (0,-.06);
\end{tikzpicture}
\caption{A forbidden subgraph of $ Z(\alpha ,\beta) $, where  $ |S_\alpha|=12 $ and  the degree of white vertices of the subgraph in $ Z(\alpha ,\beta) $ must be 5. }\label{6066}
\end{figure}
\begin{figure}
\begin{tikzpicture}[scale=.75]
\draw [fill] (0,0) circle
[radius=0.1] node  [left]  {};
\draw [fill] (2,0) circle
[radius=0.1] node  [left]  {};
\draw [] (1,1) circle
[radius=0.1] node  [left]  {};
\draw [fill] (1,-1) circle
[radius=0.1] node  [left]  {};
\draw  (0,0) -- (.95,.95);
\draw (1.05,.95) -- (2,0) -- (1,-1)-- (0,0);
\draw  (1,.9) --  (1,-1) ;
\end{tikzpicture}
\begin{tikzpicture}[scale=.75]
\draw [] (0,0) circle
[radius=0.1] node  [left]  {};
\draw [] (2,0) circle
[radius=0.1] node  [left]  {};
\draw [fill] (1,1) circle
[radius=0.1] node  [left]  {};
\draw [fill] (1,-1) circle
[radius=0.1] node  [left]  {};
\draw  (.05,.05) -- (1,1);
\draw (1,1) -- (1.95,.05) ;
\draw (1.95,-.05)-- (1,-1)-- (.05,-.05);
\draw  (1,1) --  (1,-1) ;
\end{tikzpicture}
\begin{tikzpicture}[scale=.75]
\draw [fill] (0,0) circle
[radius=0.1] node  [left]  {};
\draw [fill] (2,0) circle
[radius=0.1] node  [left]  {};
\draw [] (1,1) circle
[radius=0.1] node  [left]  {};
\draw [] (1,-1) circle
[radius=0.1] node  [left]  {};
\draw [fill] (1,0) circle
[radius=0.1] node  [left]  {};
\draw [fill] (1.5,0) circle
[radius=0.1] node  [left]  {};
\draw  (1.47,.05) -- (.98,.95);
\draw  (1.47,-.05) -- (1,-.95);
\draw  (.05,.05) -- (1,.95);
\draw (1,.95) -- (1.95,.05) ;
\draw (1.95,-.05)-- (1,-.95)-- (.05,-.05);
\draw  (1,.95) --  (1,.1) ;
\draw  (1,-.1) --  (1,-.9) ;
\end{tikzpicture}
\begin{tikzpicture}[scale=.75]
\draw [] (0,0) circle
[radius=0.1] node  [left]  {};
\draw [] (2,0) circle
[radius=0.1] node  [left]  {};
\draw [fill] (1,1) circle
[radius=0.1] node  [left]  {};
\draw [fill] (1,-1) circle
[radius=0.1] node  [left]  {};
\draw [] (1,0) circle
[radius=0.1] node  [left]  {};
\draw [] (1.5,0) circle
[radius=0.1] node  [left]  {};
\draw  (1.47,.05) -- (1,1);
\draw  (1.47,-.05) -- (1,-1);
\draw  (.05,.05) -- (1,1);
\draw (1,1) -- (1.95,.05) ;
\draw (1.95,-.05)-- (1,-1)-- (.05,-.05);
\draw  (1,1) --  (1,.1) ;
\draw  (1,-.1) --  (1,-.9) ;
\end{tikzpicture}
\begin{tikzpicture}[scale=.75]
\draw [] (0,0) circle
[radius=0.1] node  [left]  {};
\draw [] (2,0) circle
[radius=0.1] node  [right]  {};
\draw [fill] (1,1) circle
[radius=0.1] node  [left]  {};
\draw [fill] (1,-1) circle
[radius=0.1] node  [left]  {};
\draw [] (1,0) circle
[radius=0.1] node  [left]  {};
\draw [fill] (2,1) circle
[radius=0.1] node  [left]  {};
\draw   (1.1,.05) -- (2,1);
\draw   (2,.05) -- (2,1);
\draw   (.05,.05) -- (1,.95);
\draw  (1.95,.05)--(1,.95)  ;
\draw   (1,-1)-- (.05,-.05);
\draw   (1,1) --  (1,.13) ;
\draw   (1,-.1) --  (1,-.9) ;
\draw   (1,-1)--(1.95,-.05);
\end{tikzpicture}
\begin{tikzpicture}[scale=.75]
\draw [] (0,0) circle
[radius=0.1] node  [left]  {};
\draw [fill] (1,0) circle
[radius=0.1] node  [left]  {};
\draw [fill] (0,-1) circle
[radius=0.1] node  [left]  {};
\draw [] (1,-1) circle
[radius=0.1] node  [left]  {};
\draw [fill] (.5,1) circle
[radius=0.1] node  [left]  {};
\draw  (.09,.1) --  (.5,1) ;
\draw  (.5,1) --  (1,0) ;
\draw  (.1,0) --  (1,0) ;
\draw  (1,-.9) --  (1,0) ;
\draw  (.9,-1) --  (0,-1) ;
\draw  (0,-.1) --  (0,-1) ;
\end{tikzpicture}
\begin{tikzpicture}[scale=.75]
\draw [fill] (0,0) circle
[radius=0.1] node  [left]  {};
\draw [fill] (1,0) circle
[radius=0.1] node  [left]  {};
\draw [] (0,-1) circle
[radius=0.1] node  [left]  {};
\draw [] (1,-1) circle
[radius=0.1] node  [left]  {};
\draw [] (.5,1) circle
[radius=0.1] node  [left]  {};
\draw  (.09,.1) --  (.45,.95) ;
\draw  (.55,.95) --  (1,0) ;
\draw  (.1,0) --  (1,0) ;
\draw  (1,-.9) --  (1,0) ;
\draw  (.9,-1) --  (.1,-1) ;
\draw  (0,-.1) --  (0,-.9) ;
\end{tikzpicture}
\begin{tikzpicture}[scale=.75]
\draw [] (0,0) circle
[radius=0.1] node  [left]  {};
\draw [] (1,0) circle
[radius=0.1] node  [left]  {};
\draw [fill] (0,-1) circle
[radius=0.1] node  [left]  {};
\draw [fill] (1,-1) circle
[radius=0.1] node  [left]  {};
\draw [fill] (.5,1) circle
[radius=0.1] node  [left]  {};
\draw  (.09,.1) --  (.45,.95) ;
\draw  (.55,.95) --  (.95,.05) ;
\draw  (.1,0) --  (.9,0) ;
\draw  (1,-.9) --  (1,-.1) ;
\draw  (.9,-1) --  (.1,-1) ;
\draw  (0,-.1) --  (0,-.9) ;
\end{tikzpicture}
\caption{Some forbidden subgraphs of $Z(\alpha,\beta) $, where $\mathbb{F}=\mathbb{F}_2$,  $ |S_\alpha|=10 $ and   the degrees of white vertices in $ Z(\alpha,\beta)$ must be $ 4 $.} \label{tt}
\end{figure}
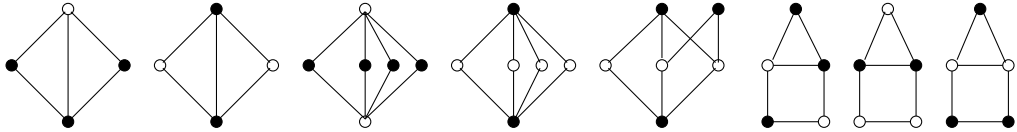
The following result extends \cite[Theorem 1.2]{PS} to arbitrary fields with a better lower bound. 
\begin{thm}
Let $ G $ be a torsion-free group and  $ \alpha , \beta $ be non-zero elements in $ \mathbb{F}[G] $ such that $ \alpha\beta=0 $. If  $ | supp(\alpha)| =4 $, then $ | supp(\beta)| \geq 7 $. 
\end{thm}
The following result improves one step the lower bound $8$ in \cite[Theorem 1.3]{PS}. 
\begin{thm}
Let $ G $ be a torsion-free group and  $ \alpha , \beta $ be non-zero elements in $ \mathbb{F}_2[G] $ such that $ \alpha\beta=0 $. If  $ | supp(\alpha)| =4 $, then $ | supp(\beta)| \geq 9 $. 
\end{thm}
Concerning Conjecture \ref{uk}, our main results  are the following:
\begin{thm} \label{aunit}
Let $G $ be a torsion-free group, $  \mathbb{F} $  an arbitrary field and $ \mathsf{a},\mathsf{b} \in \mathbb{F}[G]$ such that $ |supp(\mathsf{a})|=4 $ and $ \mathsf{a}\mathsf{b}=1 $. Then the following statements hold:
\begin{itemize}
\item[i.]$ | S_{\mathsf{a}} | \in \{10,12\} $.
\item[ii.] If  $ | S_{\mathsf{a}} | =12 $, then $ U(\mathsf{a},\mathsf{b}) $  contains no subgraph isomorphic to one of the  graphs  in   Figure {\rm \ref{606}}, where each vertex of the subgraph has degree $ 4 $ in $ U(\mathsf{a},\mathsf{b}) $.
\item[iii.]If  $ | S_{\mathsf{a}} | =10 $, then $ U(\mathsf{a},\mathsf{b}) $   contains no    subgraph isomorphic to   the graph in Figure {\rm \ref{multicycle}} for all   $ n\geq 3 $. 
\end{itemize}
\end{thm}
\begin{thm}
Let $G $ be a torsion-free group, $  \mathbb{F} $  an arbitrary field and $ \mathsf{a},\mathsf{b} \in \mathbb{F}[G]$ such that $ |supp(\mathsf{a})|=4 $ and $ \mathsf{a}\mathsf{b}=1 $. Then  $ |supp(\mathsf{b})|\geq 6 $.
\end{thm}
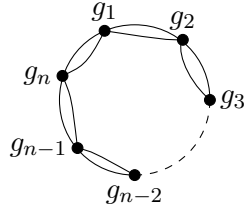
\begin{figure}[htb]
\begin{tikzpicture}[scale=.8]
\draw (0,-1.25) .. controls (-0.700,-1.25) and (-1.25,-0.7) .. (-1.25,0)
.. controls (-1.25,0.7) and (-0.7,1.25) .. (0,1.25)
.. controls (0.7,1.25) and (1.25,0.7) .. (1.25,0);
\draw (1.25,0)[dashed] .. controls (1.25,-0.7) and (0.7,-1.25) .. (0,-1.25);
\draw [fill] (0,-1.25) circle
[radius=0.09] node  [below]  {$g_{n-2}$};
\draw [fill] (-.95,-.8) circle
[radius=0.09] node  [left]  {$g_{n-1}$};
\draw [fill] (-1.20,.4) circle
[radius=0.09] node  [left]  {$g_{n}$};
\draw [fill] (-.5,1.15) circle
[radius=0.09] node  [above]  {$g_{1}$};
\draw [fill] (.8,1) circle
[radius=0.09] node  [above]  {$g_{2}$};
\draw [fill] (1.25,0) circle
[radius=0.09] node  [right]  {$g_{3}$};
\draw (0,-1.25) .. controls (-.7,-.86) and (-.7,-.86) .. (-.95,-.8).. controls (-1.05,.05) and (-1.05,.05) ..(-1.20,.4).. controls (-.9,.55) and (-.8,.55) ..(-.5,1.15).. controls (.5,1) and (.5,1) ..(.8,1).. controls (.7,.9) and (.8,.4) .. (1.25,0);
\end{tikzpicture}
\caption{ A forbidden subgraph of $ Z(\alpha,\beta) $ and $ U(\mathsf{a},\mathsf{b}) $ for all $ n\geq 3  $, where $ | S_{\alpha} |$ and $|S_\mathsf{a}|  $ are equal to 10.}\label{multicycle}
\end{figure}

\section {\bf Preliminaries}
We  encounter to the Klein bottle group in the sequel, so we give its definition and some of its properties.
\begin{rem}\label{klein}
The Klein bottle group has the  presentation $ \left\langle x,y \mid x^{2}=y^{2} \right\rangle$, also   $ \left\langle x,y \mid xyx=y \right\rangle  $ is another presentation of the Klein bottle group. It follows from Remark $ 3.16 $ of  \cite{AT} that  Conjectures {\rm\ref{KZDC}}  and {\rm\ref{uk}} are known to be hold valid for any field $ \mathbb{F}$ and  every torsion-free quotient of the Klein bottle group.
\end{rem}
\begin{defn}\label{set}
Let $G$ be a  group and $B,C$ be two finite subsets of $G$. As usual we denote by $BC$ the set $\lbrace bc\;| \;b\in B,\, c\in C\rbrace $. Also, for each element $ x\in BC$, denote by  $R_{BC}(x)$ the set $\{(b,c)\in B\times C\;|\; x=bc\}$ and let $r_{BC}(x):=|R_{BC}(x)| $. It is clear that $ r_{BC}(x)\leq min\{|B|,|C|\} $ for all  $ x\in BC $.
\end{defn} 
\begin{lem}\label{01}
Let $x,y,z$ be $3$ pairwise distinct non-trivial elements of a  torsion-free group.  Suppose that the subgroup $\langle x,y,z \rangle$ is neither abelian nor isomorphic to the Klein bottle group. If
\begin{equation*}
 S:=(\{1,x,y,z\}^{-1}\{1,x,y,z\})\setminus\{1\}=\lbrace x,y,z,x^{-1},y^{-1},z^{-1},x^{-1}y,x^{-1}z,y^{-1}x,y^{-1}z,z^{-1}x,z^{-1}y\rbrace, 
\end{equation*}
then $|S|\in \{10,12\}$. In particular, if $|S|=10$, then $\{x,y,z\} \in \big {\{} \{a,a^{-1},b\},\{ a,a^2,b\},\{a,b,ab^{-1}a\},\{a,b,ab\} \big{\}}$ for some $a,b\in \{ x,y,z \}$.
\end{lem}
\begin{proof}
Let $ X=\lbrace 1,x,y,z\rbrace $ and $ H=\langle x,y,z \rangle $. Clearly, $ | X^{-1} X | \leq 16  $. Since $x,y,z$ are pairwise distinct non-trivial elements,  $ R_{X^{-1} X}(1)=\{(1,1),(x^{-1},x),(y^{-1},y),(z^{-1},z)\} $. Hence,  $ | S | = | X^{-1} X  \setminus \lbrace 1\rbrace | \leq 12 $.  Since $1\not\in S^{-1}=S$ and the group has no element of order $2$ as it is torsion-free, it follows that $|S|$ is even and so $ | S |  \neq 11 $. 

Let us see under what conditions 
on two distinct pairs $(x_1,x_2)$ and $(x_3,x_4)$ of elements of $X$, $x_{1}^{-1}x_2 =x_{3}^{-1}x_4$.

 Let $ t=|\{x_1,x_2\}\cap \{x_3,x_4\}| $.  It is clear that $ 0\leq t\leq 2 $. If  $ t=2 $, then 
$x_1=x_4$ and $x_2=x_3$ so that $(x_1^{-1} x_2)^2=1$; since $H$ is torsion-free $x_1=x_2$ and so $(x_1,x_2)=(x_3,x_4)$, a contradiction. Thus $t\in\{0,1\}$. 
 Note that if  $ t=1 $, since $ |X|=4 $,  either $ x_1=x_4 $ or $ x_2=x_3 $.\\
  
Hence,   one of  the following cases holds:
\begin{itemize}
\item[(A)]$ t=1 $ and $ \{x_1,x_2\}\cap \{x_3,x_4\}=\{1\} $. It follows that,  in this case there exist distinct elements $ a,b\in X\setminus \{1\} $ such that $ b=a^{-1} $ and $ b^{-1}=a $.
\item[(B)]$ t=1 $ and $ 1\in \{x_1,x_2,x_3,x_4\}\setminus (\{x_1,x_2\}\cap \{x_3,x_4\})$. It follows that,  in this case there exist distinct elements $ a,b\in X\setminus \{1\} $ such that $ a^{-1}b=a $ which is equivalent to  $ b=a^{2} $.
\item[(C)]$ t=1 $ and $ 1\notin \{x_1,x_2,x_3,x_4\}$. It follows that, in this case there exist pairwise distinct elements $ a,b,c\in X\setminus \{1\} $ such that $ c^{-1}b=a^{-1}c $ which is equivalent to  $b=ca^{-1}c$.
\item[(D)]$ t=0 $. It follows that, in this case there exist pairwise distinct elements $ a,b,c\in X\setminus \{1\} $ such that $ a^{-1}b=c $ which is equivalent to $  b=ac$.
\end{itemize}
 Suppose, for a contradiction, that $ |S|<10 $. 
 Then there exist some pairs $(x_1,x_2)$, $(x_3,x_4)$, $(x'_1,x'_2)$ and $(x'_3,x'_4)$ of  elements of $X$ such that  $(x_1,x_2)\neq (x_3,x_4)$, $(x'_1,x'_2)\neq (x'_3,x'_4)$, $ x_{1}^{-1}x_2=x_{3}^{-1}x_4 $, $ {x'}_{1}^{-1}{x'}_2 = {x'}_{3}^{-1}{x'}_4 $, $ \{(x'_1,x'_2),(x'_3,x'_4)\}\neq \{(x_1,x_2),(x_3,x_4)\}$  and $ \{(x'_1,x'_2),(x'_3,x'_4)\}\neq \{({x_2},{x_1}),({x_4},{x_3})\}$.
  Thus, the pairs $(x_1,x_2)$ and $(x_3,x_4)$ satisfy the case (I) and 
  $(x'_1,x'_2)$ and $(x'_3,x'_4)$ satisfy the case (J), where (I) and (J) are one the above cases (A), (B), (C) or (D). 
  
 In the following, we show that every choice for the set $\{\rm{I,J}\}$ leads to a contradiction and so $|S|\in \{10, 12\}$.
\begin{itemize}
\item[(1)] Suppose that  $\{\rm{I,J}\}=\{\rm{A,B}\}$. Then there exist distinct elements  $ a,b\in X\setminus \{1\} $ such that $ b=a^{-1} $ and there exist distinct elements $ a',b'\in X\setminus \{1\} $ such that $ b'=a'^{2} $. Hence, one of the  following cases occurs:
\begin{itemize}
\item[1.] If $ \{a',b'\}=\{a,b\} $, then $H$ has  a non-trivial torsion element, a contradiction. 
\item[2.] If $ \{a',b'\}\neq\{a,b\} $, then $ H $ is a cyclic group, a contradiction.
\end{itemize}
\item[(2)] Suppose that  $\{\rm{I,J}\}=\{\rm{A,C}\}$. Then there exist distinct elements  $ a,b\in X\setminus \{1\} $ such that $ b=a^{-1} $ and there exist pairwise distinct elements $ a',b',c'\in X\setminus \{1\} $ such that  $a'=b'c'^{-1}b'$. Clearly, $ \{a,b\}\subseteq \{a',b',c'\} $. Thus, one of the  following cases occurs:
\begin{itemize}
\item[1.] If $ \{a',c'\}=\{a,b\} $, then $ H $ is isomorphic to the Klein bottle group, a contradiction.
\item[2.] If $ \{a',b'\}=\{a,b\} $ or  $ \{b',c'\}=\{a,b\} $, then $ H $ is a cyclic group, a contradiction.
\end{itemize}
\item[(3)] Suppose that $\{\rm{I,J}\}=\{\rm{A,D}\}$. Then there exist distinct elements  $ a,b\in X\setminus \{1\} $ such that $ b=a^{-1} $ and there exist pairwise distinct elements $ a',b',c'\in X\setminus \{1\} $ such that  $b'=a'c'$. Hence, one of the following cases occurs:
\begin{itemize}
\item[1.] If $ \{a',c'\}=\{a,b\} $, then $ b'=1 $, a contradiction.
\item[2.] If $ \{a',b'\}=\{a,b\} $ or  $ \{b',c'\}=\{a,b\} $, then $ H $ is a cyclic group, a contradiction.
\end{itemize}
\item[(4)] Suppose that $\{\rm{I,J}\}=\{\rm{B,C}\}$. Then, there exist distinct elements $ a,b\in X\setminus \{1\} $ such that  $ b=a^{2} $ and there exist pairwise distinct elements $ a',b',c'\in X\setminus \{1\} $ such that  $a'=b'c'^{-1}b'$. Thus,  one of the following cases occurs:
\begin{itemize}
\item[1.] If $ \{a',c'\}=\{a,b\} $, then $ H$ is  abelian, a contradiction.
\item[2.] If $ \{a',b'\}=\{a,b\} $ or  $(b',c')=(b,a)$, then $ H $ is a cyclic group, a contradiction.
\item[3.] If $(b',c')=(a,b)$, then $ a'=1 $, a contradiction.
\end{itemize}
\item[(5)] If $\{\rm{I,J}\}=\{\rm{B,D}\}$, then it follows that, $ H$ is a cyclic group, a contradiction.
\item[(6)]Suppose that $\{\rm{I,J}\}=\{\rm{C,D}\}$. Then  there exist pairwise distinct elements $ a,b,c\in X\setminus \{1\} $ such that  $b=ca^{-1}c$ and there exist pairwise distinct elements $ a',b',c'\in X\setminus \{1\} $ such that $  b'=a'c'$. Hence, one of the following cases occurs:
\begin{itemize}
\item[1.] If $\{a',c'\}\neq\{a,b\}$, then  $ H $  is  a cyclic group, a contradiction.
\item[2.] If $ \{a',c'\}=\{a,b\} $, then  $c=1 $, a contradiction.
\end{itemize}
\item[(7)] If $\{\rm{I,J}\}=\{\rm{A}\}$, then it follows that,  $|X|<4$, a contradiction.
 \item[(8)] Suppose that $\{\rm{I,J}\}=\{\rm{B}\}$. Then there exist distinct elements $ a,b\in X\setminus \{1\} $ such that  $ b=a^{2} $ and also there exist distinct elements $ a',b'\in X\setminus \{1\} $ such that  $ b'={a'}^{2} $. Note that since $ \{(x'_1,x'_2),(x'_3,x'_4)\}\neq \{(x_1,x_2),(x_3,x_4)\}$ and $ \{(x'_1,x'_2),(x'_3,x'_4)\}\neq \{({x_2},{x_1}),({x_4},{x_3})\}$, $ (a',b')\neq (a,b) $. Hence, one of the following cases occurs:
  \begin{itemize}
   \item[1.]If $a=b'$ and $b=a'$, then $ H $ has a non-trivial torsion element, a contradiction.
 \item[2.]If $a\neq a'$ and $b= b'$, then $ H $ is isomorphic to the Klein bottle group, a contradiction.
 \item[3.]If $a\neq b'$ and $b= a'$ or $a= b'$ and $b\neq a'$, then $ H $ is a cyclic group, a contradiction.
  \item[4.]If $a=a'$ and $b\neq  b'$, then $|X|<4$, a contradiction.
 \end{itemize}
  \item[(9)] If $\{\rm{I,J}\}=\{\rm{C}\}$, then it follows that, $ H $ has a non-trivial torsion element, a contradiction. 
  \item[(10)] Suppose that $\{\rm{I,J}\}=\{\rm{D}\}$. Then  there exist pairwise distinct elements $ a,b,c\in X\setminus \{1\} $ such that  $b=ac$ and there exist pairwise distinct elements $ a',b',c'\in X\setminus \{1\} $ such that $  b'=a'c'$. Clearly, $ (a',b',c')\neq (a,b,c) $.  Hence, one of the following cases occurs:
 \begin{itemize}
 \item[1.]If $(a',b',c')\in\{(a,c,b),(b,a,c)\} $, then $ H $ has a non-trivial torsion element, a contradiction.
 \item[2.]If $(a',b',c')\in\{(b,c,a) ,(c,a,b)\} $, then it follows that $ H $ is isomorphic to the Klein bottle group, a contradiction.
  \item[3.]If $(a',b',c')=(c,b,a) $, then it follows that $ H $ is abelian, a contradiction.
 \end{itemize}  
\end{itemize} 
Thus, $ |S|\in \{10,12\} $ and if $ |S|=10 $, then exactly one of the cases among (A), (B), (C)  and  (D)  holds. This completes the proof. Table \ref{1010} depicts the elements of $ S $ which are equal  in  each case among  (A), (B), (C)  and  (D). 
\begin{table}
\begin{tabular}{|c|c|c|c|c|c|c|c|c|c|c|c|c|}
\cline{2-13}
\multicolumn{1}{c|}
{} & $ a $ & $ b $ & $ c $ & $ a^{-1} $ & $ {b}^{-1} $ & $c^{-1} $ & $ a^{-1}b $ & $ a^{-1}c $ & $ {b}^{-1}a $ & $ {b}^{-1}c  $ & $c^{-1}a$ & $c^{-1}b $ \\
\hline
$  b=a^{-1} $ &\cellcolor{black!70!white}{ } & \cellcolor{black!30!white}{ }  &  & \cellcolor{black!30!white}{ } & \cellcolor{black!70!white}{ } &  &  &  & & & &  \\
\hline
$ b=a^{2}$ & \cellcolor{black!30!white}{ } &   &  & \cellcolor{black!70!white}{ } & &  & \cellcolor{black!30!white}{ } &  & \cellcolor{black!70!white}{ } & & &  \\
\hline
$ b=ca^{-1}c $ &  &   &  &  & &  &  & \cellcolor{black!70!white}{ } & &\cellcolor{black!30!white}{ }  &\cellcolor{black!30!white}{ }  & \cellcolor{black!70!white}{ } \\
\hline
$b=ac  $ &  &  & \cellcolor{black!70!white}{ }  &  & &\cellcolor{black!30!white}{ }  & \cellcolor{black!70!white}{ }  & &\cellcolor{black!30!white}{ }  &  &  &\\
\hline
\end{tabular}
\caption{  The same colors (black and gray) in each row  show the elements of $ X^{-1}X\setminus \{1\} $, where $ X=\{1,a,b,c\} $, that are equal  in   each case which leads to $ |X^{-1}X\setminus \{1\}|=10 $.}\label{1010}
\end{table}
\end{proof}
\begin{defn}
For each element  $ \alpha \in \mathbb{F}[G] $, let $ S_\alpha:=(supp(\alpha)^{-1} supp(\alpha)) \setminus \{1\}=\lbrace h^{-1}h^\prime \mid h\neq h^\prime, h,h^\prime \in supp(\alpha) \rbrace $.
\end{defn} 
\begin{defn}
Let $ \mathbb{F}$ be an arbitrary field and  $G$  a  group and also $ \alpha, \beta$ be non-zero elements in $ \mathbb{F}[G] $ such that  $ \alpha\beta=a $, where $ a\in \{0,1\} $. We say that $ \beta $ is a mate of $ \alpha $ if  for each non-zero element  $ \beta'\in \mathbb{F}[G]$ such that $\alpha\beta'=a$,  then $|supp(\beta)|\leq |supp(\beta')|$.
\end{defn}
\begin{rem}
Note that if $ \beta $ is a mate of $  \alpha$ then it is not necessary that $  \alpha$ is a mate of $  \beta$ {\rm(}see Example {\rm\ref{ex}}, below{\rm)}.
\end{rem} 
\begin{lem}\label{1}
Let $ \alpha$ be a zero divisor in $\mathbb{F}[G]$ for a possible torsion-free group $ G $ and arbitrary field $ \mathbb{F}$ and $ \beta $ be a mate of $ \alpha $. If $ |  supp(\alpha) | \leq 5 $, then  $ \left\langle h^{-1} supp(\alpha)\right\rangle=\left\langle  supp(\beta)g^{-1}\right\rangle $ for all  $ h \in  supp(\alpha) $ and all $ g \in supp(\beta) $.
\end{lem}
\begin{proof}
 Let $B=supp(\alpha ) $, $C=supp(\beta ) $, $ H= \left\langle h^{-1}B\right\rangle $ and $K= \left\langle Cg^{-1}\right\rangle $.  
 Partition $B=B_1\cup \cdots \cup B_l $, where  $B_i=B\cap t_iK$ and $t_iK$ is a left coset of $K$ in $G$. Clearly, since $|B|\leq 5$, $l\leq 5$.  Suppose, for a contradiction, that  $ l\geq 2 $. Since 
\begin{equation*}
h^{-1}\alpha \beta g^{-1} =  h^{-1}(\sum_{h'\in B }\alpha_{h'}h')\beta g^{-1}=  h^{-1}(\sum_{i=1}^{l}\sum_{h'\in B_i }\alpha_{h'}h')\beta g^{-1}=0,
\end{equation*}
and $t_iK \cap t_jK=\varnothing$, where $i,j$ are distinct elements in $ \{1,\ldots,l\}$, we have   $ (\sum_{h_i\in B_s }\alpha_{i}h_{i})\beta =0 $ for each  $ s\in \{1,\ldots,l\} $. By  \cite[Theorem 2.1]{PS}, $ |B_i|\notin \{1,2\} $ for all $ i\in \{1,\ldots,l\} $. On the other hand, since $ | B | \leq 5 $ and $ l\geq 2$, there exists $ i\in \{1,\ldots,l\} $ such that $ |B_i|\in \{1,2\} $, a contradiction. Hence, $l=1$ and  since $ 1\in h^{-1}B $,  $ H \subseteq K $. By  \cite[Lemma  2.5 ]{AT},  $ K\subseteq H $ and therefore $ K= H $. 
\end{proof}
\begin{lem}\label{z}
Let  $ \mathbb{F} $ be an arbitrary field and $G$  a torsion-free group. If $\alpha$ is a possible  zero divisor in $\mathbb{F}[G]$ with $|supp(\alpha)|=4$, then  $|S_\alpha|\in\{10,12\}$.
\end{lem}
\begin{proof}
Suppose that $ \beta $ is a mate of $ \alpha $. Let $B=supp(\alpha)$, $C=supp(\beta)$ and  $H=\left\langle h^{-1}B\right\rangle$, where $ h $ is an arbitrary element of $B$. By Lemma \ref{1}, $h^{-1}\alpha,\beta g^{-1}\in \mathbb{F}[H]$  for all $ g\in C $. Since $ (h^{-1}\alpha)(\beta g^{-1})=0 $,  $ H $ is not abelian \cite[Theorem 26.2]{PI} and  also $H$ is not isomorphic to the Klein bottle group  (see Remark \ref{klein}). Then by Lemma \ref{01}, $|B^{-1}B\setminus \{1\}|=|((h^{-1}B)^{-1}(h^{-1}B))\setminus \{1\}|\in \{10,12\}$. This completes the proof.
\end{proof}
\begin{lem}\label{hg}
Let  $ \mathsf{a} $  be a non-trivial unit in $\mathbb{F}[G]$ for a possible torsion-free group $ G $ and arbitrary field $ \mathbb{F}$ and $ \mathsf{b} $ be a mate of $ \mathsf{a} $. If $ |supp(\mathsf{a})| \leq 5 $, then  $ \left\langle  supp(\mathsf{a})\right\rangle=\left\langle  supp(\mathsf{b})\right\rangle $.
\end{lem} 
\begin{proof}
Let $B=supp(\mathsf{a} ) $, $C=supp(\mathsf{b} ) $, $ H= \left\langle B\right\rangle $ and $K= \left\langle C\right\rangle $. Partition $B=B_1\cup \cdots \cup B_r $, where  $B_i=B\cap t_iK$ and $t_iK$ is a  left coset of $K$ in $G$. Suppose, for a contradiction, that $ r\geq 2 $. Since 
\begin{equation*}
\mathsf{a}\mathsf{b} =  (\sum_{\mathsf{a}'\in B }\alpha_{\mathsf{a}'}\mathsf{a}')\mathsf{b} =  (\sum_{i=1}^{r}\sum_{\mathsf{a}'\in B_i }\alpha_{\mathsf{a}'}\mathsf{a}')\mathsf{b}=1,
\end{equation*}
and for each $i\neq j,\,  t_iK\cap t_jK=\varnothing $, there exists $ l\in\{1,\ldots,r\} $ such that $ ( \sum_{\mathsf{a}'\in B_l }\alpha_{\mathsf{a}'}\mathsf{a}')\mathsf{b}=1 $ and for each $ t\in \{1,\ldots,r\}\setminus\{l\}, \;( \sum_{\mathsf{a}'\in B_t }\alpha_{\mathsf{a}'}\mathsf{a}')\mathsf{b}=0 $. By  \cite[Theorem 4.2 ]{a55},   $ |B_l|\notin \{1,2\} $ and  by  \cite[Theorem 2.1]{PS}, for each $ l\neq i\in \{1,\ldots,r\} $, $ |B_i|\notin \{1,2\} $. On the other hand, since $ |B|\leq 5 $ and $ r\geq 2 $,  there exists  $ i\in \{1,\ldots,r\} $ such that $ |B_i|\in \{1,2\}  $, a contradiction.   Thus,  we must have  $ r=1 $ and so there exists $ t\in G $ such that $ H \subseteq tK$. Now, since $ \mathsf{a}\mathsf{b}=1 $, there exists $ h\in B\cap C^{-1} $ which implies $ t\in K $ and therefore $ H \subseteq K$. By \cite[Lemma 2.12]{AT}, $ K\subseteq H $ and so $ K= H $.
\end{proof}
\begin{lem}\label{unit size}
Let  $ \mathbb{F} $ be an arbitrary field and  $G$  a torsion-free group. If $\mathsf{a}$ is a possible  unit in $\mathbb{F}[G]$ with $|supp(\mathsf{a})|=4$, then  $|S_\mathsf{a}|\in\{10,12\}$.
\end{lem}
\begin{proof}
Suppose that $ \mathsf{b} $ is a mate of $ \mathsf{a} $. Let $B=supp(\mathsf{a})$, $C=supp(\mathsf{b})$ and  $H=\left\langle h^{-1}B\right\rangle$, where $ h $ is an arbitrary element of $B$. Since $ (h^{-1}\mathsf{a})(\mathsf{b} h)=1 $, Lemma \ref{1} implies $(h^{-1}\mathsf{a}),(\mathsf{b} h)\in \mathbb{F}[H]$ and therefore  $ H $ is neither abelian \cite[Theorem 26.2]{PI} nor isomorphic to the Klein bottle group (see Remark \ref{klein}). Hence, the result follows from Lemma \ref{01}.
\end{proof}
\begin{defn}\label{definitions}
Suppose that $ \alpha , \beta \in \mathbb{F}[G] $, $ B=supp(\alpha) $ and $ C=supp(\beta) $. For each $ g\in C $ and $ i\in\{1,\ldots,|B|\} $,  let $ \Theta_i(g):=\{h\in B \, |\; r_{BC}(hg)=i\} $, $ \theta_i(g):=|\Theta_i(g)| $,  $ \Delta_{(\alpha ,\beta)}^{i}:=\{s\in BC \,|\,r_{BC}(s)=i\}  $ and $ \delta_{(\alpha ,\beta)}^{i}:=|\Delta_{(\alpha ,\beta)}^{i}| $. Clearly, $ 0\leq \theta_i(g)\leq |B| $,  $ \sum_{i=1}^{|B|}\theta_i(g)=|B| $, $ \sum_{i=1}^{|B|} \delta_{(\alpha ,\beta)}^{i}=|BC| $ and $ \sum_{i=1}^{|B|}i \delta_{(\alpha ,\beta)}^{i}=|B||C| $.
\end{defn} 
\section {\bf Zero divisor  and unit graphs}\label{section2}
In \cite{AT}, the zero divisor graph  is introduced   as follows:
\begin{defn}
For any pair of non-zero elements $(\alpha,\beta)$ of a group algebra over a field $ \mathbb{F} $ and a group $ G $ such that $\alpha\beta=0$, we assign a graph $ Z(\alpha,\beta) $ to $ (\alpha,\beta) $ called the zero-divisor graph of $ (\alpha,\beta) $ as follows:\\ the vertex set is $ supp(\beta) $, the edge set is 
\[
\big{\{}\{(h,h',g,g'),(h',h,g',g)\}\,|\;h,h'\in supp(\alpha),\;g,g'\in supp(\beta),\; g\neq g',\;hg=h'g'\big{\}},
\] 
and if $ \mathcal{E}_{(\alpha,\beta)}\neq \varnothing $, the function $ \psi_{Z(\alpha,\beta)}:\mathcal{E}_{Z(\alpha,\beta)}\rightarrow {\mathcal{V}}^{2}_{_{Z(\alpha,\beta)}} $ is defined by
\[
\psi_{Z(\alpha,\beta)}(\{(h,h',g,g'),(h',h,g',g)\})=\{g,g'\},
\]
for all $ \{(h,h',g,g'),(h',h,g',g)\}\in  \mathcal{E}_{Z(\alpha,\beta)}$.
\end{defn}
$ Z(\alpha,\beta) $ is called a zero divisor graph of length $ |supp(\alpha)| $ over the field $ \mathbb{F} $ and on the group $ G $.
\begin{lem}\label{degreeconn}
Let $ \alpha,\beta $ be   non-zero elements  in $\mathbb{F}[G]$ such that $ |supp(\alpha)|=n $   and $ \alpha\beta=0 $. Then every vertex of $ Z(\alpha,\beta)  $ has degree   $n,n+1,\ldots$ or $n(n-1)$. Moreover, if $ \mathbb{F}=\mathbb{F}_2 $, then if $ n $ is even or odd, then the degrees of all vertices  of $ Z(\alpha,\beta)  $ are even or odd, respectively. 
\end{lem}
\begin{proof}
 Let $B=supp(\alpha)$, $C=supp(\beta)$.  Since $ \alpha\beta=0 $, $ \sum_{(h,g)\in R_{BC}(x)}\alpha_h\beta_g=0 $  for all $ x\in BC $. Hence, we must have $ r_{BC}(x)\geq 2 $ and by Definition \ref{set},    $ r_{BC}(x)\in \{2,3,\ldots,n\} $  for all $ x\in BC $. Suppose that $g$ is a vertex of $ Z(\alpha,\beta) $. As discussed above  for each $ h\in B $, $ r_{BC}(hg)\in \{2,3,\ldots,n\}$ and therefore  $deg_{Z(\alpha,\beta)}(g)$ satisfies the following formula:
\begin{equation}\label{deg1}
deg_{Z(\alpha,\beta)}(g)=n+\sum_{i=3}^n (i-2)\theta_i(g),\; \text{where} \;\; 0\leq \sum_{3}^{n}\theta_i(g)\leq n .
\end{equation}
Hence, $ deg_{Z(\alpha,\beta)}(g)\in\{n,n+1,\ldots,n(n-1) \}$. If  $ \mathbb{F}=\mathbb{F}_2 $, then $ \alpha\beta=\sum_{ x\in BC }r_{BC}(x)x=0 $. Thus, $ r_{BC}(x) $ is an  even number  for all $ x\in BC $. Hence, if $ i $ is an odd number, then $ \theta_i(g)=0 $  for all $ g\in supp(\beta) $ and therefore \ref{deg1} implies that if  $ n $ is  even  or odd, then $ deg_{Z(\alpha,\beta)}(g) $ is  even or odd, respectively. This completes the proof.
\end{proof}
\begin{ex}
 Let $ G' $ be a group with an element $ a\in G' $ of order $ 10 $. If we let $ \alpha=1+a+a^{4}+a^{9} $ and  $ \beta=1+a^{2}+a^{4}+a^{6}+a^{8} $, then  $ \alpha,\beta\in \mathbb{F}_2[G'] $ and $ \alpha\beta=0=\beta\alpha $. $ Z(\alpha,\beta) $ and $ Z(\beta,\alpha) $ are shown in Figure {\rm\ref{ex}}. It can be seen that $ \beta $ is a mate of $ \alpha $ but as shown in Figure {\rm\ref{ex}}, since $ Z(\beta,\alpha) $ is not connected, Lemma $2.6$ of {\rm\cite{AT}} implies that $ \alpha $ is not a mate of $ \beta $. It can be seen that $ \alpha'=1+a^{2} $ is a mate of $ \beta $. 
\end{ex}
\begin{figure}[!htb]
\begin{minipage}{.4\textwidth}\centering
\begin{tikzpicture}[scale=.75]
\draw [fill] (0,.25) circle
[radius=0.1] node  [above]  {1};
\draw [fill] (-1,-.5) circle
[radius=0.1] node  [left]  {$ a^{2} $};
\draw [fill] (1,-.5) circle
[radius=0.1] node  [right]  {$ a^{8} $};
\draw [fill] (-1,-1.5) circle
[radius=0.1] node  [left]  {$ a^{4} $};
\draw [fill] (1,-1.5) circle
[radius=0.1] node  [right]  {$ a^{6} $};
\draw (0,.25) -- (1,-.5);
\draw (0,.25) -- (-1,-.5);
\draw (0,.25) -- (1,-1.5);
\draw (0,.25) -- (-1,-1.5);
\draw (1,-.5) -- (-1,-.5);
\draw (1,-.5) -- (1,-1.5);
\draw (1,-.5) -- (-1,-1.5);
\draw (-1,-.5) -- (1,-1.5);
\draw (-1,-.5) -- (-1,-1.5);
\draw (-1,-1.5) -- (1,-1.5);
\end{tikzpicture}
\caption*{$ Z(\alpha,\beta) $}
\end{minipage}
\begin{minipage}{.4\textwidth}\centering
\begin{tikzpicture}[scale=.85]
\draw [fill] (-1,-.4) circle
[radius=0.1] node  [above]  {$ 1 $};
\draw [fill] (1,-.4) circle
[radius=0.1] node  [above]  {$ a$};
\draw [fill] (-1,-1.5) circle
[radius=0.1] node  [below]  {$ a^{4} $};
\draw [fill] (1,-1.5) circle
[radius=0.1] node  [below]  {$a^{9}$};
\draw (-1,-.5) -- (-1,-1.5);
\draw (1,-1.5) -- (1,-.5);
\draw (1,-1.5) .. controls (1.4,-.9) and (1.2,-.87) .. (1,-.4);
\draw (1,-1.5) .. controls (1.55,-.9) and (1.35,-.85) .. (1,-.4);
\draw (1,-1.5) .. controls (.6,-.9) and (.8,-.9) .. (1,-.4);
\draw (1,-1.5) .. controls (.45,-.9) and (.6,-.9) .. (1,-.4);

\draw (-1,-.4) .. controls (-1.4,-.9) and (-1.2,-.87) .. (-1,-1.5);
\draw (-1,-.4) .. controls (-1.55,-.9) and (-1.35,-.85) .. (-1,-1.5);
\draw (-1,-.4) .. controls (-.6,-.9) and (-.8,-.9) .. (-1,-1.5);
\draw (-1,-.4) .. controls (-.45,-.9) and (-.6,-.85) .. (-1,-1.5);
\end{tikzpicture}
\caption*{$ Z(\beta,\alpha) $}
\end{minipage}
\caption{$ Z(\alpha,\beta) $ and $ Z(\beta,\alpha) $,  where $ \alpha $ and $ \beta $ are two elements of $ \mathbb{F}_2[G'] $ with $ supp(\alpha)=\{1,a,a^{4},a^{9}\} $ and  $ supp(\beta)=\{1,a^{2},a^{4},a^{6},a^{8}\} $.}\label{ex}
\end{figure}
In \cite{AT}, the unit graph  is introduced   as follows:
\begin{defn}
For any pair of  elements $(\mathsf{a},\mathsf{b})$ of a group algebra over a field $ \mathbb{F} $ and a group $ G $ such that $\mathsf{a}\mathsf{b}=1$, we assign a graph $ U(\mathsf{a},\mathsf{b}) $ to $ (\mathsf{a},\mathsf{b})$ called the unit graph of $ (\mathsf{a},\mathsf{b})$ as follows:\\ the vertex set is $ supp(\mathsf{b}) $, the edge set is 
\[
\big{\{}\{(h,h',g,g'),(h',h,g',g)\}\,|\;h,h'\in supp(\mathsf{a}),\;g,g'\in supp(\mathsf{b}),\; g\neq g',\;hg=h'g'\big{\}},
\] 
and if $ \mathcal{E}_{U(\mathsf{a},\mathsf{b})}\neq \varnothing $, the function $ \psi_{U(\mathsf{a},\mathsf{b})}:\mathcal{E}_{U(\mathsf{a},\mathsf{b})}\rightarrow \mathcal{V}^{2}_{_{U(\mathsf{a},\mathsf{b})}} $ is defined by
\[
\psi_{U(\mathsf{a},\mathsf{b})}(\{(h,h',g,g'),(h',h,g',g)\})=\{g,g'\},
\]
for all $ \{(h,h',g,g'),(h',h,g',g)\}\in  \mathcal{E}_{U(\mathsf{a},\mathsf{b})}$.
\end{defn}
 $ U(\mathsf{a},\mathsf{b}) $ is called a unit graph of length $ |supp(\mathsf{a})| $ over the field $ \mathbb{F} $ and on the group $ G $.
\begin{lem}\label{degreeconnunit}
Let $ \mathsf{a},\mathsf{b} \in \mathbb{F}[G]$ such that $ |supp(\mathsf{a})|=n $   and $ \mathsf{a}\mathsf{b}=1 $. Then every vertex of $ U(\mathsf{a},\mathsf{b})  $ has degree   $n-1,n,\ldots$ or $n(n-1)$. Moreover, the number of vertices of degree $ n-1 $ in $ U(\mathsf{a},\mathsf{b})  $ is at most $ 1 $.
\end{lem}
\begin{proof}
 Let $B=supp(\mathsf{a})$ and $C=supp(\mathsf{b})$. Since $ \mathsf{a}\mathsf{b}=1 $, we must have $ 1\in BC $, $ \sum_{(h,g)\in R_{BC}(1)}\alpha_h\beta_g=1 $ and  
$ \sum_{(h,g)\in R_{BC}(x)}\alpha_h\beta_g=0 $  for all $ x\in BC\setminus \{1\} $. Thus,  by Definition \ref{set},  $ r_{BC}(1)\in \{1,2,\ldots,n\} $ and $ r_{BC}(x)\in \{2,3,\ldots,n\} $  for all $ x\in BC\setminus \{1\} $. Suppose that $g$ is a vertex of $ U(\mathsf{a},\mathsf{b}) $. As discussed above, if $ g\notin C\cap B^{-1} $, then $ r_{BC}(hg)\in \{2,3,\ldots,n\} $  for all $ h\in B$. So, $  deg_{U(\mathsf{a},\mathsf{b})}(g)$ satisfies  the following formula:\begin{equation}\label{deg1unit}
deg_{U(\mathsf{a},\mathsf{b})}(g)=n+\sum_{i=3}^n (i-2)\theta_i(g), \; \text{where} \;\; 0\leq \sum_{3}^{n}\theta_i(g)\leq n .
\end{equation}
Hence, $ deg_{U(\mathsf{a},\mathsf{b})}(g)\in\{n,n+1,\ldots, n(n-1) \}$. Now, suppose that $ g\in C\cap B^{-1}$.
Consider two cases:  $ (1) $  $  C\cap B^{-1}\neq\{g\} $: Then $ r_{BC}(1)\geq 2 $ and therefore $  deg_{U(\mathsf{a},\mathsf{b})}(g)$ satisfies  \ref{deg1unit}; $ (2) $  $  C\cap B^{-1}=\{g\} $: Then $ r_{BC}(1)=1 $ and for each $ h\in B\setminus \{g^{-1}\} $, $ r_{BC}(hg)\geq 2 $. Hence, $ deg_{U(\mathsf{a},\mathsf{b})}(g)=n-1+\sum_{i=3}^{n} (i-2)\theta_i(g) $, where $ 0\leq \sum_{3}^{n}\theta_i(g)\leq n-1 $. So, $ deg_{U(\mathsf{a},\mathsf{b})}(g)\in\{n-1,\ldots,(n-1)^2 \}$. This completes the proof.
\end{proof}
\begin{rem}\label{complete}
Suppose that $ \Gamma $ is a zero divisor graph or a unit graph on a pair of elements $ (\alpha,\beta) $ in a group algebra. Let $B=supp(\alpha)$ and $C=supp(\beta)$.  For each $ s\in \Delta_{(\alpha,\beta)}^i$,  let  $ \mathcal{V}_{(\alpha,\beta)}(s):=\{g\in C\,| \,(sg^{-1},g)\in R_{BC}(s)\} $. Clearly, the induced subgraph on $ \mathcal{V}_{(\alpha,\beta)}(s) $ in $ \Gamma $ is a complete graph.
\end{rem}
\begin{lem}\label{3part}
Let $ \alpha ,\beta \in   \mathbb{F}[G]\setminus \{0\} $ such that $ \alpha\beta=0 $. Then the following statements hold: 
\begin{itemize}
\item[(1)]Suppose that $ s\in \Delta_{(\alpha ,\beta)}^{i}  $, where $ 3\leq i\leq |supp(\alpha)| $. Then for each $ g\in \mathcal{V}_{(\alpha,\beta)}(s) $, $deg_{Z(\alpha ,\beta)}(g)> |supp(\alpha)|$.
\item[(2)]Suppose that $ s\in \Delta_{(\alpha ,\beta)}^{i}  $ and $ s'\in \Delta_{(\alpha ,\beta)}^{j}  $ such that $ s\neq s' $ and $ 2\leq i,j\leq |supp(\alpha)| $. Then each pair of distinct elements in $ \mathcal{V}_{(\alpha,\beta)}(s)\cap \mathcal{V}_{(\alpha,\beta)}(s') $ are adjacent by  more than  one edge in $ Z(\alpha,\beta) $.
\item[(3)] For each $ g\in supp(\beta) $ and $ i\in \{1,\ldots,|supp(\alpha)|\}$, $ |\{s\in \Delta_{(\alpha ,\beta)}^{i}\, |\,g\in  \mathcal{V}_{(\alpha,\beta)}(s)\}|=\theta_i(g) $.
\end{itemize}
\end{lem}
\begin{proof}
Clearly, for each $ g\in \mathcal{V}_{(\alpha,\beta)}(s) $, $ sg^{-1}\in \Theta_i(g) $ and therefore \ref{deg1} implies $deg_{Z(\alpha ,\beta)}(g)> |supp(\alpha)|$. Suppose that there are distinct elements $g_i,g_j\in\mathcal{V}_{(\alpha,\beta)}(s)\cap\mathcal{V}_{(\alpha,\beta)}(s') $.  $ s\neq s' $ implies that  \[ \{(sg_{i}^{-1},sg_{j}^{-1},g_{i} ,g_{j}),(sg_{j}^{-1},sg_{i}^{-1},g_{j} ,g_{i})\} \,\, \text{and }\, \{(s'g_{i}^{-1},s'g_{j}^{-1},g_{i} ,g_{j}),(s'g_{j}^{-1},s'g_{i}^{-1},g_{j} ,g_{i})\} \] are distinct elements of $ \mathcal{E}_{Z(\alpha,\beta)} $. This completes the proof of part $ (2)$. The proof of part $ (3) $ is clear by the definition of the set $ \Theta_i(g) $.
\end{proof}
\begin{lem}\label{3partunit}
Let $ \mathsf{a},\mathsf{b} \in \mathbb{F}[G] $ such that $ \mathsf{a}\mathsf{b}=1 $. Then the following statements hold: 
\begin{itemize}
\item[(1)]Suppose that $ s\in \Delta_{(\mathsf{a},\mathsf{b})}^{i}  $ and $ s'\in \Delta_{(\mathsf{a},\mathsf{b})}^{j}  $ such that $ s\neq s' $ and $ 2\leq i,j\leq |supp(\mathsf{a})| $. Then each pair of distinct elements in $ \mathcal{V}_{(\mathsf{a},\mathsf{b})}(s)\cap \mathcal{V}_{(\mathsf{a},\mathsf{b})}(s') $ are adjacent by  more than  one edge in $ U(\mathsf{a},\mathsf{b}) $.
\item[(2)] For each $ g\in supp(\mathsf{b}) $ and $ i\in \{1,\ldots,|supp(\mathsf{a})|\}$, $ |\{s\in \Delta_{(\mathsf{a},\mathsf{b})}^{i}\, |\,g\in  \mathcal{V}_{(\mathsf{a},\mathsf{b})}(s)\}|=\theta_i(g) $.
\end{itemize}
\end{lem}
\begin{proof}
The proof is similar to the proof of parts (2) and (3) of Lemma \ref{3part}.
\end{proof}
In the sequel, we  study the zero divisor graph and the unit graph of lengths $ 4 $ over an arbitrary field  and on any torsion-free group. If $ \alpha $ is a  zero divisor or a unit in $  \mathbb{F}[G] $ for a possible  torsion-free group $ G $ and arbitrary field $ \mathbb{F}$ with  $ |supp(\alpha)|=4 $, then by Lemmas \ref{z} and \ref{unit size},  $ | S_{\alpha} | \in\{ 10,12\} $. In Sections $  4$ and $ 5 $ we  study these graphs in the cases  $ | S_{\alpha} |=12 $ and   $ | S_{\alpha} |=10 $, respectively.
\begin{rem}\label{iden}
 Suppose that $  \alpha  $ is a  zero divisor in $ \mathbb{F}[G] $ for a possible torsion-free group $ G $  and arbitrary field $ \mathbb{F}$ with  $ |supp(\alpha)|=4 $ and $ \beta $ is a mate of $ \alpha $ . By Lemma $ 2.4 $ of  \cite{AT} and Lemma {\rm\ref{1}}, we may assume that  $ 1\in supp(\alpha)\cap supp(\beta) $ and $ G=\left\langle supp(\alpha)\right\rangle=\left\langle supp(\beta)\right\rangle $.
\end{rem}
 \begin{rem}\label{idenunit}
 Suppose that $  \mathsf{a}  $ is a  unit in $ \mathbb{F}[G] $ for a possible torsion-free group $ G $  and arbitrary field $ \mathbb{F}$ with  $ |supp(\mathsf{a})|=4 $ and $ \mathsf{b} $ is a mate of $ \mathsf{a} $. Since $ \mathsf{a}\mathsf{b}=1 $,  there exists $ h\in supp(\mathsf{a})\cap (supp(\mathsf{b}))^{-1} $. By Lemma $ 2.10 $ of  \cite{AT} and  since $ (h^{-1}\mathsf{a})(\mathsf{b} h)=1 $, we may assume that $ 1\in supp(\mathsf{a})\cap supp(\mathsf{b}) $. Also,  according to Lemma {\rm\ref{hg}}, we may assume that $ G=\left\langle supp(\mathsf{a})\right\rangle=\left\langle supp(\mathsf{b})\right\rangle $.
\end{rem}
\vspace{1cc}
\section {\bf Zero divisor  and  unit graphs of lengths $ 4 $ over an arbitrary field and on any torsion-free group, where  $ |S_\alpha|=12 $ and $ |S_{\mathsf{a}}|=12 $}
The following remark follows from Lemma \ref{degreeconn} and \cite[Lemma 2.6]{AT}.
\begin{rem}\label{deg12}
Let $ \alpha $ be a  zero divisor  in $\mathbb{F}[G]$ for a possible torsion-free group $ G $ and arbitrary field $\mathbb{F}$ with $|supp(\alpha)|=4$ and $ \beta $ be a mate of $ \alpha $. Then $  Z(\alpha,\beta) $ is  connected and also for each $ g\in supp(\beta) $, $ deg_{Z(\alpha,\beta)}(g)=4+\theta_3(g)+2\theta_4(g) $, where $ 0\leq \theta_3(g)+\theta_4(g)\leq 4 $. Moreover, if   $\mathbb{F}=\mathbb{F}_2 $, then   $ deg_{Z(\alpha,\beta)}(g)=4+2\theta_4(g) $, where $ 0\leq \theta_4(g)\leq 4 $.
\end{rem} 
\begin{thm}\label{2}
Let $ \alpha $ be a  zero divisor  in $\mathbb{F}[G]$ for a possible torsion-free group $ G $ and arbitrary field $\mathbb{F}$ with $|supp(\alpha)|=4,$  $ |S_\alpha|=12 $ and $ \beta $ be a non-zero element  of $ \mathbb{F}[G] $ such that $ \alpha\beta=0 $. Then $ Z(\alpha,\beta) $  is  the induced subgraph of the Cayley graph of $G$  with respect to  $ S_{\alpha}$ on the set $supp(\beta)$.
\end{thm}
\begin{proof}
 Let $B=supp(\alpha)$ and $C=supp(\beta)$.
 Suppose that there are two edges between distinct vertices $ g $ and $ g' $ of $ Z(\alpha,\beta) $. Hence, there exist  $ h_1,h_2,h'_1,h'_2\in B $ such that $ \{(h_{1},h_{2},g ,g'),(h_{2},h_{1},g' ,g)\} $ and $ \{(h'_{1},h'_{2},g ,g'),(h'_{2},h'_{1},g' ,g)\} $  are distinct elements of $\mathcal{E}_{Z(\alpha,\beta)} $.  Clearly, $ h_1\neq h'_1 $, $ h_2\neq h'_2 $ and  $ h_{1}^{-1}{h'}_{1}=gg'^{-1}=h_{2}^{-1}{h'}_{2} $ which imply $ |S_\alpha|<12 $, a contradiction. 
\end{proof}
 The following corollary follows from  Theorem \ref{2} and Lemma \ref{3part}.
\begin{cor}\label{cap}
Let $ \alpha $ be a  zero divisor  in $\mathbb{F}[G]$ for a possible torsion-free group $ G $ and arbitrary field $\mathbb{F}$ with $|supp(\alpha)|=4,$  $ |S_\alpha|=12 $ and $ \beta $ be a non-zero element  of $ \mathbb{F}[G] $ such that $ \alpha\beta=0 $. If $ s\in \Delta_{(\alpha ,\beta)}^{i}  $ and $ s'\in \Delta_{\alpha\beta}^{j}  $ such that $ s\neq s' $ and $  i,j\in \{2,3,4\}, $ then $ |\mathcal{V}_{(\alpha,\beta)}(s)\cap \mathcal{V}_{(\alpha,\beta)}(s')|\leq 1 $.
\end{cor}
\begin{lem}\label{5regular}
Let $ \alpha $ be a  zero divisor  in $\mathbb{F}[G]$ for a possible torsion-free group $ G $ and arbitrary field $\mathbb{F}$ with $|supp(\alpha)|=4,$  $ |S_\alpha|=12 $ and $ \beta $ be a non-zero element  of $ \mathbb{F}[G] $ such that $ \alpha\beta=0 $. If the maximum degree of $ Z(\alpha,\beta) $ is $ 5 $, then the number of vertices of degree $ 5 $ in  $ Z(\alpha,\beta) $ is a multiple of $ 6 $.
\end{lem}
\begin{proof}
Since the maximum degree of $ Z(\alpha,\beta) $ is $ 5 $,  Remark \ref{deg12} implies $ \Delta_{(\alpha ,\beta)}^{3} \neq \varnothing $ and $\Delta_{(\alpha ,\beta)}^{4} =\varnothing $. According to hypothesis and part $ (1) $ of Lemma \ref{3part}, for each $ s \in \Delta_{(\alpha ,\beta)}^{3}  $, there exist $ 3 $ vertices of degree $ 5 $ in $ Z(\alpha,\beta) $. On the other hand, according to  part $ (3) $ of Lemma \ref{3part}, if $ s $ and $ s' $ are distinct elements of $ \Delta_{(\alpha ,\beta)}^{3}  $, then $ \mathcal{V}_{(\alpha,\beta)}(s)\cap\mathcal{V}_{(\alpha,\beta)}(s')=\varnothing $. Hence, the number of vertices of degree $ 5 $ in  $ Z(\alpha,\beta) $ is a multiple of $ 3 $. The result follows from this fact that the number of vertices with odd degree in each graph is even.
\end{proof}
Suppose that $ \Gamma $ is a zero divisor graph or a unit graph on a pair of elements $ (\alpha,\beta) $ in $\mathbb{F}[G]$ for a possible torsion-free group $ G $ and arbitrary field $\mathbb{F}$ such that $ |supp(\alpha)|=4 $.
To verify if a given graph $ \Gamma_0 $ can not be appeared as a subgraph of $ \Gamma $ one may do as follows: Let $ m $ be the number of edges of $ \Gamma_0  $. Then for each edge $ e $ of $ \Gamma_0  $, there is a relation  as $ h_eg_e=h'_eg'_e $, where $ g_e, g'_e\in supp(\beta) $ are the endpoints of $ e $ and some distinct $ h_e, h'_e\in supp(\alpha)  $. Therefore, for each edge  of $ \Gamma_0  $, there are 12 possible relations as above according to how one chooses a pair $(h_e,h'_e)$ of distinct elements of $supp(\alpha)$. Hence there are $ 12^m $ systems of relations as follows:
\begin{equation}
Rel_{C}(\Gamma_0 ): \left\{
\begin{array}{ll}
h_{e} g_{e}=h'_{e}g'_{e}  \\ \text{for all} \; e\in \mathcal{E}_{\Gamma_0 }
\end{array} \right.,
\end{equation}
where $C$ runs through the set of all $12^m$ choices of the pairs $(h_e,h'_e)$.
Now fix one of the  systems of relations $Rel_{C}(\Gamma_0 )$. Take an arbitrary edge $e_0$ of $\Gamma_0 $ and consider the following new relations:

\begin{equation}
Rel'_{C}(\Gamma_0 ): \left\{
\begin{array}{ll}
h_{e_0}^{-1}h_{e} g_{e}=h_{e_0}^{-1}h'_{e}g'_{e}  \\ \text{for all} \; e\in \mathcal{E}_{\Gamma_0 }
\end{array} \right..
\end{equation}
Now consider  the group $\mathcal{G}$ with presentation $\left\langle  h_{e_0}^{-1}supp(\alpha)\,|\,Rel'_{C}(\Gamma_0 )\right\rangle   $. According to Lemma \ref{1}, $  h_{e_0}^{-1}\alpha$ and $ \beta g^{-1}  $, where  $ g\in supp(\beta) $, are non-zero elements of $\mathcal{G}$ and $ h_{e_0}^{-1}\alpha\beta g^{-1}=0 $. Therefore, 
\begin{itemize}
\item  if $\mathcal{G}$ is  abelian, then we get a contradiction e.g. by Theorem 26.2 of \cite{PI},
\item if $\mathcal{G}$ is a quotient  group  of $ BS(1,n)$\footnote{For integers $m$ and $n$, the Baumslag-Solitar group $BS(m, n)$ is the group given by the
presentation $ \left\langle a,b\,|\,ba^{m}b^{-1}=a^n\right\rangle  $.}, then we get a contradiction e.g. by Remark  3.17  of \cite{AT},  
\item if $\mathcal{G}$ has a non-trivial torsion element, then since $ \mathcal{G} $ is a torsion-free group we get a contradiction.
\end{itemize} 
 Of course, it is clear that the number of such systems is high even for a graph with a small number of edges. By using GAP \cite{a9} and some of the techniques described in Appendixs \ref{app12} and \ref{app10}, we have reduced the number of such systems.   
\begin{thm}\label{55}
Let $ \alpha $ be a  zero divisor  in $\mathbb{F}[G]$ for a possible torsion-free group $ G $ and arbitrary field $\mathbb{F}$ with $|supp(\alpha)|=4,$  $ |S_\alpha|=12 $ and $ \beta $ be a mate of $ \alpha $. Then $ Z(\alpha,\beta) $  contains no subgraph isomorphic to one of the graphs in Figure {\rm \ref{606}}, where each vertex of the subgraph has degree $ 4 $  in $ Z(\alpha,\beta) $, and the graph in Figure {\rm \ref{6066}}, where the degree of white vertices of the subgraph in $ Z(\alpha ,\beta) $ must be $5$.
\end{thm}
We refer the reader to Appendix  \ref{app12} to see details of the proof of Theorem \ref{55}. 
\begin{thm}\label{4regularlem}
Let $ \alpha $ be a zero divisor  in $\mathbb{F}_2[G]$ for a possible torsion-free group $ G $ with $|supp(\alpha)|=4$ and $ \beta $ be a mate of $ \alpha $. If $|S_\alpha|=12 $, then $Z(\alpha,\beta)$ contains a $4$-regular subgraph.
\end{thm}
\begin{proof}
Suppose that $ | supp(\beta) | = n$. If $ \Delta_{(\alpha,\beta)}^{4}=\varnothing $, then Remark \ref{deg12} implies $ deg_{Z(\alpha,\beta)}(g)=4$ for all $ g\in supp(\beta) $ and there is nothing to prove. Now, suppose that $ \Delta_{(\alpha,\beta)}^{4}\neq\varnothing $. Let $ B=\{K[s]\;|\;s\in \Delta_{(\alpha,\beta)}^{4}\}  $ and  $ B_g=\{K[s]\in B\;|\; g\in \mathcal{V}_{(\alpha,\beta)}(s)\} $  for all $ g\in supp(\beta) $. By part $ (3) $ of Lemma \ref{3part},  $ |B_g|=\theta_{4}(g) $ and by Corollary \ref{cap}, if $  K[s]$ and  $K[s']$ are distinct elements of  $ B_g $, then  $  \mathcal{V}_{(\alpha,\beta)}(s)\cap  \mathcal{V}_{(\alpha,\beta)}(s')=\{g\}  $.
Now, let $ H $ be a subgraph of $  Z(\alpha,\beta) $   obtained by deleting edges of   $\mathcal{E}_{Z(\alpha,\beta)} $ as follows: For each $K[s]\in B $, remove all edges of $K[s] $ except two non-adjacent edges. Note that $ H $ is not unique and as above, for every $K[s] $ by deleting edges from $  Z(\alpha,\beta) $, the degree of each vertex $ g\in \mathcal{V}_{(\alpha,\beta)}(s) $ is reduced   by $ 2 $. According to  above description, for each $g\in supp(\beta)$, $ deg_H(g)=deg_{Z(\alpha,\beta)}(g)-2\theta_4(g)  $. Thus,  Remark \ref{deg12} implies that $ H $ is isomorphic to a $4$-regular graph.  This completes the proof.
 \end{proof}
 The following corollary follows from Theorem \ref{4regularlem}.
\begin{cor}\label{4regular}
 If any $4$-regular graph with $ n $ vertices is a forbidden subgraph of any zero divisor graph of length $4$ over $\mathbb{F}_2$ and on the torsion-free group $G$. Then $ | supp(\beta) | \neq n$.
\end{cor}
The following remark follows  from Lemma \ref{degreeconnunit} and Lemma $ 2.11 $ of \cite{AT}: 
 \begin{rem}\label{deg12unit}
Let  $ \mathsf{a} $ be a unit  in $\mathbb{F}[G]$ for a possible torsion-free group $ G $ and arbitrary field  $\mathbb{F}$ with $|supp(\mathsf{a})|=4 $ and $ \mathsf{b} $ be a mate of $\mathsf{a} $. Then $  U(\mathsf{a},\mathsf{b}) $ is  connected  and every vertex of $ U(\mathsf{a},\mathsf{b})  $ has degree   $3,4,\ldots$ or $12$. Also, the number of vertices of degree  $ 3 $ in $ U(\mathsf{a},\mathsf{b})  $ is  at most $  1$.
\end{rem} 
\begin{thm}\label{2unit}
Let  $ \mathsf{a} $ be a unit  in $\mathbb{F}[G]$ for a possible torsion-free group $ G $ and arbitrary field  $\mathbb{F}$ with $|supp(\mathsf{a})|=4 $, $|S_\mathsf{a}|=12 $ and $ \mathsf{b} $ be an element  of $ \mathbb{F}[G] $ such that $ \mathsf{a}\mathsf{b}=1 $. Then $ U(\mathsf{a},\mathsf{b}) $  is  the induced subgraph of the Cayley graph of $G$  with respect to  $ S_{\mathsf{a}}$ on the set $supp(\mathsf{b})$.
\end{thm}
\begin{proof}
 The proof is similar to the proof of Theorem \ref{2}.
\end{proof}
The following corollary follows from  Theorem \ref{2unit} and Lemma \ref{3partunit}.
\begin{cor}\label{capunit}
Suppose that $ \mathsf{a} $ is a  unit  in $\mathbb{F}[G]$ for a possible torsion-free group $ G $ and arbitrary field  $\mathbb{F}$ with $|supp(\mathsf{a})|=4,\; |S_\mathsf{a}|=12 $ and $ \mathsf{b} $ is an element  of $ \mathbb{F}[G] $ such that $ \mathsf{a}\mathsf{b}=1 $. If $ s\in \Delta_{(\mathsf{a},\mathsf{b})}^{i}  $ and $ s'\in \Delta_{(\mathsf{a},\mathsf{b})}^{j}  $ such that $ s\neq s' $ and $  i,j\in \{2,3,4\}, $ then $ |\mathcal{V}_{(\mathsf{a},\mathsf{b})}(s)\cap \mathcal{V}_{(\mathsf{a},\mathsf{b})}(s')|\leq 1 $.
\end{cor}
\begin{thm}\label{55unit}
Let $ \mathsf{a}$ be a unit  in $\mathbb{F}[G]$ for a possible torsion-free group $ G $ and arbitrary field  $\mathbb{F}$ such that $|supp(\mathsf{a})|=4, |S_\mathsf{a}|=12 $ and $ \mathsf{b} $ be an element  of $ \mathbb{F}[G] $ such that $ \mathsf{a}\mathsf{b}=1 $. Then $ U(\mathsf{a},\mathsf{b}) $  contains no subgraph isomorphic to one of the graphs in Figure {\rm\ref{606}}, where each vertex of the subgraph has degree $ 4 $ in $ U(\mathsf{a},\mathsf{b}) $.
\end{thm}
In Appendix \ref{app12}, we give  details of the proof of Theorem  \ref{55unit}.
\section {\bf Zero divisor  and  unit graphs of lengths $ 4 $ over an arbitrary field and on any torsion-free group, where  $ |S_\alpha|=10 $ and $ |S_{\mathsf{a}}|=10 $}\label{s5}
\begin{lem}\label{83}
Suppose that $ \alpha $  is a  zero divisor in  $\mathbb{F}[G]$ for a possible torsion-free group $ G $ and arbitrary field $\mathbb{F}$ with $ |supp(\alpha)|=4 , |S_\alpha|=10 $ and $ \beta $ is a mate of $ \alpha $. Then one of the following cases occurs:
\begin{itemize}
\item[$  (i)$]There exists  $ \alpha' \in \mathbb{F}[G]$ with $supp(\alpha') =\{1,x,x^{-1},y\} $, where $ x,y $ are distinct non-trivial elements of  $ G$,   such that    $ Z(\alpha,\beta)\cong Z(\alpha',\beta)$.
\item[$  (ii)$] There exists   $ \alpha' \in \mathbb{F}[G]$ with $supp(\alpha') =\{1,x,y,xy\}$, where $ x,y $  are distinct non-trivial elements of  $ G$, such that   $ Z(\alpha,\beta)\cong Z(\alpha',\beta)$.
\end{itemize}  Furthermore, if $\mathbb{F}=\mathbb{F}_{2}$, then the case $(i)$  occurs.
\end{lem}
\begin{proof}
 Since $ G  $ is neither abelian \cite[Theorem 26.2]{PI} nor isomorphic to the Klein bottle  group (see   Remark \ref{klein}), in view of Remark \ref{iden} and  Lemma \ref{01}, we may assume that $supp(\alpha)$ is one of the sets  $\{1,x,x^{-1},y\}$, $ \{1,x,x^{2},y\}$, $ \{1,x,xy^{-1}x,y \}$ or $ \{1,x,xy,y\} $, where $ x,y $ are distinct non-trivial elements of $ G $. By   \cite[Lemma  2.4 ]{AT},  if $ supp(\alpha) $ is  one of the sets $\{1,x,x^{2},y\}  $ or $ \{1,x,xy^{-1}x,y\} $, then it is sufficient that  replacing $ \alpha $ by $ x^{-1}\alpha $.  Suppose that $\mathbb{F}=\mathbb{F}_{2}$. If $ supp(\alpha) =\{1,x,xy,y\} $, then $\alpha\beta=(1+x)(1+y)\beta=0 $ contradicting   \cite[Theorem 2.1]{PS}. This completes the proof.
\end{proof}
\begin{lem}\label{83unit}
Suppose that $ \mathsf{a} $  is a  unit in  $\mathbb{F}[G]$ for a possible torsion-free group $ G $ and arbitrary field $\mathbb{F}$ with $ |supp(\mathsf{a})|=4 , |S_\mathsf{a}|=10 $ and $ \mathsf{b} $ is a mate of $ \mathsf{a} $. Then one of the following cases occurs:
\begin{itemize}
\item[$  (i)$]There exist  $ \mathsf{a}',\mathsf{b}' \in \mathbb{F}[G]$  such that $supp(\mathsf{a}') =\{1,x,x^{-1},y\} $,  where $ x,y $  are distinct non-trivial elements of  $ G$,     and   $ U(\mathsf{a},\mathsf{b})\cong U(\mathsf{a}',\mathsf{b}')$.
\item[$  (ii)$] There exist   $ \mathsf{a}',\mathsf{b}'  \in \mathbb{F}[G]$  such that $supp(\mathsf{a}') =\{1,x,y,xy\}$, where $ x,y $  are distinct non-trivial elements of  $ G$,   and   $ U(\mathsf{a},\mathsf{b})\cong U(\mathsf{a}',\mathsf{b}')$.
\end{itemize} 
\end{lem}
\begin{proof}
 Since $ G  $ is neither abelian \cite[Theorem 26.2]{PI} nor isomorphic to the Klein bottle  group (see   Remark \ref{klein}), in  view of Remark \ref{idenunit} and  Lemma \ref{01}, we may assume that $supp(\mathsf{a})$ is one of the sets  $\{1,x,x^{-1},y\}$, $ \{1,x,x^{2},y\}$, $ \{1,x,xy^{-1}x,y \}$ or $ \{1,x,xy,y\} $, where $ x,y $ are distinct non-trivial elements of $ G $. By   \cite[Lemma  2.10 ]{AT}, if $ supp(\mathsf{a}) $ is  one of the sets $\{1,x,x^{2},y\}  $ or $ \{1,x,xy^{-1}x,y\} $, then it is sufficient that  replacing $ \mathsf{a} $ and $ \mathsf{b}$ by $ x^{-1}\mathsf{a} $ and $ \mathsf{b} x $, respectively.   This completes the proof.
\end{proof}
 \begin{rem}
 Suppose that $ \Gamma $ is a graph. For each  $ e \in \mathcal{E}_{\Gamma} $, denote  the set of its endpoints by $ E(e) $  and for each  $ v\in \mathcal{V}_{\Gamma} $, let $ M_{\Gamma}(v):=\big{\{}\{e,e'\}\subseteq \mathcal{E}_{\Gamma}\,|\,e\neq e',\, E(e)=E(e'),\,v\in E(e)\cap E(e')\big{\}} $.
 \end{rem} 
Let $ \Gamma $ be a zero divisor graph or a unit graph on a pair of elements $ (\alpha,\beta) $ in a group algebra of a possible torsion-free group $ G $ and arbitrary field $\mathbb{F}$ such that  $|supp(\alpha)|=4$, $ |S_\alpha|=10 $ and $ \beta $ be a mate of $ \alpha $. Let $ g $ and  $ g' $ be two vertices of   $ \Gamma $ which are adjacent by more than  one edge and $A=\big{\{}e\in \mathcal{E}_{\Gamma}\,|\,E(e)=\{g,g'\}\big{\}}$. Hence, for each  $ e_i\in A $, there exist distinct elements  $ h_i,h'_i\in supp(\alpha) $ such that $ \{(h_i,h'_i,g,g'),(h'_i,h_i,g',g)\}\in \mathcal{E}_{\Gamma}$ and $ gg'^{-1}={h_i}^{-1}h'_i $. Obviously,   if $ e_i,e_j $ are distinct elements of $ A $, then $ (h_i,h'_i)\neq (h_j,h'_j) $ and  
 ${h_i}^{-1}h'_i={h_j}^{-1}h'_j\in S_{\alpha}$. \\ By Lemmas \ref{83} and  \ref{83unit}, we may assume that $supp(\alpha) $ is one of the sets  $\{1,x,x^{-1},y\} $ or $ \{1,x,y,xy\} $, where $ x,y $  are distinct non-trivial elements of  $ G$. Firstly, suppose that  $supp(\alpha) =\{1,x,x^{-1},y\} $. Thus, it follows from Table \ref{1010} that $ |A|=2 $ and if $ A=\{e_1,e_2\} $, then  $$ \{(h_1,h'_1),(h_2,{h'_2}^{-1})\}\in\big{\{}\{(1,x),(x^{-1},1)\},\{(1,x^{-1}),(x,1)\}\big{\}} .$$
 Now, suppose that $supp(\alpha) =\{1,x,y,xy\} $.  Then it follows from Table \ref{1010} that $ |A|=2 $ and if $ A=\{e_1,e_2\} $, then $ \{(h_1,h'_1),(h_2,{h'_2}^{-1})\}\in\big{\{}\{(1,y),(x,xy)\},\{(xy,x),(y,1)\}\big{\}} $. 
 
  According to  above argument, we have the following remarks:
\begin{rem}\label{multi}
 Let $ \alpha $  be a  zero divisor in  $\mathbb{F}[G]$ for a possible torsion-free group $ G $  and arbitrary field $\mathbb{F}$ with $ |supp(\alpha)|=4 , |S_\alpha|=10 $ and $ \beta $ be a mate of $ \alpha $. Then $  Z(\alpha,\beta) $ is a multigraph with the following properties:
 \begin{itemize}
 \item[$(i)$] If $ g,g' $ are distinct vertices of $  Z(\alpha,\beta) $, then $ \big{|\{}e\in \mathcal{E}_{Z(\alpha,\beta)}\,|\,E(e)=\{g,g'\}\big{\}|}\leq 2 $.
 \item[$(ii)$] For each  vertex $ g $ of $  Z(\alpha,\beta) $, $ |M_{Z(\alpha,\beta)}(g)|\leq 2  $.
 \end{itemize} 
\end{rem} 
\begin{rem}\label{multiunit}
 Let $ \mathsf{a} $  be a   unit in  $\mathbb{F}[G]$ for a possible torsion-free group $ G $  and arbitrary field $\mathbb{F}$ with $ |supp(\mathsf{a})|=4 , |S_\mathsf{a}|=10 $ and $ \mathsf{b} $ be a mate of $ \mathsf{a} $. Then $U(\mathsf{a},\mathsf{b}) $ is a multigraph with the following properties:
 \begin{itemize}
 \item[$(i)$] If $ g,g' $ are distinct vertices of $  U(\mathsf{a},\mathsf{b}) $, then $ \big{|\{}e\in \mathcal{E}_{U(\mathsf{a},\mathsf{b})}\,|\,E(e)=\{g,g'\}\big{\}|}\leq 2 $.
 \item[$(ii)$] For each  vertex $ g $ of $  U(\mathsf{a},\mathsf{b}) $, $  |M_{U(\mathsf{a},\mathsf{b})}(g)|\leq 2  $.
 \end{itemize} 
\end{rem} 
 \begin{thm}\label{lemtriangle}
 Let $ \alpha $  be a   zero divisor in  $\mathbb{F}[G]$ for a possible torsion-free group $ G $  and arbitrary field $\mathbb{F}$ with $ |supp(\alpha)|=4 , |S_\alpha|=10 $ and $ \beta $ be a mate of $ \alpha $. Then for each integer number $ n\geq 3 $, $  Z(\alpha,\beta) $  contains no subgraph  isomorphic to the graph in Figure {\rm\ref{multicycle}}.
 \end{thm}
\begin{proof}
Suppose that $ Z(\alpha,\beta) $ contains a cycle as  Figure \ref{multicycle}.
By Lemma \ref{83}, we may assume that $ supp(\alpha)$ is one of the sets  $\{1,x,x^{-1},y\} $ or $\{1,x,xy,y\} $, where $ x,y $ are distinct non-trivial elements of $ G$. Firstly, suppose that $ supp(\alpha)=\{1,x,x^{-1},y\}$. According to above discussion, since $ g_1 $ is adjacent to each of the vertices $ g_2 $ and $ g_n $ by two edges,   by renumbering, we may assume that $ g_1=xg_2 $ and $ xg_1=g_n $. Then we must have $ g_2=xg_3 ,\ldots, g_{n-1}=xg_n $. Hence, $ g_1{g_2}^{-1}g_2{g_3}^{-1}\cdots g_{n-1}{g_n}^{-1}g_n{g_1}^{-1}=x^n=1 $ contradicting $ G $ is a torsion-free group.
 Similarly, if $ supp(\alpha)=\{1,x,y,xy\}$, then without loss of generality we may assume that $ g_1=yg_2 $, $ g_2=yg_3, \ldots ,g_{n-1}=yg_n  $ and $ g_n=yg_1 $. So, $ g_1{g_2}^{-1}g_2{g_3}^{-1}\cdots g_{n-1}{g_n}^{-1}g_n{g_1}^{-1}=y^n=1 $ contradicting $ G $ is a torsion-free group. This completes the proof.
\end{proof}
The following corollary follows from Lemma   \ref{3part} and Theorem \ref{lemtriangle}.
\begin{cor}\label{cap1}
 Let $ \alpha $  be a   zero divisor in  $\mathbb{F}[G]$ for a possible torsion-free group $ G $  and arbitrary field $\mathbb{F}$ with $ |supp(\alpha)|=4 , |S_\alpha|=10 $ and $ \beta $ be a mate of $ \alpha $. If $ s\in \Delta_{(\alpha ,\beta)}^{i}  $ and $ s'\in \Delta_{(\alpha ,\beta)}^{j}  $ such that $ s\neq s' $ and $  i,j\in \{3,4\}, $ then $ |\mathcal{V}_{(\alpha,\beta)}(s)\cap \mathcal{V}_{(\alpha,\beta)}(s')|\leq 2 $.
\end{cor}
\begin{thm}\label{lemtriangleunit}
 Let $ \mathsf{a} $  be a   unit in  $\mathbb{F}[G]$ for a possible torsion-free group $ G $  and arbitrary field $\mathbb{F}$ with $ |supp(\mathsf{a})|=4 , |S_\mathsf{a}|=10 $ and $ \mathsf{b} $ be a mate of $ \mathsf{a} $. Then for each integer number $ n\geq 3 $, $ U(\mathsf{a},\mathsf{b}) $  contains no subgraph  isomorphic to the graph in Figure {\rm\ref{multicycle}}.
 \end{thm}
\begin{proof}
The proof is similar to the proof of Theorem \ref{lemtriangle}. 
\end{proof}
The following corollary follows from Lemma   \ref{3partunit} and Theorem \ref{lemtriangleunit}.
\begin{cor}\label{cap1unit}
 Let $ \mathsf{a} $  be a   unit in  $\mathbb{F}[G]$ for a possible torsion-free group $ G $  and arbitrary field $\mathbb{F}$ with $ |supp(\mathsf{a})|=4 , |S_\mathsf{a}|=10 $ and $ \mathsf{b} $ be a mate of $ \mathsf{a} $. If $ s\in \Delta_{(\mathsf{a},\mathsf{b})}^{i}  $ and $ s'\in \Delta_{(\mathsf{a},\mathsf{b})}^{j}$ such that $ s\neq s' $ and $  i,j\in \{3,4\}, $ then $ |\mathcal{V}_{(\mathsf{a},\mathsf{b})}(s)\cap \mathcal{V}_{(\mathsf{a},\mathsf{b})}(s')|\leq 2 $.
\end{cor}
The following simple lemma is useful.
\begin{lem}\label{repetition}
Let $ \alpha $  be a  zero divisor in  $\mathbb{F}[G]$ for a possible torsion-free group $ G $  and arbitrary field $\mathbb{F}$ with $supp(\alpha) =\{1,x,x^{-1},y\} $, where $ x,y $  are distinct non-trivial elements of  $ G$, $ \beta $ be a mate of $ \alpha $, $ B= supp(\alpha)$ and $ C= supp(\beta)$. If  $ g $ is a vertex of $  Z(\alpha,\beta) $ such that $ |M_{Z(\alpha,\beta)}(g)|=2 $, then $ r_{BC}(g)\in \{3,4\} $ and $ \{g,xg,x^{-1}g\}\subseteq \mathcal{V}_{(\alpha,\beta)}(g) $. Moreover,  if $ \mathbb{F}=\mathbb{F}_2 $, then $ |M_{Z(\alpha,\beta)}(g)|=2 $ if and only if  $ r_{BC}(g)=4$.
\end{lem}
\begin{proof}
 Since  $ |M_{Z(\alpha,\beta)}(g)|=2 $, there exist distinct elements  $ g',g''\in C\setminus \{g\} $  such that $ g=xg' $ and $ g=x^{-1}g'' $. Therefore, $ \{(1,g),(x,x^{-1}g),(x^{-1},xg)\}\subseteq R_{BC}(g) $ and so  Definition \ref{set} leads to $ r_{BC}(g)\in \{3,4\} $. If $ \mathbb{F}=\mathbb{F}_2 $, then  $ \alpha\beta=\sum_{x\in BC}r_{BC}(x)x=0 $ which implies  $ r_{BC}(x)\in \{2,4\} $  for all $ x\in BC $. Hence, it follows from the first part of  lemma  that if $ |M_{Z(\alpha,\beta)}(g)|=2 $, then $ r_{BC}(g)=4$. Now, suppose that $ r_{BC}(g)=4$. Then there exist pairwise distinct elements $ g',g'' $ and $ \bar{g} $ of $ C\setminus \{g\} $ such that $ R_{BC}(g)=\{(1,g),(x,g'),(x^{-1},g''),(y,\bar{g})\} $  and therefore $ g $ is adjacent to each of the vertices $ g' $ and $ g'' $ by two edges. Hence, $ |M_{Z(\alpha,\beta)}(g)|=2 $. This completes the proof.
\end{proof}
\begin{rem}\label{deg4f}
Let $ \alpha $  be a  zero divisor in  $\mathbb{F}[G]$ for a possible torsion-free group $ G $  and arbitrary field $\mathbb{F}$ with $supp(\alpha) =\{1,x,x^{-1},y\} $, where $ x,y $  are distinct non-trivial elements of  $ G$, and $ \beta $ be a mate of $ \alpha $.  In view of  Remark {\rm\ref{deg12}} and Lemma {\rm\ref{repetition}}, if $ g $  is a vertex of degree $ 4 $ in $ Z(\alpha,\beta) $, then  two cases hold: $ (i)\;\;  |M_{Z(\alpha,\beta)}(g)|=0 $; $ (ii)\;\;  |M_{Z(\alpha,\beta)}(g)|=1$. More precisely, we shall speak of a vertex of degree $ 4 $ of type $ (j) $ if the vertex satisfies in the condition $ (j) $ as above {\rm(}$ j $ being $ i $ or $ ii ${\rm)}. 
\end{rem}
\begin{thm}\label{app}
Let $ \alpha $  be a  zero divisor in  $\mathbb{F}[G]$ for a possible torsion-free group $ G $  and arbitrary field $\mathbb{F}$ with $supp(\alpha) =\{1,x,x^{-1},y\} $, where $ x,y $  are distinct non-trivial elements of  $ G$, and $ \beta $ be a mate of $ \alpha $. Then $ Z(\alpha,\beta) $  contains no  subgraph isomorphic to one of  the graphs  in Figures {\rm\ref{36}} and {\rm\ref{tt}}, where the degrees of the vertex $ g $, white and  gray vertices of any subgraph in $Z(\alpha,\beta) $ must be $ 5 $, $ 4 $ and   $ 4 $ of type $ (ii) $, respectively.
\end{thm}
We refer the reader to Appendix  \ref{app10} to see details of the proof of Theorem \ref{app}.
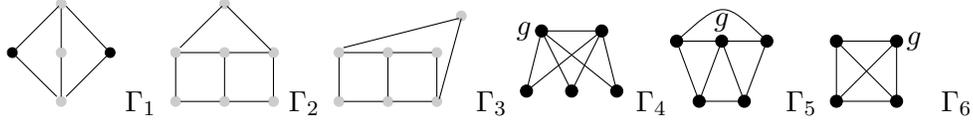
\begin{figure}
\begin{tikzpicture}[scale=.65]
\draw [fill] (0,0) circle
[radius=0.1] node  [left]  {};
\draw [fill] (2,0) circle
[radius=0.1] node  [left]  {};
\draw [fill][black!20!white] (1,1) circle
[radius=0.1] node  [left]  {};
\draw[fill][black!20!white] (1,-1) circle
[radius=0.1] node  [left]  {};
\draw [fill][black!20!white] (1,0) circle
[radius=0.1] node  [left]  {};
\draw  (.05,.05) -- (.95,.95);
\draw (1.1,.95) -- (1.95,.05) ;
\draw (1.95,-.05)-- (1.1,-.9);
\draw(.91,-.9)-- (.05,-.05);
\draw  (1,.9) --  (1,.1) ;
\draw  (1,-.1) --  (1,-.9) ;
\end{tikzpicture}
\subfloat{$ \Gamma_1$}
\begin{tikzpicture}[scale=.65]
\draw [fill][black!20!white] (0,0) circle
[radius=0.1] node  [left]  {};
\draw [fill][black!20!white] (1,0) circle
[radius=0.1] node  [left]  {};
\draw [fill][black!20!white] (0,-1) circle
[radius=0.1] node  [left]  {};
\draw [fill][black!20!white] (1,-1) circle
[radius=0.1] node  [left]  {};
\draw [fill][black!20!white] (1,1) circle
[radius=0.1] node  [left]  {};
\draw [fill][black!20!white] (2,0) circle
[radius=0.1] node  [left]  {};
\draw [fill][black!20!white] (2,-1) circle
[radius=0.1] node  [left]  {};
\draw  (.09,.1) --  (.95,.95) ;
\draw  (1.05,.95) --  (1.95,.05) ;
\draw  (.1,0) --  (.9,0) ;
\draw  (1,-.9) --  (1,-.1) ;
\draw  (.9,-1) --  (.1,-1) ;
\draw  (0,-.1) --  (0,-.9) ;
\draw  (1.1,0) --  (1.9,0) ;
\draw  (2,-.1) --  (2,-.9) ;
\draw  (1.9,-1) --  (1.1,-1) ;
\end{tikzpicture}
\subfloat{$ \Gamma_2$}
\begin{tikzpicture}[scale=.65]
\draw [fill][black!20!white] (0,0) circle
[radius=0.1] node  [left]  {};
\draw [fill][black!20!white] (1,0) circle
[radius=0.1] node  [left]  {};
\draw [fill][black!20!white] (0,-1) circle
[radius=0.1] node  [left]  {};
\draw [fill][black!20!white] (1,-1) circle
[radius=0.1] node  [left]  {};
\draw [fill][black!20!white] (2.5,.75) circle
[radius=0.1] node  [left]  {};
\draw [fill][black!20!white] (2,0) circle
[radius=0.1] node  [left]  {};
\draw [fill][black!20!white] (2,-1) circle
[radius=0.1] node  [left]  {};
\draw  (.09,.1) --  (2.45,.7) ;
\draw  (2.5,.7) --  (2.05,-.95) ;
\draw  (.1,0) --  (.9,0) ;
\draw  (1,-.9) --  (1,-.1) ;
\draw  (.9,-1) --  (.1,-1) ;
\draw  (0,-.1) --  (0,-.9) ;
\draw  (1.1,0) --  (1.9,0) ;
\draw  (2,-.1) --  (2,-.9) ;
\draw  (1.9,-1) --  (1.1,-1) ;
\end{tikzpicture}
\subfloat{$\Gamma_{3}$}
\begin{tikzpicture}[scale=.8]
\draw [fill] (0,0) circle
[radius=0.1] node  [left]  {$ g $};
\draw [fill] (1,0) circle
[radius=0.1] node  [right]  {};
\draw [fill] (-.25,-1) circle
[radius=0.1] node  [below]  {};
\draw [fill] (.5,-1) circle
[radius=0.1] node  [below]  {};
\draw [fill] (1.25,-1) circle
[radius=0.1] node  [below]  {};
\draw (0,0) -- (1,0) -- (-.25,-1) -- (0,0);
\draw  (1,0) -- (.5,-1) -- (0,0);
\draw  (1,0) -- (1.25,-1) -- (0,0);
\end{tikzpicture}
\subfloat{$\Gamma_{4}$}
\begin{tikzpicture}[scale=.8]
\draw [fill] (0,0) circle
[radius=0.1] node  [left]  {};
\draw [fill] (.75,0) circle
[radius=0.1] node  [right]  {};
\draw [fill] (-.375,1) circle
[radius=0.1] node  [below]  {};
\draw [fill] (.375,1) circle
[radius=0.1] node  [above]  {$ g $};
\draw [fill] (1.125,1) circle
[radius=0.1] node  [below]  {};
\draw (.375,1) -- (.75,0);
\draw (.375,1) -- (-.375,1);
\draw (.375,1) -- (-.375,1);
\draw (.375,1) -- (0,0);
\draw (.375,1) -- (1.125,1);
\draw (.75,0) -- (0,0);
\draw (-.375,1) -- (0,0);
\draw (.75,0) -- (1.125,1);
\draw (-.375,1) .. controls (.4,1.7)  .. (1.125,1);
\end{tikzpicture}
\subfloat{$\Gamma_{5}$}
\begin{tikzpicture}[scale=.8]
\draw [fill] (0,0) circle
[radius=0.1] node  [left]  {};
\draw [fill] (1,0) circle
[radius=0.1] node  [right]  {};
\draw [fill] (0,1) circle
[radius=0.1] node  [below]  {};
\draw [fill] (1,1) circle
[radius=0.1] node  [right]  {$ g $};
\draw (0,0) -- (1,0) -- (1,1) -- (0,1)-- (0,0);
\draw (0,0) -- (1,1);
\draw (1,0) -- (0,1);
\end{tikzpicture}
\subfloat{$\Gamma_{6}$}
\caption{ Some forbidden subgraphs of $ Z(\alpha,\beta) $, where $ supp(\alpha)=\{1,x,x^{-1},y\}$ and the degrees of the vertex $ g $ and  gray vertices of any subgrph in $Z(\alpha,\beta) $ must be $ 5 $ and   $ 4 $ of type $ (ii) $, respectively.}\label{36}
\end{figure}
\begin{thm}\label{appunit}
Let $ \mathsf{a} $ be a  unit in $ \mathbb{F}[G] $ for a possible torsion-free group $ G $ and arbitrary field $\mathbb{F}$ with $ supp(\mathsf{a})=\{1,x,x^{-1},y\}$, where $ x,y $ are distinct non-trivial elements of $ G $,  and $ \mathsf{b} $ be a mate of $ \mathsf{a} $. Then $ U(\mathsf{a},\mathsf{b}) $  contains no  subgraph isomorphic to one of  the graphs  in Figure {\rm\ref{36unit}}, where  the degrees of  gray  and white vertices of any subgraph in $U(\mathsf{a},\mathsf{b}) $ must be  $ 3$ and $ 4 $, respectively. 
\end{thm}
In Appendix \ref{app10}, we give  details of the proof of Theorem \ref{appunit}.
 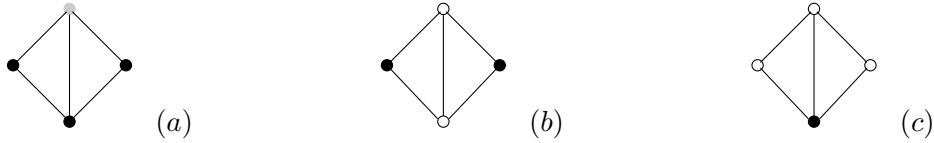
\begin{figure}[htp]
\begin{tikzpicture}[scale=.75]
\draw [fill] (0,0) circle
[radius=0.1] node  [left]  {};
\draw [fill] (2,0) circle
[radius=0.1] node  [right]  {};
\draw [fill][black!20!white] (1,1) circle
[radius=0.1] node  [left]  {};
\draw [fill] (1,-1) circle
[radius=0.1] node  [left]  {};
\draw  (0,0) -- (.95,.95);
\draw (1.05,.95) -- (2,0) -- (1,-1)-- (0,0);
\draw  (1,.9) --  (1,-1) ;
\end{tikzpicture}
\subfloat{$ (a)$}
\qquad \qquad \qquad
\begin{tikzpicture}[scale=.75]
\draw [fill] (0,0) circle
[radius=0.1] node  [left]  {};
\draw [fill] (2,0) circle
[radius=0.1] node  [right]  {};
\draw [] (1,1) circle
[radius=0.1] node  [left]  {};
\draw [] (1,-1) circle
[radius=0.1] node  [left]  {};
\draw  (0,0) -- (.95,.95);
\draw (1.05,.95) -- (2,0) -- (1.1,-.95);
\draw (.9,-.95)-- (0,0);
\draw  (1,.9) --  (1,-.9) ;
\end{tikzpicture}
\subfloat{$(b)$}
\qquad \qquad \qquad
\begin{tikzpicture}[scale=.75]
\draw [] (0,0) circle
[radius=0.1] node  [left]  {};
\draw [] (2,0) circle
[radius=0.1] node  [right]  {};
\draw [] (1,1) circle
[radius=0.1] node  [left]  {};
\draw [fill] (1,-1) circle
[radius=0.1] node  [left]  {};
\draw  (.05,.05) -- (.95,.95);
\draw(1.94,-.05) -- (1.1,-.95);
\draw (1.05,.95) -- (1.94,.05);
\draw (.9,-.95)-- (.05,-.05);
\draw  (1,.9) --  (1,-.9) ;
\end{tikzpicture}
\subfloat{$(c)$}
\caption{Three forbidden subgraphs of $ U(\mathsf{a},\mathsf{b}) $, where $ supp(\mathsf{a})=\{1,x,x^{-1},y\}$ and the degrees of  white and gray vertices of any subgraph in $U(\mathsf{a},\mathsf{b}) $ must be   $ 3 $ and $ 4 $, respectively.}\label{36unit}
\end{figure}
\begin{rem}\label{repetitionunit}
Suppose that $ \mathsf{a} $ is a  unit in $ \mathbb{F}[G] $ for a possible torsion-free group $ G $ and arbitrary field $\mathbb{F}$ with  $ supp(\mathsf{a})=\{1,x,x^{-1},y\}$, where $ x,y $ are distinct non-trivial elements of $ G$, and  $ \mathsf{b} $ is a mate of $  \mathsf{a} $. Let $ B= supp(\mathsf{a})$, $ C= supp(\mathsf{b})$ and $ g $ be a vertex of $ U(\mathsf{a},\mathsf{b}) $. By the same argument as Lemma {\rm\ref{repetition}}, it can be  seen that if $ |M_{U(\mathsf{a},\mathsf{b})}(g)|=2 $, then $ r_{BC}(g)\in \{3,4\} $ and $ \{g,xg,x^{-1}g\}\subseteq \mathcal{V}_{(\mathsf{a},\mathsf{b})}(g) $.
\end{rem}
\section{ \textbf{Zero divisors whose supports are of size $ 4 $ in  group algebras of  torsion-free groups} }
Throughout this section let $ \alpha $ be a zero divisor
in $ \mathbb{F}[G]  $ for a possible torsion-free group $ G $ and arbitrary field $ \mathbb{F} $ with  $|supp(\alpha)|= 4 $ and $ \beta $ be a mate of $ \alpha $. It is known that $ |supp(\beta)|> 2 $ \cite[Theorem 2.1]{PS}. In  this section we will show that $ |supp(\beta)|\geq 7 $.
  
Let $G$ be a group and $B,C$ be two finite subsets of $G$. Then $ |BC|\geq |B|+|C|-1 $ \cite{BF}. The extremal sets in the inequality are characterized  in the following theorem:
\begin{thm}[\cite{set2}] \label{set2}
Let $G$ be a group and $B,C$ be two  subsets of $G$ such that $min\{ |B|, |C|\}\geq 2 $ and $ |BC|= |B|+|C|-1 $. Then the subsets $ B $ and $ C $ have the form $ B=\{b,bq,\ldots,bq^{l-1}\}$ and $ C=\{c,qc,\ldots,q^{k-1}c\}$, where $ b,c,q\in G $ and $ q\neq 1 $.
\end{thm}
\begin{thm}[The main result of \cite{a7}] \label{4}
 Let $ C $ be a finite generating subset of a nonabelian torsion-free group
$ G $ such that $ 1\in C $. Then for all $ B\subset G $ with $  |B|\geq 4 $,
$ | BC|\geq |B|+|C|+1$. 
\end{thm}
  Let $B=supp(\alpha)$, $C=supp(\beta)$. In view of Remark \ref{iden}, we may assume that $ C $ is a generating set of  $G$ and  $ 1\in B\cap C $. Since  $G$ is not abelian \cite[Theorem 26.2]{PI}, Theorem \ref{4} implies  $ | BC|\geq | B | + | C |+1$.\\
Since $ \alpha\beta=0 $, we must have $ \sum_{(h,g)\in R_{BC}(x)}\alpha_h\beta_g=0 $  for all $ x\in BC$. Hence, $ r_{BC}(x)\geq 2 $ and by Definition \ref{set},  $ r_{BC}(x)\in \{2,3,4\} $  for all $ x\in BC$. 
 So, $| BC|\leq 2|C|$ because if there is an integer number $ i\geq 1 $ such that  $|BC|\leq 2|C|-i$, then by Definition \ref{definitions}, $\delta_{(\alpha,\beta)}^{3}+2\delta_{(\alpha,\beta)}^{4}=-2i $, a contradiction. Thus, we have the following:
\begin{rem}\label{fieldf1}
Let   $ \alpha $ be a zero divisor
in $ \mathbb{F}[G]  $ for a possible torsion-free group $ G $ and arbitrary field $ \mathbb{F} $ with  $|supp(\alpha)|= 4 $ and $ \beta $ be a mate of $ \alpha $. Then $|supp(\beta)|+5\leq |supp(\alpha)supp(\beta)|\leq 2|supp(\beta)| $. Also, according to Definition {\rm\ref{definitions}}, $ \delta_{(\alpha,\beta)}^{2}+\delta_{(\alpha,\beta)}^{3}+\delta_{(\alpha,\beta)}^{4}=|supp(\alpha)supp(\beta)| $ and $2\,\delta_{(\alpha,\beta)}^{2}+3\delta_{(\alpha,\beta)}^{3}+4\delta_{(\alpha,\beta)}^{4}=4|supp(\beta)| $.
\end{rem}
\begin{lem}\label{fieldf}
Let  $ \alpha $ be a zero divisor
in $ \mathbb{F}[G]  $ for a possible torsion-free group $ G $  and arbitrary field $ \mathbb{F} $ with  $|supp(\alpha)|= 4 $ and $ \beta $ be a mate of $ \alpha $. Then $ |supp(\beta)|\geq 6 $.
\end{lem}
\begin{proof}
By Remark \ref{fieldf1}, it is clear that $|C|\geq 5$. Suppose that  $ |C| = 5 $, then  Remark \ref{fieldf1} implies $ |BC|=10 $ and $ r_{BC}(x)= 2 $ for each $ x\in BC $. Hence, if we let   $ \alpha'=\sum_{h\in B}h $ and $ \beta'=\sum_{g\in C}g $, then $ \alpha',\beta' \in \mathbb{F}_2 [G] $ and $ \alpha'\beta'=0 $ contradicting   \cite[Theorem 1.3 ]{PS}. This completes the proof.
\end{proof}
\begin{thm}\label{endxy}
Let  $ \alpha $ be a  zero divisor in $\mathbb{F}[G]$ for a possible torsion-free group $ G $  and arbitrary field $ \mathbb{F} $ with  $ supp(\alpha)= \{1,x,y,xy\} $, where $ x,y $ are distinct non-trivial elements of $ G$, and $ \beta $ be a mate of $ \alpha $. Then $  |supp(\alpha)supp(\beta)|\leq 2|supp(\beta)|-2$. 
\end{thm}
\begin{proof}
 Let $B=supp(\alpha)$ and $C=supp(\beta)$. Since   $ {\alpha_1}^{-1}\alpha\beta =0 $, where $\alpha_1\in \mathbb{F}\setminus\{0\} $ ,  we may assume that $ \alpha= 1+\lambda x+\gamma y+\delta xy$, where $ \gamma ,\lambda,\delta\in \mathbb{F}\setminus\{0\} $. Note that if $ \delta=\lambda\gamma $, then $ \alpha\beta=(1+\lambda x)(1+\gamma y)\beta=0 $ contradicting  \cite[Theorem 2.1]{PS}. Hence, $ \mathbb{F}\neq \mathbb{F}_2 $. It is clear that 
$BC=\{1,x\}\{1,y\}C$.  Let $ H=C\cup yC $ and $ |C\cap yC|=t $.  We will show that $ t\geq 4 $. As discussed in the beginning of this section, $ r_{BC}(a)\in \{2,3,4\} $  for all $ a\in BC $. Firstly, we show that for each $ a\in  C\cap y C $, 
\begin{equation}\label{maximum}
max\{r_{BC}(a),r_{BC}(xa)\}\geq 3.
\end{equation}
Suppose that $ r_{BC}(a)=2$, $r_{BC}(xa)=2 $ and $ R_{BC}(a)=\{(1,g), (y,g')\} $  which imply $ R_{BC}(xa)=\{(x,g), (xy,g')\} $. Then $ \beta_g+\gamma \beta_{g'}=0 $ and $ \lambda \beta_g+ \delta \beta_{g'}=0$ which imply $ \delta=\lambda\gamma $ that is a contradiction. Hence, \ref{maximum}   holds. Now suppose that $ a\in (H\cup xH) \setminus (H\cap xH) $. Since $ r_{BC}(a)\geq 2 $, if $ a\in H $, then  $ a\in C\cap yC $ and  if  $ a\in xH $, then $ x^{-1}a\in C\cap yC $. Therefore, we can define a map $ f:(H\cup xH) \setminus (H\cap xH)\longrightarrow C\cap y C $ such that for each $ a\in (H\cup xH) \setminus (H\cap xH) $, 
\[
f(a)=\left\{
\begin{array}{ll}
a &  a\in H\\
x^{-1}a &  a\in xH\\
\end{array} \right. .
\]
Let us show that $ f $ is injective. Indeed, $ f(a)=f(a') $ and $ a\neq a' $ would imply either $ a\in H $ and $ a'\in xH $ or $ a'\in H $ and $ a\in xH $. Without loss of generality, let $ f(a)=a $ and $ f(a')=x^{-1}a' $. Since $ a,a'\notin H\cap xH $, $ r_{BC}(a)=2$ and  $r_{BC}(a')=2 $ leads to $r_{BC}(xa)=2 $ contradicting \ref{maximum}. Then if we let $ s=|(H\cup xH)\setminus(H\cap xH)\}| $, then  $ 0\leq s\leq t  $. On the other hand, Theorem \ref{set2} implies:
\begin{equation}\label{tm}
|BC|=|\{1,x\}H|\geq 2+|H|. 
\end{equation} 
 Now, since $ |BC|=|H|+\frac{s}{2} $,  by \ref{tm},  $ s\geq 4 $ which leads to $ t\geq 4 $.\\
 Clearly, if $ \Delta_{(\alpha,\beta)}^{3}=\varnothing $, then  for each $ a\in BC $, $ r_{BC}(a)\in\{2,4\}$ and thus if we let $ \alpha'=\sum_{b\in B}b $ and $ \beta'=\sum_{c\in C}c $, then $ \alpha',\beta' \in \mathbb{F}_2[G]$  and $ \alpha'\beta'=0 $ that is a contradiction. On the other hand, \ref{maximum} implies that if  $ \Delta_{\alpha\beta}^{4}=\varnothing $, then $ \delta_{(\alpha,\beta)}^{3}\geq t $ and therefore $ \delta_{(\alpha,\beta)}^{3}\geq 4 $. By Remark \ref{fieldf1}, $ |BC|\leq 2|C| $ and  if $ |BC|= 2|C| $ or $ |BC|= 2|C|-1 $, then either $ \Delta_{(\alpha,\beta)}^{3}=\varnothing $ or $ \delta_{(\alpha,\beta)}^{3}=2$ and  $ \Delta_{(\alpha,\beta)}^{4}=\varnothing $  that are contradictions. Then $ |BC|\leq 2|C|-2 $.  This completes the proof.
\end{proof}
\subsection{ The size of  $supp(\beta) $ must be at least $ 7 $}
By Lemma \ref{z}, $ |S_\alpha|\in\{10,12\} $.
At first we show that if $ |S_\alpha|=12 $, then  $ |supp(\beta)|\neq 6 $ . Suppose that $ |S_\alpha|=12 $ and $ |supp(\beta)|= 6 $. So, Remark \ref{deg12} and  Lemma \ref{2} imply that $ Z(\alpha,\beta) $ is a simple and connected graph with $ 6 $ vertices of degree  $ 4 $ or $ 5$. According to Lemma \ref{5regular},  $ Z(\alpha,\beta) $ is isomorphic to a $ 4 $-regular graph or a $ 5 $-regular graph. It can be seen that  the number of non-isomorphic $ 4 $-regular graphs with $ 6 $ vertices is $ 1 $.  This graph contains the graph $ K_{1,2,2} $ in  Figure \ref{606} as a subgraph. Also, the number of non isomorphic $ 5 $-regular graphs with $ 6 $ vertices is $ 1 $.  This graph contains the graph in Figure \ref{6066} as a subgraph. Therefore, 
according to above discussion and Lemma \ref{fieldf} and Theorem \ref{55}, we have the following theorem:
\begin{thm}\label{set vertex1}
Let $ \alpha$ be a  zero divisor  in $\mathbb{F}[G]$ for a possible torsion-free group $ G $ and arbitrary field $ \mathbb{F} $ with $|supp(\alpha)|=4, |S_\alpha|=12 $ and $ \beta $ be a mate of $ \alpha $. Then the size of the vertex set of $ Z(\alpha,\beta) $ is at least $ 7 $. 
\end{thm}
The following corollary  immediately  follows from Theorem \ref{set vertex1}.
\begin{cor}\label{end12}
Let $ \alpha $ be a  zero divisor  in $\mathbb{F}[G]$ for a possible torsion-free group $ G $ and arbitrary field $ \mathbb{F} $ with $ | supp(\alpha) | = 4, | S_{\alpha} |=12$ and $ \beta $ be a mate of $ \alpha $. Then $ | supp(\beta) |\geq 7 $. 
\end{cor}
Now, suppose that $ |S_\alpha|=10 $. By Lemma \ref{83}, we may assume that  $ supp(\alpha)$ is one of the sets  $\{1,x,y,xy\} $ or  $ \{1,x,x^{-1},y\} $, where $ x,y $ are distinct non-trivial elements of $ G$. Combining Remark \ref{fieldf1} and Theorem \ref{endxy}, we obtain the following corollary.
\begin{cor}\label{endf10xy}
Let $ \alpha $ be a  zero divisor  in $\mathbb{F}[G]$ for a possible torsion-free group $ G $  with $ supp(\alpha)=\{1,x,y,xy\} $, where $ x,y $ are distinct non-trivial elements of $ G$, and $ \beta $ be a mate of $ \alpha $. Then $ |supp(\beta)|\geq 7 $.
\end{cor}
Now, suppose that $ supp(\alpha)=\{1,x,x^{-1},y\} $ and $ |supp(\beta)|= 6 $. Remark \ref{fieldf1} implies that $ |BC|\in\{11 , 12\} $ and also if $ |BC|=12 $, then $ \delta_{(\alpha,\beta)}^{2}=12$, $ \delta_{(\alpha,\beta)}^{3}=0$, $ \delta_{(\alpha,\beta)}^{4}=0$ and if  $ |BC|=11 $, then either $ \delta_{(\alpha,\beta)}^{2}=10$, $ \delta_{(\alpha,\beta)}^{3}=0$, $ \delta_{(\alpha,\beta)}^{4}=1$ or $ \delta_{(\alpha,\beta)}^{2}=9$, $ \delta_{(\alpha,\beta)}^{3}=2$, $ \delta_{(\alpha,\beta)}^{4}=0$. As discussed before, if $ \Delta_{(\alpha,\beta)}^{3}=\varnothing $, then  for each $ a\in BC $, $ r_{BC}(a)\in\{2,4\}$ and thus if we let $ \alpha'=\sum_{b\in B}b $ and $ \beta'=\sum_{c\in C}c $, then $ \alpha',\beta' \in \mathbb{F}_2[G]$  and $ \alpha'\beta'=0 $ contradicting  \cite[Theorem 1.3]{PS}. Hence, it is enough  consider the case that $ \delta_{(\alpha,\beta)}^{2}=9$, $ \delta_{(\alpha,\beta)}^{3}=2$ and $ \delta_{(\alpha,\beta)}^{4}=0$. Suppose that  $ \Delta_{(\alpha,\beta)}^{3}=\{s,s'\} $.  By Corollary \ref{cap1}, $ |\mathcal{V}_{(\alpha,\beta)}(s)\cap\mathcal{V}_{(\alpha,\beta)}(s')|\leq 2 $. Then one of the following cases holds:
\begin{itemize}
\item[(i)] $ |\mathcal{V}_{(\alpha,\beta)}(s)\cap\mathcal{V}_{(\alpha,\beta)}(s')|=0 $: Hence, $ C=\mathcal{V}_{(\alpha,\beta)}(s)\cup\mathcal{V}_{(\alpha,\beta)}(s') $. Part $ (1) $ of Lemma  \ref{3part} implies $ deg_{Z(\alpha,\beta) }(g)=5 $  for all $ g\in C $. Clearly, there are at least two distinct elements $ g_1,g_2 \in \mathcal{V}_{(\alpha,\beta)}(s) $ such that $ s\notin \{g_1,g_2\} $. So, Lemma \ref{repetition} implies that $ |M_{Z(\alpha,\beta)}(g_{1})|\leq 1 $ and  $ |M_{Z(\alpha,\beta)}(g_{2})|\leq 1 $. Thus, there exist $ g,g'\in \mathcal{V}_{(\alpha,\beta)}(s') $ and $ \bar{g},\tilde{g} \in \mathcal{V}_{(\alpha,\beta)}(s') $  such that $ g_{1}\sim g $, $ g_{1}\sim g' $, $ g_{2}\sim \bar{g} $ and $ g_{2}\sim \tilde{g} $. Clearly, $ |\{g,g'\}\cap\{\bar{g},\tilde{g}\}|\geq 1 $. It is easy to see that if $ \{g,g'\}=\{\bar{g},\tilde{g}\} $, then $ Z(\alpha,\beta) $ includes both graphs $ \Gamma_{4} $ and $ \Gamma_{6} $  in Figure \ref{36} and if $ |\{g,g'\}\cap\{\bar{g},\tilde{g}\}|=1 $, then $ Z(\alpha,\beta) $ includes the graph $ \Gamma_{5} $  in Figure \ref{36} as a subgraph which give contradictions.
\item[(ii)]$ |\mathcal{V}_{(\alpha,\beta)}(s)\cap\mathcal{V}_{(\alpha,\beta)}(s')|=1 $: Suppose that  $\mathcal{V}_{(\alpha,\beta)}(s)\cap\mathcal{V}_{(\alpha,\beta)}(s')= \{\bar{g}\} $ and  $ C \setminus(\mathcal{V}_{(\alpha,\beta)}(s)\cup\mathcal{V}_{(\alpha,\beta)}(s'))=\{g\} $. By Remark \ref{deg12}, $ deg_{Z(\alpha,\beta) }(\bar{g})=6 $, $ deg_{Z(\alpha,\beta) }(g)=4 $ and $ deg_{Z(\alpha,\beta) }(g')=5$  for all $ g'\in (\mathcal{V}_{(\alpha,\beta)}(s)\cup\mathcal{V}_{(\alpha,\beta)}(s'))\setminus \{\bar{g}\} $. It follows from Remark \ref{deg4f} that $ g $ is adjacent to at least three distinct vertices 
 in $Z(\alpha,\beta)$. Consider two cases: $(1)$ $\bar{g}\sim g  $: Clearly in this case $Z(\alpha,\beta) $ contains the first graph of Figure \ref{tt} as a subgraph, a contradiction; $(2)$ $\bar{g}\nsim g  $: So $ |M_{Z(\alpha,\beta) }(\bar{g})|=2 $. Thus, by  Lemma \ref{repetition}, $\bar{g}\in\{ s,s'\} $. Without loss of generality, we may assume that $ \bar{g}=s $. So, Lemma \ref{repetition} implies that $ |M_{Z(\alpha,\beta) }(g_{1})|\leq 1 $ and $ |M_{Z(\alpha,\beta) }(g_{2})|\leq 1  $, where $ \{g_1,g_2\}=\mathcal{V}_{(\alpha,\beta)}(s)\setminus \{\bar{g}\} $. Therefore, for each $ i\in\{1,2\} $ there are two distinct vertices in $  \{g'_{1},g'_{2},g\} $  such that $ g_{i} $ is adjacent to each of them, where $ \{g'_1,g'_2\}=\mathcal{V}_{(\alpha,\beta)}(s')\setminus \{\bar{g}\} $. If there is $ i\in \{1,2\} $ such that $ g_{i}\sim g'_{1} $ and $ g_{i}\sim g'_{2} $, then clearly $ Z(\alpha,\beta) $ contains the graph $ \Gamma_{4} $ in Figure \ref{36} as a subgraph, a contradiction. Hence, $ g_{1}\sim g $ and $ g_{2}\sim g $ and therefore $ Z(\alpha,\beta) $ contains the first graph of Figure \ref{tt} as a subgraph, a contradiction.
\item[(iii)]$ |\mathcal{V}_{(\alpha,\beta)}(s)\cap\mathcal{V}_{(\alpha,\beta)}(s')|=2 $: Suppose that  $\mathcal{V}_{(\alpha,\beta)}(s)\cap\mathcal{V}_{(\alpha,\beta)}(s')=\{ g_{1},g_{2}\} $ and  $ C \setminus(\mathcal{V}_{(\alpha,\beta)}(s)\cup\mathcal{V}_{(\alpha,\beta)}(s'))=\{g,g'\} $.  By Remark \ref{deg12}, $ g $ and $ g' $ are vertices of degree $ 4 $ in $ Z(\alpha,\beta) $ and if $ \mathcal{V}_{(\alpha,\beta)}(s)\setminus \{g_1,g_2\}=\{g'_1\} $ and $ \mathcal{V}_{(\alpha,\beta)}(s')\setminus \{g_1,g_2\}=\{g'_2\} $, then   $ deg_{ Z(\alpha,\beta)}(g'_{1})=5 $ and $ deg_{ Z(\alpha,\beta)}(g'_{2})=5 $. Also, according to part $ (2) $ of Lemma \ref{3part} and our discussion in Section $ 5 $, $\{ s,s'\}=\{g_{1},g_{2}\} $ and therefore 
 Lemma \ref{repetition} implies $ |M_{ Z(\alpha,\beta)}(g'_{1})|\leq 1 $ and $ |M_{ Z(\alpha,\beta)}(g'_{2})|\leq 1  $. Thus,   there are two distinct vertices in $  \{g'_{2},g,g'\} $  such that $ g'_{1} $ is adjacent to each of them and   there are two distinct vertices in $  \{g'_{1},g,g'\} $  such that $ g'_{2} $ is adjacent to each of them. If $ g'_{1}\sim g'_{2}$, then clearly $ Z(\alpha,\beta) $ contains the graph $ \Gamma_{6} $ in Figure \ref{36} as a subgraph, a contradiction. Thus, $ g'_{1}\sim g$, $ g'_{1}\sim g'$, $ g'_{2}\sim g$ and  $ g'_{2}\sim g'$. Therefore,  $ Z(\alpha,\beta) $ contains the second graph of Figure \ref{tt} as a subgraph, a contradiction.
\end{itemize}
According to above argument, we have the following theorem:
\begin{thm}\label{F2vertex}
Let $ \alpha $ be a  zero divisor  in $\mathbb{F}[G]$ for a possible torsion-free group $ G $ and arbitrary field $\mathbb{F}$ with $ supp(\alpha)=\{1,x,x^{-1},y\} $, where $ x,y $ are distinct non-trivial elements of $ G$, and $ \beta $ be a mate of $ \alpha $. Then the size of the vertex set of $ Z(\alpha,\beta) $ is at least $ 7 $. 
\end{thm}
The following corollary  follows from Theorem \ref{F2vertex}.
\begin{cor}\label{F2}
Let $ \alpha $ be a  zero divisor  in $\mathbb{F}[G]$ for a possible torsion-free group $ G $ and arbitrary field $\mathbb{F}$ with $ supp(\alpha)=\{1,x,x^{-1},y\} $, where $ x,y $ are distinct non-trivial elements of $ G$, and $ \beta $ be a mate of $ \alpha $. Then $ | supp(\beta) |\geq 7 $. 
\end{cor}
\begin{thm}\label{final}
Let $ \alpha $ be a zero divisor in  $ \mathbb{F}[G] $ for a possible torsion-free group $ G $ and arbitrary field $\mathbb{F}$ with $ | supp(\alpha) | = 4$ and $ \beta $ be a mate of $ \alpha $, then $ | supp(\beta) |\geq 7 $.
\end{thm}
\begin{proof}
The result follows from  Lemmas \ref{z}, \ref{83}  and Corollaries \ref{end12}, \ref{endf10xy} and \ref{F2}.
\end{proof}
\section{ \textbf{Units whose supports are of size $ 4 $ in  group algebras of  torsion-free groups } }
Throughout this section let $  \mathsf{a} $ be a unit
in $ \mathbb{F}[G]  $ for a possible torsion-free group $ G $ and arbitrary field $\mathbb{F}$ with  $|supp(\mathsf{a})|= 4 $ and $  \mathsf{b} $ be a mate of $ \mathsf{a} $. It is known that $ |supp(\mathsf{b})|> 2 $ \cite[Theorem 4.2]{a55}. In  this section we will show that $ |supp(\mathsf{b})|\geq 6 $.\\
Let $B=supp(\mathsf{a})$, $C=supp(\mathsf{b})$. In view of Remark \ref{idenunit}, we may assume that $ C $ is a generating set of  $G$ and  $ 1\in B\cap C $. Since  $ G $ is not abelian \cite[Theorem 26.2]{PI},  Theorem \ref{4} implies $ |BC| \geq |B|+|C|+1 $.
\\ Since $ \mathsf{a}\mathsf{b}=1 $, we must have $1\in BC $, $ \sum_{(h,g)\in R_{BC}(1)}\alpha_h\beta_g=1 $ and  
$ \sum_{(h,g)\in R_{BC}(x)}\alpha_h\beta_g=0 $  for all $ x\in BC\setminus \{1\} $. Therefore,  by Definition \ref{set},  $ r_{BC}(1)\in \{1,2,\ldots,n\} $ and $ r_{BC}(x)\in \{2,3,\ldots,n\} $  for all $ x\in BC\setminus \{1\} $.
 According to above discussion and  Definition \ref{definitions}, $ \delta_{(\mathsf{a},\mathsf{b})}^{1}+\delta_{(\mathsf{a},\mathsf{b})}^{2}+\delta_{(\mathsf{a},\mathsf{b})}^{3}+\delta_{(\mathsf{a},\mathsf{b})}^{4}=|BC| $ and $\delta_{(\mathsf{a},\mathsf{b})}^{1}+2\delta_{(\mathsf{a},\mathsf{b})}^{2}+3\delta_{(\mathsf{a},\mathsf{b})}^{3}+4\delta_{(\mathsf{a},\mathsf{b})}^{4}=4|C| $, where $ \delta_{(\mathsf{a},\mathsf{b})}^{1}\in\{0,1\} $. So, if there is an integer number $ i\geq 1 $ such that $ |BC|=2|C|+i $, then $ \delta_{(\mathsf{a},\mathsf{b})}^{3}+2\delta_{(\mathsf{a},\mathsf{b})}^{4}<0 $, a contradiction. Hence, $ |BC|\leq 2|C| $.
 The following remark  follows from  above discussion.
\begin{rem}\label{sizeunit}
Let   $ \mathsf{a}  $ be a unit
in $ \mathbb{F}[G]  $ for a possible torsion-free group $ G $ and  arbitrary field $\mathbb{F}$ with $ |supp(\mathsf{a})|=4 $ and $ \mathsf{b}  $ be a mate of $ \mathsf{a}  $. Then $|supp(\mathsf{b})|+5\leq |supp(\mathsf{a})supp(\mathsf{b})|\leq 2|supp(\mathsf{b})| $. Also, according to Definition {\rm\ref{definitions}}, $ \delta_{(\mathsf{a},\mathsf{b})}^{1}+\delta_{(\mathsf{a},\mathsf{b})}^{2}+\delta_{(\mathsf{a},\mathsf{b})}^{3}+\delta_{(\mathsf{a},\mathsf{b})}^{4}=|supp(\mathsf{a})supp(\mathsf{b})| $ and $\delta_{(\mathsf{a},\mathsf{b})}^{1}+2\delta_{(\mathsf{a},\mathsf{b})}^{2}+3\delta_{(\mathsf{a},\mathsf{b})}^{3}+4\delta_{(\mathsf{a},\mathsf{b})}^{4}=4|supp(\mathsf{b})| $, where $ \delta_{(\mathsf{a},\mathsf{b})}^{1}\in\{0,1\} $.
\end{rem}
\begin{rem}\label{unit4}
Note that $ \mathbb{F}_2[G]  $ has no unit with support of size $ 4 $. Suppose, for a contradiction, that $ \mathsf{a}$ is a unit in  $ \mathbb{F}_2[G] $ with $ |supp(\mathsf{a})|=4 $ and $ \mathsf{b} $ is a mate of $ \mathsf{a} $.  $ \mathsf{a}\mathsf{b}=\sum_{x\in BC}r_{BC}(x)x=1 $ implies that $ 1\in BC $, $ r_{BC}(1)\in \{1,3\} $ and $ r_{BC}(x)\in \{2,4\} $  for all $ x\in BC\setminus \{1\} $, where $B=supp(\mathsf{a})$ and $C=supp(\mathsf{b})$. Hence, by Definition {\rm\ref{definitions}}, $ i+2\delta_{(\mathsf{a},\mathsf{b})}^{2}+4\delta_{(\mathsf{a},\mathsf{b})}^{4}=4|C| $, where $ i\in \{1,3\} $, a contradiction.
\end{rem}
The following lemma follows from Remark \ref{sizeunit}.
\begin{lem}
Let $ \mathsf{a} $ be a unit in  $ \mathbb{F}[G] $ for a possible torsion-free group $ G $ and arbitrary field $\mathbb{F}$ with $ |supp(\mathsf{a})|=4 $ and $ \mathsf{b} $ be a mate of $ \mathsf{a} $. Then $ |supp(\mathsf{b})|\geq 5 $.
\end{lem}
\begin{thm}\label{endxyunit}
Let  $ \mathsf{a} $ be a  unit in $\mathbb{F}[G]$ for a possible torsion-free group $ G $ and arbitrary field $\mathbb{F}$ such that $ supp(\mathsf{a})= \{1,x,y,xy\} $, where $ x,y $ are distinct non-trivial elements of $ G $, and $ \mathsf{b} $ be a mate of $ \mathsf{a} $. Then $  |supp(\mathsf{a})supp(\mathsf{b})|\leq 2|supp(\mathsf{b})|-1$. 
\end{thm}
\begin{proof}
 Let $B=supp(\mathsf{a})$ and $C=supp(\mathsf{b})$. Replacing $ \mathsf{a} $ by $ {\alpha_1}^{-1}\mathsf{a} $ and $ \mathsf{b} $ by $ \mathsf{b}{\alpha_1} $, we may assume that $ \mathsf{a}= 1+\lambda x+\gamma y+\delta xy$, where $ \gamma ,\lambda,\delta\in \mathbb{F}\setminus\{0\} $.  Note that if $ \delta=\lambda\gamma $, then $ \mathsf{a}\mathsf{b}=(1+\lambda x)(1+\gamma y)\mathsf{b}=1 $ contradicting   \cite[Theorem 4.2]{a55}. By Remark \ref{sizeunit}, $ |BC|\leq 2|C| $ and if  $ |BC|= 2|C| $, then two cases hold: $ (1) $ $ \delta_{(\mathsf{a},\mathsf{b})}^{2}=2|C| $ and $ \delta_{(\mathsf{a},\mathsf{b})}^{i}=0 $  for all $ i\in \{1,3,4\} $; $ (2) $ $ \delta_{(\mathsf{a},\mathsf{b})}^{2}=2|C|-2 $, $ \delta_{(\mathsf{a},\mathsf{b})}^{1}=1 $, $ \delta_{(\mathsf{a},\mathsf{b})}^{3}=1 $ and $ \delta_{(\mathsf{a},\mathsf{b})}^{4}=0 $.
 If the case $ (1) $ holds, then if we let $ \mathsf{a}'=\sum_{h\in B}h $ and $ \mathsf{b}'=\sum_{g\in C}g $, then $ \mathsf{a}',\mathsf{b}'\in \mathbb{F}_2[G] $ and $ \mathsf{a}'\mathsf{b}'=0 $  contradicting \cite[Theorem 2.1]{PS}. Hence,  the case $ (2) $ holds. In other words, if  $ |BC|= 2|C| $, then $ r_{BC}(1)=1 $ and there exists $ s\in BC $ such that $ r_{BC}(s)=3 $ and $ r_{BC}(m)=2 $  for all $ m\in BC\setminus \{1,s\} $. Let $ H=C\cup yC $ and $ |C\cap yC|=t $.  By the same argument as in the proof of  \ref{maximum}, it can be seen that  for each $ z\in  C\cap y C $, 
\begin{equation}\label{maximumunit}
max\{r_{BC}(z),r_{BC}(xz)\}\geq 3.
\end{equation}
Clearly, if $ \Delta_{(\mathsf{a},\mathsf{b})}^{3}\neq \varnothing $, then $t\geq 1  $. Thus,  by \ref{maximumunit}, if $ |BC|= 2|C| $, then $ t=1 $.\\
 Now, suppose that  $ z\in (H\cup xH) \setminus (\{1\}\cup (H\cap xH)) $. Since $ r_{BC}(z)\geq 2 $, if $ z\in H $, then  $ z\in C\cap yC $ and  if  $ z\in xH $, then $ x^{-1}z\in C\cap yC $. Therefore, we can define a map $ f:(H\cup xH) \setminus (\{1\}\cup (H\cap xH))\longrightarrow C\cap y C $ such that for each $ z\in (H\cup xH) \setminus (\{1\}\cup (H\cap xH)) $, 
\[
f(z)=\left\{
\begin{array}{ll}
z &  z\in H\\
x^{-1}z &  z\in xH\\
\end{array} \right. .
\]
By a similar argument as in the proof of Theorem \ref{endxy}, it can be seen that $ f $ is injective and also it follows from  Theorem \ref{set2}    that $ t\geq 3 $ and therefore $ |BC|\neq 2|C| $. This completes the proof.  
\end{proof}
\subsection*{The size of $ supp(\mathsf{b}) $ must be at least $ 6 $}  By Lemma \ref{unit size}, $ |S_{\mathsf{a}}|\in \{10,12\} $. At first suppose that $ |S_{\mathsf{a}}|=12 $. If $ |supp(\mathsf{b})|=5 $. Then by Remark \ref{deg12unit} and Lemma \ref{2unit}  and since the number of vertices with odd degree in each graph is even, $ U(\mathsf{a},\mathsf{b}) $ is isomorphic to the complete graph with $ 5 $ vertices. This graph  contains the graph $ K_{1,1,3} $ in Figure \ref{606} as a subgraph contradicting Theorem \ref{55unit}. Thus, we have the following theorem:
\begin{thm}\label{unit vertex 12}
Let $ \mathsf{a} $ be a unit in  $ \mathbb{F}[G] $ for a possible torsion-free group $ G $ and arbitrary field $\mathbb{F}$ with $ |supp(\mathsf{a})|=4 $, $ |S_{a}|=12 $  and $\mathsf{b}$ be a mate of $\mathsf{a} $. Then the size of the vertex set of $ U(\mathsf{a},\mathsf{b}) $ is at least $ 6 $.
\end{thm}
The following corollary follows  from Theorem \ref{unit vertex 12}.
 \begin{cor}\label{unit support 12}
Let $ \mathsf{a} $ be a unit in  $ \mathbb{F}[G] $ for a possible torsion-free group $ G $ and arbitrary field $\mathbb{F}$ with $ |supp(\mathsf{a})|=4 $, $ |S_{\mathsf{a}}|=12 $  and $\mathsf{b}$ be a mate of $ \mathsf{a} $. Then $ |supp(\mathsf{b})|\geq 6$.
\end{cor}
Now, suppose that $ |S_\mathsf{a}|=10 $. By Lemma \ref{83unit}, we may assume that  $ supp(\mathsf{a})$ is one of the sets  $\{1,x,y,xy\} $ or  $ \{1,x,x^{-1},y\} $, where $ x,y $ are distinct non-trivial elements of $ G$. 
By Remark \ref{sizeunit} and Theorem \ref{endxyunit}, we have the following result:
 \begin{thm}\label{unit x,y,xy}
Let $ \mathsf{a} $ be a unit in  $ \mathbb{F}[G] $ for a possible torsion-free group $ G $ and arbitrary field $\mathbb{F}$ with $ supp(\mathsf{a})=\{1,x,y,xy\} $, where $ x $ and $ y $ are distinct non-trivial elements of $ G$, and $\mathsf{b} $ be a mate of $ \mathsf{a} $. Then $ |supp(\mathsf{b})|\geq 6$.
\end{thm}  
Now, suppose that $ supp(\mathsf{a})=\{1,x,x^{-1},y\} $, where $ x $ and $ y $ are distinct non-trivial elements of $ G$, and $ |supp(\mathsf{b})|=5 $. By Remark  \ref{sizeunit},  $ |BC|=10$  and  we have two cases to consider: $ (1) $ $ \delta_{(\mathsf{a},\mathsf{b})}^{2}=10 $ and $ \delta_{(\mathsf{a},\mathsf{b})}^{i}=0 $  for all $ i\in \{1,3,4\} $; $ (2) $ $ \delta_{(\mathsf{a},\mathsf{b})}^{2}=8 $, $ \delta_{(\mathsf{a},\mathsf{b})}^{1}=1 $, $ \delta_{(\mathsf{a},\mathsf{b})}^{3}=1 $ and $ \delta_{(\mathsf{a},\mathsf{b})}^{4}=0 $.
 If the case $ (1) $ holds, then if we let $ \bar{\mathsf{a}}=\sum_{h\in B}h $ and $ \bar{\mathsf{b}}=\sum_{g\in C}g $, then $ \bar{\mathsf{a}},\bar{\mathsf{b}}\in \mathbb{F}_2[G] $ and $ \bar{\mathsf{a}}\bar{\mathsf{b}}=0 $ contradicting Theorem \ref{final}. Hence,  the case $ (2) $ holds.  Suppose that $ \Delta_{(\mathsf{a},\mathsf{b})}^{3}=\{s\} $ and $ R_{BC}(1)=\{(a',b')\} $. At first suppose that $ b'\notin \mathcal{V}_{(\mathsf{a},\mathsf{b})}(s) $. Let  $ \{g'\}=C\setminus(\mathcal{V}_{(\mathsf{a},\mathsf{b})}(s)\cup \{b'\}) $. According to the proof of  Lemma \ref{degreeconnunit}, $ deg_{U(\mathsf{a},\mathsf{b})}(b')=3 $,  $deg_{U(\mathsf{a},\mathsf{b})}(g')=4 $ and $ deg_{U(\mathsf{a},\mathsf{b})}(g)=5 $ for all $ g\in \mathcal{V}_{(\mathsf{a},\mathsf{b})}(s) $. It is clear that if $ b'\sim g $ for all $ g\in \mathcal{V}_{(\mathsf{a},\mathsf{b})}(s)$, then $ U(\mathsf{a},\mathsf{b}) $ contains the subgraph $ (a) $ in Figure \ref{36unit}, a contradiction. Hence, there exists $ g\in \mathcal{V}_{(\mathsf{a},\mathsf{b})}(s) $ such that $ g\nsim b' $. So, $ |M_{U(\mathsf{a},\mathsf{b})}(g)|=2 $ and therefore by Remark \ref{repetitionunit}, $g=s$, $ \mathcal{V}_{(\mathsf{a},\mathsf{b})}(g)=\{g,xg,x^{-1}g\} $ and  $ |M_{U(\mathsf{a},\mathsf{b})}(\bar{g})|=1  $ for each $ \bar{g}\in \{xg,x^{-1}g\} $. On the other hand, if $ b'\sim g' $, then $ U(\mathsf{a},\mathsf{b}) $ contains the subgraph $ (a) $ in Figure \ref{36unit}, a contradiction. Thus, we must have $ b'\nsim g' $, $ |M_{U(\mathsf{a},\mathsf{b})}(b')|=1 $ and  $ |M_{U(\mathsf{a},\mathsf{b})}(g')|=1 $. So,  $ \{g',b'\}=\{x^2g,x^{-2}g\}$ and therefore $ G $ is a cyclic group contradicting \cite[Theorem 26.2]{PI}. \\
 Now, suppose that $ b'\in \mathcal{V}_{(\mathsf{a},\mathsf{b})}(s) $. Let $ \{g'_1,g'_2\}=C\setminus\mathcal{V}_{(\mathsf{a},\mathsf{b})}(s) $ and $ \{g_1,g_2\}=\mathcal{V}_{(\mathsf{a},\mathsf{b})}(s)\setminus \{b'\} $. According to the proof of  Lemma \ref{degreeconnunit}, $ deg_{U(\mathsf{a},\mathsf{b})}(b')=4 $,  $deg_{U(\mathsf{a},\mathsf{b})}(g'_i)=4 $ and $ deg_{U(\mathsf{a},\mathsf{b})}(g_i)=5 $ for each $ i\in \{1,2\} $. If $|M_{U(\mathsf{a},\mathsf{b})}(g_i)|=1  $ for each $ i\in \{1,2\} $, then  $ U(\mathsf{a},\mathsf{b}) $ contains the subgraph $ (b) $ in Figure \ref{36unit}, a contradiction. Hence, we may assume that $|M_{U(\mathsf{a},\mathsf{b})}(g_1)|=2 $. So,  by Remark \ref{repetitionunit},   $|M_{U(\mathsf{a},\mathsf{b})}(b')|=1 $ and $|M_{U(\mathsf{a},\mathsf{b})}(g_2)|=1 $. With this conditions, we must have $|M_{U(\mathsf{a},\mathsf{b})}(g'_i)|=1 $ for each $ i\in \{1,2\} $ and $ g'_1\sim g'_2 $. Since there exists  $ i\in\{1,2\} $ such that $ b'\sim g'_i $ and $ g_2 $ is adjacent to all vertices in $ U(\mathsf{a},\mathsf{b}) $, $ U(\mathsf{a},\mathsf{b}) $ contains the subgraph $ (c) $ in Figure \ref{36unit}, a contradiction.
 According to above discussion, we have the following theorem:
\begin{thm}\label{unit vertex 10}
Let $ \mathsf{a} $ be a unit in  $ \mathbb{F}[G] $ for a possible torsion-free group $ G $ and arbitrary field $\mathbb{F}$  with $ supp(\mathsf{a})=\{1,x,x^{-1},y\} $, where $ x $ and $ y $ are distinct non-trivial elements of $ G$,  and $ \mathsf{b} $ be a mate of $ \mathsf{a} $. Then the size of the vertex set of $ U(\mathsf{a},\mathsf{b}) $ is at least $ 6 $.
\end{thm}
\begin{cor}\label{unit support 10}
Let $ \mathsf{a} $ be a unit in  $ \mathbb{F}[G] $ for a possible torsion-free group $ G $ and arbitrary field $\mathbb{F}$  with $ supp(\mathsf{a})=\{1,x,x^{-1},y\} $, where $ x $ and $ y $ are distinct non-trivial elements of $ G$,  and $ \mathsf{b} $ be a mate of $\mathsf{a} $. Then $ |supp(\mathsf{b})|\geq 6$.
\end{cor}
\begin{thm}\label{finalunit}
Let $ \mathsf{a} $ be a unit in  $ \mathbb{F}[G] $ for a possible torsion-free group $ G $ and arbitrary field $\mathbb{F}$  with $ | supp(\mathsf{a}) | = 4$ and $\mathsf{b} $ be a mate of $ \mathsf{a} $, then $ | supp(\mathsf{b}) |\geq 6 $.
\end{thm}
\begin{proof}
The result follows from  Lemmas \ref{unit size}, \ref{83unit} and Corollaries  \ref{unit support 12}, \ref{unit x,y,xy} and  \ref{unit support 10}.
\end{proof}
\section{\textbf{The zero divisor graph of length $4 $ over  $ \mathbb{F}_2 $ and on torsion-free groups }}
Throughout this section, let $ \alpha$ be a  zero divisor  in $\mathbb{F}_{2}[G]$ for a possible torsion-free group $ G $ with $|supp(\alpha)|=4,$  $ \beta $ be a mate of $ \alpha $,  $B= supp(\alpha) $ and $C= supp(\beta) $. 
It is known that   $|C|\geq 8$ \cite[Theorem 1.3]{PS}. In this section, we improve the latter to $|C|\geq 9$. By Lemma \ref{z}, $ S_{\alpha}\in\{10,12\} $. In the sequel, we consider the size of the  vertex set of $Z(\alpha,\beta) $ in each of the cases $ S_{\alpha}=12 $ and $ S_{\alpha}=10 $.
\subsection{\textbf{The size of the  vertex set of} $Z\mathbf{(\alpha,\beta)} $,\textbf{ where} $ \mathbf{|S_\alpha|=12} $}  Theorem \ref{2} and Remark \ref{deg12} imply that if $ |S_\alpha|=12 $ and  $ |C|=n $, then  $Z(\alpha,\beta) $ is a simple and connected graph with $ n $ vertices such that the degree of each vertex in $ Z(\alpha,\beta) $ is $ 4,6,8,10 $ or $ 12$. Also, Corollary \ref{4regular} implies that it is sufficient that we study all non-isomorphic connected $ 4 $-regular graphs with $ n $ vertices. We obtain all  non-isomorphic connected $ 4 $-regular graphs when $ n<10 $. In Table \ref{forbidden}, we give all results about the existence of the forbidden subgraphs (see Figure \ref{606}) in the non-isomorphic connected $ 4 $-regular graphs with $ n $ vertices, where $ n\in\{5,6,7,8,9\} $. The entry $(n,\Gamma)$ of Table \ref{forbidden}, where $n\in\{5,6,7,8,9\}$ and $\Gamma$ is one of the graphs  $ K_{1,1,3}, K_{1,2,2},  \Gamma_1,  \Gamma_2,  \Gamma_3,  \Gamma_4,  \Gamma_5, \Gamma_6, \Gamma_7 $, shows the number of graphs of order $n$ having a subgraph isomorphic to $\Gamma$. The following theorem follows from  the results in Table \ref{forbidden} and Theorem \ref{4regularlem}.
\begin{table}
\begin{tabular}{|c|c|c|c|c|c|c|c|c|c|c|}
\cline{2-11}
\multicolumn{1}{c|}
{} & Total & $ K_{1,1,3} $ & $  K_{1,2,2} $ & $ \Gamma_1 $ & $ \Gamma_2 $& $ \Gamma_3 $& $ \Gamma_4 $& $ \Gamma_5 $& $ \Gamma_6 $& $ \Gamma_7 $\\
\hline
$ n=5 $&$ 1 $&$ 1 $&$ 0 $&$ 0 $&$ 0 $&$ 0 $&$ 0 $&$ 0 $&$ 0 $&$ 0 $\\
\hline
$ n=6 $&$ 1 $&$ 0 $&$ 1 $&$ 0 $&$ 0 $&$ 0 $&$ 0 $&$ 0 $&$ 0 $&$ 0 $\\
\hline
$ n=7 $&$ 2 $&$ 1$&$ 0 $&$ 0 $&$ 0 $&$ 0 $&$ 1 $&$ 0 $&$ 0 $&$ 0 $\\
\hline
$ n=8 $&$ 6 $&$ 1 $&$ 0 $&$ 0 $&$ 1 $&$ 2 $&$ 0 $&$ 1 $&$ 0 $&$ 1 $\\
\hline
$ n=9 $&$ 16 $&$ 3 $&$ 1 $&$ 1 $&$ 0 $&$ 1 $&$ 0 $&$ 2 $&$ 0 $&$ 1 $\\
\hline
\end{tabular}
\caption{Existence of the forbidden subgraphs in the non-isomorphic connected $ 4 $-regular graphs. }\label{forbidden}
\end{table}
\begin{thm}\label{set vertex}
Let $ \alpha$ be a   zero divisor  in $\mathbb{F}_{2}[G]$ for a possible torsion-free group $ G $ with $|supp(\alpha)|=4,$  $ |S_\alpha|=12 $ and $ \beta $ be a mate of $ \alpha $. Then the size of the vertex set of $ Z(\alpha,\beta) $ is at least $ 9 $. Moreover, if the size of the vertex set of $ Z(\alpha,\beta) $ is  $ 9 $, then  $ Z(\alpha,\beta) $  contains  a subgraph isomorphic to one of the graphs in Figure {\rm\ref{regular}}.
\end{thm}
\begin{figure}
\begin{tikzpicture}[scale=.8]
\draw [fill] (0,1.25) circle
[radius=0.09] node  [above]  {};
\draw [fill] (.85,.9) circle
[radius=0.09] node  [right]  {};
\draw [fill] (1.15,.3) circle
[radius=0.09] node  [right]  {};
\draw [fill] (1.05,-.4) circle
[radius=0.09] node  [right]  {};
\draw [fill] (.53,-.9) circle
[radius=0.09] node  [below]  {};
\draw [fill] (-.5,-.85) circle
[radius=0.09] node  [below]  {};
\draw [fill] (-1.05,-.4) circle
[radius=0.09] node  [left]  {};
\draw [fill] (-1.15,.3) circle
[radius=0.09] node  [left]  {};
\draw [fill] (-.85,.9) circle
[radius=0.09] node  [left]  {};
\draw  (0,1.25) -- (.85,.9);
\draw  (0,1.25) -- (1.15,.3);
\draw  (0,1.25) -- (1.05,-.4);
\draw  (0,1.25) -- (.53,-.9);
\draw  (.85,.9) -- (1.15,.3);
\draw  (.85,.9) -- (-.5,-.85);
\draw  (.85,.9) -- (-1.05,-.4);
\draw  (-1.15,.3) -- (1.15,.3);
\draw  (-.85,.9) -- (1.15,.3);
\draw  (-1.15,.3) -- (1.05,-.4);
\draw  (-1.05,-.4) -- (1.05,-.4);
\draw  (-.5,-.85) -- (1.05,-.4);
\draw  (-.85,.9) -- (.53,-.9);
\draw  (-1.15,.3) -- (.53,-.9);
\draw  (-.5,-.85) -- (.53,-.9);
\draw  (-.5,-.85) -- (-.85,.9);
\draw  (-1.05,-.4) -- (-.85,.9);
\draw  (-1.05,-.4) -- (-1.15,.3);
\end{tikzpicture}
\begin{tikzpicture}[scale=.8]
\draw [fill] (0,1.25) circle
[radius=0.09] node  [above]  {};
\draw [fill] (.85,.9) circle
[radius=0.09] node  [right]  {};
\draw [fill] (1.15,.3) circle
[radius=0.09] node  [right]  {};
\draw [fill] (1.05,-.4) circle
[radius=0.09] node  [right]  {};
\draw [fill] (.53,-.9) circle
[radius=0.09] node  [below]  {};
\draw [fill] (-.5,-.85) circle
[radius=0.09] node  [below]  {};
\draw [fill] (-1.05,-.4) circle
[radius=0.09] node  [left]  {};
\draw [fill] (-1.15,.3) circle
[radius=0.09] node  [left]  {};
\draw [fill] (-.85,.9) circle
[radius=0.09] node  [left]  {};
\draw  (0,1.25) -- (.85,.9);
\draw  (0,1.25) -- (1.15,.3);
\draw  (0,1.25) -- (1.05,-.4);
\draw  (0,1.25) -- (.53,-.9);
\draw  (.85,.9) -- (1.15,.3);
\draw  (.85,.9) -- (-.5,-.85);
\draw  (.85,.9) -- (-1.05,-.4);
\draw  (-1.15,.3) -- (1.15,.3);
\draw  (-.85,.9) -- (1.15,.3);
\draw  (-1.15,.3) -- (1.05,-.4);
\draw  (-.5,-.85) -- (1.05,-.4);
\draw  (.53,-.9) -- (1.05,-.4);
\draw  (.53,-.9) -- (-1.05,-.4);
\draw  (.53,-.9) -- (-.85,.9);
\draw  (-.5,-.85) -- (-1.05,-.4);
\draw  (-.5,-.85) -- (-.85,.9);
\draw  (-1.15,.3) -- (-1.05,-.4);
\draw  (-1.15,.3) -- (-.85,.9);
\end{tikzpicture}
\begin{tikzpicture}[scale=.8]
\draw [fill] (0,1.25) circle
[radius=0.09] node  [above]  {};
\draw [fill] (.85,.9) circle
[radius=0.09] node  [right]  {};
\draw [fill] (1.15,.3) circle
[radius=0.09] node  [right]  {};
\draw [fill] (1.05,-.4) circle
[radius=0.09] node  [right]  {};
\draw [fill] (.53,-.9) circle
[radius=0.09] node  [below]  {};
\draw [fill] (-.5,-.85) circle
[radius=0.09] node  [below]  {};
\draw [fill] (-1.05,-.4) circle
[radius=0.09] node  [left]  {};
\draw [fill] (-1.15,.3) circle
[radius=0.09] node  [left]  {};
\draw [fill] (-.85,.9) circle
[radius=0.09] node  [left]  {};
\draw  (0,1.25) -- (.85,.9);
\draw  (0,1.25) -- (1.15,.3);
\draw  (0,1.25) -- (1.05,-.4);
\draw  (0,1.25) -- (.53,-.9);
\draw  (.85,.9) -- (1.15,.3);
\draw  (.85,.9) -- (-.5,-.85);
\draw  (.85,.9) -- (-1.05,-.4);
\draw  (-1.15,.3) -- (1.15,.3);
\draw  (-.85,.9) -- (1.15,.3);
\draw  (-1.15,.3) -- (1.05,-.4);
\draw  (-.5,-.85) -- (1.05,-.4);
\draw  (.53,-.9) -- (1.05,-.4);
\draw  (.53,-.9) -- (-1.05,-.4);
\draw  (-.85,.9) -- (.53,-.9);
\draw  (-1.05,-.4) -- (.53,-.9);
\draw  (-1.05,-.4) -- (-.5,-.85);
\draw  (-1.15,.3) -- (-.5,-.85);
\draw  (-1.05,-.4) -- (-.85,.9);
\draw  (-1.15,.3) -- (-.85,.9);
\end{tikzpicture}
\begin{tikzpicture}[scale=.8]
\draw [fill] (0,1.25) circle
[radius=0.09] node  [above]  {};
\draw [fill] (.85,.9) circle
[radius=0.09] node  [right]  {};
\draw [fill] (1.15,.3) circle
[radius=0.09] node  [right]  {};
\draw [fill] (1.05,-.4) circle
[radius=0.09] node  [right]  {};
\draw [fill] (.53,-.9) circle
[radius=0.09] node  [below]  {};
\draw [fill] (-.5,-.85) circle
[radius=0.09] node  [below]  {};
\draw [fill] (-1.05,-.4) circle
[radius=0.09] node  [left]  {};
\draw [fill] (-1.15,.3) circle
[radius=0.09] node  [left]  {};
\draw [fill] (-.85,.9) circle
[radius=0.09] node  [left]  {};
\draw  (0,1.25) -- (.85,.9);
\draw  (0,1.25) -- (1.15,.3);
\draw  (0,1.25) -- (1.05,-.4);
\draw  (0,1.25) -- (.53,-.9);
\draw  (.85,.9) -- (1.15,.3);
\draw  (.85,.9) -- (-.5,-.85);
\draw  (.85,.9) -- (1.05,-.4);
\draw  (-1.15,.3) -- (1.15,.3);
\draw  (-1.05,-.4) -- (1.15,.3);
\draw  (-1.05,-.4) -- (1.05,-.4);
\draw  (-.85,.9) -- (1.05,-.4);
\draw  (-.85,.9) -- (.53,-.9);
\draw  (-1.15,.3) -- (.53,-.9);
\draw  (-1.05,-.4) -- (.53,-.9);
\draw  (-.85,.9) -- (-.5,-.85);
\draw  (-1.15,.3) -- (-.5,-.85);
\draw  (-1.05,-.4) -- (-.5,-.85);
\draw  (-.85,.9) -- (-1.15,.3);
\end{tikzpicture}
\begin{tikzpicture}[scale=.8]
\draw [fill] (0,1.25) circle
[radius=0.09] node  [above]  {};
\draw [fill] (.85,.9) circle
[radius=0.09] node  [right]  {};
\draw [fill] (1.15,.3) circle
[radius=0.09] node  [right]  {};
\draw [fill] (1.05,-.4) circle
[radius=0.09] node  [right]  {};
\draw [fill] (.53,-.9) circle
[radius=0.09] node  [below]  {};
\draw [fill] (-.5,-.85) circle
[radius=0.09] node  [below]  {};
\draw [fill] (-1.05,-.4) circle
[radius=0.09] node  [left]  {};
\draw [fill] (-1.15,.3) circle
[radius=0.09] node  [left]  {};
\draw [fill] (-.85,.9) circle
[radius=0.09] node  [left]  {};
\draw  (0,1.25) -- (.85,.9);
\draw  (0,1.25) -- (1.15,.3);
\draw  (0,1.25) -- (1.05,-.4);
\draw  (0,1.25) -- (.53,-.9);
\draw  (.85,.9) -- (1.15,.3);
\draw  (.85,.9) -- (-.5,-.85);
\draw  (.85,.9) -- (1.05,-.4);
\draw  (.53,-.9) -- (1.15,.3);
\draw  (-1.05,-.4) -- (1.15,.3);
\draw  (-.5,-.85) -- (1.05,-.4);
\draw  (-1.15,.3) -- (1.05,-.4);
\draw  (-.85,.9) -- (.53,-.9);
\draw  (-1.05,-.4) -- (.53,-.9);
\draw  (-.85,.9) -- (-.5,-.85);
\draw  (-1.15,.3) -- (-.5,-.85);
\draw  (-.85,.9) -- (-1.05,-.4);
\draw  (-1.15,.3) -- (-1.05,-.4);
\draw  (-1.15,.3) -- (-.85,.9);
\end{tikzpicture}
\begin{tikzpicture}[scale=.8]
\draw [fill] (0,1.25) circle
[radius=0.09] node  [above]  {};
\draw [fill] (.85,.9) circle
[radius=0.09] node  [right]  {};
\draw [fill] (1.15,.3) circle
[radius=0.09] node  [right]  {};
\draw [fill] (1.05,-.4) circle
[radius=0.09] node  [right]  {};
\draw [fill] (.53,-.9) circle
[radius=0.09] node  [below]  {};
\draw [fill] (-.5,-.85) circle
[radius=0.09] node  [below]  {};
\draw [fill] (-1.05,-.4) circle
[radius=0.09] node  [left]  {};
\draw [fill] (-1.15,.3) circle
[radius=0.09] node  [left]  {};
\draw [fill] (-.85,.9) circle
[radius=0.09] node  [left]  {};
\draw  (0,1.25) -- (.85,.9);
\draw  (0,1.25) -- (1.15,.3);
\draw  (0,1.25) -- (1.05,-.4);
\draw  (0,1.25) -- (.53,-.9);
\draw  (.85,.9) -- (1.15,.3);
\draw  (.85,.9) -- (-.5,-.85);
\draw  (.85,.9) -- (1.05,-.4);
\draw  (.53,-.9) -- (1.15,.3);
\draw  (-1.05,-.4) -- (1.15,.3);
\draw  (-1.05,-.4) -- (1.05,-.4);
\draw  (-1.15,.3) -- (1.05,-.4);
\draw  (-.85,.9) -- (.53,-.9);
\draw  (-1.15,.3) -- (.53,-.9);
\draw  (-.85,.9) -- (-.5,-.85);
\draw  (-1.15,.3) -- (-.5,-.85);
\draw  (-.85,.9) -- (-1.05,-.4);
\draw  (-.5,-.85) -- (-1.05,-.4);
\draw  (-1.15,.3) -- (-.85,.9);
\end{tikzpicture}
\begin{tikzpicture}[scale=.8]
\draw [fill] (0,1.25) circle
[radius=0.09] node  [above]  {};
\draw [fill] (.85,.9) circle
[radius=0.09] node  [right]  {};
\draw [fill] (1.15,.3) circle
[radius=0.09] node  [right]  {};
\draw [fill] (1.05,-.4) circle
[radius=0.09] node  [right]  {};
\draw [fill] (.53,-.9) circle
[radius=0.09] node  [below]  {};
\draw [fill] (-.5,-.85) circle
[radius=0.09] node  [below]  {};
\draw [fill] (-1.05,-.4) circle
[radius=0.09] node  [left]  {};
\draw [fill] (-1.15,.3) circle
[radius=0.09] node  [left]  {};
\draw [fill] (-.85,.9) circle
[radius=0.09] node  [left]  {};
\draw  (0,1.25) -- (.85,.9);
\draw  (0,1.25) -- (1.15,.3);
\draw  (0,1.25) -- (1.05,-.4);
\draw  (0,1.25) -- (.53,-.9);
\draw  (.85,.9) -- (1.15,.3);
\draw  (.85,.9) -- (-.5,-.85);
\draw  (.85,.9) -- (1.05,-.4);
\draw  (.53,-.9) -- (1.15,.3);
\draw  (-.5,-.85) -- (1.15,.3);
\draw  (-1.05,-.4) -- (1.05,-.4);
\draw  (-1.15,.3) -- (1.05,-.4);
\draw  (-.85,.9) -- (.53,-.9);
\draw  (-1.05,-.4) -- (.53,-.9);
\draw  (-.85,.9) -- (-.5,-.85);
\draw  (-1.15,.3) -- (-1.05,-.4);
\draw  (-.85,.9) -- (-1.05,-.4);
\draw  (-.5,-.85) -- (-1.15,.3);
\draw  (-1.15,.3) -- (-.85,.9);
\end{tikzpicture}
\caption{Possible $4$-regular subgraphs occuring in a zero divisor graph of length 4 with 9 vertices over $\mathbb{F}_2$. }\label{regular}
\end{figure}
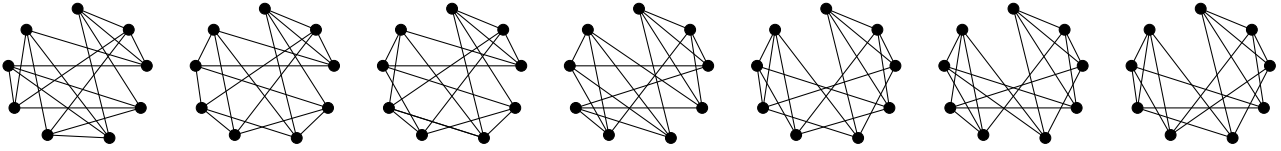
\subsection{\textbf{The zero divisor graph $Z(\alpha,\beta) $, where  $ |S_\alpha|=10 $}} According to Remark \ref{multi}, if $ |S_\alpha|=10 $, then $  Z(\alpha,\beta) $ is a multigraph such that  every pair of its vertices are adjacent by  at most two edges and  for each  $ g\in supp(\beta) $, $ 0\leq |M_{Z(\alpha,\beta)}(g)|\leq 2  $. By Lemma \ref{83},  we may assume that $supp(\alpha)=\{1,x,x^{-1},y\} $, where $x$ and $y$ are distinct non-trivial elements of $G$, and therefore if  $ \{(h_1,h'_1,g,g'),(h'_1,h_1,g',g)\}$ and $ \{(h_2,h'_2,g,g'),(h'_2,h_2,g',g)\} $ are distinct elements of $ \mathcal{E}_{Z(\alpha,\beta)}$, then 
\[\{(h_1,h'_1),(h_2,{h'_2}^{-1})\}\in\big{\{}\{(1,x),(x^{-1},1)\} \;\text{or}\; \{(1,x^{-1}),(x,1)\}\big{\}}
 .\]  
Therefore we may say that multiple edges in the zero divisor graph  are ``corresponding" to $x$ or $x^{-1}$ whenever $|S_{\alpha}|=10$ and each of these edges gives an equivalent equality $xg_1=g_2$ or $g_1=x^{-1}g_2$, where $g_1$ and $g_2$ are endpoints of the multiple edges. So we may replace these 
multiple edges by one edge, as we do in the following:
 \begin{defn}\label{cal}
Let $ \alpha$ be a   zero divisor  in $\mathbb{F}_{2}[G]$ for a possible torsion-free group $ G $ with $|supp(\alpha)|=4,$  $ |S_\alpha|=10 $ and $ \beta $ be a mate of $ \alpha $. Let $ \mathcal{Z}(\alpha,\beta) $ be the graph obtained from $ Z(\alpha,\beta) $ by replacing  multiple edges by a single edge. Clearly, $ \mathcal{Z}(\alpha,\beta) $ is a simple and connected graph.
 \end{defn}
\begin{lem}\label{r}
The degree of each  vertex of $ \mathcal{Z}(\alpha ,\beta) $ is  $ 3,4,5,6,7,8 $ or  $10$.
\end{lem}
\begin{proof} 
 According to  Remarks \ref{deg12} and \ref{multi} and  Definition \ref{cal}, the degree of each vertex of $ \mathcal{Z}(\alpha,\beta) $ is $ 2,3,\ldots,11 $ or $ 12 $.
Therefore, it is sufficient that we prove $ deg_{\mathcal{Z}(\alpha ,\beta)}(g)\notin\{ 2,9,11 , 12 \}$  for all $ g\in supp(\beta) $. By Lemma \ref{83},  we may assume that $supp(\alpha)=\{1,x,x^{-1},y\} $, where $x$ and $y$ are distinct non-trivial elements of $G$.\\ Suppose that $ g\in supp(\beta) $.  It  follows from  Remark \ref{deg12} and Lemma \ref{repetition} that if  $ deg_{Z(\alpha ,\beta)}(g)= 4 $, then $ |M_{Z(\alpha ,\beta)}(g)|\leq 1$. So, $ deg_{\mathcal{Z}(\alpha ,\beta)}(g)\neq 2 $.
By Remark \ref{deg12}, if $ g $ is a vertex of degree $  10 $ or 12, then  $1\in \Theta_4(g)$. Hence, Lemma  \ref{repetition} implies that $ |M_{Z(\alpha ,\beta)}(g)|= 2 $. Thus, $ deg_{\mathcal{Z}(\alpha ,\beta)}(g)\notin\{9,11,12\} $. This completes the proof.
\end{proof}
\begin{rem}\label{completemathcalz}
 Let $ \alpha$ be a   zero divisor  in $\mathbb{F}_{2}[G]$ for a possible torsion-free group $ G $ with $supp(\alpha)=\{1,x,x^{-1},y\} $, where $x$ and $y$ are distinct non-trivial elements of $G$, and $ \beta $ be a mate of $ \alpha $. Suppose that $ s\in \Delta_{(\alpha,\beta)}^4$.  Obviously, $ \mathcal{V}_{(\alpha,\beta)}(s)=\{s,x^{-1}s,xs,y^{-1}s\} $. Also, the following statements follow from Remark {\rm\ref{deg12}} and  Lemma {\rm\ref{repetition}}:
 \begin{itemize}
 \item[-] $ deg_{\mathcal{Z}(\alpha,\beta)}(s) $ is even and  $ deg_{\mathcal{Z}(\alpha,\beta)}(g)\geq 5 $ for all $ g\in \{xs,x^{-1}s,y^{-1}s\} $.
  \item[-] If the degrees of $ xs$ and  $x^{-1}s $ in $ \mathcal{Z}(\alpha ,\beta) $ are even, then $ deg_{\mathcal{Z}(\alpha,\beta)}(s)\in\{8,10\} $.
 \end{itemize}    
\end{rem}
In the sequel, we study the properties of vertices with different degrees in $ \mathcal{Z}(\alpha ,\beta) $. Note that by Lemma \ref{83},  we may assume that $supp(\alpha)=\{1,x,x^{-1},y\} $, where $x$ and $y$ are distinct non-trivial elements of $G$. Let $ g $ be a  vertex  in $ \mathcal{Z}(\alpha ,\beta) $ such that $deg_{\mathcal{Z}(\alpha ,\beta)}(g)= n $.\\ Clearly, if $  n=3$, then  $deg_{Z(\alpha ,\beta)}(g)= 4  $ and $ |M_{Z(\alpha ,\beta)}(g)|= 1 $. If $  n=4$, then one of the following cases occurs:
\begin{itemize}\label{typeii}
\item[ $ (i) $:]
$deg_{Z(\alpha ,\beta)}(g)= 4  $ and $ |M_{Z(\alpha ,\beta)}(g)|= 0 $. Hence, Remark \ref{deg12} implies  $\theta_4(g)=0$.   
\item[ $ (ii) $:]  $ deg_{Z(\alpha ,\beta)}(g)=6 $ and $ |M_{Z(\alpha ,\beta)}(g)|= 2 $. So, by  Remark  \ref{deg12}  and  Lemma \ref{repetition}, $ \Theta_4(g)=\{1\} $.
\end{itemize}
More precisely, we shall speak of a vertex of degree $ 4 $ of type $ (j) $  in $ \mathcal{Z}(\alpha ,\beta) $ if the vertex satisfies in the condition $ (j) $ in the above list ($ j $ being $ i $ or $ ii $). \\The following remark follows from above discussion, Lemma \ref{3part} and Remark  \ref{completemathcalz}.
\begin{rem}
If $ g $ is a vertex of degree $4$ of type $ (ii) $ in $ \mathcal{Z}(\alpha ,\beta) $, then there exists exactly an element  $ s\in \Delta_{(\alpha,\beta)}^ {4} $ such that $ g\in \mathcal{V}_{(\alpha,\beta)}(s)$. Moreover,   the degrees of all vertices of $ K[s]\setminus\{g\} $ are greater than or equal to $5$ in $ \mathcal{Z}(\alpha,\beta) $ and $ K[s] $ has at least two  vertices with  odd degrees in $ \mathcal{Z}(\alpha,\beta) $. Also, if $ g $ is a vertex of degree $ 3 $ or $4$ of type $ (i) $ in $ \mathcal{Z}(\alpha ,\beta) $, then $ g\notin \mathcal{V}_{(\alpha,\beta)}(s) $ for all $ s\in \Delta_{(\alpha,\beta)}^ {4} $.
\end{rem}
\begin{lem}\label{s4}
If $ 2 $ vertices of degree $ 4 $ in $ \mathcal{Z}(\alpha,\beta) $ are adjacent, then at least one of them is of type $ (i) $.
\end{lem}
\begin{proof}
Suppose that  $ g$  and $g' $ are $2$ vertices  of degree $ 4 $ in $ \mathcal{Z}(\alpha,\beta) $ such that $ g \sim g'$. If $ g$  and $g' $ are of type $ (ii) $, then $\Theta_4(g)=\{1\}=\Theta_4(g')$ and therefore $ g'\notin\{x^{-1}g,xg,y^{-1}g\}$ and $ g\notin\{x^{-1}g',xg',y^{-1}g'\}$. Hence, we must have $ yg=yg' $, a contradiction. 
\end{proof}
 Now suppose that $  n=5$. Hence,  $deg_{Z(\alpha ,\beta)}(g)= 6  $ and $ |M_{Z(\alpha ,\beta)}(g)|= 1 $. Thus,  by  Remarks \ref{deg12} and \ref{completemathcalz} and  Lemma \ref{repetition} we have the following:
 \begin{rem}\label{deg5}
 If $ g $ is a vertex of degree $ 5 $ in $ \mathcal{Z}(\alpha,\beta) $, then there exists exactly an element $ s\in \Delta_{(\alpha,\beta)}^{4} $ such that $ g\in\mathcal{V}_{(\alpha,\beta)}(s) $. Moreover, $ K[s] $ has at least  a vertex with even degree and  at least three vertices of degree greater than $ 4 $ in $ \mathcal{Z}(\alpha,\beta) $. Also, if the degrees of all vertices of $ \mathcal{V}_{(\alpha,\beta)}(s)\setminus\{g\} $ in $ \mathcal{Z}(\alpha,\beta) $ are even, then the degrees of all vertices of $ K[s] $ must be  greater than $4$ in $ \mathcal{Z}(\alpha,\beta) $. 
\end{rem}
The following lemma follows from Remark \ref{deg5}.
\begin{lem}
 $ \mathcal{Z}(\alpha,\beta) $ is not isomorphic to any $ 5 $-regular graph.
\end{lem}
If $  n=6$, then one of the following cases occurs:
\begin{itemize}
\item[ (a) :]
$deg_{Z(\alpha ,\beta)}(g)= 6  $ and $ |M_{Z(\alpha ,\beta)}(g)|= 0 $. Hence, by Remark \ref{deg12},   $\theta_4(g)=1$. Note that in this case $ \Theta_4(g)=\{y\} $ since otherwise $ |M_{Z(\alpha ,\beta)}(g)|\geq 1 $.
\item[ (b) :]  $ deg_{Z(\alpha ,\beta)}(g)=8 $ and $ |M_{Z(\alpha ,\beta)}(g)|= 2$. So, it follows from   Remark \ref{deg12} and  Lemma \ref{repetition} that in this case $ \Theta_4(g) $ is one of the sets $ \{1,x\} $, $ \{1,x^{-1}\} $ or $ \{1,y\} $.
\end{itemize}
Combining Lemma \ref{repetition}, Remark \ref{completemathcalz} and above discussion, we have the following: 
\begin{rem}\label{deg6}
 If  $ g $ is a vertex of degree $ 6 $ in $ \mathcal{Z}(\alpha,\beta) $, then exactly one of the following  holds:
\begin{itemize}
\item[(i)]  There exists exactly an element  $ s\in \Delta_{(\alpha,\beta)}^ {4} $ such that $ g\in \mathcal{V}_{(\alpha,\beta)}(s)$. Moreover, $ K[s] $ has at least three vertices of degree greater than $ 4 $ in $ \mathcal{Z}(\alpha,\beta) $ and if the degrees of all vertices of $ K[s] $ in $ \mathcal{Z}(\alpha,\beta) $ are even, then $ K[s] $ must have  a vertex of degree  $8$ or $ 10 $ in $ \mathcal{Z}(\alpha,\beta) $ {\rm(}this case corresponds  to the above case {\rm(a) )}.
\item[(ii)]  There exist exactly two distinct elements $ s,s'\in \Delta_{(\alpha,\beta)}^ {4} $ such that $ g\in \mathcal{V}_{(\alpha,\beta)}(s)\cap \mathcal{V}_{(\alpha,\beta)}(s') $. Moreover,  $ |\mathcal{V}_{(\alpha,\beta)}(s)\cap \mathcal{V}_{(\alpha,\beta)}(s')|=1 $, the degrees of all vertices of $ K[s] $ in $ \mathcal{Z}(\alpha,\beta) $ are  greater than $ 4 $ and $ K[s] $ has at least two vertices with odd degrees in $ \mathcal{Z}(\alpha,\beta) $. Also, $ K[s'] $ has at least three vertices of degree greater than $ 4 $ in $ \mathcal{Z}(\alpha,\beta) $ and if the degrees of all vertices of $ K[s'] $ in $ \mathcal{Z}(\alpha,\beta) $ are even, then $ K[s'] $  must have  a vertex of degree  $8$ or $ 10 $  in $ \mathcal{Z}(\alpha,\beta) $ {\rm(}this case corresponds to the above case {\rm(b)}, where $ \Theta_4(g)=\{1,y\} ${\rm)}. 
\item[(iii)]There exist exactly two distinct elements $ s,s'\in \Delta_{(\alpha,\beta)}^ {4} $ such that $ g\in \mathcal{V}_{(\alpha,\beta)}(s)\cap \mathcal{V}_{(\alpha,\beta)} (s') $. Moreover, $ |\mathcal{V}_{(\alpha,\beta)}(s)\cap \mathcal{V}_{(\alpha,\beta)} (s')|=2 $, the degrees of all vertices of $ K[s] $ and $ K[s'] $ in $ \mathcal{Z}(\alpha,\beta) $ are  greater than $ 4 $ and if $ \{g'\}=(\mathcal{V}_{(\alpha,\beta)}(s)\cap \mathcal{V}_{(\alpha,\beta)} (s'))\setminus \{g\} $, then $ deg_{ \mathcal{Z}(\alpha,\beta)}(g') $ must be even. Also,  
$ K[s] $ has  exactly two vertices with odd degrees and  if the degrees of all vertices of $ K[s'] $ in $ \mathcal{Z}(\alpha,\beta) $ are even, then $ deg_{ \mathcal{Z}(\alpha,\beta)}(g')\in\{8,10\} $ {\rm(}this case corresponds to the above case {\rm(b)}, where $ \Theta_4(g)\in\big{\{} \{1,x\},\{1,x^{-1}\}\big{\}} ${\rm)}.
\end{itemize}
\end{rem}
The following lemma follows from  Remark \ref{deg6}:
\begin{lem}
$ \mathcal{Z}(\alpha,\beta) $ is not isomorphic to any $ 6 $-regular graph.
\end{lem}
If $  n=7$, then $deg_{Z(\alpha ,\beta)}(g)=8  $ and $ |M_{Z(\alpha ,\beta)}(g)|= 1 $. Thus, by Remark \ref{deg12},  $\theta_4(g)=2$ and by Lemma \ref{repetition}, $ 1\notin\Theta_4(g) $. Therefore, $ \Theta_4(g)\neq \{x,x^{-1}\} $ since otherwise $ 1\in\Theta_4(g) $. Hence, $ \Theta_4(g)$ is one of the sets $ \{y,x^{-1}\} $ or $ \{x,y\} $. \\
The following remark follows from Remark \ref{completemathcalz} and above discussion. 
\begin{rem}\label{deg7}
If $ g $ is a vertex of degree $ 7 $ in $ \mathcal{Z}(\alpha,\beta)  $, then there exist exactly two distinct elements $ s,s'\in \Delta_{(\alpha,\beta)}^{4} $ such that $ g\in \mathcal{V}_{(\alpha,\beta)}(s)\cap  \mathcal{V}_{(\alpha,\beta)} (s') $. Moreover, $ |\mathcal{V}_{(\alpha,\beta)}(s)\cap  \mathcal{V}_{(\alpha,\beta)} (s')|=1 $,  each of the $ K[s] $ and $ K[s'] $ has at least three vertices  of degree greater than $ 4 $ and   at least  a vertex with even degree  in $ \mathcal{Z}(\alpha,\beta) $. Also, if the degrees of all vertices in $ \mathcal{V}_{(\alpha,\beta)}(s)\setminus \{g\} $ in $ \mathcal{Z}(\alpha,\beta)  $ are even, then $ K[s] $ must have  a vertex of degree  $8$ or $ 10 $  in $ \mathcal{Z}(\alpha,\beta) $.
\end{rem}
  The following Lemma directly follows  from  Remark \ref{deg7}.
\begin{lem}
 $ \mathcal{Z}(\alpha,\beta) $ is not isomorphic to any $ 7 $-regular graph.
\end{lem}
If $  n=8$, then one of the following cases occurs:
\begin{itemize}
\item[ $ (i) $:]
$deg_{Z(\alpha ,\beta)}(g)= 8  $ and $ |M_{Z(\alpha ,\beta)}(g)|= 0$. Hence, Remark \ref{deg12} implies  $\theta_4(g)=2$. Since $|M_{Z(\alpha ,\beta)}(g)|= 0$, $ \Theta_4(g) $ does not contain any elements of  $\{1,x,x^{-1}\} $ and therefore this case is impossible.
\item[ $ (ii) $:]  $ deg_{Z(\alpha ,\beta)}(g)=10 $ and $ |M_{Z(\alpha ,\beta)}(g)|= 2 $. According to Lemma \ref{repetition} and Remark \ref{deg12}, $1\in \Theta_4(g) $ and $ \theta_4(g)=3 $. So, $ \Theta_4(g) $ is one of the sets $ \{1,x,x^{-1}\} $, $ \{1,x^{-1},y\} $ or $ \{1,y,x\} $.
\end{itemize}
If $  n=10$, then one of the following cases occurs:
\begin{itemize}
\item[ $ (i) $:]
$deg_{Z(\alpha ,\beta)}(g)= 10  $ and $ |M_{Z(\alpha ,\beta)}(g)|= 0 $. Hence, Remark \ref{deg12} implies  $\theta_4(g)=3$. Since $|M_{Z(\alpha ,\beta)}(g)|= 0$, $ \Theta_4(g) $ does not contain any elements of  $\{1,x,x^{-1}\} $ and therefore this case is impossible.
\item[ $ (ii) $:]  $ deg_{Z(\alpha ,\beta)}(g)=12 $ and $ |M_{Z(\alpha ,\beta)}(g)|=2 $. So, Remark \ref{deg12} implies  $ \Theta_4(g)=supp(\alpha) $.
\end{itemize}
\begin{thm}\label{appmatcal}
Let $ \alpha $ be a zero divisor in $ \mathbb{F}_2[G] $ for a possible torsion-free group $ G $ with $ |supp(\alpha)|=4$, $ |S_{\alpha}|=10 $,  and $ \beta $ be a mate of $ \alpha $. Then $ \mathcal{Z}(\alpha,\beta) $ does not contain any  subgraph isomorphic to one of  the graphs  in Figure {\rm\ref{forbidden mathcal}}, where the degrees of vertices  $ g $ and $ g' $ of any subgraph in   $ \mathcal{Z}(\alpha,\beta) $ must be  $ 4 $ of type $ (ii) $ and $ 5  $, respectively.
\end{thm}
We refer the reader to Appendix  \ref{app10} to see details of the proof of Theorem  \ref{appmatcal}.
\begin{figure}
\begin{tikzpicture}[scale=.75]
\draw [fill] (0,0) circle
[radius=0.1] node  [left]  {$ g $};
\draw [fill] (1,0) circle
[radius=0.1] node  [left]  {};
\draw [fill] (0,-1) circle
[radius=0.1] node  [left]  {};
\draw [fill] (1,-1) circle
[radius=0.1] node  [left]  {};
\draw [fill] (.5,1) circle
[radius=0.1] node  [left]  {};
\draw  (.09,.1) --  (.45,.95) ;
\draw  (.55,.95) --  (.95,.05) ;
\draw  (.1,0) --  (.9,0) ;
\draw  (1,-.9) --  (1,-.1) ;
\draw  (.9,-1) --  (.1,-1) ;
\draw  (0,-.1) --  (0,-.9) ;
\draw (1,-.9) .. controls (1.5,1) and (1.75,1)  .. (.5,1) ;
\draw  (1,-.93) --  (.1,-.1) ;
\draw   (.99,-.05)  --  (0,-.93) ;
\end{tikzpicture}
\begin{tikzpicture}[scale=.75]
\draw [fill] (0,0) circle
[radius=0.1] node  [left]  {$g'$};
\draw [fill] (1,0) circle
[radius=0.1] node  [left]  {};
\draw [fill] (0,-1) circle
[radius=0.1] node  [left]  {};
\draw [fill] (1,-1) circle
[radius=0.1] node  [left]  {};
\draw [fill] (.5,1) circle
[radius=0.1] node  [left]  {};
\draw [fill] (-.5,1) circle
[radius=0.1] node  [left]  {};
\draw  (.09,.1) --  (.45,.95) ;
\draw  (.55,.95) --  (.95,.05) ;
\draw  (.1,0) --  (.9,0) ;
\draw  (1,-.9) --  (1,-.1) ;
\draw  (.9,-1) --  (.1,-1) ;
\draw  (0,-.1) --  (0,-.9) ;
\draw  (1,-.93) --  (.1,-.1) ;
\draw   (.99,-.05)  --  (0,-.93) ;
\draw  (-.5,1) --  (.95,.05) ;
\draw  (-.5,1) --  (0,.1) ;
\end{tikzpicture}
\begin{tikzpicture}[scale=.75]
\draw [fill] (0,0) circle
[radius=0.1] node  [left]  {$g'$};
\draw [fill] (1,0) circle
[radius=0.1] node  [right]  {$g$};
\draw [fill] (0,-1) circle
[radius=0.1] node  [left]  {};
\draw [fill] (1,-1) circle
[radius=0.1] node  [left]  {};
\draw [fill] (.5,1) circle
[radius=0.1] node  [left]  {};
\draw [fill] (-.5,1) circle
[radius=0.1] node  [left]  {};
\draw  (.09,.1) --  (.45,.95) ;
\draw  (.55,.95) --  (.95,.05) ;
\draw  (.1,0) --  (.9,0) ;
\draw  (1,-.9) --  (1,-.1) ;
\draw  (.9,-1) --  (.1,-1) ;
\draw  (0,-.1) --  (0,-.9) ;
\draw  (1,-.93) --  (.1,-.1) ;
\draw   (.99,-.05)  --  (0,-.93) ;
\draw  (-.5,1) --  (0,.1) ;
\draw  (-.5,1) --  (.5,1) ;
\end{tikzpicture}
\caption{Some forbidden subgraphs of   $ \mathcal{Z}(\alpha,\beta) $, where the degrees of vertices  $ g $ and $ g' $ of any subgraph in   $ \mathcal{Z}(\alpha,\beta) $ must be  $ 4 $ of type $ (ii) $ and $ 5  $, respectively.}\label{forbidden mathcal}
\end{figure}
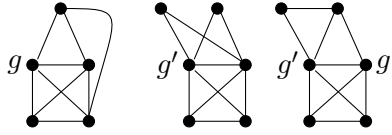
\subsection{\textbf{The size of the  vertex set of $Z(\alpha,\beta) $, where  $ |S_\alpha|=10 $}}   
According to Remark \ref{deg4f} and our discussion in this section,  $ g $ is a vertex of degree 4 of type $ (i) $ and $ (ii) $ in $ Z(\alpha,\beta) $ if and only if $ g $ is a vertex of degree 4 of type $ (i) $ and 3 in $ \mathcal{Z}(\alpha,\beta) $, respectively. So, the following Theorem follows from   Theorem \ref{app}.
\begin{thm}\label{graphf2}
Let $ \alpha $ be a  zero divisor in $ \mathbb{F}_2[G] $ for a possible torsion-free group $ G $ such that  $ | supp(\alpha)|=4
 $, $ | S_\alpha|=10$  and $ \beta $ be a mate of $ \alpha $. Then $ \mathcal{Z}(\alpha,\beta) $ does not contain any  subgraph isomorphic to one of  the graphs $ \Gamma_1,\Gamma_2 $ and $ \Gamma_3 $ in  Figure {\rm\ref{36}} and all graphs in Figure {\rm\ref{tt}}, where  the  degrees of gray and white vertices  in any subgraph  in $ \mathcal{Z}(\alpha,\beta) $ must be $ 3 $  and  $ 3 $ or $ 4 $ of type $ (i) $, respectively. 
\end{thm}
\begin{rem}\label{h1}
If $ \mathcal{Z}(\alpha,\beta) $ has a vertex of degree $ |supp(\beta)|-1 $ and a vertex of degree $ 3 $ or $ 4 $ of type $ (i) $, then clearly $  Z(\alpha,\beta) $ contains a subgraph isomorphic to the first graph  in Figure {\rm\ref{tt}} that is a contradiction.
\end{rem}
 Suppose that $ \alpha $ is a zero divisor in $ \mathbb{F}_2[G] $ for a possible torsion-free group $ G $ with  $  supp(\alpha)=\{1,x,x^{-1},y\}$, where $ x,y $ are distinct non-trivial elements of $ G $,   and $ \beta $ is a mate of $ \alpha $. Let $ B=supp(\alpha) $, $ C=supp(\beta) $, $ |C|=n $.  We know that for each $ a\in BC $, $ r_{BC}(a)\in\{2,4\} $. So, by Remark \ref{fieldf1}, we have the following statements:
 \begin{itemize}
 \item[-] $n+5 \leq |BC|\leq 2n $.
 \item[-] If $ |BC|= 2n  $, then $ r_{BC}(a)=2 $  for all $ a\in BC $ and therefore the degrees of all vertices in $  \mathcal{Z}(\alpha,\beta) $ are $ 3 $ or $ 4 $ of type $ (i) $.
 \item[-] If $ |BC|= 2n-i  $, where $ i $ is an integer number greater than or equal to $1$, then  $ \delta_{(\alpha,\beta)}^{4}=i $.
 \end{itemize}
  Note that in the sequel, we  apply repeatedly Remarks \ref{deg12}, \ref{completemathcalz} and the explanations that are given  about the vertices of different degrees in $  \mathcal{Z}(\alpha,\beta) $ in this section, with out referring.  \\Suppose that $ |BC|= 2n-1  $ and  $\Delta_{(\alpha,\beta)}^{4}=\{a\} $. Thus,  $ a $ is a vertex of degree $ 4 $ of type $ (ii) $ in $  \mathcal{Z}(\alpha,\beta) $, the degrees of vertices $ xa, x^{-1}a $ in $  \mathcal{Z}(\alpha,\beta) $ are $ 5 $ and  $ deg_{\mathcal{Z}(\alpha,\beta)}(y^{-1}a)\in\{ 5,6\} $. Also,   for each $ g\in C\setminus \mathcal{V}_{(\alpha,\beta)}(a) $, $ g $ is a vertex of degree $ 3 $ or $ 4 $ of type $ (i) $ in $ \mathcal{Z}(\alpha,\beta) $. In other words, if $ |BC|= 2|C|-1  $, then $ \mathcal{Z}(\alpha,\beta) $ must have $ n-3 $ vertices of degree  $ 3 $ or $ 4 $  and either  $ 3 $ vertices of degree $ 5 $ or  $ 2 $ vertices of degree $ 5 $ and a vertex of degree $ 6 $.\\ Now, suppose that $ |BC|= 2|C|-2  $. Then there exist distinct elements $ a,a'\in BC $ such that $ \Delta_{(\alpha,\beta)}^{4}=\{a,a'\} $. Let $S=\mathcal{V}_{(\alpha,\beta)}(a)\cap \mathcal{V}_{(\alpha,\beta)}(a')$. By Corollary \ref{cap1},  $ |S|\leq 2  $. \\Suppose that $ n=7 $ and $ |BC|= 2|C|-2  $. Clearly, $ |S|\neq 0 $. If $ |S|=1 $, then we must have $ a=y^{-1}a' $. Hence, $ deg_{\mathcal{Z}(\alpha,\beta)}(x^{-1}a')=5 $ and therefore there exist two distinct vertices in $ \mathcal{V}_{(\alpha,\beta)}(a)\setminus\{a\} $ such that $ x^{-1}a' $ is adjacent to each of them. So, $ \mathcal{Z}(\alpha,\beta) $ contains the second graph in Figure \ref{forbidden mathcal}, a contradiction. Then,  $ |S|=2 $.  Thus,  according to  part $ (2) $ of Lemma \ref{3part}, either $ a=xa' $ and $ x^{-1}a=a' $ or $ a=x^{-1}a' $ and $ xa=a' $. Therefore, $ \mathcal{Z}(\alpha,\beta) $ has a vertex of degree $ 6 $  and a vertex of degree  $ 3 $ or $ 4 $ of type $ (i) $  contradicting Remark \ref{h1}.\\ 
Suppose that $  n=8$ and  $ |BC|= 2|C|-2  $. Clearly, if $ |S| =0 $, then $ a $ and $a' $ are of degree  $ 4 $ of type $ (ii) $ in $ \mathcal{Z}(\alpha,\beta) $, $ deg_{\mathcal{Z}(\alpha,\beta)}(g)=5 $ for all $g\in\{ xa,x^{-1}a,xa', x^{-1}a'\} $ and $ y^{-1}a $, $ y^{-1}a' $ are vertices of degree  $ 5 $ or $ 6 $ in $ \mathcal{Z}(\alpha,\beta) $. If $ |S|=1 $, then two different cases hold:  (1)  $ S=\{a\} $ (resp. $ S=\{a'\} $): Then   $ a=y^{-1}a' $ (resp. $ a'=y^{-1}a $) since otherwise $ |S|=2 $. Therefore, the degree of $ a $ (resp. $ a' $) in $  \mathcal{Z}(\alpha,\beta) $ is $ 6 $, $ a' $ (resp. $ a $) is a  vertex of degree $ 4 $ of type $ (ii) $ in $  \mathcal{Z}(\alpha,\beta) $, $ deg_{\mathcal{Z}(\alpha,\beta)}(g)=5 $ for all $g\in\{ xa,x^{-1}a,xa', x^{-1}a'\} $ and $ y^{-1}a $ (resp. $ y^{-1}a' $) is of degree  $ 5 $ or $ 6 $ in $  \mathcal{Z}(\alpha,\beta) $. Also, the degree of $ g\in C\setminus (\mathcal{V}_{(\alpha,\beta)}(a)\cup \mathcal{V}_{(\alpha,\beta)}(a')) $ in $  \mathcal{Z}(\alpha,\beta) $ is  $ 3 $ or $ 4 $ of type $ (i) $; (2)  $ a\notin S$ and $a'\notin S $: Then we have $ y^{-1}a =x^{-1}a',\,y^{-1}a =xa',\,y^{-1}a' =xa$  or $y^{-1}a'=x^{-1}a $ since otherwise $ \delta_{(\alpha,\beta)}^{4}>2 $. It is clear that  in this case the degree of   $g\in S $ in $ \mathcal{Z}(\alpha,\beta) $ is $ 7 $ and the degree of  $ g'\in C\setminus (\mathcal{V}_{(\alpha,\beta)}(a)\cup \mathcal{V}_{(\alpha,\beta)}(a'))$ in $ \mathcal{Z}(\alpha,\beta) $ is  $ 3 $ or $ 4 $ of type $ (i) $  contradicting Remark \ref{h1}. If $ |S|=2 $, then by the same discussion about similar conditions for $  n=7$,  $ \mathcal{Z}(\alpha,\beta) $ has $ 2 $ vertices of degree $ 3 $ or $ 4 $, $ 2 $ vertices of degree $ 6$, $ 2 $ vertices of degree $5 $  and 2    vertices of degree $ 5 $ or $ 6 $. \\
 Now, suppose that $  n=8$ and  $ |BC|= 2|C|-3  $. Then there exist distinct elements $ a,a',a''\in BC $ such that  $ \Delta_{(\alpha,\beta)}^{4}= \{a,a',a''\} $. Let $ A_1= \mathcal{V}_{(\alpha,\beta)}(a) $, $ A_2= \mathcal{V}_{(\alpha,\beta)}(a') $ and $ A_3= \mathcal{V}_{(\alpha,\beta)}(a'') $. By inclusion-exclusion principle, $ |A_1\cup A_2\cup A_3|=12-|A_1\cap A_2|-|A_1\cap A_3|-|A_2\cap A_3|+|A_1\cap A_2\cap A_3|\leq 8 $. 
It follows from Lemma \ref{cap1} that $ |A_i\cap A_j|\leq 2 $ for each pair  $ (i,j)$ of  distinct elements of $\{1,2,3\} $. Note that if $ g\in A_1\cap A_2\cap A_3 $, then $ \theta_4(g)=3 $ which implies $ deg_{\mathcal{Z}}(g)>7 $, a contradiction. On the other hand, if $ |A_i\cap A_j|= 2 $ and $ |A_i\cap A_k|=2 $, where $ \{i,j,k\} = \{1,2,3\} $, then there exists  $ g\in A_1\cap A_2\cap A_3 $, a contradiction. So, without loss of generality we may assume that $ |A_1\cap A_2|=2 $, $ |A_1\cap A_3|=1$,  $|A_2\cap A_3|=1$,  $|A_1\cap A_2\cap A_3|=0  $ and $ a=xa',\, x^{-1}a=a' $. Let us show that it is impossible that $A_1\cap A_3= \{xa\} $ and $A_2\cap A_3= \{x^{-1}a'\} $. Suppose that  $ xa\in A_1\cap A_3 $. It is clear that $ xa\neq xa''$. On the other hand, if $ xa=a'' $, then contradicting  $ | A_1\cap A_2\cap A_3 |=0$ and if $ xa=x^{-1}a'' $, then $\delta_{(\alpha,\beta)}^{4}>3$, a contradiction. Hence, $ xa=y^{-1}a'' $. Similarly, if $A_2\cap A_3= \{x^{-1}a'\} $, then $ x^{-1}a'=y^{-1}a'' $ which implies $ | A_1\cap A_2\cap A_3 |=1$, a contradiction. Obviously,  $ ya=hg $, $ yx^{-1}a=h'g' $, where $ g,g' $  are distinct elements of $ \{x^{-1}a'',xa'',y^{-1}a''\} $ and $ h,h'\in \{1,x,x^{-1}\} $. If  $ xa=y^{-1}a'' $ (resp. $ x^{-1}a'=y^{-1}a'' $), then $y^{-1}a'=a'' $ (resp. $ y^{-1}a=a'' $) and  therefore according to above discussion   in this case $ \mathcal{Z}(\alpha,\beta) $ has  a vertex of degree  $ 7 $, $ 4 $ vertices of degree $ 6 $ and $ 3 $ vertices of degree $ 5 $. Now, suppose that  $A_1\cap A_3= \{y^{-1}a\} $ and $A_2\cap A_3= \{y^{-1}a'\} $. Then according to before discussion  $ a''\in\{y^{-1}a,y^{-1}a'\} $ and therefore in this case $ \mathcal{Z}(\alpha,\beta) $ has a vertex of degree $ 7 $, $ 4 $ vertices of degree $ 6 $ and $ 3 $ vertices of degree $ 5 $.
The following remark follows from  above discussion.
\begin{lem}\label{degreese}
Let $ \alpha $ be a  zero divisor in $ \mathbb{F}_2[G] $ for a possible torsion-free group $ G $ such that  $ | supp(\alpha)|=4$, $ | S_\alpha|=10$  and $ \beta $ be a mate of $ \alpha $. Then  the following items hold:
 \begin{itemize}
 \item[(i)] If $ | supp(\beta)|=6$, then either the degree of all vertices in $ \mathcal{Z}(\alpha,\beta) $ is $ 3 $ or $ 4 $ or the degree sequence {\rm(}i.e., non-increasing sequence of degrees{\rm)} of $ \mathcal{Z}(\alpha,\beta) $ is $ [5,5,5,4,4,3] $.
 \item[(ii)] If $ | supp(\beta)|=7$, then either the degree of all vertices in $ \mathcal{Z}(\alpha,\beta) $ is $ 3 $ or $ 4 $ or the degree sequence of $ \mathcal{Z}(\alpha,\beta) $ is $ [6,5,5,4,4,3,3] $, $ [6,5,5,4,4,4,4] $, $ [5,5,5,4,3,3,3] $ or $ [5,5,5,4,4,4,3] $.
  \item[(iii)] If $ | supp(\beta)|=8$, then either the degree of all vertices in $ \mathcal{Z}(\alpha,\beta) $ is $ 3 $ or $ 4 $ or the degree sequence of $ \mathcal{Z}(\alpha,\beta) $ is $ [6,5,5,4,3,3,3,3] $, $ [6,5,5,4,4,4,3,3] $, $[6,5,5,4,4,4,4,4] $, $ [5,5,5,4,4,3,3,3] $, $ [5,5,5,4,4,4,4,3] $, $ [6,6,5,5,5,5,4,4] $, $ [5,5,5,5,5,5,4,4] $, $ [6,5,5,5,5,5,4,3] $, $ [6,6,6,6,5,5,4,4] $, $ [6,6,6,6,5,5,3,3] $,  $ [6,6,6,5,5,5,4,3] $,  $ [6,6,5,5,5,5,3,3] $ or  $ [7,6,6,6,6,5,5,5] $.
 \end{itemize}
\end{lem}
It can be seen that the number of all non-isomorphic connected graphs with $ 7 $ and $8$ vertices which satisfy Lemma \ref{degreese} is   $ 51 $ and $ 567 $, respectively. We checked all of these graphs, it follows from Theorems \ref{graphf2}, \ref{appmatcal} and Remarks \ref{deg5}, \ref{deg6} and \ref{deg7} that $ \mathcal{Z}(\alpha,\beta) $ is not isomorphic to any of these graphs. Then, by  above discussion and Theorem \ref{final} we have the following theorem:
\begin{thm}\label{set vertex 10}
Let $ \alpha $ be a  zero divisor in $ \mathbb{F}_2[G] $ for a possible torsion-free group $ G $ with $ |Supp(\alpha)|=4 $, $ |S_\alpha|=10 $ and $ \beta $ be a mate of $ \alpha $. Then the size of the vertex set of $  \mathcal{Z}(\alpha,\beta) $ is at least $ 9 $.
\end{thm}
\subsection{\textbf{The size of $ supp(\beta)$ must be  equal or greater than  $ 9 $} }
The following corollary directly follows  from  Theorem \ref{set vertex}.
\begin{cor}\label{end1}
Let $ \alpha $ be a  zero divisor in $ \mathbb{F}_2[G] $ for a possible torsion-free group $ G $ with $ | supp(\alpha) | = 4,$   $ | S_{\alpha} |=12$ and $ \beta $ be a mate of $ \alpha $. Then $ | supp(\beta) |\geq 9 $. 
\end{cor}
The following corollary follows immediately from  Theorem \ref{set vertex 10}.
\begin{cor}\label{end2}
Let $ \alpha $ be a  zero divisor in $ \mathbb{F}_2[G] $ for a possible torsion-free group $ G $ with $ | supp(\alpha)|=4
 $, $ | S_\alpha|=10$  and $ \beta $ be a mate of $ \alpha $. Then $  | supp(\beta)|\geq 9 $.
\end{cor}
Combining  Lemma \ref{z} and Corollaries  \ref{end1} and \ref{end2}, we obtain the following theorem:
\begin{thm}
Let $ \alpha $ be a  zero divisor in $ \mathbb{F}_2[G] $ for a possible torsion-free group $ G $ with $ | supp(\alpha)|=4 $ and $ \beta $ be a mate of $ \alpha $. Then $  | supp(\beta)|\geq 9 $.
\end{thm}

\small\noindent{\bf Alireza Abdollahi} \\
Department of Mathematics, University of Isfahan, Isfahan 81746-73441,
 Iran,\\ and \\ School of Mathematics, Institute for Research in Fundamental Sciences (IPM), P.O.Box 19395-5746, Tehran, Iran\\
E-mail address: {\tt a.abdollahi@math.ui.ac.ir}\\

\noindent {\bf Fatemeh Jafari} \\
 Department of Mathematics, University of Isfahan, Isfahan 81746-73441, Iran.\\
E-mail address: {\tt f.jafari@sci.ui.ac.ir}
\section{Appendix}\label{app12}
In Theorems \ref{55} and \ref{55unit}, we determine some graphs as forbidden subgraphs of $ Z(\alpha,\beta) $ and $ U(\alpha,\beta) $, where $ |S_\alpha|=12 $,  without details. In this section we give details about finding these graphs.  Before we do this, we need some explanation  which are stated below.\\
In this section, let    $ \alpha$ be   a  zero divisor or a unit in $\mathbb{F}[G]$ for a possible torsion free group $G$ and arbitrary field $\mathbb{F}$ with  $ |supp(\alpha)|=4 $, $  | S_\alpha |=12 $ and $ \beta $ be a mate of $ \alpha $. By Remarks \ref{iden} and \ref{idenunit}, we may assume that $ 1\in supp(\alpha) $ and $ G=\left\langle supp(\alpha)\right\rangle=\left\langle supp(\beta)\right\rangle  $. Since $G$ is not  abelian \cite[Theorem 26.2]{PI} or isomorphic to the Klein bottle group (see Remark \ref{klein}),  Theorem \ref{01} implies the following: 
\begin{rem}\label{one}
Let  $ \alpha $ be a zero divisor or a unit in $\mathbb{F}[G]$ for a possible torsion free group $G$ and arbitrary field $\mathbb{F}$ with  $ |supp(\alpha)|=4 $ and $  | S_\alpha |=12 $. If $ supp(\alpha)=\{1,x,y,z\} $, where $ x,y,z $ are pairwise distinct non-trivial elements of $ G$, then   $\{ x,y,z\} $ is not any of  sets $\{a,a^{-1},b\},\{ a,a^2,b\},\{a,b,ab^{-1}a\}$ or $\{a,b,ab\}$ for some $a,b\in \{ x,y,z \}$. 
\end{rem}
 Suppose  that $ C $ is a cycle of length $ n $ as Figure \ref{cn} in $ Z(\alpha,\beta) $ or $ U(\alpha,\beta) $. Suppose further that $ \{g_1,\ldots,g_n\}\subseteq supp(\beta) $ is the  vertex set of $ C $ such that $ g_i\sim g_{i+1} $ for all $ i\in \{1,\ldots,n-1\} $ and $ g_1\sim g_n $. By an arrangement $ l $ of the vertex set $C$, we mean a sequence of all vertices as $x_1,\ldots,x_k$ such that $ x_i\sim x_{i+1} $ for all $ i\in \{1,\ldots,n-1\} $ and $ x_1\sim x_n $.
\begin{figure}[htp]
\begin{tikzpicture}
\draw (0,-1.25) .. controls (-0.700,-1.25) and (-1.25,-0.7) .. (-1.25,0)
.. controls (-1.25,0.7) and (-0.7,1.25) .. (0,1.25)
.. controls (0.7,1.25) and (1.25,0.7) .. (1.25,0);
\draw (1.25,0)[dashed] .. controls (1.25,-0.7) and (0.7,-1.25) .. (0,-1.25);
\draw [fill] (0,-1.25) circle
[radius=0.09] node  [below]  {$g_{n-2}$};
\draw [fill] (-.95,-.8) circle
[radius=0.09] node  [left]  {$g_{n-1}$};
\draw [fill] (-1.20,.4) circle
[radius=0.09] node  [left]  {$g_{n}$};
\draw [fill] (-.5,1.15) circle
[radius=0.09] node  [above]  {$g_{1}$};
\draw [fill] (.8,1) circle
[radius=0.09] node  [above]  {$g_{2}$};
\draw [fill] (1.25,0) circle
[radius=0.09] node  [right]  {$g_{3}$};
\end{tikzpicture}
\caption{A cycle of length $ n $ in $ Z(\alpha,\beta) $ or $ U(\alpha,\beta) $.}\label{cn}
\end{figure}
Since $ g_i\sim g_{i+1} $ for all $ i\in \{1,\ldots,n-1\} $ and $ g_1\sim g_n $ for each $ i\in \{1,2,\ldots,n\} $, there exist distinct elements $ h_i,h'_i\in supp(\alpha) $ satisfying the following relations:
\begin{equation}\label{cycle}
C: \left\{
\begin{array}{ll}
 h_{1} g_{1}=h'_{1}g_{2} \rightarrow g_{1}g_{2}^{-1}= h^{-1}_{1}h'_{1}\\  
h_{2}g_{2}=h'_{2}g_{3} \rightarrow g_{2} g_{3}^{-1}={h}_{2}^{-1}h'_{2}\\
\vdots\\
h_{n}g_{n}=h'_{n}g_{1} \rightarrow g_{n}g_{1}^{-1}=h_{n}^{-1}h'_{n}
\end{array} \right. .
\end{equation} 
 We assign a $ 2n $-tuple $ T_C^l=[h_1,h'_1,h_2,h'_2,\ldots ,h_n,h'_n] $  to the cycle $ C $ corresponding to the above arrangement $ l $ of the vertex set of $C$. We denote by $ R( T_C^l) $ the above set $R$ of relations.
It can be derived from the relations \ref{cycle} that $ r( T_C^l):=(h^{-1}_{1}h'_{1})({h}_{2}^{-1}h'_{2})\ldots (h_{n}^{-1}h'_{n}) $ is equal to 1. It
follows from Lemmas \ref{2} and \ref{2unit} that if $ T_C^{l'} $ is  a $ 2n $-tuple of $ C $ corresponding to another arrangement $ l' $ of the vertex set of $C$, then $ T_C^{l'} $ is one of the following $ 2n $-tuples:
 \begin{center}
  \begin{tabular}{c}
  $ [h_1,h'_1,h_2,h'_2,\ldots ,h_n,h'_n] $,\\
 $ [h_2,h'_2,h_3,h'_3,\ldots ,h_n,h'_n,h_1,h'_1] $,\\
 $ [h_3,h'_3,h_4,h'_4,\ldots ,h_1,h'_1,h_2,h'_2] $,\\
 $ \vdots $\\
 $ [h_n,h'_n,h_1,h'_1,\ldots ,h_{n-2},h'_{n-2},h_{n-1},h'_{n-1}] $,\\
$ [h'_n,h_n,h'_{n-1},h_{n-1},\ldots ,h'_2,h_2,h'_1,h_1] $,\\
$ \vdots $\\
$ [h'_1,h_1,h'_{n},h_{n},\ldots ,h_3,h'_3,h_2,h'_2] $.
 \end{tabular}
 \end{center}
 The set of all such $ 2n $-tuples will be denoted by $ \mathcal{T}(C) $.  Also, $ \mathcal{R}(T)=\{R(C)|T\in  \mathcal{T}(C) \}$.
 \begin{defn}
Let $ \Gamma $ be a zero divisor graph or a unit graph on a pair of elements $ (\alpha,\beta) $ in $\mathbb{F}[G]$ for a possible torsion free group $G$ and arbitrary field $\mathbb{F}$ such that   $ |supp(\alpha)|=4 $ and $  | S_\alpha |=12 $. Let $ C $ be a cycle of length $ n $ in $ \Gamma $. Since $ r(T_1)=1 $ if and only if $ r(T_2)=1 $, for all $ T_1,T_2\in  \mathcal{T}(C) $, a member of $ \{r(T)|T \in  \mathcal{T}(C) \} $ is given as a representative and denoted by $ r(C) $. Also, $ r(C)=1 $  is called the relation of $ C $.
\end{defn}
 \begin{defn}\label{eqquidef}
Let $ \Gamma $ be a zero divisor graph or a unit graph on a pair of elements $ (\alpha,\beta) $ in $\mathbb{F}[G]$ for a possible torsion free group $G$ and arbitrary field $\mathbb{F}$ such that   $ |supp(\alpha)|=4 $ and $  | S_\alpha |=12 $. Let $ C $ and $ C' $ be two cycles of length $ n $ in $  \Gamma$. We say that these two
cycles are equivalent, if $ \mathcal{T}(C)\cap \mathcal{T}(C')\neq\varnothing $ and otherwise are non-equivalent.
\end{defn}
 \begin{rem}\label{eqquirem}
 Suppose  that $ C $ and $ C' $ are two cycles of length $ n $ in $ \Gamma $. Then  $ C $ and $ C' $ are equivalent if $ \mathcal{T}(C)=\mathcal{T}(C')$.
 \end{rem}
 \begin{rem}\label{tuple}
 Let  $ \alpha $ be a zero divisor in $\mathbb{F}[G]$ for a possible torsion free group $G$ and arbitrary field $\mathbb{F}$ with  $ |supp(\alpha)|=4 $ and $  | S_\alpha |=12 $ and $ \beta $ be a mate of $ \alpha $. By Remark {\rm\ref{deg12}}, if $ g $ is a vertex of degree $ 4 $ in $ Z(\alpha,\beta) $, then $ \theta_{3}(g)=0 $ and  $ \theta_{4}(g)=0 $. Thus, if $ C $ is a cycle of length $ n $  on vertices of degree $ 4 $ in $ Z(\alpha,\beta) $ and $ T_1=[h_1,h'_1,h_2,h'_2,\ldots ,h_n,h'_n] $ is  a $ 2n $-tuple of $ C $, then $ h'_i\neq h_{i+1} $ for all $ i\in \{1,\ldots,n-1\} $ and $ h'_n\neq h_{1} $.
 \end{rem}
 \begin{rem}\label{tupleunit}
 Let  $ \alpha $ be a unit in $\mathbb{F}[G]$ for a possible torsion free group $G$ and arbitrary field $\mathbb{F}$ with  $ |supp(\alpha)|=4 $ and $  | S_\alpha |=12 $ and $ \beta $ be a mate of $ \alpha $. Let $ F=\{g|g\in \mathcal{V}_{U(\alpha,\beta)}, deg_U(g)=4\} $. In view of Proof of Lemma  {\rm\ref{degreeconnunit}}, either for each  $ g\in F $, $ \theta_{3}(g)=0 $ and $ \theta_{4}(g)=0 $ or there exists  $ g'\in F $ such that  $ \theta_{3}(g')=1 $ and $ \theta_{4}(g')=0 $ and for each  $ g\in F\setminus \{g'\} $, $ \theta_{3}(g)=0 $ and $ \theta_{4}(g)=0 $. Hence, if $ C $ is a cycle of length $ n $  on vertices of degree $ 4 $ in $ U(\alpha,\beta) $ and $ T_1=[h_1,h'_1,h_2,h'_2,\ldots ,h_n,h'_n] $ is  a $ 2n $-tuple of $ C $, then $ h'_i\neq h_{i+1} $ for all $ i\in \{1,\ldots,n-1\} $ and $ h'_n\neq h_{1} $.
 \end{rem}
  \begin{rem}
 For integers $m$ and $ n $, the Baumslag-Solitar group $BS(m,n)$
is the group given by the presentation $\left\langle a,b|ba^mb^{-1}=a^n\right\rangle$. According to   \cite[Remark 3.17]{AT} every torsion-free quotient of $BS(1,n)$  satisfies both Conjectures  {\rm\ref{KZDC}} and {\rm\ref{uk}}.
 \end{rem} 
 \begin{rem}
Let  $ \Gamma $ be a graph such that the degree of all its vertices is at most $ 12 $ and   $ B $  be the set of all cycles in $ \Gamma $. Let $\mathcal{G}(\Gamma):=\left\langle supp(\alpha) |r(C),C\in B\right\rangle.$
If at least one of the following cases happens, then clearly  $ \Gamma $ is a forbidden subgraph of $ Z(\alpha,\beta) $ and $ U(\alpha,\beta) $:
\begin{itemize}
\item[1.]$ \mathcal{G}(\Gamma) $ is an abelian group.
\item[3.]$ \mathcal{G}(\Gamma)$ is a quotient  group  of
 $ BS(1,n)$.
\item[4.]$ \mathcal{G}(\Gamma) $ has a non-trivial torsion element.
\end{itemize} 
 \end{rem}
\textbf{Proof of Theorems \ref{55} and \ref{55unit}:}
\begin{figure}
\subfloat{$ H_1 $}
\begin{tikzpicture}[scale=.75]
\draw [fill] (.75,1.25) circle
[radius=0.1] node  [right]  {$g_1$};
\draw [fill] (-.5,0) circle
[radius=0.1] node  [below]  {$g_3$};
\draw [fill] (.75,0) circle
[radius=0.1] node  [below]  {$g_2$};
\draw [fill] (2,0) circle
[radius=0.1] node  [below]  {$g_4$};
\draw  (2,0) --  (.75,1.2) ;
\draw   (.75,0)--(1.9,0);
\draw   (.75,1.25)--(.75,.1);
\draw   (.75,0) -- (-.5,0) -- (.75,1.25);
\draw   (.75,0) -- (-.4,0);
\end{tikzpicture}
\subfloat{$ K_{1,1,3} $}
\begin{tikzpicture}[scale=.75]
\draw [fill] (0,0) circle
[radius=0.1] node  [left]  {$g_1$};
\draw [fill] (1.5,0) circle
[radius=0.1] node  [right]  {$g_2$};
\draw [fill] (-.75,-1.25) circle
[radius=0.1] node  [below]  {$g_3$};
\draw [fill] (.75,-1.25) circle
[radius=0.1] node  [below]  {$g_4$};
\draw [fill] (2,-1.25) circle
[radius=0.1] node  [below]  {$g_5$};
\draw   (-.75,-1.25) -- (0,-.1);
\draw   (.75,-1.25) -- (0,-.1);
\draw   (2,-1.25) -- (0,-.1);
\draw   (1.5,0) -- (.75,-1.2);
\draw   (0,0) -- (1.4,0);
\draw   (1.5,0) -- (2,-1.2);
\draw   (1.5,0) -- (-.75,-1.2);
\end{tikzpicture}
\subfloat{$  K_{1,2,2} $}
\begin{tikzpicture}[scale=.75]
\draw [fill] (0,0) circle
[radius=0.1] node  [left]  {$g_2$};
\draw [fill] (2,0) circle
[radius=0.1] node  [right]  {$g_3$};
\draw [fill] (0,-1.25) circle
[radius=0.1] node  [below]  {$g_4$};
\draw [fill] (2,-1.25) circle
[radius=0.1] node  [below]  {$g_5$};
\draw [fill] (1,-.65) circle
[radius=0.1] node  [below]  {$g_1$};
\draw (2,-1.25)--(2,-.1);
\draw (2,-1.25)--(.1,-1.25);
\draw  (0,0)--(1.9,0);
\draw (0,0)--(0,-1.2);
\draw  (1,-.65)--(0,-.05);
\draw  (2,-.05)--(1.05,-.65);
\draw  (0,-1.25)--(.95,-.7);
\draw  (1,-.65)--(2,-1.2) ;
\end{tikzpicture}
\caption{   $H_1$ and two forbidden subgraphs of $ Z(\alpha,\beta) $  and $ U(\alpha,\beta) $ which contains it, where the degrees of all vertices of any subgraph in $ Z(\alpha,\beta) $  and $ U(\alpha,\beta) $ must be 4. }\label{707}
\end{figure}
Let $ \Gamma $ be a zero divisor graph or a unit graph on a pair of elements $ (\alpha,\beta) $ in $\mathbb{F}[G]$ for a possible torsion free group $G$ and arbitrary field $\mathbb{F}$ such that   $ supp(\alpha)=\{1,x,y,z\} $, where $ x,y,z $ are pairwise distinct non-trivial elements of $ G $, and $  | S_\alpha |=12 $. In the sequel, we give details about finding all graphs in Figure \ref{606} as forbidden subgraphs of $ \Gamma $. \\ 
 $\mathbf{K_{1,1,3}}$: Suppose that $ C $ is a cycle of length $3$ in $ \Gamma $ with the vertex set  $\mathcal{V}_C=\{g_1,g_2,g_3\}\subset supp(\beta) $  such that  $deg_\Gamma(g)=4$  for all $ g\in \mathcal{V}_C  $.  Suppose further that $ T\in \mathcal{T}(C) $ and $ T=[h_1,h'_1,h_2,h'_2,h_3,h'_3] $. 
 Remarks \ref{tuple} and \ref{tupleunit} imply that $ h_1\neq h'_1\neq h_2\neq h'_2\neq h_3\neq h'_3\neq h_1 $. Hence,  by using GAP \cite{a9}, it can be seen that there exist $126$ non-equivalent cases for $ T $. The relations of such  non-equivalent cases  are listed in Table \ref{t1}. Among these cases, the cases  $1,56,93,114,123,126$   marked by "*"s in Table \ref{t1} lead to contradictions because  each of these cases leads to $ \mathcal{G}(C) $ having a non-trivial torsion element. Hence, there are $ 120 $ cases which may lead
to the existence of $ C $ in $ \Gamma $.
\begin{longtable}{|l|l|l|l|}
\caption{Non-equivalent relations of a cycle of length $3$ on vertices of degrees $ 4 $ in $ Z(\alpha,\beta) $ and $U(\alpha,\beta) $.}\label{t1}\\
\hline
\endfirsthead
\multicolumn{4}{c}%
{{\bfseries \tablename\ \thetable{} -- continued from previous page}} \\\hline
\endhead

\hline \multicolumn{4}{|r|}{{Continued on next page}} \\ \hline
\endfoot

\hline \hline
\endlastfoot

$1.x^3=1^{*}$&
$2.x^2y=1$&
$3.x^2z=1$&
$4.x^2y^{-1}x=1$\\
$5.x^2y^{-1}z=1$&
$6.x^2z^{-1}x=1$&
$7.x^2z^{-1}y=1$&
$8.xy^2=1$\\
$9.xyz=1$&
$10.xyx^{-1}y=1$&
$11.xyx^{-1}z=1$&
$12.xyz^{-1}x=1$\\
$13.xyz^{-1}y=1$&
$14.xzy=1$&
$15.xz^2=1$&
$16.xzx^{-1}y=1$\\
$17.xzx^{-1}z=1$&
$18.xzy^{-1}x=1$&
$19.xzy^{-1}z=1$&
$20.xy^{-1}x^{-1}y=1$\\
$21.xy^{-1}x^{-1}z=1$&
$22.xy^{-2}x=1$&
$23.xy^{-2}z=1$&
$24.xy^{-1}z^{-1}x=1$\\
$25.xy^{-1}z^{-1}y=1$&
$26.xy^{-1}xy=1$&
$27.xy^{-1}xz=1$&
$28.(xy^{-1})^2x=1$\\
$29.(xy^{-1})^2z=1$&
$30.xy^{-1}xz^{-1}x=1$&
$31.xy^{-1}xz^{-1}y=1$&
$32.xy^{-1}zy=1$\\
$33.xy^{-1}z^2=1$&
$34.xy^{-1}zx^{-1}y=1$&
$35.xy^{-1}zx^{-1}z=1$&
$36.xy^{-1}zy^{-1}x=1$\\
$37.x(y^{-1}z)^2=1$&
$38.xz^{-1}x^{-1}y=1$&
$39.xz^{-1}x^{-1}z=1$&
$40.xz^{-1}y^{-1}x=1$\\
$41.xz^{-1}y^{-1}z=1$&
$42.xz^{-2}x=1$&
$43.xz^{-2}y=1$&
$44.xz^{-1}xy=1$\\
$45.xz^{-1}xz=1$&
$46.xz^{-1}xy^{-1}x=1$&
$47.xz^{-1}xy^{-1}z=1$&
$48.(xz^{-1})^2x=1$\\
$49.(xz^{-1})^2y=1$&
$50.xz^{-1}y^2=1$&
$51.xz^{-1}yz=1$&
$52.xz^{-1}yx^{-1}y=1$\\
$53.xz^{-1}yx^{-1}z=1$&
$54.xz^{-1}yz^{-1}x=1$&
$55.x(z^{-1}y)^2=1$&
$56.y^3=1^{*}$\\
$57.y^2z=1$&
$58.y^2x^{-1}y=1$&
$59.y^2x^{-1}z=1$&
$60.y^2z^{-1}x=1$\\
$61.y^2z^{-1}y=1$&
$62.yz^2=1$&
$63.yzx^{-1}y=1$&
$64.yzx^{-1}z=1$\\
$65.yzy^{-1}x=1$&
$66.yzy^{-1}z=1$&
$67.yx^{-1}yz=1$&
$68.(yx^{-1})^2y=1$\\
$69.(yx^{-1})^2z=1$&
$70.yx^{-1}yz^{-1}x=1$&
$71.yx^{-1}yz^{-1}y=1$&
$72.yx^{-1}z^2=1$\\
$73.yx^{-1}zx^{-1}y=1$&
$74.y(x^{-1}z)^2=1$&
$75.yx^{-1}zy^{-1}x=1$&
$76.yx^{-1}zy^{-1}z=1$\\
$77.yz^{-1}x^{-1}y=1$&
$78.yz^{-1}x^{-1}z=1$&
$79.yz^{-1}y^{-1}x=1$&
$80.yz^{-1}y^{-1}z=1$\\
$81.yz^{-2}x=1$&
$82.yz^{-2}y=1$&
$83.yz^{-1}xz=1$&
$84.yz^{-1}xy^{-1}x=1$\\
$85.yz^{-1}xy^{-1}z=1$&
$86.y(z^{-1}x)^2=1$&
$87.yz^{-1}xz^{-1}y=1$&
$88.yz^{-1}yz=1$\\
$89.yz^{-1}yx^{-1}y=1$&
$90.yz^{-1}yx^{-1}z=1$&
$91.(yz^{-1})^2x=1$&
$92.(yz^{-1})^2y=1$\\
$93.z^3=1^{*}$&
$94.z^2x^{-1}y=1$&
$95.z^2x^{-1}z=1$&
$96.z^2y^{-1}x=1$\\
$97.z^2y^{-1}z=1$&
$98.z(x^{-1}y)^2=1$&
$99.zx^{-1}yx^{-1}z=1$&
$100.zx^{-1}yz^{-1}x=1$\\
$101.zx^{-1}yz^{-1}y=1$&
$102.(zx^{-1})^2y=1$&
$103.(zx^{-1})^2z=1$&
$104.zx^{-1}zy^{-1}x=1$\\
$105.zx^{-1}zy^{-1}z=1$&
$106.z(y^{-1}x)^2=1$&
$107.zy^{-1}xy^{-1}z=1$&
$108.zy^{-1}xz^{-1}x=1$\\
$109.zy^{-1}xz^{-1}y=1$&
$110.zy^{-1}zx^{-1}y=1$&
$111.zy^{-1}zx^{-1}z=1$&
$112.(zy^{-1})^2x=1$\\
$113.(zy^{-1})^2z=1$&
$114.(x^{-1}y)^3=1^{*}$&
$115.(x^{-1}y)^2x^{-1}z=1$&
$116.(x^{-1}y)^2z^{-1}y=1$\\
$117.x^{-1}y(x^{-1}z)^2=1$&
$118.x^{-1}yx^{-1}zy^{-1}z=1$&
$119.x^{-1}yz^{-1}xy^{-1}z=1$&
$120.x^{-1}yz^{-1}xz^{-1}y=1$\\
$121.x^{-1}yz^{-1}yx^{-1}z=1$&
$122.x^{-1}(yz^{-1})^2y=1$&
$123.(x^{-1}z)^3=1^{*}$&
$124.(x^{-1}z)^2y^{-1}z=1$\\
$125.x^{-1}(zy^{-1})^2z=1$&
$126.(y^{-1}z)^3=1^{*}$&&\\
\hline
\end{longtable}
 Now, suppose that $ C' $ is another cycle of length $  3$  in $\Gamma$ with vertex set  $\mathcal{V}_{C'} =\{ g_1,g_2,g_4 \}\subset supp(\beta)$ (i.e.,  $ C $ and $ C' $ have exactly an edge in common (see the graph $H_1$ of Figure \ref{707})), such that  $deg_\Gamma(g)=4$  for all $ g\in \mathcal{V}_{C'}  $.  Further suppose that $ T'\in \mathcal{T}(C') $ and $ T'=[h_1,h'_1,t_2,t'_2,t_3,t'_3] $, where the first two components
are related to the common edge between these cycles.
Remarks \ref{tuple} and \ref{tupleunit} imply that $ h_1\neq h'_1\neq t_2\neq t'_2\neq t_3\neq t'_3\neq  h_1 $, $ h_{2}\neq {t}_{2} $ and $ h'_{3}\neq t'_{3} $.\\
Using GAP \cite{a9}, it can be seen that there are $1006$ different cases for the relations of two cycles of length $3$ with this structure. We examined these relations, the cases which lead to contradictions are listed  in Table \ref{t2}. Note that in Table \ref{t2} the column  labelled  by $E$ shows the structure of   $\mathcal{G}(H_1)$, which leads to a contradiction, $ r_1 $ and $ r_2 $ are the  relations of $ C $ and $ C' $, respectively, and also  we use the symbol T when  $\mathcal{G}(H_1)$ has a non-trivial torsion element.\\
\begin{center}
\begin{longtable}{|l|l||l|l|}
 \caption{The relations of two cycles of length 3 with a common  edge   on vertices of degree $ 4 $ in $ Z(\alpha,\beta) $ and $ U(\alpha,\beta) $ which lead to contradictions.}\label{t2}\\
\hline
\endfirsthead

\multicolumn{4}{c}%
{{\bfseries \tablename\ \thetable{} -- continued from previous page}} \\
\hline 
 \multicolumn{1}{|l|}{\textbf{$ r_1,r_2 $}} & \multicolumn{1}{c|}{\textbf{$E$}}&  \multicolumn{1}{|l|}{\textbf{$ r_1,r_2 $}} & \multicolumn{1}{c|}{\textbf{$E$}} \\ \hline  
\endhead

\hline \multicolumn{4}{|r|}{{Continued on next page}} \\ \hline
\endfoot

\hline \hline
\endlastfoot

$r_1,r_2$ & $E$ & $r_1,r_2$ & $ E $ \\
\hline
$1.x ^{2} y=1,x y ^{-1} x ^{-1} z=1$  &   Abelian   &  $2.x ^{2} y=1,x y ^{-2} x=1$  &  T  \\
\hline
  $3.x ^{2} y=1,x y ^{-2} z=1$  &  Abelian  &  $4.x ^{2} y=1,x y ^{-1} z ^{-1} x=1$  &  Abelian  \\
\hline
  $5.x ^{2} y=1,x y ^{-1} x z=1$  &  Abelian  &  $6.x ^{2} y=1,(x y ^{-1}) ^{2} x=1$  &  T  \\
\hline
  $7.x ^{2} y=1,(x y ^{-1}) ^{2} z=1$  &  Abelian  &  $8.x ^{2} y=1,x y ^{-1} x z ^{-1} x=1$  &  Abelian  \\
\hline
  $9.x ^{2} y=1,x y ^{-1} z y ^{-1} x=1$  &  Abelian &  $10.x ^{2} y=1,x (y ^{-1} z) ^{2}=1$  &  Abelian \\
\hline
  $11.x ^{2} y=1,x z ^{-1} x ^{-1} z=1$  &  Abelian  &  $12.x ^{2} y=1,x z ^{-1} y ^{-1} x=1$ & Abelian \\
\hline
 $13.x ^{2} y=1,x z ^{-1} y ^{-1} z=1$  &  BS$(2,-1)$  &  $14.x ^{2} y=1,x z ^{-2} x=1$  &  BS$(1,-1)$  \\
\hline
\end{longtable}
\end{center} 
Thus, there are  $458$ cases which may lead
to the existence of a subgraph isomorphic to the graph $H_1$ as Figure \ref{707} on vertices of degree $ 4 $ in $ \Gamma $.  These cases are listed in Table \ref{t222}.
\begin{longtable}{|c|l|l|c|l|l|}
 \caption{The possible relations of  two cycles of length $  3$ with a common  edge  on vertices of degree $ 4 $  in $ Z(\alpha,\beta) $ and $ U(\alpha,\beta) $.}\label{t222}\\
\hline
\endfirsthead

\multicolumn{6}{c}%
{{\bfseries \tablename\ \thetable{} -- continued from previous page}} \\
\hline \multicolumn{1}{|c|}{\textbf{$ n $}} & \multicolumn{1}{c|}{\textbf{$ r_1 $}} & \multicolumn{1}{c|}{\textbf{$ r_2 $}} & \multicolumn{1}{|c|}{\textbf{$ n $}} & \multicolumn{1}{c|}{\textbf{$r_1 $}} & \multicolumn{1}{c|}{\textbf{$ r_2 $}} \\ \hline 
\endhead

\hline \multicolumn{6}{|r|}{{Continued on next page}} \\ \hline
\endfoot

\hline \hline
\endlastfoot
$ n $&$ r_1 $&$ r_2 $&$ n $&$ r_1 $&$r_2 $\\
\hline
$1$&$x^2y=1$&$xy^{-1}z^2=1$&
$2$&$x^2y=1$&$xy^{-1}zx^{-1}z=1$\\
$3$&$x^2y=1$&$xz^{-1}yz^{-1}x=1$&
$4$&$x^2z=1$&$xy^{-1}zy^{-1}x=1$\\
$5$&$x^2z=1$&$xz^{-1}y^2=1$&
$6$&$x^2z=1$&$xz^{-1}yx^{-1}y=1$\\
$7$&$x^2y^{-1}x=1$&$xy^{-1}x^{-1}y=1$&
$8$&$x^2y^{-1}x=1$&$xz^{-2}y=1$\\
$9$&$x^2y^{-1}x=1$&$(xz^{-1})^2y=1$&
$10$&$x^2y^{-1}z=1$&$xz^{-1}y^{-1}x=1$\\
$11$&$x^2y^{-1}z=1$&$xz^{-1}xy^{-1}x=1$&
$12$&$x^2z^{-1}x=1$&$xy^{-2}z=1$\\
$13$&$x^2z^{-1}x=1$&$(xy^{-1})^2z=1$&
$14$&$x^2z^{-1}x=1$&$xz^{-1}x^{-1}z=1$\\
\hline
\end{longtable}
Now, suppose that $ C'' $ is another cycle of length 3  in $\Gamma$ with  vertex set  $\mathcal{V}_{C''}=\{ g_1,g_2,g_5\} $ (i.e., $ C,C' $ and  $ C'' $ have exactly one edge in common)  such that  $ deg_{\Gamma} (g)=4 $ for all   $ g\in \mathcal{V}_{C''} $. Further suppose that $ T''\in \mathcal{T}(C'') $ and $ T''=[h_1,h'_1,r_2,r'_2,r_3,r'_3] $, where the first two components
are related to the common edge between these cycles.
  Remarks \ref{tuple} and \ref{tupleunit} imply that $ h_1\neq h'_1\neq r_2\neq r'_2\neq r_3\neq r'_3\neq h_1 $ $, r_{2}\neq {t}_{2} $, $ r_{2}\neq {h}_{2} $ and $ r'_{3}\neq h'_{3} $ and $ r'_{3}\neq t'_{3} $. Note that replacing $ \alpha $  by $ {h_1}^{-1}\alpha $, we may assume that $ h_1=1 $.\\
It can be seen that there are $2264$ different cases for the relations of $3$ cycles of length $3$ which have exactly one edge in common. Using GAP \cite{a9}, we see that in  $2016$ cases among these $2264$ cases, a free group with generators $ x,y,z $ and the  relations of each of such cases  is finite or abelian that is a contradiction. We checked the other cases,  all of them  lead to  contradictions (see Table \ref{t3}) and therefore $ \Gamma $ does not contain any subgraph isomorphic to  $ K_{1,1,3} $  on vertices of degree $ 4 $. Note that  in Table \ref{t3} the column   marked  by $E$ shows the structure of   $G(K_{1,1,3})$,  which leads to a contradiction, $ r_1,\; r_2 $ and $ r_3 $ are the relations of $ C'', \; C $ and $ C' $, respectively, and also  we use the symbol T when  $G(K_{1,1,3})$ has a non-trivial torsion element.
\begin{longtable}{|l|l|l|l||l|}
\caption{The relations of $ 3 $ cycles of length $ 3 $ which have exactly one edge in common.}\label{t3}\\
\hline
\multicolumn{1}{|l|}{\textbf{$ n $}} & \multicolumn{1}{c|}{\textbf{$r_1$}} & \multicolumn{1}{c|}{\textbf{$r_2$}}& \multicolumn{1}{c|}{\textbf{$r_3$}}& \multicolumn{1}{c|}{\textbf{$E$}} \\ \hline 
\endfirsthead

\multicolumn{5}{c}%
{{\bfseries \tablename\ \thetable{} -- continued from previous page}} \\
\hline \multicolumn{1}{|c|}{\textbf{$ n $}} & \multicolumn{1}{c|}{\textbf{$r_1$}} & \multicolumn{1}{c|}{\textbf{$r_2$}}& \multicolumn{1}{c|}{\textbf{$r_3$}}& \multicolumn{1}{c|}{\textbf{$E$}}   \\ \hline 
\endhead

\hline \multicolumn{5}{|r|}{{Continued on next page}} \\ \hline
\endfoot

\hline \hline
\endlastfoot

\hline
$1$&$x^2y=1$&$xz^{-1}xz=1$&$xy^{-2}x=1$&T\\
\hline
$2$&$x^2y=1$&$xz^{-1}yz=1$&$xy^{-2}x=1$&T\\
\hline
$3$&$x^2y=1$&$xz^{-1}y^{-1}z=1$&$(xy^{-1})^2x=1$&T\\
\hline
$4$&$x^2y=1$&$xz^{-1}xz=1$&$(xy^{-1})^2x=1$&T\\
\hline
$5$&$x^2y=1$&$xz^{-1}yz=1$&$(xy^{-1})^2x=1$&T\\
\hline
$6$&$x^2y=1$&$xz^{-1}yx^{-1}z=1$&$(xy^{-1})^2x=1$&T\\
\hline
$7$&$x^2z=1$&$xz^{-2}x=1$&$xy^{-1}x^{-1}y=1$&T\\
\hline
$8$&$x^2z=1$&$(xz^{-1})^2x=1$&$xy^{-1}x^{-1}y=1$&T\\
\hline
$9$&$x^2z=1$&$xz^{-1}yz^{-1}x=1$&$xy^{-1}x^{-1}y=1$&Abelian\\
\hline
$10$&$x^2z=1$&$(xz^{-1})^2x=1$&$xy^{-1}z^{-1}y=1$&T\\
\hline
$11$&$x^2z=1$&$xz^{-2}x=1$&$xy^{-1}xy=1$&T\\
\hline
$12$&$x^2z=1$&$(xz^{-1})^2x=1$&$xy^{-1}xy=1$&T\\
\hline
$13$&$x^2z=1$&$xz^{-2}x=1$&$xy^{-1}zy=1$&T\\
\hline
$14$&$x^2z=1$&$(xz^{-1})^2x=1$&$xy^{-1}zy=1$&T\\
\hline
$15$&$x^2y^{-1}x=1$&$xz^{-1}y^{-1}z=1$&$xy^{-1}x^{-1}y=1$&BS$(3,1)$\\
\hline
$16$&$x^2y^{-1}x=1$&$xz^{-1}xz=1$&$xy^{-1}x^{-1}y=1$&BS$(1,-1)$\\
\hline
$17$&$x^2y^{-1}x=1$&$xz^{-1}xy^{-1}z=1$&$xy^{-1}x^{-1}y=1$&BS$(2,1)$\\
\hline
$18$&$x^2y^{-1}x=1$&$xz^{-1}yz=1$&$xy^{-1}x^{-1}y=1$&BS$(3,-1)$\\
\hline
$19$&$x^2y^{-1}x=1$&$xz^{-1}yx^{-1}z=1$&$xy^{-1}x^{-1}y=1$&BS$(2,-1)$\\
\hline
$20$&$x^2y^{-1}x=1$&$xz^{-1}x^{-1}z=1$&$xy^{-1}xz^{-1}y=1$&Abelian\\
\hline
$21$&$x^2y^{-1}z=1$&$xz^{-2}x=1$&$xy^{-1}x^{-1}y=1$&Abelian\\
\hline
$22$&$x^2y^{-1}z=1$&$(xz^{-1})^2x=1$&$xy^{-1}x^{-1}y=1$&Abelian\\
\hline
$23$&$x^2y^{-1}z=1$&$xz^{-1}yz^{-1}x=1$&$xy^{-1}x^{-1}y=1$&Abelian\\
\hline
$24$&$x^2y^{-1}z=1$&$x(z^{-1}y)^2=1$&$xy^{-1}z^{-1}x=1$&T\\
\hline
$25$&$x^2y^{-1}z=1$&$xz^{-2}x=1$&$xy^{-1}z^{-1}y=1$&$x=z$\\
\hline
$26$&$x^2z^{-1}x=1$&$xz^{-1}x^{-1}z=1$&$xy^{-1}x^{-1}y=1$&Abelian\\
\hline
$27$&$x^2z^{-1}x=1$&$xz^{-1}y^{-1}z=1$&$xy^{-1}x^{-1}y=1$&Abelian\\
\hline
$28$&$x^2z^{-1}x=1$&$xz^{-1}x^{-1}z=1$&$xy^{-1}z^{-1}y=1$&BS$(3,1)$\\
\hline
$29$&$x^2z^{-1}x=1$&$xz^{-1}x^{-1}z=1$&$xy^{-1}xy=1$&BS$(1,-1)$\\
\hline
$30$&$x^2z^{-1}x=1$&$xz^{-1}xz=1$&$xy^{-1}xy=1$&BS$(1,-1)$\\
\hline
$31$&$x^2z^{-1}x=1$&$xz^{-2}y=1$&$xy^{-1}xz=1$&Abelian\\
\hline
$32$&$x^2z^{-1}x=1$&$xz^{-1}x^{-1}z=1$&$xy^{-1}xz^{-1}y=1$&BS$(2,1)$\\
\hline
$33$&$x^2z^{-1}x=1$&$xz^{-1}xz=1$&$xy^{-1}xz^{-1}y=1$&T\\
\hline
$34$&$x^2z^{-1}x=1$&$xz^{-1}x^{-1}z=1$&$xy^{-1}zy=1$&BS$(3,-1)$\\
\hline
$35$&$x^2z^{-1}x=1$&$xz^{-1}xz=1$&$xy^{-1}zy=1$&T\\
\hline
$36$&$x^2z^{-1}x=1$&$xz^{-1}x^{-1}z=1$&$xy^{-1}zx^{-1}y=1$&BS$(2,-1)$\\
\hline
$37$&$x^2z^{-1}x=1$&$xz^{-1}xz=1$&$xy^{-1}zx^{-1}y=1$&T\\
\hline
$38$&$x^2z^{-1}y=1$&$xz^{-1}x^{-1}z=1$&$xy^{-2}x=1$&BS$(1,-1)$\\
\hline
$39$&$x^2z^{-1}y=1$&$xz^{-1}y^{-1}x=1$&$x(y^{-1}z)^2=1$&T\\
\hline
$40$&$xy^2=1$&$xz^{-1}x^{-1}z=1$&$xy^{-2}x=1$&T\\
\hline
$41$&$xy^2=1$&$xz^{-1}xz=1$&$xy^{-2}x=1$&T\\
\hline
$42$&$xy^2=1$&$xz^{-1}yz=1$&$xy^{-2}x=1$&T\\
\hline
$43$&$xy^2=1$&$xz^{-1}x^{-1}z=1$&$(xy^{-1})^2x=1$&T\\
\hline
$44$&$xy^2=1$&$xz^{-1}xz=1$&$(xy^{-1})^2x=1$&T\\
\hline
$45$&$xyz=1$&$xz^{-2}x=1$&$xy^{-1}x^{-1}y=1$&Abelian\\
\hline
$46$&$xyz=1$&$(xz^{-1})^2x=1$&$xy^{-1}x^{-1}y=1$&Abelian\\
\hline
$47$&$xyz=1$&$xz^{-2}x=1$&$xy^{-1}z^{-1}y=1$&BS$(1,-1)$\\
\hline
$48$&$xyz=1$&$(xz^{-1})^2y=1$&$xy^{-1}xz^{-1}x=1$&T\\
\hline
$49$&$xyz=1$&$xz^{-1}y^2=1$&$xy^{-1}xz^{-1}x=1$&T\\
\hline
$50$&$xyx^{-1}y=1$&$xz^{-1}y^{-1}x=1$&$xy^{-1}x^{-1}z=1$&BS$(1,-1)$\\
\hline
$51$&$xyx^{-1}y=1$&$xz^{-1}x^{-1}z=1$&$xy^{-2}x=1$&T\\
\hline
$52$&$xyx^{-1}y=1$&$xz^{-1}y^{-1}z=1$&$xy^{-2}x=1$&T\\
\hline
$53$&$xyx^{-1}y=1$&$xz^{-1}xz=1$&$xy^{-2}x=1$&T\\
\hline
$54$&$xyx^{-1}y=1$&$xz^{-1}xy^{-1}z=1$&$xy^{-2}x=1$&T\\
\hline
$55$&$xyx^{-1}y=1$&$xz^{-1}yz=1$&$xy^{-2}x=1$&T\\
\hline
$56$&$xyx^{-1}y=1$&$xz^{-1}yx^{-1}z=1$&$xy^{-2}x=1$&T\\
\hline
$57$&$xyx^{-1}y=1$&$xz^{-2}x=1$&$xy^{-2}z=1$&BS$(1,-1)$\\
\hline
$58$&$xyx^{-1}y=1$&$xz^{-1}yx^{-1}z=1$&$xy^{-1}z^{-1}x=1$&BS$(1,-1)$\\
\hline
$59$&$xyx^{-1}y=1$&$xz^{-1}x^{-1}z=1$&$(xy^{-1})^2x=1$&T\\
\hline
$60$&$xyx^{-1}y=1$&$xz^{-1}xz=1$&$(xy^{-1})^2x=1$&T\\
\hline
$61$&$xyx^{-1}y=1$&$xz^{-2}x=1$&$(xy^{-1})^2z=1$&BS$(1,-1)$\\
\hline
$62$&$xyx^{-1}y=1$&$xz^{-1}yx^{-1}z=1$&$xy^{-1}xz^{-1}x=1$&BS$(1,-1)$\\
\hline
$63$&$xyx^{-1}y=1$&$xz^{-1}x^{-1}z=1$&$xy^{-1}zy^{-1}x=1$&BS$(1,-1)$\\
\hline
$64$&$xyx^{-1}y=1$&$xz^{-1}yx^{-1}z=1$&$xy^{-1}zy^{-1}x=1$&BS$(1,-1)$\\
\hline
$65$&$xyx^{-1}z=1$&$xz^{-2}x=1$&$xy^{-1}z^{-1}y=1$&BS$(1,-1)$\\
\hline
$66$&$xyz^{-1}x=1$&$xz^{-1}x^{-1}z=1$&$xy^{-1}x^{-1}y=1$&Abelian\\
\hline
$67$&$xyz^{-1}x=1$&$xz^{-1}x^{-1}y=1$&$xy^{-1}x^{-1}z=1$&BS$(1,-1)$\\
\hline
$68$&$xyz^{-1}x=1$&$x(z^{-1}y)^2=1$&$xy^{-1}x^{-1}z=1$&BS$(1,-1)$\\
\hline
$69$&$xyz^{-1}x=1$&$xz^{-1}yz=1$&$xy^{-1}z^{-1}y=1$&$x=y^{-1}$\\
\hline
$70$&$xyz^{-1}x=1$&$xz^{-1}xz=1$&$xy^{-1}xy=1$&BS$(1,-1)$\\
\hline
$71$&$xyz^{-1}x=1$&$x(z^{-1}y)^2=1$&$xy^{-1}xz=1$&BS$(1,2)$\\
\hline
$72$&$xyz^{-1}x=1$&$xz^{-1}yx^{-1}z=1$&$xy^{-1}xz^{-1}y=1$&T\\
\hline
$73$&$xyz^{-1}x$&$xz^{-1}xy=1$&$xy^{-1}z^2=1$&Abelian\\
\hline
$74$&$xyz^{-1}x=1$&$xz^{-1}xy=1$&$xy^{-1}zx^{-1}z=1$&Abelian\\
\hline
$75$&$xyz^{-1}x=1$&$xz^{-1}x^{-1}y=1$&$x(y^{-1}z)^2=1$&BS$(1,-2)$\\
\hline
$76$&$xyz^{-1}y=1$&$xz^{-1}x^{-1}z=1$&$xy^{-2}x=1$&Abelian\\
\hline
$77$&$xyz^{-1}y=1$&$xz^{-1}y^{-1}z=1$&$xy^{-2}x=1$&BS$(1,-1)$\\
\hline
$78$&$xyz^{-1}y=1$&$xz^{-1}x^{-1}z=1$&$xy^{-1}zy^{-1}x=1$&T\\
\hline
$79$&$xyz^{-1}y=1$&$xz^{-1}xz=1$&$xy^{-1}zy^{-1}x=1$&T\\
\hline
$80$&$xzy=1$&$xz^{-1}y^{-1}z=1$&$xy^{-2}x=1$&BS$(1,-1)$\\
\hline
$81$&$xzy=1$&$xz^{-1}xy^{-1}x=1$&$(xy^{-1})^2z=1$&T\\
\hline
$82$&$xzy=1$&$xz^{-1}xy^{-1}x=1$&$xy^{-1}z^2=1$&T\\
\hline
$83$&$xz^2=1$&$xz^{-2}x=1$&$xy^{-1}x^{-1}y=1$&T\\
\hline
$84$&$xz^2=1$&$(xz^{-1})^2x=1$&$xy^{-1}x^{-1}y=1$&T\\
\hline
$85$&$xz^2=1$&$xz^{-1}yz^{-1}x=1$&$xy^{-1}x^{-1}y=1$&Abelian\\
\hline
$86$&$xz^2=1$&$xz^{-2}x=1$&$xy^{-1}z^{-1}y=1$&T\\
\hline
$87$&$xz^2=1$&$(xz^{-1})^2x=1$&$xy^{-1}z^{-1}y=1$&T\\
\hline
$88$&$xz^2=1$&$xz^{-2}x=1$&$xy^{-1}xy=1$&T\\
\hline
$89$&$xz^2=1$&$(xz^{-1})^2x=1$&$xy^{-1}xy=1$&T\\
\hline
$90$&$xz^2=1$&$xz^{-2}x=1$&$xy^{-1}zy=1$&T\\
\hline
$91$&$xzx^{-1}y=1$&$xz^{-1}y^{-1}z=1$&$xy^{-2}x=1$&BS$(1,-1)$\\
\hline
$92$&$xzx^{-1}z=1$&$xz^{-2}x=1$&$xy^{-1}x^{-1}y=1$&T\\
\hline
$93$&$xzx^{-1}z=1$&$(xz^{-1})^2x=1$&$xy^{-1}x^{-1}y=1$&T\\
\hline
$94$&$xzx^{-1}z=1$&$xz^{-1}yz^{-1}x=1$&$xy^{-1}x^{-1}y=1$&BS$(1,-1)$\\
\hline
$95$&$xzx^{-1}z=1$&$xz^{-2}y=1$&$xy^{-2}x=1$&BS$(1,-1)$\\
\hline
$96$&$xzx^{-1}z=1$&$(xz^{-1})^2y=1$&$xy^{-2}x=1$&BS$(1,-1)$\\
\hline
$97$&$xzx^{-1}z=1$&$xz^{-1}x^{-1}y=1$&$xy^{-1}z^{-1}x=1$&BS$(1,-1)$\\
\hline
$98$&$xzx^{-1}z=1$&$xz^{-2}x=1$&$xy^{-1}z^{-1}y=1$&T\\
\hline
$99$&$xzx^{-1}z=1$&$xz^{-2}x=1$&$xy^{-1}xy=1$&T\\
\hline
$100$&$xzx^{-1}z=1$&$(xz^{-1})^2x=1$&$xy^{-1}xy=1$&T\\
\hline
$101$&$xzx^{-1}z=1$&$xz^{-2}x=1$&$xy^{-1}xz^{-1}y=1$&T\\
\hline
$102$&$xzx^{-1}z=1$&$xz^{-2}x=1$&$xy^{-1}zy=1$&T\\
\hline
$103$&$xzx^{-1}z=1$&$(xz^{-1})^2x=1$&$xy^{-1}zy=1$&T\\
\hline
$104$&$xzx^{-1}z=1$&$xz^{-1}y^{-1}x=1$&$xy^{-1}zx^{-1}y=1$&BS$(1,-1)$\\
\hline
$105$&$xzx^{-1}z=1$&$xz^{-2}x=1$&$xy^{-1}zx^{-1}y=1$&T\\
\hline
$106$&$xzx^{-1}z=1$&$xz^{-1}xy^{-1}x=1$&$xy^{-1}zx^{-1}y=1$&BS$(1,-1)$\\
\hline
$107$&$xzx^{-1}z=1$&$xz^{-1}yz^{-1}x=1$&$xy^{-1}zx^{-1}y=1$&BS$(1,-1)$\\
\hline
$108$&$xzy^{-1}x=1$&$xz^{-1}x^{-1}z=1$&$xy^{-1}x^{-1}y=1$&Abelian\\
\hline
$109$&$xzy^{-1}x=1$&$xz^{-1}x^{-1}y=1$&$xy^{-1}x^{-1}z=1$&BS$(1,-1)$\\
\hline
$110$&$xzy^{-1}x=1$&$x(z^{-1}y)^2=1$&$xy^{-1}x^{-1}z=1$&BS$(1,-2)$\\
\hline
$111$&$xzy^{-1}x=1$&$xz^{-1}xz=1$&$xy^{-1}xy=1$&Abelian\\
\hline
$112$&$xzy^{-1}x=1$&$xz^{-2}y=1$&$xy^{-1}xz=1$&BS$(1,-1)$\\
\hline
$113$&$xzy^{-1}x=1$&$(xz^{-1})^2y=1$&$xy^{-1}xz=1$&Abelian\\
\hline
$114$&$xzy^{-1}x=1$&$xz^{-1}x^{-1}z=1$&$xy^{-1}xz^{-1}y=1$&Abelian\\
\hline
$115$&$xzy^{-1}x=1$&$xz^{-1}y^{-1}z=1$&$xy^{-1}zy=1$&$x=y$\\
\hline
$116$&$xzy^{-1}x=1$&$xz^{-1}x^{-1}z=1$&$xy^{-1}zx^{-1}y=1$&Abelian\\
\hline
$117$&$xzy^{-1}x=1$&$xz^{-1}xy^{-1}z=1$&$xy^{-1}zx^{-1}y=1$&T\\
\hline
$118$&$xzy^{-1}x=1$&$xz^{-1}x^{-1}y=1$&$x(y^{-1}z)^2=1$&BS$(1,-1)$\\
\hline
$119$&$xzy^{-1}x=1$&$xz^{-1}xy=1$&$x(y^{-1}z)^2=1$&BS$(1,2)$\\
\hline
$120$&$xzy^{-1}z=1$&$xz^{-2}x=1$&$xy^{-1}x^{-1}y=1$&BS$(1,-1)$\\
\hline
$121$&$xzy^{-1}z=1$&$(xz^{-1})^2x=1$&$xy^{-1}x^{-1}y=1$&Abelian\\
\hline
$122$&$xzy^{-1}z=1$&$xz^{-1}yz^{-1}x=1$&$xy^{-1}x^{-1}y=1$&T\\
\hline
$123$&$xzy^{-1}z=1$&$xz^{-2}x=1$&$xy^{-1}z^{-1}y=1$&BS$(1,-1)$\\
\hline
$124$&$xzy^{-1}z=1$&$xz^{-1}yz^{-1}x=1$&$xy^{-1}xy=1$&T\\
\hline
$125$&$xy^{-1}x^{-1}y=1$&$xz^{-1}xy^{-1}x=1$&$x^2y^{-1}z=1$&Abelian\\
\hline
$126$&$xy^{-1}x^{-1}y=1$&$xz^{-1}x^{-1}z=1$&$x^2z^{-1}x=1$&Abelian\\
\hline
$127$&$xy^{-1}x^{-1}y=1$&$xz^{-1}xy^{-1}z=1$&$xzy^{-1}x=1$&Abelian\\
\hline
$128$&$xy^{-1}x^{-1}y=1$&$xz^{-1}xy^{-1}x=1$&$xzy^{-1}z=1$&Abelian\\
\hline
$129$&$xy^{-2}x=1$&$xz^{-1}x^{-1}z=1$&$xy^2=1$&T\\
\hline
$130$&$xy^{-2}x=1$&$xz^{-1}xz=1$&$xy^2=1$&T\\
\hline
$131$&$xy^{-2}x=1$&$xz^{-1}xy^{-1}z=1$&$xy^2=1$&T\\
\hline
$132$&$xy^{-2}x=1$&$xz^{-1}yx^{-1}z=1$&$xy^2=1$&T\\
\hline
$133$&$xy^{-2}x=1$&$xz^{-1}x^{-1}z=1$&$xyx^{-1}y=1$&T\\
\hline
$134$&$xy^{-2}x=1$&$xz^{-1}y^{-1}z=1$&$xyx^{-1}y=1$&T\\
\hline
$135$&$xy^{-2}x=1$&$xz^{-1}xz=1$&$xyx^{-1}y=1$&T\\
\hline
$136$&$xy^{-2}x=1$&$xz^{-1}xy^{-1}z=1$&$xyx^{-1}y=1$&T\\
\hline
$137$&$xy^{-2}x=1$&$xz^{-1}yz=1$&$xyx^{-1}y=1$&T\\
\hline
$138$&$xy^{-2}x=1$&$xz^{-1}yx^{-1}z=1$&$xyx^{-1}y=1$&T\\
\hline
$139$&$xy^{-2}x=1$&$xz^{-1}x^{-1}y=1$&$xyx^{-1}z=1$&BS$(1,-1)$\\
\hline
$140$&$xy^{-2}x=1$&$xz^{-1}x^{-1}z=1$&$xyz^{-1}y=1$&BS$(1,-1)$\\
\hline
$141$&$xy^{-2}x=1$&$xz^{-1}y^{-1}z=1$&$xyz^{-1}y=1$&BS$(1,-1)$\\
\hline
$142$&$xy^{-2}x=1$&$xz^{-1}y^{-1}z=1$&$xzy=1$&BS$(1,-1)$\\
\hline
$143$&$xy^{-2}x=1$&$xz^{-1}y^{-1}z=1$&$xzx^{-1}y=1$&BS$(1,-1)$\\
\hline
$144$&$xy^{-2}x=1$&$xz^{-1}y^2=1$&$xzx^{-1}z=1$&BS$(1,-1)$\\
\hline
$145$&$xy^{-2}x=1$&$xz^{-1}yx^{-1}y=1$&$xzx^{-1}z=1$&BS$(1,-1)$\\
\hline
$146$&$xy^{-2}z=1$&$xz^{-2}x=1$&$xyx^{-1}y=1$&BS$(1,-1)$\\
\hline
$147$&$xy^{-1}z^{-1}x=1$&$xz^{-1}yx^{-1}z=1$&$xyx^{-1}y=1$&BS$(1,-1)$\\
\hline
$148$&$xy^{-1}z^{-1}y=1$&$xz^{-1}x^{-1}z=1$&$x^2z^{-1}x=1$&BS$(3,1)$\\
\hline
$149$&$xy^{-1}z^{-1}y=1$&$xz^{-1}yz=1$&$xyz^{-1}x=1$&$x=y^{-1}$\\
\hline
$150$&$xy^{-1}xy$&$xz^{-1}x^{-1}z=1$&$x^2z^{-1}x=1$&BS$(1,-1)$\\
\hline
$151$&$xy^{-1}xz=1$&$xz^{-2}y=1$&$xzy^{-1}x=1$&Abelian\\
\hline
$152$&$xy^{-1}xz=1$&$(xz^{-1})^2y=1$&$xzy^{-1}x=1$&Abelian\\
\hline
$153$&$xy^{-1}xz=1$&$xz^{-1}y^2=1$&$xzy^{-1}x=1$&Abelian\\
\hline
$154$&$xy^{-1}xz=1$&$xz^{-1}yx^{-1}y=1$&$xzy^{-1}x=1$&Abelian\\
\hline
$155$&$(xy^{-1})^2x=1$&$xz^{-1}x^{-1}z=1$&$xy^2=1$&T\\
\hline
$156$&$(xy^{-1})^2x=1$&$xz^{-1}xz=1$&$xy^2=1$&T\\
\hline
$157$&$(xy^{-1})^2x=1$&$xz^{-1}x^{-1}z=1$&$xyx^{-1}y=1$&T\\
\hline
$158$&$(xy^{-1})^2x=1$&$xz^{-1}xz=1$&$xyx^{-1}y=1$&T\\
\hline
$159$&$(xy^{-1})^2x=1$&$xz^{-1}xy^{-1}z=1$&$xyx^{-1}y=1$&T\\
\hline
$160$&$(xy^{-1})^2z=1$&$xz^{-2}x=1$&$xyx^{-1}y=1$&BS$(1,-1)$\\
\hline
$161$&$(xy^{-1})^2z=1$&$xz^{-1}xy^{-1}x=1$&$xzy=1$&T\\
\hline
$162$&$xy^{-1}xz^{-1}x=1$&$(xz^{-1})^2y=1$&$xyz=1$&T\\
\hline
$163$&$xy^{-1}xz^{-1}x=1$&$xz^{-1}y^2=1$&$xyz=1$&T\\
\hline
$164$&$xy^{-1}xz^{-1}x=1$&$xz^{-1}yx^{-1}z=1$&$xyx^{-1}y=1$&BS$(1,-1)$\\
\hline
$165$&$xy^{-1}xz^{-1}y=1$&$xz^{-1}x^{-1}z=1$&$x^2z^{-1}x=1$&BS$(2,1)$\\
\hline
$166$&$xy^{-1}xz^{-1}y=1$&$xz^{-1}yx^{-1}z=1$&$xyz^{-1}x=1$&T\\
\hline
$167$&$xy^{-1}zy$&$xz^{-1}x^{-1}z=1$&$x^2z^{-1}x=1$&BS$(3,-1)$\\
\hline
$168$&$xy^{-1}zy=1$&$xz^{-1}y^{-1}z=1$&$xzy^{-1}x=1$&$x=z^{-1}$\\
\hline
$169$&$xy^{-1}z^2=1$&$xz^{-2}y=1$&$x^2y^{-1}x=1$&T\\
\hline
$170$&$xy^{-1}z^2=1$&$xz^{-2}x=1$&$xyx^{-1}y=1$&T\\
\hline
$171$&$xy^{-1}z^2=1$&$xz^{-1}xy^{-1}x=1$&$xzy=1$&T\\
\hline
$172$&$xy^{-1}zx^{-1}y=1$&$xz^{-1}y^{-1}x=1$&$x^2y^{-1}z=1$&$y=x^{2}$\\
\hline
$173$&$xy^{-1}zx^{-1}y=1$&$xz^{-1}x^{-1}z=1$&$x^2z^{-1}x=1$&BS$(2,-1)$\\
\hline
$174$&$xy^{-1}zx^{-1}y=1$&$xz^{-1}xy^{-1}z=1$&$xzy^{-1}x=1$&T\\
\hline
$175$&$xy^{-1}zx^{-1}z=1$&$xz^{-2}x=1$&$xyx^{-1}y=1$&BS$(1,-1)$\\
\hline
$176$&$xy^{-1}zy^{-1}x=1$&$xz^{-1}x^{-1}z=1$&$xyx^{-1}y=1$&BS$(1,-1)$\\
\hline
$177$&$xy^{-1}zy^{-1}x=1$&$xz^{-1}yx^{-1}z=1$&$xyx^{-1}y=1$&BS$(1,-1)$\\
\hline
$178$&$xy^{-1}zy^{-1}x=1$&$xz^{-1}x^{-1}z=1$&$xyz^{-1}y=1$&T\\
\hline
$179$&$xy^{-1}zy^{-1}x$&$xz^{-1}xz=1$&$xyz^{-1}y=1$&T\\
\hline
$180$&$x(y^{-1}z)^2=1$&$(xz^{-1})^2y=1$&$x^2y^{-1}x=1$&Abelian\\
\hline
$181$&$xz^{-1}x^{-1}y=1$&$x(y^{-1}z)^2=1$&$xzy^{-1}x=1$&BS$(1,-1)$\\
\hline
$182$&$xz^{-1}x^{-1}z=1$&$xy^{-1}xz^{-1}y=1$&$xzy^{-1}x=1$&Abelian\\
\hline
$183$&$xz^{-1}y^{-1}x=1$&$xy^{-1}x^{-1}y=1$&$xz^2=1$&Abelian\\
\hline
$184$&$xz^{-1}y^{-1}x=1$&$xy^{-1}zx^{-1}y=1$&$xzx^{-1}z=1$&BS$(1,-1)$\\
\hline
$185$&$xz^{-1}y^{-1}z=1$&$xy^{-1}x^{-1}y=1$&$x^2y^{-1}x=1$&BS$(3,1)$\\
\hline
$186$&$xz^{-1}y^{-1}z=1$&$xy^{-1}zy=1$&$xzy^{-1}x=1$&T\\
\hline
$187$&$xz^{-2}x=1$&$xy^{-1}z^{-1}y=1$&$xyz=1$&BS$(1,-1)$\\
\hline
$188$&$xz^{-2}x=1$&$xy^{-1}z^2=1$&$xyx^{-1}y=1$&T\\
\hline
$189$&$xz^{-2}x=1$&$xy^{-1}zx^{-1}z=1$&$xyx^{-1}y=1$&BS$(1,-1)$\\
\hline
$190$&$xz^{-2}x=1$&$xy^{-1}z^{-1}y=1$&$xyx^{-1}z=1$&BS$(1,-1)$\\
\hline
$191$&$xz^{-2}x=1$&$xy^{-1}x^{-1}y=1$&$xz^2=1$&T\\
\hline
$192$&$xz^{-2}x=1$&$xy^{-1}xy=1$&$xz^2=1$&T\\
\hline
$193$&$xz^{-2}x=1$&$xy^{-1}x^{-1}z=1$&$xzx^{-1}y=1$&BS$(1,-1)$\\
\hline
$194$&$xz^{-2}x=1$&$xy^{-1}x^{-1}y=1$&$xzx^{-1}z=1$&T\\
\hline
$195$&$xz^{-2}x=1$&$xy^{-1}z^{-1}y=1$&$xzx^{-1}z=1$&T\\
\hline
$196$&$xz^{-2}x=1$&$xy^{-1}xy=1$&$xzx^{-1}z=1$&T\\
\hline
$197$&$xz^{-2}x=1$&$xy^{-1}xz^{-1}y=1$&$xzx^{-1}z=1$&T\\
\hline
$198$&$xz^{-2}x=1$&$xy^{-1}zy=1$&$xzx^{-1}z=1$&T\\
\hline
$199$&$xz^{-2}x=1$&$xy^{-1}zx^{-1}y=1$&$xzx^{-1}z=1$&T\\
\hline
$200$&$xz^{-2}x=1$&$xy^{-1}x^{-1}y=1$&$xzy^{-1}z=1$&BS$(1,-1)$\\
\hline
$201$&$xz^{-2}x=1$&$xy^{-1}z^{-1}y=1$&$xzy^{-1}z=1$&BS$(1,-1)$\\
\hline
$202$&$xz^{-2}y=1$&$xy^{-2}x=1$&$xzx^{-1}z=1$&BS$(1,-1)$\\
\hline
$203$&$xz^{-1}xy=1$&$xy^{-1}z^2=1$&$xyz^{-1}x=1$&Abelian\\
\hline
$204$&$xz^{-1}xy=1$&$xy^{-1}zx^{-1}z=1$&$xyz^{-1}x=1$&Abelian\\
\hline
$205$&$xz^{-1}xy=1$&$x(y^{-1}z)^2=1$&$xzy^{-1}x=1$&BS$(1,2)$\\
\hline
$206$&$xz^{-1}xz=1$&$xy^{-1}x^{-1}y=1$&$x^2y^{-1}x=1$&BS$(1,-1)$\\
\hline
$207$&$xz^{-1}xy^{-1}x=1$&$(xy^{-1})^2z=1$&$xzy=1$&T\\
\hline
$208$&$xz^{-1}xy^{-1}x=1$&$xy^{-1}z^2=1$&$xzy=1$&T\\
\hline
$209$&$xz^{-1}xy^{-1}x=1$&$xy^{-1}x^{-1}y=1$&$xz^2=1$&Abelian\\
\hline
$210$&$xz^{-1}xy^{-1}x=1$&$xy^{-1}x^{-1}y=1$&$xzx^{-1}z=1$&Abelian\\
\hline
$211$&$xz^{-1}xy^{-1}x=1$&$xy^{-1}zx^{-1}y=1$&$xzx^{-1}z=1$&BS$(1,-1)$\\
\hline
$212$&$xz^{-1}xy^{-1}x=1$&$xy^{-1}x^{-1}y=1$&$xzy^{-1}z=1$&Abelian\\
\hline
$213$&$xz^{-1}xy^{-1}z=1$&$xy^{-1}x^{-1}y=1$&$x^2y^{-1}x=1$&BS$(2,1)$\\
\hline
$214$&$xz^{-1}xy^{-1}z=1$&$xy^{-1}zx^{-1}y=1$&$xzy^{-1}x=1$&T\\
\hline
$215$&$(xz^{-1})^2x=1$&$xy^{-1}x^{-1}y=1$&$xz^2=1$&T\\
\hline
$216$&$(xz^{-1})^2x=1$&$xy^{-1}xy=1$&$xz^2=1$&T\\
\hline
$217$&$(xz^{-1})^2x=1$&$xy^{-1}x^{-1}y=1$&$xzx^{-1}z=1$&T\\
\hline
$218$&$(xz^{-1})^2x=1$&$xy^{-1}xy=1$&$xzx^{-1}z=1$&T\\
\hline
$219$&$(xz^{-1})^2x=1$&$xy^{-1}xz^{-1}y=1$&$xzx^{-1}z=1$&T\\
\hline
$220$&$(xz^{-1})^2y=1$&$xy^{-1}xz^{-1}x=1$&$xyz=1$&T\\
\hline
$221$&$(xz^{-1})^2y=1$&$xy^{-2}x=1$&$xzx^{-1}z=1$&BS$(1,-1)$\\
\hline
$222$&$xz^{-1}y^2=1$&$xy^{-2}z=1$&$x^2z^{-1}x=1$&T\\
\hline
$223$&$xz^{-1}y^2=1$&$xy^{-1}xz^{-1}x=1$&$xyz=1$&T\\
\hline
$224$&$xz^{-1}y^2=1$&$xy^{-2}x=1$&$xzx^{-1}z=1$&T\\
\hline
$225$&$xz^{-1}yz=1$&$xy^{-1}x^{-1}y=1$&$x^2y^{-1}x=1$&BS$(3,-1)$\\
\hline
$226$&$xz^{-1}yz=1$&$xy^{-1}z^{-1}y=1$&$xyz^{-1}x=1$&T\\
\hline
$227$&$xz^{-1}yx^{-1}y=1$&$xy^{-2}x=1$&$xzx^{-1}z=1$&BS$(1,-1)$\\
\hline
$228$&$xz^{-1}yx^{-1}z=1$&$xy^{-1}x^{-1}y=1$&$x^2y^{-1}x=1$&BS$(2,-1)$\\
\hline
$229$&$xz^{-1}yx^{-1}z=1$&$xy^{-1}z^{-1}x=1$&$x^2z^{-1}y=1$&$z=x^{2}$\\
\hline
$230$&$xz^{-1}yx^{-1}z=1$&$xy^{-1}xz^{-1}y=1$&$xyz^{-1}x=1$&T\\
\hline
$231$&$xz^{-1}yz^{-1}x=1$&$xy^{-1}x^{-1}y=1$&$xz^2=1$&Abelian\\
\hline
$232$&$xz^{-1}yz^{-1}x=1$&$xy^{-1}x^{-1}y=1$&$xzx^{-1}z=1$&BS$(1,-1)$\\
\hline
$233$&$xz^{-1}yz^{-1}x=1$&$xy^{-1}zx^{-1}y=1$&$xzx^{-1}z=1$&BS$(1,-1)$\\
\hline
$234$&$xz^{-1}yz^{-1}x=1$&$xy^{-1}x^{-1}y=1$&$xzy^{-1}z=1$&T\\
\hline
$235$&$xz^{-1}yz^{-1}x=1$&$xy^{-1}xy=1$&$xzy^{-1}z=1$&T\\
\hline
$236$&$y^2z^{-1}x=1$&$yz^{-1}x^{-1}y=1$&$y(x^{-1}z)^2=1$&T\\
\hline
$237$&$y^2z^{-1}y=1$&$yz^{-1}y^{-1}z=1$&$yx^{-1}yz^{-1}x=1$&BS$(2,1)$\\
\hline
$238$&$y^2z^{-1}y=1$&$yz^{-1}yz=1$&$yx^{-1}yz^{-1}x=1$&T\\
\hline
$239$&$y^2z^{-1}y=1$&$yz^{-1}y^{-1}z=1$&$yx^{-1}zy^{-1}x=1$&BS$(2,-1)$\\
\hline
$240$&$y^2z^{-1}y=1$&$yz^{-1}yz=1$&$yx^{-1}zy^{-1}x=1$&T\\
\hline
$241$&$yzx^{-1}y=1$&$yz^{-1}y^{-1}x=1$&$y(x^{-1}z)^2=1$&BS$(1,-1)$\\
\hline
$242$&$yzx^{-1}y=1$&$yz^{-1}yx^{-1}z=1$&$yx^{-1}zy^{-1}x=1$&T\\
\hline
$243$&$yzy^{-1}z=1$&$yz^{-2}y=1$&$yx^{-1}yz^{-1}x=1$&T\\
\hline
$244$&$yzy^{-1}z=1$&$yz^{-1}x^{-1}y=1$&$yx^{-1}zy^{-1}x=1$&BS$(1,-1)$\\
\hline
$245$&$yzy^{-1}z=1$&$yz^{-2}y=1$&$yx^{-1}zy^{-1}x=1$&T\\
\hline
$246$&$yzy^{-1}z=1$&$yz^{-1}xz^{-1}y=1$&$yx^{-1}zy^{-1}x=1$&BS$(1,-1)$\\
\hline
$247$&$yzy^{-1}z=1$&$yz^{-1}yx^{-1}y=1$&$yx^{-1}zy^{-1}x=1$&BS$(1,-1)$\\
\hline
$248$&$yx^{-1}yz^{-1}x=1$&$yz^{-1}y^{-1}z=1$&$y^2z^{-1}y=1$&BS$(2,1)$\\
\hline
$249$&$yx^{-1}z^2=1$&$yz^{-2}x=1$&$y^2x^{-1}y=1$&T\\
\hline
$250$&$yx^{-1}zy^{-1}x=1$&$yz^{-1}y^{-1}z=1$&$y^2z^{-1}y=1$&BS$(2,-1)$\\
\hline
$251$&$yx^{-1}zy^{-1}x=1$&$yz^{-1}yx^{-1}z=1$&$yzx^{-1}y=1$&T\\
\hline
$252$&$yz^{-1}x^{-1}y=1$&$yx^{-1}zy^{-1}x=1$&$yzy^{-1}z=1$&BS$(1,-1)$\\
\hline
$253$&$yz^{-2}y=1$&$yx^{-1}yz^{-1}x=1$&$yzy^{-1}z=1$&T\\
\hline
$254$&$yz^{-2}y=1$&$yx^{-1}zy^{-1}x=1$&$yzy^{-1}z=1$&T\\
\hline
$255$&$yz^{-1}xz^{-1}y=1$&$yx^{-1}zy^{-1}x=1$&$yzy^{-1}z=1$&BS$(1,-1)$\\
\hline
$256$&$yz^{-1}yx^{-1}y=1$&$yx^{-1}zy^{-1}x=1$&$yzy^{-1}z=1$&BS$(1,-1)$\\
\hline
$257$&$yz^{-1}yx^{-1}z=1$&$yx^{-1}zy^{-1}x=1$&$yzx^{-1}y=1$&T\\
\hline
$258$&$(yz^{-1})^2y=1$&$yx^{-1}yz^{-1}x=1$&$yzy^{-1}z=1$&T\\
\hline
\end{longtable}
$\mathbf{K_{1,2,2}}$: With the same argument such as about $ K_{1,1,3} $, we see that there are $123$ different cases for the relations of $4$ cycles of length $3$  with the structure of  $ K_{1,2,2} $ as Figure \ref{707} on vertices of degree $ 4 $ in $ \Gamma $. Using GAP \cite{a9}, a free group with generators $ x,y,z $ and the relations of these $ 4 $ cycles which are between $122$ cases of these $123$ cases  is finite or abelian, that is a
contradiction.   In the following it can be seen that another case leads to a contradiction and so, the graph $ \Gamma $ contains no subgraph isomorphic to the graph $ K_{1,2,2} $ on vertices of degree $ 4 $ in $ \Gamma $. Another case is as follows:\\
$r_1:y^{2}z^{-1}y=1,\;\;r_2:yz^{-1}y^{-1}z=1,\;\;r_3:x^{-1}yz^{-1}xz^{-1}y=1,\;\;r_4:x^{-1}yx^{-1}zy^{-1}z=1.$ \\
$ r_1 $ and $ r_3 $ imply that $ (y^{-2})^x=y^{2} $ (a), $ r_1 $ and $ r_4 $ imply that $ x^{-1}yx^{-1}y^5=1 $ (b) and $ r_3 $ and $ r_4 $ imply that $ x^{-1}y^{-1}x^{-1}y^3=1 $ (c). By (a) and (c), $ x^{-1}yx^{-1}y^{-3}=1 $ (d). (d) and (b) lead to  $ y^{8}=1 $, a contradiction. \\
\textbf{$ \mathbf{C_4} $:}
Suppose that $ C $ is a cycle of length $4$ in $ \Gamma $ with the vertex set  $\mathcal{V}_C=\{g_1,g_2,g_3,g_4\}\subset supp(\beta) $  such that  $deg_\Gamma(g)=4$  for all $ g\in \mathcal{V}_C  $.  Suppose further that $ T\in \mathcal{T}(C) $ and $ T=[h_1,h'_1,h_2,h'_2,h_3,h'_3,h_4,h'_4] $. 
 Remarks \ref{tuple} and \ref{tupleunit} imply that $ h_1\neq h'_1\neq h_2\neq h'_2\neq h_3\neq h'_3\neq h_4\neq h'_4\neq h_1 $. Thus,  by using GAP \cite{a9}, it can be seen that there exist $834$ non-equivalent cases for $ T $. The relations of such  non-equivalent cases  are listed in Table \ref{tc4}. Among these cases, cases  $1,455,698,799,829,834 $ marked by "*"s in Table \ref{tc4} lead to contradictions because  each of these cases leads to $ G(C) $ having a non-trivial torsion element and also according to Remark \ref{one}, the cases $56,110,211,263,362,412,492,528,563,633,666,719,745,764,782,808,$  $822,832$ marked by "*"s in Table \ref{tc4} lead to contradictions. Hence, there are $ 810 $ cases which may lead
to the existence of $ C $ in $ \Gamma $.
 \begin{center}

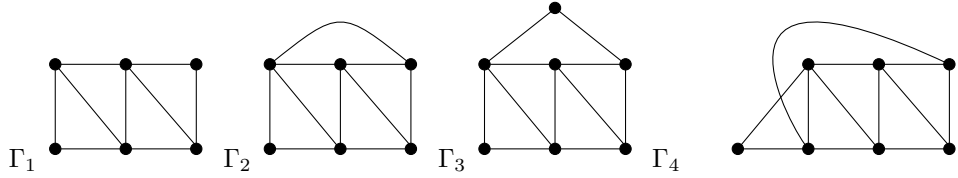
\caption{ The graph $\Gamma_1$ and some forbidden subgraphs which contains it, where the degrees of all vertices of any subgraph  in  $ Z(\alpha,\beta) $ and $ U(\alpha,\beta) $ must be $ 4 $. }\label{gamma}
\end{figure}
 $\mathbf{\Gamma_1}$ \textbf{of Figure \ref{gamma}:}  Taking into account the relations from Table \ref{t222}, it can be seen that  there are  $9623$ different cases for the relations of $4$ cycles of length $3$ on vertices of degree $ 4 $ in $ \Gamma $ with the structure of the graph $ \Gamma_1 $ as Figure \ref{gamma}. Using GAP \cite{a9}, a free group with generators $ x,y,z $ and  the  relations of these $ 4 $ cycles which are between $9380$ cases of these $9623$ cases  is finite or abelian, that is a
contradiction. Hence, there are $ 243 $ cases which may lead
to the existence of $ \Gamma_1$ in $ \Gamma $. We checked these such cases. In Table \ref{tgamma1}, it can be seen that $ 166 $ cases among these $ 243 $ cases lead to contradictions and so there are just $ 77 $ cases which may lead
to the existence of $ \Gamma _1$ in $ \Gamma $ which listed in Table \ref{t4}.
\begin{longtable}{|l|l|l|}
\caption{The relations of $ 4 $ cycles of length $ 3 $ on vertices of degree $ 4 $ in $ Z(\alpha,\beta) $ and $ U(\alpha,\beta) $ with the structure of $ \Gamma_1 $ as Figure \ref{gamma} which lead to  contradictions. }\label{tgamma1}\\
\hline
\endfirsthead
\multicolumn{3}{c}%
{{\bfseries \tablename\ \thetable{} -- continued from previous page}} \\\hline
\hline 
 \multicolumn{1}{|l|}{\textbf{$ n $}} & \multicolumn{1}{l|}{\textbf{$r_1,\;r_2,\;r_3,\;r_4$}}&  \multicolumn{1}{|l|}{\textbf{$E $}}  \\ \hline
\endhead

\hline \multicolumn{3}{|r|}{{Continued on next page}} \\ \hline
\endfoot

\hline 
\endlastfoot
\hline
\textbf{$ n $}& \textbf{$r_1,\;r_2,\;r_3,\;r_4$} & \textbf{$E $}\\
\hline
$1$&$xz^{-2}x=1,\;xy^2=1,\;(xz^{-1})^2x=1,\;xy^{-1}x^{-1}y=1$&$x=z$\\\hline
$2$&$xz^{-2}x=1,\;xy^2=1,\;(xz^{-1})^2x=1,\;xy^{-1}xy=1$&$x=z$\\
\hline
$3$&$xz^{-2}x=1,\;xyx^{-1}y=1,\;(xz^{-1})^2x=1,\;xy^{-1}x^{-1}y=1$&T\\\hline
$4$&$xz^{-2}x=1,\;xyx^{-1}y=1,\;(xz^{-1})^2x=1,\;xy^{-1}xy=1$&$x=z$\\
\hline
$5$&$xy^{-2}x=1,\;xz^2=1,\;(xy^{-1})^2x=1,\;xz^{-1}x^{-1}z=1$&$x=y$\\\hline
$6$&$xy^{-2}x=1,\;xzx^{-1}z=1,\;(xy^{-1})^2x=1,\;xz^{-1}x^{-1}z=1$&$x=y$\\
\hline
$7$&$xz^{-1}xz=1,\;xy^{-1}x^{-1}y=1,\;z^2x^{-1}z=1,\;(zy^{-1})^2x=1$&T\\\hline
$8$&$xz^{-1}xz=1,\;xy^{-1}xy=1,\;z^2x^{-1}z=1,\;(zy^{-1})^2x=1$&T\\
\hline
$9$&$xz^{-2}y=1,\;(xy^{-1})^2z=1,\;xz^{-1}yx^{-1}y=1,\;xy^{-1}z^2=1$&T\\\hline
$10$&$yz^{-2}x=1,\;(yx^{-1})^2z=1,\;yz^{-1}xy^{-1}x=1,\;yx^{-1}z^2=1$&T\\
\hline
$11$&$xy^2=1,\;xz^{-1}x^{-1}z=1,\;xyx^{-1}y=1,\;xz^{-2}x=1$&$x=z$\\\hline
$12$&$xy^2=1,\;xz^{-1}x^{-1}z=1,\;xyx^{-1}y=1,\;(xz^{-1})^2x=1$&T\\
\hline
$13$&$xy^2=1,\;xz^{-1}xz=1,\;xyx^{-1}y=1,\;(xz^{-1})^2x=1$&T\\\hline
$14$&$xy^2=1,\;xz^{-1}xz=1,\;xyx^{-1}y=1,\;xz^{-2}x=1$&T\\
\hline
$15$&$xz^2=1,\;xy^{-1}x^{-1}y=1,\;xzx^{-1}z=1,\;xy^{-2}x=1$&T\\\hline
$16$&$xz^2=1,\;xy^{-1}x^{-1}y=1,\;xzx^{-1}z=1,\;(xy^{-1})^2x=1$&T\\
\hline
$17$&$xz^2=1,\;xy^{-1}xy=1,\;xzx^{-1}z=1,\;(xy^{-1})^2x=1$&T\\\hline
$18$&$xy^{-2}z=1,\;(xz^{-1})^2y=1,\;xy^{-1}zx^{-1}z=1,\;xz^{-1}y^2=1$&T\\
\hline
$19$&$xy^{-1}xy=1,\;xz^{-1}x^{-1}z=1,\;y^2x^{-1}y=1,\;(yz^{-1})^2x=1$&T\\\hline
$20$&$xy^{-1}xy=1,\;xz^{-1}xz=1,\;y^2x^{-1}y=1,\;(yz^{-1})^2x=1$&T\\
\hline
$21$&$xy^{-2}x=1,\;xz^2=1,\;(xy^{-1})^2x=1,\;xz^{-1}xz=1$&T\\\hline
$22$&$xy^{-2}x=1,\;xzx^{-1}z=1,\;(xy^{-1})^2x=1,\;xz^{-1}xz=1$&$x=y$\\
\hline
$23$&$xz^2=1,\;xy^{-1}xy=1,\;xzx^{-1}z=1,\;xy^{-2}x=1$&T\\\hline
$24$&$xz^{-1}x^{-1}z=1,\;x^2z^{-1}x=1,\;yx^{-1}zy^{-1}x=1,\;yz^{-1}y^{-1}z=1$&BS$(1,-2)$\\
\hline
$25$&$xz^{-1}x^{-1}z=1,\;x^2z^{-1}x=1,\;yx^{-1}zy^{-1}z=1,\;yz^{-1}y^{-1}x=1$&BS$(1,3)$\\\hline
$26$&$xz^{-1}x^{-1}z=1,\;x^2z^{-1}x=1,\;x^{-1}yx^{-1}zy^{-1}z=1,\;x^{-1}yz^{-1}xz^{-1}y=1$&T\\
\hline
$27$&$xz^{-1}x^{-1}z=1,\;xyx^{-1}y=1,\;yx^{-1}zx^{-1}y=1,\;yz^{-1}y^{-1}z=1$&BS$(1,-1)$\\\hline
$28$&$xz^{-2}x=1,\;xyx^{-1}y=1,\;x^{-1}yz^{-1}xz^{-1}y=1,\;x^{-1}yx^{-1}zy^{-1}z=1$&T\\
\hline
$29$&$xy^{-1}zy^{-1}x=1,\;xyx^{-1}z=1,\;(xy^{-1})^2z=1,\;xzx^{-1}y=1$&Abelian\\\hline
$30$&$xy^{-1}z^2=1,\;xyz^{-1}x=1,\;xzy^{-1}z=1,\;xy^{-1}x^{-1}y=1$&Abelian\\
\hline
$31$&$xz^{-1}y^2=1,\;xyz^{-1}x=1,\;xyz^{-1}x=1,\;xz^{-1}y^{-1}z=1$&Abelian\\\hline
$32$&$xz^{-1}y^2=1,\;xyz^{-1}x=1,\;xyz^{-1}y=1,\;xz^{-1}x^{-1}z=1$&Abelian\\
\hline
$33$&$xz^{-1}y^2=1,\;xyz^{-1}x=1,\;xyz^{-1}y=1,\;xz^{-1}y^{-1}z=1$&Abelian\\\hline
$34$&$xz^{-1}yz^{-1}x=1,\;xzy=1,\;(xz^{-1})^2y=1,\;xyz=1$&Abelian\\
\hline
$35$&$xz^{-1}yz^{-1}x=1,\;xzx^{-1}y=1,\;(xz^{-1})^2y=1,\;xyx^{-1}z=1$&Abelian\\\hline
$36$&$xy^{-1}x^{-1}y=1,\;xzx^{-1}z=1,\;(yx^{-1})^2y=1,\;y(z^{-1}x)^2=1$&T\\
\hline
$37$&$xy^{-1}zx^{-1}y=1,\;xzx^{-1}z=1,\;(yx^{-1})^2y=1,\;yzy^{-1}z=1$&Abelian\\\hline
$38$&$xy^{-1}z^2=1,\;xzy^{-1}x=1,\;xzy^{-1}x=1,\;xy^{-1}z^{-1}y=1$&$x=z$\\
\hline
$39$&$xy^{-1}z^2=1,\;xzy^{-1}x=1,\;xzy^{-1}z=1,\;xy^{-1}x^{-1}y=1$&$x=z$\\\hline
$40$&$xy^{-1}z^2=1,\;xzy^{-1}x=1,\;xzy^{-1}z=1,\;xy^{-1}z^{-1}y=1$&Abelian\\
\hline
$41$&$xz^{-1}y^2=1,\;xzy^{-1}x=1,\;xyz^{-1}y=1,\;xz^{-1}x^{-1}z=1$&Abelian\\\hline
$42$&$xy^{-1}x^{-1}y=1,\;xzy^{-1}z=1,\;(yx^{-1})^2y=1,\;yz^{-1}y^{-1}z=1$&Abelian\\
\hline
$43$&$xz^{-1}x^{-1}z=1,\;xy^{-1}x^{-1}y=1,\;yx^{-1}zx^{-1}y=1,\;yz^{-1}x^{-1}z=1$&Abelian\\\hline
$44$&$xz^{-1}x^{-1}z=1,\;xy^{-1}x^{-1}y=1,\;yx^{-1}zx^{-1}y=1,\;yz^{-1}y^{-1}x=1$&Abelian\\
\hline
$45$&$xz^{-1}x^{-1}z=1,\;xy^{-1}x^{-1}y=1,\;yx^{-1}zx^{-1}y=1,\;yz^{-1}y^{-1}z=1$&Abelian\\\hline
$46$&$xz^{-1}x^{-1}z=1,\;xy^{-1}x^{-1}y=1,\;yx^{-1}zx^{-1}y=1,\;yz^{-2}x=1$&Abelian\\
\hline
$47$&$xz^{-1}x^{-1}z=1,\;xy^{-1}x^{-1}y=1,\;yx^{-1}zx^{-1}y=1,\;yz^{-1}xz=1$&Abelian\\\hline
$48$&$xz^{-1}x^{-1}z=1,\;xy^{-1}x^{-1}y=1,\;yx^{-1}zx^{-1}y=1,\;yz^{-1}xy^{-1}z=1$&Abelian\\
\hline
$49$&$xz^{-1}x^{-1}z=1,\;xy^{-1}x^{-1}y=1,\;yx^{-1}zx^{-1}y=1,\;yz^{-1}yz=1$&Abelian\\\hline
$50$&$xz^{-1}x^{-1}z=1,\;xy^{-1}x^{-1}y=1,\;yx^{-1}zx^{-1}y=1,\;yz^{-1}yx^{-1}z=1$&Abelian\\
\hline
$51$&$xz^{-1}x^{-1}z=1,\;xy^{-1}x^{-1}y=1,\;x^{-1}y(x^{-1}z)^2=1,\;x^{-1}(yz^{-1})^2y=1$&Abelian\\\hline
$52$&$xz^{-1}y^{-1}z=1,\;xy^{-1}x^{-1}y=1,\;xy^{-1}zx^{-1}y=1,\;xz^{-1}x^{-1}z=1$&Abelian\\
\hline
$53$&$xz^{-1}y^{-1}z=1,\;xy^{-1}x^{-1}y=1,\;zy^{-1}zx^{-1}y=1,\;zx^{-1}yx^{-1}z=1$&$x=y$\\\hline
$54$&$xz^{-2}x=1,\;xy^{-1}x^{-1}y=1,\;(xz^{-1})^2x=1,\;xy^{-1}x^{-1}y=1$&$x=z$\\
\hline
$55$&$xz^{-2}x=1,\;xy^{-1}x^{-1}y=1,\;(xz^{-1})^2x=1,\;xy^{-1}xy=1$&$x=z$\\\hline
$56$&$xz^{-2}x=1,\;xy^{-1}x^{-1}y=1,\;(xz^{-1})^2x=1,\;x(y^{-1}z)^2=1$&$x=z$\\
\hline
$57$&$xz^{-2}x=1,\;xy^{-1}x^{-1}y=1,\;y(z^{-1}x)^2=1,\;yx^{-1}yz=1$&Abelian\\\hline
$58$&$xz^{-2}x=1,\;xy^{-1}x^{-1}y=1,\;yz^{-1}xz^{-1}y=1,\;(yx^{-1})^2z=1$&Abelian\\
\hline
$59$&$xz^{-1}yz^{-1}x=1,\;xy^{-1}x^{-1}y=1,\;(xz^{-1})^2y=1,\;x^2y^{-1}x=1$&Abelian\\\hline
$60$&$xz^{-1}yz^{-1}x=1,\;xy^{-1}x^{-1}y=1,\;(xz^{-1})^2y=1,\;xyz=1$&Abelian\\
\hline
$61$&$xz^{-1}yz^{-1}x=1,\;xy^{-1}x^{-1}y=1,\;(xz^{-1})^2y=1,\;xzy^{-1}x=1$&Abelian\\\hline
$62$&$xz^{-1}x^{-1}y=1,\;xy^{-1}x^{-1}z=1,\;(yx^{-1})^2y=1,\;yz^{-1}y^{-1}z=1$&Abelian\\
\hline
$63$&$xz^{-1}x^{-1}z=1,\;xy^{-2}x=1,\;yx^{-1}zy^{-1}x=1,\;yz^{-1}y^{-1}z=1$&BS$(1,-1)$\\\hline
$64$&$xz^{-1}x^{-1}y=1,\;xy^{-2}z=1,\;(x^{-1}y)^2x^{-1}z=1,\;x^{-1}(yz^{-1})^2y=1$&BS$(-1,2)$\\
\hline
$65$&$xz^{-1}x^{-1}y=1,\;xy^{-1}z^{-1}x=1,\;(x^{-1}y)^2z^{-1}y=1,\;x^{-1}yz^{-1}xy^{-1}z=1$&BS$(3,-1)$\\\hline
$66$&$xz^{-2}y=1,\;xy^{-1}z^{-1}x=1,\;x^{-1}(yz^{-1})^2y=1,\;(x^{-1}y)^2x^{-1}z=1$&BS$(-1,3)$\\
\hline
$67$&$xz^{-1}yz=1,\;xy^{-1}z^{-1}y=1,\;xzx^{-1}y=1,\;xz^{-1}yz=1$&BS$(1,-1)$\\\hline
$68$&$xz^{-1}x^{-1}z=1,\;xy^{-1}xy=1,\;yx^{-1}zx^{-1}y=1,\;yz^{-1}y^{-1}z=1$&BS$(1,-1)$\\
\hline
$69$&$xz^{-1}x^{-1}z=1,\;xy^{-1}xy=1,\;yx^{-1}zx^{-1}y=1,\;yz^{-1}xy^{-1}z=1$&BS$(1,-1)$\\\hline
$70$&$xz^{-2}x=1,\;xy^{-1}xy=1,\;(xz^{-1})^2x=1,\;xy^{-1}x^{-1}y=1$&$x=z$\\
\hline
$71$&$xz^{-2}x,\;xy^{-1}xy=1,\;(xz^{-1})^2x=1,\;xy^{-1}xy=1$&$x=z$\\\hline
$72$&$xz^{-2}x=1,\;xy^{-1}xy=1,\;(xz^{-1})^2x,\;x(y^{-1}z)^2=1$&$x=z$\\
\hline
$73$&$xz^{-1}xz=1,\;xy^{-1}xy=1,\;yzx^{-1}y=1,\;yz^{-1}xy^{-1}z=1$&BS$(1,-1)$\\\hline
$74$&$xz^{-1}x^{-1}z=1,\;(xy^{-1})^2x=1,\;yx^{-1}zy^{-1}x=1,\;yz^{-1}y^{-1}z=1$&Abelian\\
\hline
$75$&$xz^{-1}y^2=1,\;xy^{-1}xz^{-1}x=1,\;xyz^{-1}y=1,\;xz^{-1}x^{-1}z=1$&Abelian\\\hline
$76$&$xz^{-1}x^{-1}z=1,\;xy^{-1}xz^{-1}y=1,\;yx^{-1}zx^{-1}y=1,\;yz^{-1}y^{-1}z=1$&BS$(1,-1)$\\
\hline
$77$&$xz^{-1}x^{-1}z=1,\;xy^{-1}zx^{-1}y=1,\;yx^{-1}zx^{-1}y=1,\;yz^{-1}y^{-1}z=1$&$x=y$\\\hline
$78$&$xz^{-2}x=1,\;xy^{-1}zx^{-1}y=1,\;(xz^{-1})^2x=1,\;xy^{-1}x^{-1}y=1$&T\\
\hline
$79$&$xz^{-2}x=1,\;xy^{-1}zx^{-1}y=1,\;(xz^{-1})^2x=1,\;xy^{-1}z^{-1}y=1$&T\\\hline
$80$&$xz^{-2}x=1,\;xy^{-1}zx^{-1}y=1,\;(xz^{-1})^2x=1,\;xy^{-1}xy=1$&T\\
\hline
$81$&$xz^{-2}x=1,\;xy^{-1}zx^{-1}y=1,\;(xz^{-1})^2x=1,\;xy^{-1}zy=1$&T\\\hline
$82$&$xz^{-1}yz^{-1}x=1,\;xy^{-1}zx^{-1}y=1,\;(xz^{-1})^2y=1,\;xyz^{-1}x=1$&T\\
\hline
$83$&$xz^{-1}y^{-1}x=1,\;xy^{-1}zx^{-1}z=1,\;(xy^{-1})^2z=1,\;xz^{-1}y^{-1}x=1$&Abelian\\\hline
$84$&$xz^{-1}x^{-1}z=1,\;xy^{-1}zy^{-1}x=1,\;yx^{-1}zy^{-1}x=1,\;yz^{-1}x^{-1}z=1$&Abelian\\
\hline
$85$&$xz^{-1}x^{-1}z=1,\;xy^{-1}zy^{-1}x=1,\;yx^{-1}zy^{-1}x=1,\;yz^{-1}y^{-1}z=1$&$x=y$\\\hline
$86$&$xz^{-1}x^{-1}z=1,\;xy^{-1}zy^{-1}x=1,\;zx^{-1}zy^{-1}z=1,\;z(y^{-1}x)^2=1$&Abelian\\
\hline
$87$&$xz^{-1}y^{-1}z=1,\;xy^{-1}zy^{-1}x=1,\;xy^{-1}zy^{-1}x=1,\;xz^{-1}x^{-1}z=1$&Abelian\\\hline
$88$&$yx^{-1}zx^{-1}y=1,\;y^2z=1,\;(yx^{-1})^2z=1,\;yzx^{-1}y=1$&T\\
\hline
$89$&$yz^{-1}y^{-1}z=1,\;y^2z^{-1}y=1,\;xy^{-1}zx^{-1}y=1,\;xz^{-1}x^{-1}z=1$&BS$(1,-2)$\\\hline
$90$&$yz^{-1}y^{-1}z=1,\;y^2z^{-1}y=1,\;xy^{-1}zx^{-1}z=1,\;xz^{-1}x^{-1}y=1$&BS$(1,3)$\\
\hline
$91$&$yx^{-1}zy^{-1}x=1,\;yzy^{-1}z=1,\;(xy^{-1})^2x=1,\;xzx^{-1}z=1$&Abelian\\\hline
$92$&$yz^{-1}y^{-1}z=1,\;(yx^{-1})^2y=1,\;xy^{-1}zx^{-1}y=1,\;xz^{-1}x^{-1}z=1$&Abelian\\
\hline
$93$&$yz^{-1}x^{-1}y=1,\;(yx^{-1})^2z=1,\;z(x^{-1}y)^2=1,\;(zy^{-1})^2z=1$&Abelian\\\hline
$94$&$yz^{-1}x^{-1}y=1,\;yx^{-1}yz^{-1}x=1,\;yx^{-1}yz^{-1}x=1,\;yz^{-1}y^{-1}z=1$&Abelian\\
\hline
$95$&$yz^{-1}y^{-1}z=1,\;yx^{-1}yz^{-1}x=1,\;xy^{-1}zy^{-1}x=1,\;xz^{-1}x^{-1}z=1$&BS$(1,-1)$\\\hline
$96$&$yz^{-1}y^{-1}z=1,\;yx^{-1}yz^{-1}x=1,\;xy^{-1}zy^{-1}x=1,\;xz^{-1}xz=1$&BS$(1,-1)$\\
\hline
$97$&$yz^{-1}x^{-1}z=1,\;yx^{-1}zx^{-1}y=1,\;yx^{-1}zx^{-1}y=1,\;yz^{-1}y^{-1}z=1$&Abelian\\\hline
$98$&$yz^{-1}y^{-1}z=1,\;yx^{-1}zx^{-1}y=1,\;xy^{-1}zx^{-1}y=1,\;xz^{-1}x^{-1}z=1$&$y=x$\\
\hline
$99$&$yz^{-1}y^{-1}z=1,\;yx^{-1}zy^{-1}x=1,\;xy^{-1}zy^{-1}x=1,\;xz^{-1}x^{-1}z=1$&$y=x$\\\hline
$100$&$x^2z^{-1}x=1,\;xz^{-1}x^{-1}z=1,\;yz^{-1}xy^{-1}z=1,\;yzy^{-1}x=1$&BS$(1,-3)$\\
\hline
$101$&$x^2z^{-1}x=1,\;xz^{-1}x^{-1}z=1,\;yz^{-1}xy^{-1}x=1,\;yzy^{-1}z=1$&BS$(1,2)$\\\hline
$102$&$xy^2=1,\;xz^{-1}x^{-1}z=1,\;xyx^{-1}y=1,\;xz^{-1}x^{-1}z=1$&T\\
\hline
$103$&$xy^2=1,\;xz^{-1}x^{-1}z=1,\;xyx^{-1}y=1,\;xz^{-1}y^{-1}z=1$&T\\\hline
$104$&$xy^2=1,\;xz^{-1}x^{-1}z=1,\;xyx^{-1}y=1,\;xz^{-1}xz=1$&T\\
\hline
$105$&$xy^2=1,\;xz^{-1}x^{-1}z=1,\;xyx^{-1}y=1,\;xz^{-1}xy^{-1}z=1$&T\\\hline
$106$&$xy^2=1,\;xz^{-1}xz=1,\;xyx^{-1}y=1,\;xz^{-1}x^{-1}z=1$&T\\
\hline
$107$&$xy^2=1,\;xz^{-1}xz=1,\;xyx^{-1}y=1,\;xz^{-1}xz=1$&T\\\hline
$108$&$xyz^{-1}y=1,\;xz^{-1}x^{-1}z=1,\;xz^{-1}yx^{-1}z=1,\;xyx^{-1}y=1$&Abelian\\
\hline
$109$&$xyz^{-1}y=1,\;xz^{-1}x^{-1}z=1,\;xz^{-1}yx^{-1}z=1,\;xyz^{-1}y=1$&Abelian\\\hline
$110$&$xyz^{-1}y=1,\;xz^{-2}x=1,\;xz^{-1}yz^{-1}x=1,\;xyx^{-1}y=1$&BS$(1,-1)$\\
\hline
$111$&$xyz^{-1}y=1,\;xz^{-1}yx^{-1}z=1,\;xz^{-1}yx^{-1}z=1,\;xyx^{-1}y=1$&BS$(1,-1)$\\\hline
$112$&$xyz^{-1}y=1,\;xz^{-1}yz^{-1}x=1,\;xz^{-1}yz^{-1}x=1,\;xyx^{-1}y=1$&BS$(1,-1)$\\
\hline
$113$&$xz^2=1,\;xy^{-1}x^{-1}y=1,\;xzx^{-1}z=1,\;xy^{-1}x^{-1}y=1$&T\\\hline
$114$&$xz^2=1,\;xy^{-1}x^{-1}y=1,\;xzx^{-1}z=1,\;xy^{-1}xy=1$&T\\
\hline
$115$&$xz^2=1,\;xy^{-1}xy=1,\;xzx^{-1}z=1,\;xy^{-1}x^{-1}y=1$&T\\\hline
$116$&$xz^2=1,\;xy^{-1}xy=1,\;xzx^{-1}z=1,\;xy^{-1}xy=1$&T\\
\hline
$117$&$xzx^{-1}z=1,\;xy^{-1}x^{-1}y=1,\;yx^{-1}zx^{-1}y=1,\;yzx^{-1}z=1$&Abelian\\\hline
$118$&$xzx^{-1}z=1,\;xy^{-2}x=1,\;yx^{-1}zy^{-1}x=1,\;yzx^{-1}z=1$&BS$(1,-1)$\\
\hline
$119$&$xzx^{-1}z=1,\;xy^{-2}x=1,\;yx^{-1}zy^{-1}x=1,\;yzy^{-1}z=1$&BS$(1,-1)$\\\hline
$120$&$xzx^{-1}z=1,\;(xy^{-1})^2x=1,\;yx^{-1}zy^{-1}x=1,\;yzy^{-1}z=1$&Abelian\\
\hline
$121$&$xzx^{-1}z=1,\;xy^{-1}zx^{-1}y=1,\;yx^{-1}zx^{-1}y=1,\;yzy^{-1}z=1$&$x=y$\\\hline
$122$&$xzx^{-1}z=1,\;xy^{-1}zx^{-1}y=1,\;yx^{-1}zx^{-1}y=1,\;yzx^{-1}z=1$&$x=y$\\
\hline
$123$&$xzx^{-1}z=1,\;xy^{-1}zy^{-1}x=1,\;yx^{-1}zy^{-1}x=1,\;yzx^{-1}z=1$&$x=y$\\\hline
$124$&$xzx^{-1}z=1,\;xy^{-1}zy^{-1}x=1,\;yx^{-1}zy^{-1}x=1,\;yzy^{-1}z=1$&$x=y$\\
\hline
$125$&$xzy^{-1}x=1,\;xz^{-1}yx^{-1}z=1,\;(xy^{-1})^2z=1,\;xzy^{-1}x=1$&$x=y^{-1}$\\\hline
$126$&$xzy^{-1}z=1,\;xy^{-1}x^{-1}y=1,\;xy^{-1}zx^{-1}y=1,\;xzx^{-1}z=1$&Abelian\\
\hline
$127$&$xzy^{-1}z=1,\;xy^{-1}x^{-1}y=1,\;xy^{-1}zx^{-1}y=1,\;xzy^{-1}z=1$&Abelian\\\hline
$128$&$xzy^{-1}z=1,\;xy^{-2}x=1,\;xy^{-1}zy^{-1}x=1,\;xzx^{-1}z=1$&BS$(1,-1)$\\
\hline
$129$&$xzy^{-1}z=1,\;xy^{-2}x=1,\;xy^{-1}zy^{-1}x=1,\;xzy^{-1}z=1$&BS$(1,-1)$\\\hline
$130$&$xzy^{-1}z=1,\;xy^{-1}zx^{-1}y=1,\;xy^{-1}zx^{-1}y=1,\;xzx^{-1}z=1$&$y=x$\\
\hline
$131$&$xzy^{-1}z=1,\;xy^{-1}zy^{-1}x=1,\;xy^{-1}zy^{-1}x=1,\;xzx^{-1}z=1$&BS$(1,-1)$\\\hline
$132$&$xy^{-1}x^{-1}y=1,\;xz^{-1}x^{-1}z=1,\;(yx^{-1})^2z=1,\;yz^{-1}x^{-1}y=1$&Abelian\\
\hline
$133$&$xy^{-1}x^{-1}y=1,\;xz^{-1}x^{-1}z=1,\;(yx^{-1})^2z=1,\;yz^{-2}x=1$&Abelian\\\hline
$134$&$xy^{-1}x^{-1}y=1,\;xz^{-1}x^{-1}z=1,\;(yx^{-1})^2z=1,\;yz^{-2}y=1$&Abelian\\
\hline
$135$&$xy^{-1}x^{-1}y=1,\;xz^{-1}x^{-1}z=1,\;(yx^{-1})^2z=1,\;yz^{-1}xz^{-1}y=1$&Abelian\\\hline
$136$&$xy^{-1}x^{-1}y=1,\;xz^{-1}y^{-1}z=1,\;(yx^{-1})^2y=1,\;yz^{-1}yz=1$&BS$(1,-1)$\\
\hline
$137$&$xy^{-1}x^{-1}y=1,\;xz^{-2}x=1,\;(x^{-1}y)^2z^{-1}y=1,\;x^{-1}yz^{-1}xy^{-1}z=1$&Abelian\\\hline
$138$&$xy^{-1}x^{-1}y=1,\;xz^{-1}xz=1,\;(yx^{-1})^2y=1,\;yz^{-1}x^{-1}z=1$&BS$(1,-1)$\\
\hline
$139$&$xy^{-1}x^{-1}y=1,\;xz^{-1}xz=1,\;zx^{-1}yx^{-1}z=1,\;zy^{-1}xz^{-1}y=1$&BS$(1,-1)$\\\hline
$140$&$xy^{-1}x^{-1}y=1,\;(xz^{-1})^2x=1,\;yx^{-1}yz^{-1}x=1,\;yz^{-1}y^{-1}z=1$&Abelian\\
\hline
$141$&$xy^{-1}x^{-1}y=1,\;(xz^{-1})^2x=1,\;(x^{-1}y)^2z^{-1}y=1,\;x^{-1}yz^{-1}xy^{-1}z=1$&Abelian\\\hline
$142$&$xy^{-2}x=1,\;xz^{-1}x^{-1}z=1,\;(xy^{-1})^2x=1,\;xz^{-1}x^{-1}z=1$&$x=y$\\
\hline
$143$&$xy^{-2}x=1,\;xz^{-1}x^{-1}z=1,\;(xy^{-1})^2x=1,\;xz^{-1}xz=1$&$x=y$\\\hline
$144$&$xy^{-2}x=1,\;xz^{-1}x^{-1}z=1,\;(xy^{-1})^2x=1,\;x(z^{-1}y)^2=1$&$x=y$\\
\hline
$145$&$xy^{-2}x=1,\;xz^{-1}xz=1,\;(xy^{-1})^2x=1,\;xz^{-1}x^{-1}z=1$&T\\\hline
$146$&$xy^{-2}x=1,\;xz^{-1}xz=1,\;(xy^{-1})^2x=1,\;xz^{-1}xz=1$&$x=y$\\
\hline
$147$&$xy^{-2}x=1,\;xz^{-1}xz=1,\;(xy^{-1})^2x=1,\;x(z^{-1}y)^2=1$&$x=y$\\\hline
$148$&$xy^{-1}z^{-1}x=1,\;(xz^{-1})^2y=1,\;y(z^{-1}x)^2=1,\;(yx^{-1})^2y=1$&Abelian\\
\hline
$149$&$xy^{-1}z^{-1}y=1,\;xz^{-1}yz^{-1}x=1,\;xz^{-1}yz^{-1}x=1,\;xy^{-1}x^{-1}y=1$&Abelian\\\hline
$150$&$xy^{-1}zy=1,\;xz^{-1}x^{-1}z=1,\;xyx^{-1}z=1,\;xy^{-1}zy=1$&Abelian\\
\hline
$151$&$xy^{-1}z^2=1,\;xz^{-1}y^{-1}x=1,\;xzy^{-1}z=1,\;xy^{-1}x^{-1}y=1$&Abelian\\\hline
$152$&$xy^{-1}z^2=1,\;xz^{-1}y^{-1}x=1,\;xzy^{-1}z=1,\;xy^{-1}z^{-1}x=1$&Abelian\\
\hline
$153$&$xy^{-1}z^2=1,\;xz^{-1}xy^{-1}x=1,\;xzy^{-1}z=1,\;xy^{-1}x^{-1}y=1$&Abelian\\\hline
$154$&$xy^{-1}z^2=1,\;xz^{-1}xy^{-1}x=1,\;xzy^{-1}z=1,\;xy^{-1}xz^{-1}x=1$&Abelian\\
\hline
$155$&$y^2z^{-1}y=1,\;yz^{-1}y^{-1}z=1,\;xz^{-1}yx^{-1}z=1,\;xyx^{-1}y=1$&BS$(1,-1)$\\\hline
$156$&$y^2z^{-1}y=1,\;yz^{-1}y^{-1}z=1,\;xz^{-1}yx^{-1}y=1,\;xyx^{-1}z=1$&BS$(1,2)$\\
\hline
$157$&$y^2z^{-1}y=1,\;yz^{-1}y^{-1}z=1,\;xz^{-1}yx^{-1}z=1,\;xzx^{-1}y=1$&BS$(1,-3)$\\\hline
$158$&$y^2z^{-1}y=1,\;yz^{-1}y^{-1}z=1,\;xz^{-1}yx^{-1}y=1,\;xzx^{-1}z=1$&BS$(1,2)$\\
\hline
$159$&$yz^2=1,\;yx^{-1}yz^{-1}x=1,\;yzy^{-1}z=1,\;yx^{-1}yz^{-1}x=1$&T\\\hline
$160$&$yz^2=1,\;yx^{-1}yz^{-1}x=1,\;yzy^{-1}z=1,\;yx^{-1}zy^{-1}x=1$&T\\
\hline
$161$&$yz^2=1,\;yx^{-1}zy^{-1}x=1,\;yzy^{-1}z=1,\;yx^{-1}yz^{-1}x=1$&T\\\hline
$162$&$yz^2=1,\;yx^{-1}zy^{-1}x=1,\;yzy^{-1}z=1,\;yx^{-1}zy^{-1}x=1$&T\\
\hline
$163$&$yzx^{-1}y=1,\;yz^{-1}xy^{-1}z=1,\;(yx^{-1})^2z=1,\;yzx^{-1}y=1$&Abelian\\\hline
$164$&$yzy^{-1}z=1,\;yx^{-1}zx^{-1}y=1,\;xy^{-1}zx^{-1}y=1,\;xzx^{-1}z=1$&$y=x$\\\hline
$165$&$yzy^{-1}z=1,\;yx^{-1}zy^{-1}x=1,\;xy^{-1}zy^{-1}x=1,\;xzx^{-1}z=1$&$y=x$\\\hline
$166$&$yzx^{-1}z=1,\;yx^{-1}zx^{-1}y=1,\;yx^{-1}zx^{-1}y=1,\;yzy^{-1}z=1$&$y=x^{-1}$\\\hline
\hline
\end{longtable}
\begin{longtable}{|l|l|l|l|l|}
\caption{The  relations for existing the graph $ \Gamma_1 $ as Figure \ref{gamma} in $Z(\alpha,\beta)$ and $U(\alpha,\beta)$. }\label{t4}\\
\hline
\endfirsthead
\multicolumn{5}{c}%
{{\bfseries \tablename\ \thetable{} -- continued from previous page}} \\\hline
\hline 
 \multicolumn{1}{|l|}{\textbf{$ n $}} & \multicolumn{1}{l|}{\textbf{$r_1$}}&  \multicolumn{1}{|l|}{\textbf{$r_2 $}} & \multicolumn{1}{l|}{\textbf{$r_3$}}&  \multicolumn{1}{|l|}{\textbf{$r_4 $}} \\ \hline
\endhead

\hline \multicolumn{5}{|r|}{{Continued on next page}} \\ \hline
\endfoot

\hline 
\endlastfoot
$ n $& $ r_1 $ &$ r_2 $ &$ r_3 $ &$ r_4 $ \\
\hline
$1$&$xy^{-1}xy=1$&$xz^{-1}x^{-1}z=1$&$(zx^{-1})^2z=1$&$(zy^{-1})^2x=1$\\
$2$&$xz^{-1}xz=1$&$xy^{-1}x^{-1}y=1$&$(yx^{-1})^2y=1$&$(yz^{-1})^2x=1$\\
$3$&$xyx^{-1}z=1$&$xy^{-1}zy=1$&$xyx^{-1}z=1$&$xy^{-1}zy=1$\\
$4$&$xyx^{-1}z=1$&$xz^{-1}y^2=1$&$xyx^{-1}z=1$&$xz^{-1}y^2=1$\\
$5$&$xyz^{-1}x=1$&$xy^{-1}zy=1$&$xyz^{-1}x=1$&$xy^{-1}zy=1$\\
$6$&$xyz^{-1}x=1$&$xz^{-1}y^2=1$&$xyz^{-1}x=1$&$xz^{-1}y^2=1$\\
$7$&$xyz^{-1}y=1$&$xz^{-2}x=1$&$x^{-1}yz^{-1}xz^{-1}y=1$&$x^{-1}yx^{-1}zy^{-1}z=1$\\
$8$&$xzx^{-1}y=1$&$xy^{-1}z^2=1$&$xzx^{-1}y=1$&$xy^{-1}z^2=1$\\
$9$&$xzx^{-1}y=1$&$xz^{-1}yz=1$&$xzx^{-1}y=1$&$xz^{-1}yz=1$\\
$10$&$xzx^{-1}z=1$&$xy^{-1}x^{-1}y=1$&$(yx^{-1})^2y=1$&$yz^{-1}y^{-1}z=1$\\
$11$&$xzx^{-1}z=1$&$xy^{-1}x^{-1}y=1$&$(yx^{-1})^2y=1$&$yz^{-1}yz=1$\\
$12$&$xzy^{-1}x=1$&$xy^{-1}z^2=1$&$xzy^{-1}x=1$&$xy^{-1}z^2=1$\\
$13$&$xzy^{-1}x=1$&$x(y^{-1}z)^2=1$&$x(y^{-1}z)^2=1$&$xzy^{-1}x=1$\\
$14$&$xzy^{-1}x=1$&$xz^{-1}yz=1$&$xzy^{-1}x=1$&$xz^{-1}yz=1$\\
$15$&$xy^{-1}x^{-1}y=1$&$xz^{-1}x^{-1}z=1$&$(zx^{-1})^2z=1$&$(zy^{-1})^2x=1$\\
$16$&$xy^{-1}x^{-1}y=1$&$xz^{-1}y^{-1}z=1$&$zy^{-1}zx^{-1}y=1$&$zx^{-1}zy^{-1}x=1$\\
$17$&$xy^{-1}x^{-1}y=1$&$xz^{-2}x=1$&$x^{-1}yz^{-1}xz^{-1}y=1$&$x^{-1}yx^{-1}zy^{-1}z=1$\\
$18$&$xy^{-2}x=1$&$xz^{-1}x^{-1}z=1$&$x^{-1}yx^{-1}zy^{-1}z=1$&$x^{-1}yz^{-1}xz^{-1}y=1$\\
$19$&$xy^{-2}x=1$&$xz^{-1}y^{-1}z=1$&$(zy^{-1})^2x=1$&$zx^{-1}yx^{-1}z=1$\\
$20$&$xy^{-1}z^{-1}y=1$&$xz^{-1}x^{-1}z=1$&$yx^{-1}zx^{-1}y=1$&$yz^{-1}y^{-1}x=1$\\
$21$&$xy^{-1}z^{-1}y=1$&$xz^{-1}x^{-1}z=1$&$yx^{-1}zx^{-1}y=1$&$yz^{-1}yx^{-1}z=1$\\
$22$&$xy^{-1}z^{-1}y=1$&$xz^{-2}x=1$&$yz^{-1}xz^{-1}y=1$&$yx^{-1}yz^{-1}x=1$\\
$23$&$xy^{-1}z^{-1}y=1$&$xz^{-2}x=1$&$x^{-1}yz^{-1}xz^{-1}y=1$&$x^{-1}yx^{-1}zy^{-1}z=1$\\
$24$&$xy^{-1}xy=1$&$xz^{-2}x=1$&$x^{-1}yz^{-1}xz^{-1}y=1$&$x^{-1}yx^{-1}zy^{-1}z=1$\\
$25$&$(xy^{-1})^2z=1$&$xz^{-1}y^{-1}x=1$&$(xy^{-1})^2z=1$&$xz^{-1}y^{-1}x=1$\\
$26$&$xy^{-1}zy=1$&$xz^{-2}x=1$&$x^{-1}yz^{-1}xz^{-1}y=1$&$x^{-1}yx^{-1}zy^{-1}z=1$\\
$27$&$xy^{-1}zx^{-1}y=1$&$xz^{-1}y^{-1}z=1$&$xy^{-1}zx^{-1}y=1$&$xz^{-1}y^{-1}z=1$\\
$28$&$xy^{-1}zy^{-1}x=1$&$xz^{-1}y^{-1}z=1$&$xy^{-1}zy^{-1}x=1$&$xz^{-1}y^{-1}z=1$\\
$29$&$yzx^{-1}y=1$&$yx^{-1}z^2=1$&$yzx^{-1}y=1$&$yx^{-1}z^2=1$\\
$30$&$yzx^{-1}y=1$&$yz^{-1}xz=1$&$yzx^{-1}y=1$&$yz^{-1}xz=1$\\
$31$&$yzy^{-1}x=1$&$yx^{-1}z^2=1$&$yzy^{-1}x=1$&$yx^{-1}z^2=1$\\
$32$&$(yx^{-1})^2z=1$&$yz^{-1}x^{-1}y=1$&$(yx^{-1})^2z=1$&$yz^{-1}x^{-1}y=1$\\
$33$&$yx^{-1}yz^{-1}x=1$&$yz^{-1}x^{-1}y=1$&$yx^{-1}yz^{-1}x=1$&$yz^{-1}x^{-1}y=1$\\
$34$&$yx^{-1}yz^{-1}x=1$&$yz^{-1}x^{-1}y=1$&$yx^{-1}yz^{-1}x=1$&$yz^{-1}xy^{-1}z=1$\\
$35$&$yx^{-1}zx^{-1}y=1$&$yz^{-1}x^{-1}z=1$&$yx^{-1}zx^{-1}y=1$&$yz^{-1}x^{-1}z=1$\\
$36$&$yx^{-1}zy^{-1}x=1$&$yz^{-1}x^{-1}z=1$&$yx^{-1}zy^{-1}x=1$&$yz^{-1}x^{-1}z=1$\\
$37$&$xz^{-1}xy^{-1}z=1$&$xyz^{-1}x=1$&$xz^{-1}xy^{-1}z=1$&$xyz^{-1}x=1$\\
$38$&$(xz^{-1})^2y=1$&$xyz^{-1}x=1$&$(xz^{-1})^2y=1$&$xyz^{-1}x=1$\\
$39$&$xz^{-1}yx^{-1}z=1$&$xyz^{-1}y=1$&$xz^{-1}yx^{-1}z=1$&$xyz^{-1}y=1$\\
$40$&$xz^{-1}yz^{-1}x=1$&$xyz^{-1}y=1$&$xz^{-1}yz^{-1}x=1$&$xyz^{-1}y=1$\\
$41$&$(xy^{-1})^2z=1$&$xzy^{-1}x=1$&$(xy^{-1})^2z=1$&$xzy^{-1}x=1$\\
$42$&$xy^{-1}xz^{-1}y=1$&$xzy^{-1}x=1$&$xy^{-1}xz^{-1}y=1$&$xzy^{-1}x=1$\\
$43$&$xy^{-1}zx^{-1}y=1$&$xzy^{-1}z=1$&$xy^{-1}zx^{-1}y=1$&$xzy^{-1}z=1$\\
$44$&$xy^{-1}zy^{-1}x=1$&$xzy^{-1}z=1$&$xy^{-1}zy^{-1}x=1$&$xzy^{-1}z=1$\\
$45$&$xz^{-1}x^{-1}z=1$&$xy^{-1}x^{-1}y=1$&$(yx^{-1})^2y=1$&$yz^{-1}yz=1$\\
$46$&$xz^{-1}x^{-1}z=1$&$xy^{-1}x^{-1}y=1$&$(yx^{-1})^2y=1$&$(yz^{-1})^2x=1$\\
$47$&$xz^{-1}y^{-1}z=1$&$xy^{-1}x^{-1}y=1$&$(yx^{-1})^2y=1$&$yz^{-1}x^{-1}z=1$\\
$48$&$xz^{-1}y^{-1}z=1$&$xy^{-1}x^{-1}y=1$&$(yx^{-1})^2y=1$&$yz^{-1}xz=1$\\
$49$&$xz^{-1}y^{-1}z=1$&$xy^{-1}x^{-1}y=1$&$zx^{-1}yx^{-1}z=1$&$zy^{-1}zx^{-1}y=1$\\
$50$&$xz^{-1}xz=1$&$xy^{-1}x^{-1}y=1$&$(yx^{-1})^2y=1$&$yz^{-1}y^{-1}z=1$\\
$51$&$xz^{-1}xz=1$&$xy^{-1}x^{-1}y=1$&$(yx^{-1})^2y=1$&$yz^{-1}yz=1$\\
$52$&$xz^{-1}xy^{-1}z=1$&$xy^{-1}x^{-1}y=1$&$(yx^{-1})^2y=1$&$yz^{-1}x^{-1}z=1$\\
$53$&$xz^{-1}xy^{-1}z=1$&$xy^{-1}x^{-1}y=1$&$(yx^{-1})^2y=1$&$yz^{-1}y^{-1}z=1$\\
$54$&$xz^{-1}xy^{-1}z=1$&$xy^{-1}x^{-1}y=1$&$(yx^{-1})^2y=1$&$yz^{-1}xz=1$\\
$55$&$xz^{-1}xy^{-1}z=1$&$xy^{-1}x^{-1}y=1$&$(yx^{-1})^2y=1$&$yz^{-1}yz=1$\\
$56$&$xz^{-1}yz=1$&$xy^{-1}x^{-1}y=1$&$(yx^{-1})^2y=1$&$yz^{-1}x^{-1}z=1$\\
$57$&$xz^{-1}yx^{-1}z=1$&$xy^{-1}x^{-1}y=1$&$(yx^{-1})^2y=1$&$yz^{-1}x^{-1}z=1$\\
$58$&$xz^{-1}yx^{-1}z=1$&$xy^{-1}x^{-1}y=1$&$(yx^{-1})^2y=1$&$yz^{-1}y^{-1}z=1$\\
$59$&$xz^{-1}yx^{-1}z=1$&$xy^{-1}x^{-1}y=1$&$(yx^{-1})^2y=1$&$yz^{-1}xz=1$\\
$60$&$xz^{-1}yx^{-1}z=1$&$xy^{-1}x^{-1}y=1$&$(yx^{-1})^2y=1$&$yz^{-1}yz=1$\\
$61$&$xz^{-1}x^{-1}y=1$&$xy^{-1}x^{-1}z=1$&$yx^{-1}zx^{-1}y=1$&$yz^{-1}y^{-1}z=1$\\
$62$&$xz^{-1}y^{-1}x=1$&$xy^{-1}x^{-1}z=1$&$x^{-1}yx^{-1}zy^{-1}z=1$&$x^{-1}yz^{-1}xz^{-1}y=1$\\
$63$&$xz^{-1}y^{-1}z=1$&$xy^{-2}x=1$&$zy^{-1}xy^{-1}z=1$&$zx^{-1}zy^{-1}x=1$\\
$64$&$xz^{-1}x^{-1}y=1$&$xy^{-1}z^{-1}x=1$&$x^{-1}yz^{-1}xz^{-1}y=1$&$x^{-1}yx^{-1}zy^{-1}z=1$\\
$65$&$xz^{-1}y^{-1}z=1$&$xy^{-1}z^{-1}x=1$&$x^{-1}yz^{-1}xy^{-1}z=1$&$(x^{-1}y)^2z^{-1}y=1$\\
$66$&$xz^{-2}y=1$&$xy^{-1}z^{-1}x=1$&$x^{-1}yz^{-1}xz^{-1}y=1$&$x^{-1}yx^{-1}zy^{-1}z=1$\\
$67$&$(xz^{-1})^2y=1$&$xy^{-1}z^{-1}x=1$&$(xz^{-1})^2y=1$&$xy^{-1}z^{-1}x=1$\\
$68$&$xz^{-1}x^{-1}z=1$&$xy^{-1}z^{-1}y=1$&$yz^{-1}yx^{-1}z=1$&$yx^{-1}yz^{-1}x=1$\\
$69$&$xz^{-1}x^{-1}z=1$&$xy^{-1}z^{-1}y=1$&$yz^{-1}yx^{-1}z=1$&$yx^{-1}zx^{-1}y=1$\\
$70$&$xz^{-2}x=1$&$xy^{-1}z^{-1}y=1$&$(yz^{-1})^2x=1$&$yx^{-1}zx^{-1}y=1$\\
$71$&$xz^{-1}yx^{-1}z=1$&$xy^{-1}z^{-1}y=1$&$xz^{-1}yx^{-1}z=1$&$xy^{-1}z^{-1}y=1$\\
$72$&$xz^{-1}yz^{-1}x=1$&$xy^{-1}z^{-1}y=1$&$xz^{-1}yz^{-1}x=1$&$xy^{-1}z^{-1}y=1$\\
$73$&$xz^{-1}y^{-1}z=1$&$xy^{-1}zy=1$&$xyx^{-1}z=1$&$xy^{-1}zy=1$\\
$74$&$(yx^{-1})^2z=1$&$yzx^{-1}y=1$&$(yx^{-1})^2z=1$&$yzx^{-1}y=1$\\
$75$&$yx^{-1}yz^{-1}x=1$&$yzx^{-1}y=1$&$yx^{-1}yz^{-1}x=1$&$yzx^{-1}y=1$\\
$76$&$yx^{-1}zy^{-1}x=1$&$yzx^{-1}z=1$&$yx^{-1}zy^{-1}x=1$&$yzx^{-1}z=1$\\
$77$&$yx^{-1}zx^{-1}y=1$&$yzx^{-1}z=1$&$yx^{-1}zx^{-1}y=1$&$yzx^{-1}z=1$\\
\hline
\end{longtable}
 $\mathbf{\Gamma_2}$\textbf{ of Figure \ref{gamma}:} Taking into account the relations from table \ref{t1} which are not disproved and the relations from Table \ref{t4}, it can be seen that  there are  $16$ different cases for the relations of $5$ cycles of length $3$ on vertices of degree $ 4 $ in $ \Gamma $ with the structure of the graph $ \Gamma_2 $ as Figure \ref{gamma}. Using GAP \cite{a9}, a free group with generators $ x,y,z $ and     the relations of these $ 5 $ cycles which are among these cases   is finite or abelian, that is a
contradiction. Hence, $ \Gamma $ does not contain any subgraph isomorphic to $\Gamma_2$ on vertices of degree $ 4 $.\\
$\mathbf{\Gamma_3}$ \textbf{of Figure \ref{gamma}:}
In view of   the relations from table \ref{tc4} which are not rejected and the relations from Table \ref{t4},, it can be seen that there are  $71$ different cases for the relations of $ 4 $ cycles of length $ 3 $ and a cycle of length $ 4 $  on vertices of degree $ 4 $ in $\Gamma$  with the structure of $ \Gamma_3 $. Using GAP \cite{a9},  a free group with generators $ x,y,z $ and the relations of these $ 5 $ cycles of   $  67$  cases among these $71$ cases is finite or abelian that is a contradiction. Hence, there are just  $ 4 $ cases which may lead
to the existence of $ \Gamma_1$ in $ \Gamma $. In the following, we show that all such cases lead to contradictions. Thus, $ \Gamma $ does not contain any subgraph isomorphic to the graph $\Gamma_3$.
 \begin{itemize}
 \item[(1)] $ r_1:xz^{-1}xz=1,\;\;r_2:xy^{-1}x^{-1}y=1,\;\;r_3:(yx^{-1})^2y=1,\;\;r_4:yz^{-1}y^{-1}z=1 ,\;\;\; r_5:z(x^{-1}y)^2z^{-1}y=1 $.\\
 $ r_2 $ and $ r_4 $ imply that  $\left\langle  x,y\right\rangle  $ and $\left\langle  z,y\right\rangle  $ are  abelian groups. Hence,  $ r_3 $ leads to $ x^2=y^3 $. On the other hand, by $ r_1 $, $ x^2=(x^{-2})^z $, thereupon $ y^3=(y^{-3})^z $ which implies  $ y^6=1 $, a contradiction.
\item[(2)] $ r_1:xz^{-1}xz=1,\;\;\;r_2:xy^{-1}x^{-1}y=1,\;\;\;r_3:(yx^{-1})^2y=1,\;\;\;r_4:yz^{-1}yz=1,\;\;\; r_5:z(x^{-1}y)^2z^{-1}y=1 $.\\
$ r_3 $ and $ r_5 $ imply that $\left\langle  z,y\right\rangle  $ is an abelian group. Then $ r_4 $ implies  $ y^2=1 $, a contradiction.
\item[(3)] $ r_1:xz^{-1}yz=1,\;\;r_2:xy^{-1}x^{-1}y=1,\;\;r_3:(yx^{-1})^2y=1,\;\;r_4:yz^{-1}x^{-1}z=1,\;\; r_5:z(x^{-1}y)^2z^{-1}x=1 $.\\
$ r_2 $ implies that $\left\langle  x,y\right\rangle  $ is an abelian group.  Hence,  $ r_3 $ leads to $ x^2=y^3 $. Also by  $ r_1 $ and $ r_4 $, $ x^2=(y^2)^{z^{-1}} $ and $ x^2=(y^{-2})^z $ which imply  $ (y^2)^{z^{-1}}y^{-3}=1 $ and $ (y^{-2})^zy^{-3}=1  $. Thereupon, $ (y^5)^{z^{-1}}y^{-1}=1 $.  Therefore, $ \left\langle x,y,z\right\rangle  $ is isomorphic to a quotient group of BS$ (1,5) $, that is a contradiction. 
\item[(4)]$r_1:xy^{-1}x^{-1}y=1,\;\;r_2:xz^{-1}y^{-1}z=1,\;\;r_3:zy^{-1}zx^{-1}y=1,\;\; r_4: zx^{-1}zy^{-1}x=1,\;\; r_5:(x^{-1}y)^2z^{-1}xy^{-1}z=1$.\\
In this case we have $x=y$, a contradiction.
 \end{itemize}
$\mathbf{\Gamma_4}$ \textbf{of Figure \ref{gamma}:}Taking into account the relations from table \ref{tc4} which are not rejected and the relations from Table \ref{t4}, it can be seen that  there are  $26$ different cases for the relations of $4$ cycles of length $3$ and a cycle of length $ 4 $  with the structure of the graph $ \Gamma_4 $ as Figure \ref{gamma} on vertices of degree $ 4 $ in $ \Gamma $.
Using GAP \cite{a9}, a free group with generators $ x,y,z $ and    the  relations of these $ 5 $ cycles which are among $25$ cases of these $26$ cases  is finite or abelian, that is a
contradiction. The other case is the following:
\[
r_1:xz^{-1}xz=1,\;r_2:xy^{-1}x^{-1}y=1,\;r_3:(yx^{-1})^2y=1,\;r_4:yz^{-1}yz=1,\;r_5:zx^{-1}yz^{-1}x^{-1}y=1.   
\]
Now, according to above relation and the relations from table \ref{t1} which are not rejected,  it can be seen that there is no cases for the existence of  $\Gamma_4$  of Figure \ref{gamma} as a subgraph  of $\Gamma$. Then,   $ \Gamma $ does not contain any subgraph isomorphic to $\Gamma_4$.\\
 \begin{figure}
\subfloat{$ G_1 $}
\begin{tikzpicture}[scale=.75]
\draw [fill] (0,0) circle
[radius=0.1] node  [above]  {};
\draw [fill] (1,0) circle
[radius=0.1] node  [right]  {};
\draw [fill] (0,-1) circle
[radius=0.1] node  [above right]  {};
\draw [fill] (1,-1) circle
[radius=0.1] node  [below]  {};
\draw [fill] (-1,0) circle
[radius=0.1] node  [left]  {};
\draw [fill] (-1,-1) circle
[radius=0.1] node  [below]  {};
\draw [fill] (0,-2) circle
[radius=0.1] node  [below]  {};
\draw  (0,0) -- (1,0) ;
\draw   (0,-1) -- (0,-.1);
\draw   (-1,0)-- (0,-.95);
\draw  (-1,-1)-- (0,-1);
\draw (0,0)--(-1,0);
\draw  (1,0) -- (1,-1);
\draw  (0,0) -- (-1,0);
\draw  (0,0) -- (-1,-1);
\draw  (1,-1) -- (0,-1);
\draw   (0,-2)-- (1,-1) ;
\draw  (-1,-1) -- (0,-2);
\end{tikzpicture}
\subfloat{$  G_2 $}
\begin{tikzpicture}[scale=.75]
\draw [fill] (0,0) circle
[radius=0.1] node  [above]  {};
\draw [fill] (1,0) circle
[radius=0.1] node  [right]  {};
\draw [fill] (0,-1) circle
[radius=0.1] node  [above right]  {};
\draw [fill] (1,-1) circle
[radius=0.1] node  [below right]  {};
\draw [fill] (-1,0) circle
[radius=0.1] node  [left]  {};
\draw [fill] (-1,-1) circle
[radius=0.1] node  [below]  {};
\draw [fill] (0,-2) circle
[radius=0.1] node  [below]  {};
\draw  (0,0) -- (1,0) ;
\draw   (0,-1) -- (0,-.1);
\draw   (-1,0)-- (0,-.95);
\draw  (-1,-1)-- (0,-1);
\draw (0,0)--(-1,0);
\draw  (1,0) -- (1,-1);
\draw  (0,0) -- (-1,0);
\draw  (0,0) -- (-1,-1);
\draw  (1,-1) -- (0,-1);
\draw   (0,-2)-- (1,-1) ;
\draw  (-1,-1) -- (0,-2);
\draw  (0,-2) -- (-1,-1);
\draw  (0,-2) -- (-1,0);
\end{tikzpicture}
\subfloat{$  G_3 $}
\begin{tikzpicture}[scale=.75]
\draw [fill] (0,0) circle
[radius=0.1] node  [above]  {};
\draw [fill] (1,0) circle
[radius=0.1] node  [right]  {};
\draw [fill] (0,-1) circle
[radius=0.1] node  [above right]  {};
\draw [fill] (1,-1) circle
[radius=0.1] node  [below]  {};
\draw [fill] (-1,0) circle
[radius=0.1] node  [left]  {};
\draw [fill] (-1,-1) circle
[radius=0.1] node  [below]  {};
\draw [fill] (0,-2) circle
[radius=0.1] node  [below]  {};
\draw [fill] (-1.5,-2) circle
[radius=0.1] node  [below]  {};
\draw  (0,0) -- (1,0) ;
\draw   (0,-1) -- (0,-.1);
\draw   (-1,0)-- (0,-.95);
\draw  (-1,-1)-- (0,-1);
\draw (0,0)--(-1,0);
\draw  (1,0) -- (1,-1);
\draw  (0,0) -- (-1,0);
\draw  (0,0) -- (-1,-1);
\draw  (1,-1) -- (0,-1);
\draw   (0,-2)-- (1,-1) ;
\draw   (-1.5,-2)-- (1,-1) ;
\draw   (-1,-1) -- (0,-2);
\draw   (-1,0) -- (-1.5,-2);
\end{tikzpicture}
\caption{ The graph $G_1$ and two forbidden subgraphs which contains it, where the degrees of all vertices of any subgraph in $ Z(\alpha,\beta) $ and $ U(\alpha,\beta) $ must be $ 4 $. }\label{g1234}
\end{figure}
 $\mathbf{G_1}$ \textbf{of Figure \ref{g1234}:} Taking into account the relations from Table \ref{t222} and the relations from table \ref{tc4} which are not rejected , it can be see that  there are  $11052$ different cases for the relations of $ 2 $ cycles of length $ 3 $ and  $ 2 $ cycles of length $ 4 $ with the structure of the graph $G_1$ of Figure \ref{g1234} on vertices of degree $ 4 $ in $ \Gamma $. Using GAP \cite{a9}, a free group with generators $ x,y,z $ and the relations of these $ 4 $ cycles which are between $10952$ cases of these $11052$ cases  is finite or abelian, that is a
contradiction. Hence, there are $ 100 $ cases which may lead
to the existence of $ G_1$ in $ \Gamma $. We checked these such cases. In Table \ref{t6}, it can be seen that $ 84 $ cases among these $ 100 $ cases lead to contradictions and so there are just $ 16 $ cases which may lead
to the existence of $ G _1$ in $ \Gamma $ which listed in Table \ref{t66}.
\begin{longtable}{|l|l|l|}
\caption{The relations of  $ 2 $ cycles of length $ 3 $ and  $ 2 $ cycles of length $ 4 $ with the structure of the graph $G_1$ of Figure \ref{g1234} which lead to  contradictions. }\label{t6}\\
\hline
\endfirsthead
\multicolumn{3}{c}%
{{\bfseries \tablename\ \thetable{} -- continued from previous page}} \\\hline
\hline 
 \multicolumn{1}{|l|}{\textbf{$ n $}} & \multicolumn{1}{l|}{\textbf{$r_1,\;r_2,\;r_3,\;r_4$}}&  \multicolumn{1}{|l|}{\textbf{$E $}}  \\ \hline
\endhead

\hline \multicolumn{3}{|r|}{{Continued on next page}} \\ \hline
\endfoot

\hline 
\endlastfoot

\hline
\textbf{$ n $} & \textbf{$r_1,\;r_2,\;r_3,\;r_4$} & \textbf{$E$}\\
\hline
$1$&$x^2y^{-1}x=1,\;xy^{-1}x^{-1}y=1,\;xz^{-1}xzy^{-1}z=1,\;x^{-1}y(z^{-1}x)^2y^{-1}z$&BS$(1,3)$\\
\hline
$2$&$x^2y^{-1}x=1,\;xy^{-1}x^{-1}y=1,\;(xz^{-1})^2x^{-1}z=1,\;x^{-1}yz^{-1}xy^{-1}zx^{-1}z=1$&BS$(1,2)$\\\hline
$3$&$xzy^{-1}x=1,\;xy^{-1}z^{-1}y=1,\;(xz^{-1})^2x^{-1}z=1,\;x^{-1}yx^{-1}zy^{-1}zx^{-1}z=1$&BS$(1,2)$\\\hline
$4$&$xzy^{-1}x=1,\;xy^{-1}z^{-1}y=1,\;xz^{-1}yz^{-1}y^{-1}z=1,\;x^{-1}y(x^{-1}z)^2y^{-1}z=1$&$z=x$\\\hline
$5$&$xzy^{-1}x,\;xy^{-1}xz^{-1}y=1,\;xz^{-1}yz^{-1}xy^{-1}z=1,\;x^{-1}yx^{-1}(zy^{-1})^2z=1$&Abelian\\\hline
$6$&$xzy^{-1}x=1,\;xy^{-1}zy=1,\;xz^{-1}xzx^{-1}z=1,\;(x^{-1}y)^2z^{-1}yx^{-1}z=1$&BS$(1,-3)$\\\hline
$7$&$xzy^{-1}x=1,\;xy^{-1}zy=1,\;xz^{-1}xzy^{-1}z=1,\;(x^{-1}y)^2z^{-1}xy^{-1}z=1$&Abelian\\\hline
$8$&$xzy^{-1}x=1,\;xy^{-1}zy=1,\;xz^{-1}xy^{-1}x^{-1}z=1,\;(x^{-1}y)^2z^{-1}yx^{-1}z=1$&BS$(-1,3)$\\\hline
$9$&$xzy^{-1}x=1,\;xy^{-1}zy=1,\;xz^{-1}xy^{-2}z=1,\;(x^{-1}y)^2z^{-1}xy^{-1}z=1$&Abelian\\\hline
$10$&$xzy^{-1}x=1,\;xy^{-1}zy=1,\;xz^{-1}yzx^{-1}z=1,\;(x^{-1}y)^2z^{-1}yx^{-1}z=1$&BS$(-1,3)$\\\hline
$11$&$xzy^{-1}x=1,\;xy^{-1}zy=1,\;xz^{-1}yzy^{-1}z=1,\;(x^{-1}y)^2z^{-1}xy^{-1}z=1$&Abelian\\\hline
$12$&$xzy^{-1}x=1,\;xy^{-1}zx^{-1}y=1,\;xz^{-2}x^{-1}z=1,\;(x^{-1}y)^2(x^{-1}z)^2=1$&BS$(1,2)$\\\hline
$13$&$xzy^{-1}x=1,\;xz^{-2}y=1,\;xy^{-1}z^2y^{-1}z=1,\;x^{-1}yx^{-1}zy^{-1}xy^{-1}z=1$&Abelian\\\hline
$14$&$xzy^{-1}x=1,\;xz^{-1}yz=1,\;xy^{-1}zx^{-1}z^{-1}y=1,\;(x^{-1}y)^2(z^{-1}y)^2=1$&$y=x^{-1}$\\\hline
$15$&$xy^{-2}x=1,\;xz^{-1}x^{-1}z=1,\;x^3z^{-1}y=1,\;x^{-1}yx^{-1}zx^{-1}yz^{-1}y=1$&Abelian\\\hline
$16$&$xy^{-2}x=1,\;xz^{-1}x^{-1}z=1,\;x^2yz^{-1}y=1,\;x^{-1}yx^{-1}zy^{-1}xz^{-1}y=1$&BS$(1,-1)$\\\hline
$17$&$xy^{-2}x=1,\;xz^{-1}x^{-1}z=1,\;x^2y^{-1}z^{-1}y=1,\;x^{-1}yx^{-1}zx^{-1}yz^{-1}y=1$&BS$(1,-1)$\\\hline
$18$&$xy^{-2}x=1,\;xz^{-1}x^{-1}z=1,\;x^2y^{-1}z^{-1}y=1,\;x^{-1}yx^{-1}zy^{-1}xz^{-1}y=1$&BS$(1,-1)$\\\hline
$19$&$xy^{-2}x=1,\;xz^{-1}x^{-1}z=1,\;x^2y^{-1}xz^{-1}y=1,\;x^{-1}yx^{-1}zx^{-1}yz^{-1}y=1$&BS$(1,-1)$\\\hline
$20$&$xy^{-2}x=1,\;xz^{-1}x^{-1}z=1,\;x^2z^{-2}y=1,\;x^{-1}yx^{-1}zx^{-1}yz^{-1}y=1$&Abelian\\\hline
$21$&$xy^{-2}x=1,\;xz^{-1}x^{-1}z=1,\;x^2z^{-2}y=1,\;x^{-1}yx^{-1}zy^{-1}xz^{-1}y=1$&Abelian\\\hline
$22$&$xy^{-2}x=1,\;xz^{-1}x^{-1}z=1,\;x(xz^{-1})^2y=1,\;x^{-1}yx^{-1}zx^{-1}yz^{-1}y=1$&Abelian\\\hline
$23$&$xy^{-2}x=1,\;xz^{-1}x^{-1}z=1,\;xyxz^{-1}y=1,\;x^{-1}yx^{-1}zx^{-1}yz^{-1}y=1$&BS$(1,-1)$\\\hline
$24$&$xy^{-2}x=1,\;xz^{-1}x^{-1}z=1,\;xy^2z^{-1}y=1,\;x^{-1}yx^{-1}zy^{-1}xz^{-1}y=1$&BS$(1,-1)$\\\hline
$25$&$xy^{-2}x=1,\;xz^{-1}x^{-1}z=1,\;xyx^{-1}z^{-1}y=1,\;x^{-1}yx^{-1}zx^{-1}yz^{-1}y=1$&BS$(1,-1)$\\\hline
$26$&$xy^{-2}x=1,\;xz^{-1}x^{-1}z=1,\;xyx^{-1}z^{-1}y=1,\;x^{-1}yx^{-1}zy^{-1}xz^{-1}y=1$&BS$(1,-1)$\\\hline
$27$&$xy^{-2}x=1,\;xz^{-1}x^{-1}z=1,\;xyx^{-1}yz^{-1}y=1,\;x^{-1}yx^{-1}zy^{-1}xz^{-1}y=1$&BS$(1,-1)$\\\hline
$28$&$xy^{-2}x=1,\;xz^{-1}x^{-1}z=1,\;xyz^{-2}y=1,\;x^{-1}yx^{-1}zx^{-1}yz^{-1}y=1$&Abelian\\\hline
$29$&$xy^{-2}x=1,\;xz^{-1}x^{-1}z=1,\;xzxz^{-1}y=1,\;x^{-1}yx^{-1}zx^{-1}yz^{-1}y=1$&Abelian\\\hline
$30$&$xy^{-2}x=1,\;xz^{-1}x^{-1}z=1,\;xzyz^{-1}y=1,\;x^{-1}yx^{-1}zy^{-1}xz^{-1}y=1$&T\\\hline
$31$&$xy^{-2}x=1,\;xz^{-1}x^{-1}z=1,\;xzx^{-1}z^{-1}y=1,\;x^{-1}yx^{-1}zx^{-1}yz^{-1}y=1$&Abelian\\\hline
$32$&$xy^{-2}x=1,\;xz^{-1}x^{-1}z=1,\;xzx^{-1}z^{-1}y=1,\;x^{-1}yx^{-1}zy^{-1}xz^{-1}y=1$&Abelian\\\hline
$33$&$xy^{-2}x=1,\;xz^{-1}x^{-1}z=1,\;xzx^{-1}yz^{-1}y=1,\;x^{-1}yx^{-1}zy^{-1}xz^{-1}y=1$&T\\\hline
$34$&$xy^{-2}x=1,\;xz^{-1}y^{-1}z=1,\;x^3z^{-1}y=1,\;x^{-1}yx^{-1}zx^{-1}yz^{-1}y=1$&BS$(1,-1)$\\\hline
$35$&$xy^{-2}x=1,\;xz^{-1}y^{-1}z=1,\;x^2yz^{-1}y=1,\;x^{-1}yx^{-1}zy^{-1}xz^{-1}y=1$&Abelian\\\hline
$36$&$xy^{-2}x=1,\;xz^{-1}y^{-1}z=1,\;x^2y^{-1}z^{-1}y=1,\;x^{-1}yx^{-1}zx^{-1}yz^{-1}y=1$&Abelian\\\hline
$37$&$xy^{-2}x=1,\;xz^{-1}y^{-1}z=1,\;x^2y^{-1}z^{-1}y=1,\;x^{-1}yx^{-1}zy^{-1}xz^{-1}y=1$&Abelian\\\hline
$38$&$xy^{-2}x=1,\;xz^{-1}y^{-1}z=1,\;x^2y^{-1}xz^{-1}y=1,\;x^{-1}yx^{-1}zx^{-1}yz^{-1}y=1$&Abelian\\\hline
$39$&$xy^{-2}x=1,\;xz^{-1}y^{-1}z=1,\;x^2z^{-2}y=1,\;x^{-1}yx^{-1}zx^{-1}yz^{-1}y=1$&Abelian\\\hline
$40$&$xy^{-2}x=1,\;xz^{-1}y^{-1}z=1,\;x^2z^{-2}y=1,\;x^{-1}yx^{-1}zy^{-1}xz^{-1}y=1$&Abelian\\\hline
$41$&$xy^{-2}x=1,\;xz^{-1}y^{-1}z=1,\;x(xz^{-1})^2y=1,\;x^{-1}yx^{-1}zx^{-1}yz^{-1}y=1$&Abelian\\\hline
$42$&$xy^{-2}x=1,\;xz^{-1}y^{-1}z=1,\;x^2(z^{-1}y)^2=1,\;x^{-1}yx^{-1}zy^{-1}xz^{-1}y=1$&BS$(1,-1)$\\\hline
$43$&$xy^{-2}x=1,\;xz^{-1}y^{-1}z=1,\;xyx^{-1}z^{-1}y=1,\;x^{-1}yx^{-1}zy^{-1}xz^{-1}y=1$&$z=xy$\\\hline
$44$&$xy^{-2}x=1,\;xz^{-1}y^{-1}z=1,\;xzy^{-1}xz^{-1}y=1,\;x^{-1}yx^{-1}zx^{-1}yz^{-1}y=1$&T\\\hline
$45$&$xy^{-2}x=1,\;xz^{-1}xz=1,\;xzx^{-1}z^{-1}y=1,\;x^{-1}yx^{-1}zx^{-1}yz^{-1}y=1$&BS$(1,-1)$\\\hline
$46$&$xy^{-2}x=1,\;xz^{-1}xz=1,\;xzx^{-1}z^{-1}y=1,\;x^{-1}yx^{-1}zy^{-1}xz^{-1}y=1$&BS$(1,-1)$\\\hline
$47$&$xy^{-2}x=1,\;xz^{-1}xz=1,\;xzx^{-1}yz^{-1}y=1,\;x^{-1}yx^{-1}zy^{-1}xz^{-1}y=1$&T\\\hline
$48$&$xy^{-2}x=1,\;xz^{-1}xz=1,\;xzy^{-1}z^{-1}y=1,\;x^{-1}yx^{-1}zx^{-1}yz^{-1}y=1$&$y=x$\\\hline
$49$&$xy^{-2}x=1,\;xz^{-1}xz=1,\;xzy^{-1}z^{-1}y=1,\;x^{-1}yx^{-1}zy^{-1}xz^{-1}y=1$&T\\\hline
$50$&$xy^{-2}x=1,\;(xz^{-1})^2y=1,\;xzyx^{-1}z=1,\;(x^{-1}y)^2(x^{-1}z)^2=1$&T\\\hline
$51$&$xy^{-2}x=1,\;(xz^{-1})^2y=1,\;xzy^{-1}x^{-1}z=1,\;(x^{-1}y)^2(x^{-1}z)^2=1$&$y=x$\\\hline
$52$&$xy^{-2}x=1,\;xz^{-1}yz=1,\;xzx^{-1}z^{-1}y=1,\;x^{-1}yx^{-1}zx^{-1}yz^{-1}y=1$&$y=x$\\\hline
$53$&$xy^{-2}x=1,\;xz^{-1}yz=1,\;xzx^{-1}z^{-1}y=1,\;x^{-1}yx^{-1}zy^{-1}xz^{-1}y=1$&BS$(1,2)$\\\hline
$54$&$xy^{-2}x=1,\;xz^{-1}yz=1,\;xzy^{-1}z^{-1}y=1,\;x^{-1}yx^{-1}zx^{-1}yz^{-1}y=1$&$x=y$\\\hline
$55$&$xy^{-2}x=1,\;xz^{-1}yz=1,\;xzy^{-1}z^{-1}y=1,\;x^{-1}yx^{-1}zy^{-1}xz^{-1}y=1$&$x=y$\\\hline
$56$&$xz^{-1}y^{-1}x=1,\;x^2y^{-1}z=1,\;xy^{-1}z^{-1}xz^{-1}y=1,\;x^{-1}yz^{-1}x(z^{-1}y)^2$=1&T\\\hline
$57$&$xy^{-1}zy^{-1}x=1,\;xz^{-1}x^{-1}z=1,\;x^2(z^{-1}y)^2=1,\;(x^{-1}y)^2z^{-1}xz^{-1}y=1$&BS$(1,-1)$\\\hline
$58$&$xy^{-1}zy^{-1}x=1,\;xz^{-1}x^{-1}z=1,\;xyx^{-1}z^{-1}y=1,\;(x^{-1}y)^2(z^{-1}y)^2=1$&BS$(1,-1)$\\\hline
$59$&$xy^{-1}zy^{-1}x=1,\;xz^{-1}x^{-1}z=1,\;xzy^{-1}z^{-1}y=1,\;(x^{-1}y)^2(z^{-1}y)^2=1$&BS$(1,-1)$\\\hline
$60$&$xy^{-1}zy^{-1}x=1,\;xz^{-1}x^{-1}z=1,\;xzy^{-1}xz^{-1}y=1,\;(x^{-1}y)^2(z^{-1}y)^2=1$&BS$(1,-1)$\\\hline
$61$&$xy^{-1}zy^{-1}x=1,\;xz^{-2}y=1,\;x^2y^{-2}z=1,\;x^{-1}yx^{-1}zy^{-1}xy^{-1}z=1$&Abelian\\\hline
$62$&$xy^{-1}zy^{-1}x=1,\;xz^{-1}xz=1,\;xyxz^{-1}y=1,\;(x^{-1}y)^2(z^{-1}y)^2=1$&BS$(1,-1)$\\\hline
$63$&$xy^{-1}zy^{-1}x=1,\;xz^{-1}xz=1,\;xyx^{-1}z^{-1}y=1,\;(x^{-1}y)^2(z^{-1}y)^2=1$&BS$(1,-1)$\\\hline
$64$&$xy^{-1}zy^{-1}x=1,\;xz^{-1}xz=1,\;xzy^{-1}z^{-1}y=1,\;(x^{-1}y)^2(z^{-1}y)^2=1$&BS$(1,-1)$\\\hline
$65$&$xy^{-1}zy^{-1}x=1,\;xz^{-1}xz=1,\;xzy^{-1}xz^{-1}y=1,\;(x^{-1}y)^2(z^{-1}y)^2=1$&BS$(1,-1)$\\\hline
$66$&$yzy^{-1}x=1,\;yz^{-1}x^{-1}z=1,\;yx^{-1}zy^{-1}z^{-1}y=1,\;x^{-1}yx^{-1}zx^{-1}yz^{-1}y=1$&BS$(1,-1)$\\\hline
$67$&$yzy^{-1}x=1,\;yz^{-1}x^{-1}z=1,\;yx^{-1}zy^{-1}z^{-1}y=1,\;x^{-1}yx^{-1}zy^{-1}xz^{-1}y=1$&BS$(1,-1)$\\\hline
$68$&$yx^{-1}zy^{-1}x=1,\;yz^{-1}y^{-1}z=1,\;yzx^{-1}yz^{-1}y=1,\;(x^{-1}y)^2z^{-1}xz^{-1}y=1$&Abelian\\\hline
$69$&$zx^{-1}zy^{-1}x=1,\;zy^{-1}xz^{-1}y=1,\;z(zy^{-1})^2z=1,\;x^{-1}yx^{-1}zy^{-1}xy^{-1}z=1$&BS$(1,-1)$\\\hline
$70$&$zx^{-1}zy^{-1}x=1,\;zy^{-1}zx^{-1}y=1,\;z(zx^{-1})^2z=1,\;(x^{-1}y)^2z^{-1}yx^{-1}z=1$&BS$(1,-1)$\\\hline
$71$&$zx^{-1}zy^{-1}x=1,\;zy^{-1}zx^{-1}y=1,\;z^2x^{-1}zy^{-1}z=1,\;(x^{-1}y)^2z^{-1}xy^{-1}z=1$&Abelian\\\hline
$72$&$zx^{-1}zy^{-1}x=1,\;zy^{-1}zx^{-1}y=1,\;z^2y^{-1}zx^{-1}z=1,\;(x^{-1}y)^2z^{-1}yx^{-1}z=1$&Abelian\\\hline
$73$&$zx^{-1}zy^{-1}x=1,\;zy^{-1}zx^{-1}y=1,\;z(zy^{-1})^2z=1,\;(x^{-1}y)^2z^{-1}xy^{-1}z=1$&BS$(1,-1)$\\\hline
$74$&$xz^{-1}y^{-1}x=1,\;x^2y^{-1}z=1,\;xy^{-2}xz^{-1}y=1,\;x^{-1}yz^{-1}x(z^{-1}y)^2=1$&T\\\hline
$75$&$xz^{-1}y^{-1}x=1,\;x^2y^{-1}z=1,\;xy^{-1}zxz^{-1}y=1,\;x^{-1}yz^{-1}x(z^{-1}y)^2=1=$&T\\\hline
$76$&$xz^{-1}y^{-1}x=1,\;x^2y^{-1}z=1,\;xy^{-1}zx^{-1}z^{-1}y=1,\;x^{-1}yz^{-1}x(z^{-1}y)^2=1$&T\\\hline
$77$&$xz^{-1}y^{-1}x=1,\;x^2y^{-1}z=1,\;xy^{-1}zx^{-1}z^{-1}y=1,\;x^{-1}(yz^{-1})^2xz^{-1}y=1$&BS$(1,-1)$\\\hline
$78$&$xz^{-1}yz^{-1}x=1,\;xy^{-1}x^{-1}y=1,\;xyz^{-1}y^{-1}z=1,\;(x^{-1}z)^2(y^{-1}z)^2=1$&BS$(1,-1)$\\\hline
$79$&$xz^{-1}yz^{-1}x=1,\;xy^{-1}x^{-1}y=1,\;xyz^{-1}xy^{-1}z=1,\;(x^{-1}z)^2(y^{-1}z)^2=1$&BS$(1,-1)$\\\hline
$80$&$xz^{-1}yz^{-1}x=1,\;xy^{-1}x^{-1}y=1,\;xzxy^{-1}z=1,\;(x^{-1}z)^2(y^{-1}z)^2=1$&BS$(1,-1)$\\\hline
$81$&$xz^{-1}y^{-1}x=1,\;xy^{-1}zx^{-1}y=1,\;xyx^{-1}y^{-1}z=1,\;x^{-1}yz^{-1}yx^{-1}zy^{-1}z=1$&BS$(1,-1)$\\\hline
$82$&$xz^{-1}y^{-1}x=1,\;xy^{-1}zx^{-1}y=1,\;xzxy^{-1}z=1,\;x^{-1}yz^{-1}yx^{-1}zy^{-1}z=1$&T\\\hline
$83$&$yz^{-2}x=1,\;y^2x^{-1}y=1,\;yx^{-1}zy^{-2}z=1,\;(x^{-1}z)^2(y^{-1}z)^2=1$&T\\\hline
$84$&$yx^{-1}zy^{-1}x=1,\;yzy^{-1}z=1,\;yz^{-1}x^{-1}yz^{-1}y=1,\;(x^{-1}y)^2z^{-1}xz^{-1}y=1$&BS$(1,-1)$\\\hline
\end{longtable}
\begin{longtable}{|l|l|l|l|l|}
\caption{The relations of the  existence of  $ G_1 $ as Figure \ref{g1234}  in  $ Z(\alpha,\beta) $ and $ U(\alpha,\beta) $. }\label{t66}\\
\hline
\endfirsthead
\multicolumn{5}{c}%
{{\bfseries \tablename\ \thetable{} -- continued from previous page}} \\\hline
\hline 
 \multicolumn{1}{|l|}{\textbf{$ n $}} & \multicolumn{1}{l|}{\textbf{$r_1$}}&  \multicolumn{1}{|l|}{\textbf{$r_2 $}} & \multicolumn{1}{l|}{\textbf{$r_3$}}&  \multicolumn{1}{|l|}{\textbf{$r_4 $}} \\ \hline
\endhead

\hline \multicolumn{5}{|r|}{{Continued on next page}} \\ \hline
\endfoot

\hline 
\endlastfoot

\hline
\textbf{$ n $} & \textbf{$r_1$} & \textbf{$r_2$}& \textbf{$r_3$} & \textbf{$r_4$}\\
\hline
$ 1 $&$ x^2y^{-1}x=1  $&$xy^{-1}x^{-1}y =1  $&$ xz^{-1}x^{-1}zy^{-1}z=1  $&$ x^{-1}(yz^{-1})^2xy^{-1}z =1 $\\
$ 2 $&$ x^2y^{-1}x=1 $&$ xy^{-1}x^{-1}y =1 $&$  xz^{-2}xy^{-1}z=1 $&$ x^{-1}yz^{-1}yx^{-1}zy^{-1}z=1  $\\
$ 3 $&$ x^2y^{-1}x =1$&$ xy^{-1}x^{-1}y =1 $&$ (xz^{-1})^2y^{-1}z=1 $&$ x^{-1}yz^{-1}yx^{-1}zy^{-1}z=1  $\\
$ 4 $&$ x^2y^{-1}x=1 $&$ xy^{-1}x^{-1}y =1 $&$ xz^{-1}yz^{-1}y^{-1}z =1 $&$ x^{-1}y(z^{-1}x)^2y^{-1}z=1  $\\
$ 5 $&$ xy^{-2}x =1 $&$ xz^{-1}x^{-1}z =1 $&$ x^2(z^{-1}y)^2=1  $&$  x^{-1}yx^{-1}zy^{-1}xz^{-1}y=1$\\
$ 6 $&$  xy^{-2}x=1 $&$ xz^{-1}x^{-1}z=1 $&$ xzy^{-1}xz^{-1}y =1 $&$x^{-1}yx^{-1}zx^{-1}yz^{-1}y=1   $\\
$  7$&$  xy^{-2}x=1 $&$ xz^{-1}y^{-1}z=1  $&$ xyz^{-2}y =1$&$ x^{-1}yx^{-1}zy^{-1}xz^{-1}y =1 $\\
$8  $&$  xy^{-2}x=1 $&$  xz^{-1}y^{-1}z=1 $&$ xzx^{-1}yz^{-1}y =1 $&$ x^{-1}yx^{-1}zy^{-1}xz^{-1}y=1 $\\
$ 9 $&$ xy^{-2}x=1 $&$ xz^{-1}xz=1 $&$ xzy^{-1}xz^{-1}y =1 $&$ x^{-1}yx^{-1}zx^{-1}yz^{-1}y =1 $\\
$ 10 $&$ xy^{-2}x =1 $&$ xz^{-1}yz=1 $&$ xzx^{-1}yz^{-1}y =1 $&$ x^{-1}yx^{-1}zy^{-1}xz^{-1}y =1 $\\
$ 11 $&$ xy^{-2}x=1  $&$  xz^{-1}yz=1 $&$ xzy^{-1}xz^{-1}y =1 $&$ x^{-1}yx^{-1}zx^{-1}yz^{-1}y=1$\\
$ 12 $&$ xz^{-1}y^{-1}x=1 $&$ xyz^{-1}y =1 $&$ xy^{-1}z^{-1}yx^{-1}z =1 $&$ (x^{-1}y)^2(x^{-1}z)^2=1 $\\
$ 13 $&$ xz^{-1}yz^{-1}x=1  $&$ xy^{-1}x^{-1}y =1 $&$ xzx^{-1}y^{-1}z =1 $&$ (x^{-1}z)^2(y^{-1}z)^2=1  $\\
$ 14 $&$ xz^{-1}xy^{-1}x =1 $&$ xy^{-1}z^{-1}y=1 $&$x^2z^{-1}y^{-1}z=1   $&$ x^{-1}y(z^{-1}x)^2y^{-1}z=1  $\\
$ 15 $&$ xz^{-1}xy^{-1}x =1 $&$ xy^{-1}z^{-1}y =1 $&$ x(zy^{-1})^2z=1  $&$ x^{-1}y(z^{-1}x)^2y^{-1}z =1 $\\
$  16$&$ xz^{-1}y^{-1}x =1 $&$ xy^{-1}zy=1  $&$ xzx^{-1}yx^{-1}z=1 $&$ (x^{-1}y)^2(x^{-1}z)^2=1  $\\
\hline
\end{longtable}
$\mathbf{G_2}$ \textbf{of Figure \ref{g1234}:} 
Taking into account the relations from Table \ref{t66} and the relations from Table \ref{tc4} which are not rejected, it can be see that  there are  $53$ different cases for the relations of $2$ cycles of length $3$ and $ 3 $ cycles of length $ 4 $  with the structure of the graph $ G_2 $ as Figure \ref{g1234} on vertices of degree $ 4 $ in $ \Gamma $. Using GAP \cite{a9}, a free group with generators $ x,y,z $ and   the relations of such cases  is finite, that is a contradiction. Hence, $ \Gamma $ does not contain any subgraph isomorphic to $G_2$ of Figure \ref{g1234} on vertices of degree $ 4 $.\\
$\mathbf{G_3}$ \textbf{of Figure \ref{g1234}:} According to the  the relations from Table \ref{t66} and the relations from Table \ref{tc4} which are not rejected, it can be see that  there are  $222$ different cases for the relations of $2$ cycles of length $3$ and $ 3 $ cycles of length $ 4 $ with the structure of the graph $ G_3 $ as Figure \ref{g1234} on vertices of degree $ 4 $ in $ \Gamma $ . Using GAP \cite{a9}, a free group with generators $ x,y,z $ and the  relations of these $ 5 $ cycles which are between $221$ cases of these $222$ cases  is finite, that is a contradiction. Another case is as follows:\\
$ r_1:x^2y^{-1}x=1,\; r_2:xy^{-1}x^{-1}y=1,\; r_3:(xz^{-1})^2y^{-1}z=1,\; 
r_4:x^{-1}yz^{-1}yx^{-1}zy^{-1}z=1,\;r_5:y^{-1}xz^{-1}yz^{-1}xy^{-1}z=1$.\\
It is not hard to see that in this $ \left\langle x,y,z\right\rangle   $  is a quotient  group  of
 $ BS(-1,2)$, a contradiction. Therefore,  $ \Gamma $ does not contain any subgraph isomorphic to $G_3$ of Figure \ref{g1234} on vertices of degree $ 4 $.\\
$\mathbf{K_1}$ \textbf{of \ref{125}:} According to the  the relations from Table \ref{t222} and the relations from Table \ref{tc4} which are not rejected, it can be seen that  there are  $54385$ different cases for the relations of $4$ cycles of length $3$ and a cycle of length $ 4 $ with the structure of the graph $ K_1 $ as Figure \ref{125} on vertices of degree $ 4 $ in $ \Gamma $ . Using GAP \cite{a9}, a free group with generators $ x,y,z $ and  the  relations of these $ 5 $ cycles which are between $54293$ cases of these $54385$ cases  is finite or abelian, that is a
contradiction. Hence, there are $ 92 $ cases which may lead
to the existence of $ K_1$ in $ \Gamma $. In Table \ref{t7}, it can be seen that all such cases  lead to contradictions and therefore $ \Gamma $ does not contain any subgraph isomorphic to $K_1$ of Figure \ref{125} on vertices of degree $ 4 $.\\
\begin{figure}
\subfloat{$ K_1 $}
\begin{tikzpicture}[scale=.75,shorten >=2pt]
\draw [fill] (0,0) circle
[radius=0.1] node  [above]  {};
\draw [fill] (1.5,0) circle
[radius=0.1] node  [above]  {};
\draw [fill] (0,-1.25) circle
[radius=0.1] node  [below]  {};
\draw [fill] (1.5,-1.25) circle
[radius=0.1] node  [below]  {};
\draw [fill] (-1.5,0) circle
[radius=0.1] node  [left]  {};
\draw [fill] (-1.5,-1.25) circle
[radius=0.1] node  [below]  {};
\draw [fill] (3,0) circle
[radius=0.1] node  [right]  {};
\draw [fill] (3,-1.25) circle
[radius=0.1] node  [below]  {};
\draw   (3,-1.25)--(1.5,0)  ;
\draw  (1.5,-1.25)-- (3,-1.25);
\draw  (1.5,-1.25)-- (3,0);
\draw   (3,0)-- (1.5,0);
\draw  (-1.5,0) -- (0,0) -- (1.5,0) ;
\draw  (1.5,0) -- (1.5,-1.25) ;
\draw  (1.5,-1.25) -- (0,-1.25);
\draw  (0,-1.25) -- (0,0);
\draw (-1.5,0)-- (0,-1.25);
\draw  (0,0) -- (-1.5,0);
\draw  (-1.5,-1.25)-- (0,-1.25);
\draw  (0,0) -- (-1.5,-1.25);
\end{tikzpicture}
\subfloat{$K_2$}
\begin{tikzpicture}[scale=.75,shorten >=2pt]
\draw [fill] (0,0) circle
[radius=0.1] node  [left]  {};
\draw [fill] (1.5,0) circle
[radius=0.1] node  [left]  {};
\draw [fill] (3,0) circle
[radius=0.1] node  [left]  {};
\draw [fill] (4.5,0) circle
[radius=0.1] node  [right]  {};
\draw [fill] (1.5,1.1) circle
[radius=0.1] node  [left]  {};
\draw [fill] (1.5,-1.1) circle
[radius=0.1] node  [left]  {};
\draw [fill] (0,-2.2) circle
[radius=0.1] node  [left]  {};
\draw  (1.5,-1.1)--(1.5,0); 
\draw  (1.5,0) --(1.5,1.1) ;
\draw  (1.5,1.1) -- (0,0) ;
\draw  (0,0)-- (1.5,-1.1);
\draw  (1.5,1.1) -- (3,0) ;
\draw  (3,0)-- (1.5,-1.1);
\draw  (1.5,1.1) -- (4.5,0) ;
\draw  (4.5,0)-- (1.5,-1.1);
\draw  (0,0)--(0,-2.2);
\draw  (0,-2.2) -- (4.5,0);
\draw  (0,-2.2) -- (1.5,0);
\end{tikzpicture}
\caption{ Two forbidden subgraphs of $Z(\alpha,\beta)$ and $U(\alpha,\beta)$, where the degree of all vertices of any subgraph in $Z(\alpha,\beta)$ and $U(\alpha,\beta)$ must be  $ 4 $. }\label{125}
\end{figure}
\begin{longtable}{|l|l|l|}
\caption{The relations of  $  4$ cycles of length $ 3 $ and a cycle of length $ 4 $ with the  structure of the graph  $ K_1 $ as Figure \ref{125}  which lead to  contradictions. }\label{t7}\\
\hline
\endfirsthead
\multicolumn{3}{c}%
{{\bfseries \tablename\ \thetable{} -- continued from previous page}} \\\hline
\hline 
 \multicolumn{1}{|l|}{\textbf{$ n $}} & \multicolumn{1}{l|}{\textbf{$R_1,\;R_2,\;R_3,\;R_4,\;R_5$}}&  \multicolumn{1}{|l|}{\textbf{$E $}}  \\ \hline
\endhead

\hline \multicolumn{3}{|r|}{{Continued on next page}} \\ \hline
\endfoot

\hline 
\endlastfoot

\hline
\textbf{$ n $} & \textbf{$R_1,\;R_2,\;R_3,\;R_4,\;R_5$} & \textbf{$E$}\\
\hline
$1$&$x^2y^{-1}x=1,\;xy^{-1}x^{-1}y=1,\;xy^{-1}x^{-1}y=1,\;xz^{-1}y^{-1}z=1,\;xzxz^{-1}x=1$&T\\\hline
$2$&$x^2y^{-1}x=1,\;xy^{-1}x^{-1}y=1,\;xy^{-1}x^{-1}y=1,\;xz^{-1}yx^{-1}z=1,\;xzxz^{-1}x=1$&BS$(2,-1)$\\\hline
$3$&$x^2y^{-1}x=1,\;xy^{-1}x^{-1}y=1,\;(xy^{-1})^2x=1,\;xz^{-1}xz=1,\;xzxz^{-1}y=1$&T\\\hline
$4$&$x^2z^{-1}x=1,\;xz^{-1}x^{-1}z=1,\;xy^{-1}z^{-1}y=1,\;xz^{-1}x^{-1}z=1,\;xyxy^{-1}x=1$&T\\\hline
$5$&$x^2z^{-1}x=1,\;xz^{-1}x^{-1}z=1,\;xy^{-1}xy=1,\;(xz^{-1})^2x=1,\;xyxy^{-1}z=1$&T\\\hline
$6$&$x^2z^{-1}x=1,\;xz^{-1}x^{-1}z=1,\;xy^{-1}zx^{-1}y=1,\;xz^{-1}x^{-1}z=1,\;xyxy^{-1}x=1$&BS$(2,-1)$\\\hline
$7$&$x^2z^{-1}y=1,\;xy^{-1}z^{-1}x=1,\;x^2y^{-1}z=1,\;xz^{-1}y^{-1}x=1,\;xy^{-1}zxz^{-1}y=1$&T\\\hline
$8$&$xy^2=1,\;xz^{-1}xz=1,\;xy^{-1}x^{-1}y=1,\;xz^{-1}xz=1,\;x^3y^{-1}x=1$&T\\\hline
$9$&$xy^2=1,\;xz^{-1}xz=1,\;xy^{-1}x^{-1}y=1,\;xz^{-1}yz=1,\;x^3y^{-1}x=1$&T\\\hline
$10$&$xyz=1,\;xy^{-1}xz^{-1}x=1,\;xy^{-1}xy=1,\;xz^{-1}xz=1,\;xyxz^{-1}x=1$&$x=z$\\\hline
$11$&$xyx^{-1}y=1,\;xz^{-1}xz=1,\;xy^{-1}x^{-1}y=1,\;xz^{-1}xz=1,\;x^3y^{-1}x=1$&T\\\hline
$12$&$xyx^{-1}y=1,\;xz^{-1}xz=1,\;xy^{-1}x^{-1}y=1,\;xz^{-1}xy^{-1}z=1,\;x^3y^{-1}x=1$&T\\\hline
$13$&$xyx^{-1}y=1,\;xz^{-1}xy^{-1}z=1,\;xy^{-1}x^{-1}y=1,\;xz^{-1}xz=1,\;x^3y^{-1}x=1$&T\\\hline
$14$&$xyx^{-1}y=1,\;xz^{-1}xy^{-1}z=1,\;xy^{-1}x^{-1}y=1,\;xz^{-1}xy^{-1}z=1,\;x^3y^{-1}x=1$&T\\\hline
$15$&$xyx^{-1}y=1,\;xz^{-1}yx^{-1}z=1,\;xy^{-1}x^{-1}y=1,\;xz^{-1}yx^{-1}z=1,\;x^3y^{-1}x=1$&T\\\hline
$16$&$xyx^{-1}z=1,\;x(z^{-1}y)^2=1,\;xy^{-1}xy=1,\;xz^{-2}x=1,\;x^3y^{-1}z=1$&T\\\hline
$17$&$xzy=1,\;xz^{-1}xy^{-1}x=1,\;xy^{-1}xy=1,\;xz^{-1}xz=1,\;xzxy^{-1}x=1$&$x=y$\\\hline
$18$&$xz^2=1,\;xy^{-1}x^{-1}y=1,\;x^2z^{-1}x=1,\;xz^{-1}x^{-1}z=1,\;xy^{-1}x^2z^{-1}y=1$&T\\\hline
$19$&$xz^2=1,\;xy^{-1}x^{-1}y=1,\;xy^{-1}x^{-1}y=1,\;xz^{-1}x^{-1}z=1,\;x^3z^{-1}x=1$&T\\\hline
$20$&$xz^2=1,\;xy^{-1}xy=1,\;xy^{-1}xy=1,\;xz^{-1}x^{-1}z=1,\;x^3z^{-1}x=1$&T\\\hline
$21$&$xz^2=1,\;xy^{-1}xy=1,\;xy^{-1}zy=1,\;xz^{-1}x^{-1}z=1,\;x^3z^{-1}x=1$&T\\\hline
$22$&$xzx^{-1}z=1,\;xy^{-1}x^{-1}y=1,\;x^2z^{-1}x=1,\;xz^{-1}x^{-1}z=1,\;xy^{-1}x^2z^{-1}y=1$&T\\\hline
$23$&$xzx^{-1}z=1,\;xy^{-1}x^{-1}y=1,\;xy^{-1}x^{-1}y=1,\;xz^{-1}x^{-1}z=1,\;x^3z^{-1}x=1$&T\\\hline
$24$&$xzx^{-1}z=1,\;xy^{-1}xy=1,\;xy^{-1}z^{-1}y=1,\;xz^{-1}xz=1,\;x^3z^{-1}x=1$&T\\\hline
$25$&$xzx^{-1}z=1,\;xy^{-1}xy=1,\;xy^{-1}xy=1,\;xz^{-1}x^{-1}z=1,\;x^3z^{-1}x=1$&T\\\hline
$26$&$xzx^{-1}z=1,\;xy^{-1}xy=1,\;xy^{-1}xz^{-1}y=1,\;xz^{-1}x^{-1}z=1,\;x^3z^{-1}x=1$&T\\\hline
$27$&$xzx^{-1}z=1,\;xy^{-1}xz^{-1}y=1,\;xy^{-1}xy=1,\;xz^{-1}x^{-1}z=1,\;x^3z^{-1}x=1$&T\\\hline
$28$&$xzx^{-1}z=1,\;xy^{-1}xz^{-1}y=1,\;xy^{-1}xz^{-1}y=1,\;xz^{-1}x^{-1}z=1,\;x^3z^{-1}x=1$&T\\\hline
$29$&$xzx^{-1}z=1,\;xy^{-1}zx^{-1}y=1,\;xy^{-1}zx^{-1}y=1,\;xz^{-1}x^{-1}z=1,\;x^3z^{-1}x=1$&T\\\hline
$30$&$xzy^{-1}z=1,\;xy^{-1}x^{-1}y=1,\;x^2z^{-1}x=1,\;xz^{-1}x^{-1}z=1,\;xy^{-1}x^2z^{-1}y=1$&Abelian\\\hline
$31$&$xzy^{-1}z=1,\;xy^{-1}x^{-1}y=1,\;xy^{-1}x^{-1}y=1,\;xz^{-1}x^{-1}z=1,\;x^3z^{-1}x=1$&Abelian\\\hline
$32$&$xy^{-1}x^{-1}y=1,\;xz^{-1}x^{-1}z=1,\;xy^{-1}x^{-1}y=1,\;xz^{-2}x=1,\;x^3z=1$&T\\\hline
$33$&$xy^{-1}x^{-1}y=1,\;xz^{-1}x^{-1}z=1,\;xy^{-1}x^{-1}y=1,\;(xz^{-1})^2x=1,\;x^3z=1$&T\\\hline
$34$&$xy^{-1}x^{-1}y=1,\;xz^{-1}x^{-1}z=1,\;xy^{-1}x^{-1}y=1,\;xz^{-1}yz^{-1}x=1,\;x^3z=1$&Abelian\\\hline
$35$&$xy^{-1}x^{-1}y=1,\;xz^{-1}y^{-1}z=1,\;xy^{-2}x=1,\;xz^{-1}y^{-1}z=1,\;x^3y$&T\\\hline
$36$&$xy^{-1}x^{-1}y=1,\;xz^{-1}xz=1,\;xy^{-2}x=1,\;xz^{-1}xz=1,\;x^3y=1$&T\\\hline
$37$&$xy^{-1}x^{-1}y=1,\;xz^{-1}xz=1,\;xy^{-2}x=1,\;xz^{-1}yz=1,\;x^3y=1$&T\\\hline
$38$&$xy^{-1}x^{-1}y=1,\;xz^{-1}xz=1,\;(xy^{-1})^2x=1,\;xz^{-1}xz=1,\;x^3y=1$&T\\\hline
$39$&$xy^{-1}x^{-1}y=1,\;xz^{-1}xy^{-1}z=1,\;(xy^{-1})^2x=1,\;xz^{-1}xz=1,\;x^3y=1$&T\\\hline
$40$&$xy^{-1}x^{-1}y=1,\;xz^{-1}yz=1,\;xy^{-2}x=1,\;xz^{-1}xz=1,\;x^3y=1$&T\\\hline
$41$&$xy^{-1}x^{-1}y=1,\;xz^{-1}yz=1,\;xy^{-2}x=1,\;xz^{-1}yz=1,\;x^3y=1$&T\\\hline
$42$&$xy^{-1}x^{-1}y=1,\;xz^{-1}yx^{-1}z=1,\;(xy^{-1})^2x=1,\;xz^{-1}yx^{-1}z=1,\;x^3y=1$&T\\\hline
$43$&$xy^{-1}z^{-1}y=1,\;xz^{-1}x^{-1}z=1,\;xy^{-1}z^{-1}y=1,\;xz^{-2}x=1,\;x^3z=1$&T\\\hline
$44$&$xy^{-1}xy=1,\;xz^{-1}x^{-1}z=1,\;xy^{-1}xy=1,\;xz^{-2}x=1,\;x^3z=1$&T\\\hline
$45$&$xy^{-1}xy=1,\;xz^{-1}x^{-1}z=1,\;xy^{-1}xy=1,\;(xz^{-1})^2x=1,\;x^3z=1$&T\\\hline
$46$&$xy^{-1}xy=1,\;xz^{-1}x^{-1}z=1,\;(xy^{-1})^2x=1,\;xz^{-1}x^{-1}z=1,\;x^3y=1$&T\\\hline
$47$&$xy^{-1}xy=1,\;xz^{-1}x^{-1}z=1,\;xy^{-1}zy=1,\;xz^{-2}x=1,\;x^3z=1$&T \\\hline
$48$&$xy^{-1}xz^{-1}y=1,\;xz^{-1}x^{-1}z=1,\;xy^{-1}xy=1,\;(xz^{-1})^2x=1,\;x^3z=1$&T\\\hline
$49$&$xy^{-1}zy=1,\;xz^{-1}x^{-1}z=1,\;xy^{-1}xy=1,\;xz^{-2}x=1,\;x^3z=1$&T\\\hline
$50$&$xy^{-1}zy=1,\;xz^{-1}x^{-1}z=1,\;xy^{-1}zy=1,\;xz^{-2}x=1,\;x^3z=1$&T\\\hline
$51$&$y^2z=1,\;yx^{-1}zx^{-1}y=1,\;x^2z=1,\;xy^{-1}zy^{-1}x=1,\;xz^{-1}xyz^{-1}y=1$&T\\\hline
$52$&$y^2z=1,\;yx^{-1}zx^{-1}y=1,\;xzx^{-1}z=1,\;xy^{-2}x=1,\;xz^{-1}xyz^{-1}y=1$&T\\\hline
$53$&$y^2z=1,\;yx^{-1}zx^{-1}y=1,\;xzx^{-1}z=1,\;xy^{-1}zy^{-1}x=1,\;xz^{-1}xyz^{-1}y=1$&Abelian\\\hline
$54$&$y^2z=1,\;yx^{-1}zx^{-1}y=1,\;xy^{-2}x=1,\;xz^{-1}x^{-1}z=1,\;x^2yz^{-1}y=1$&T\\\hline
$55$&$y^2z=1,\;yx^{-1}zx^{-1}y=1,\;xy^{-2}x=1,\;xz^{-1}xz=1,\;x^2yz^{-1}y=1$&T\\\hline
$56$&$y^2z=1,\;yx^{-1}zx^{-1}y=1,\;xy^{-1}zy^{-1}x=1,\;xz^{-1}x^{-1}z=1,\;x^2yz^{-1}y=1$&T\\\hline
$57$&$y^2z=1,\;yx^{-1}zx^{-1}y=1,\;xy^{-1}zy^{-1}x=1,\;xz^{-1}xz=1,\;x^2yz^{-1}y=1$&T\\\hline
$58$&$y^2x^{-1}y=1,\;yz^{-2}x=1,\;x^2y^{-1}x=1,\;xz^{-2}y=1,\;xy^{-1}zyx^{-1}z=1$&T\\\hline
$59$&$y^2x^{-1}y=1,\;yz^{-2}x=1,\;xyx^{-1}y=1,\;xz^{-2}x=1,\;xy^{-1}zyx^{-1}z=1$&T\\\hline
$60$&$yzy^{-1}z=1,\;yx^{-1}yz^{-1}x=1,\;xy^{-1}xy=1,\;xz^{-1}x^{-1}z=1,\;xy^2z^{-1}x=1$&BS$(1,-1)$\\\hline
$61$&$yzy^{-1}z=1,\;yx^{-1}yz^{-1}x=1,\;xy^{-1}zx^{-1}y=1,\;xz^{-1}x^{-1}z=1,\;xy^2z^{-1}x=1$&BS$(1,-1)$\\\hline
$62$&$yzy^{-1}z=1,\;yx^{-1}yz^{-1}x=1,\;yx^{-1}yz^{-1}x=1,\;yz^{-1}y^{-1}z=1,\;y^3z^{-1}y=1$&T\\\hline
$63$&$yzy^{-1}z=1,\;yx^{-1}zx^{-1}y=1,\;x^2z=1,\;xy^{-1}zy^{-1}x=1,\;xz^{-1}xyz^{-1}y=1$&T\\\hline
$64$&$yzy^{-1}z=1,\;yx^{-1}zx^{-1}y=1,\;xzx^{-1}z=1,\;xy^{-2}x=1,\;xz^{-1}xyz^{-1}y=1$&BS$(1,-1)$\\\hline
$65$&$yzy^{-1}z=1,\;yx^{-1}zx^{-1}y=1,\;xzx^{-1}z=1,\;xy^{-1}zy^{-1}x=1,\;xz^{-1}xyz^{-1}y1=$&T\\\hline
$66$&$yzy^{-1}z=1,\;yx^{-1}zx^{-1}y=1,\;xy^{-2}x=1,\;xz^{-1}x^{-1}z=1,\;x^2yz^{-1}y=1$&T\\\hline
$67$&$yzy^{-1}z=1,\;yx^{-1}zx^{-1}y=1,\;xy^{-2}x=1,\;xz^{-1}xz=1,\;x^2yz^{-1}y=1$&T\\\hline
$68$&$yzy^{-1}z=1,\;yx^{-1}zx^{-1}y=1,\;xy^{-1}zy^{-1}x=1,\;xz^{-1}x^{-1}z=1,\;x^2yz^{-1}y=1$&T\\\hline
$69$&$yzy^{-1}z=1,\;yx^{-1}zx^{-1}y=1,\;xy^{-1}zy^{-1}x=1,\;xz^{-1}xz=1,\;x^2yz^{-1}y=1$&T\\\hline
$70$&$yzy^{-1}z=1,\;yx^{-1}zy^{-1}x=1,\;xyz=1,\;xy^{-1}zx^{-1}y=1,\;xz^{-1}y^2z^{-1}x=1$&$z=yx$\\\hline
$71$&$yzy^{-1}z=1,\;yx^{-1}zy^{-1}x=1,\;xzx^{-1}z=1,\;xy^{-1}zx^{-1}y=1,\;xz^{-1}y^2z^{-1}x=1$&BS$(1,-1)$\\\hline
$72$&$yzy^{-1}z=1,\;yx^{-1}zy^{-1}x=1,\;xy^{-1}xy=1,\;xz^{-1}x^{-1}z=1,\;xy^2z^{-1}x=1$&BS$(1,-1)$\\\hline
$73$&$yzy^{-1}z=1,\;yx^{-1}zy^{-1}x=1,\;xy^{-1}zx^{-1}y=1,\;xz^{-1}x^{-1}z=1,\;xy^2z^{-1}x=1$&Abelian\\\hline
$74$&$yzy^{-1}z=1,\;yx^{-1}zy^{-1}x=1,\;yx^{-1}zy^{-1}x=1,\;yz^{-1}y^{-1}z=1,\;y^3z^{-1}y=1$&T\\\hline
$75$&$(yx^{-1})^2y=1,\;yz^{-1}x^{-1}z=1,\;(xy^{-1})^2x=1,\;xz^{-1}y^{-1}z=1,\;x^2y^2=1$&T\\\hline
$76$&$(yx^{-1})^2y=1,\;yz^{-1}x^{-1}z=1,\;(xy^{-1})^2x=1,\;xz^{-1}xz=1,\;x^2y^2=1$&T\\\hline
$77$&$(yx^{-1})^2y=1,\;yz^{-1}yz=1,\;(xy^{-1})^2x=1,\;xz^{-1}y^{-1}z=1,\;x^2y^2=1$&T\\\hline
$78$&$(yx^{-1})^2y=1,\;yz^{-1}yz=1,\;(xy^{-1})^2x=1,\;xz^{-1}xz=1,\;x^2y^2=1$&T\\\hline
$79$&$yx^{-1}zx^{-1}y=1,\;yz^{-1}y^{-1}z=1,\;x^2z=1,\;xy^{-1}zy^{-1}x=1,\;xz^{-1}xy^2=1$&T\\\hline
$80$&$yx^{-1}zx^{-1}y=1,\;yz^{-1}y^{-1}z=1,\;xzx^{-1}z=1,\;xy^{-2}x=1,\;xz^{-1}xy^2=1$&T\\\hline
$81$&$yx^{-1}zx^{-1}y=1,\;yz^{-1}y^{-1}z=1,\;xzx^{-1}z=1,\;xy^{-1}zy^{-1}x=1,\;xz^{-1}xy^2=1$&T\\\hline
$82$&$yx^{-1}zx^{-1}y=1,\;yz^{-1}y^{-1}z=1,\;xy^{-2}x=1,\;xz^{-1}x^{-1}z=1,\;x^2y^2=1$&T\\\hline
$83$&$yx^{-1}zx^{-1}y=1,\;yz^{-1}y^{-1}z=1,\;xy^{-2}x=1,\;xz^{-1}xz=1,\;x^2y^2=1$&T\\\hline
$84$&$yx^{-1}zx^{-1}y=1,\;yz^{-1}y^{-1}z=1,\;xy^{-1}zy^{-1}x=1,\;xz^{-1}x^{-1}z=1,\;x^2y^2=1$&BS$(1,-1)$\\\hline
$85$&$yx^{-1}zx^{-1}y=1,\;yz^{-1}y^{-1}z=1,\;xy^{-1}zy^{-1}x=1,\;xz^{-1}xz=1,\;x^2y^2=1$&BS$(1,-1)$\\\hline
$86$&$yx^{-1}zx^{-1}y=1,\;yz^{-1}yz=1,\;x^2z=1,\;xy^{-1}zy^{-1}x=1,\;xz^{-1}xy^2=1$&BS$(1,-1)$\\\hline
$87$&$yx^{-1}zx^{-1}y=1,\;yz^{-1}yz=1,\;xzx^{-1}z=1,\;xy^{-2}x=1,\;xz^{-1}xy^2=1$&BS$(1,-1)$\\\hline
$88$&$yx^{-1}zx^{-1}y=1,\;yz^{-1}yz=1,\;xzx^{-1}z=1,\;xy^{-1}zy^{-1}x=1,\;xz^{-1}xy^2=1$&T\\\hline
$89$&$yx^{-1}zx^{-1}y=1,\;yz^{-1}yz=1,\;xy^{-2}x=1,\;xz^{-1}x^{-1}z=1,\;x^2y^2=1$&T\\\hline
$90$&$yx^{-1}zx^{-1}y=1,\;yz^{-1}yz=1,\;xy^{-2}x=1,\;xz^{-1}xz=1,\;x^2y^2=1$&T\\\hline
$91$&$yx^{-1}zx^{-1}y=1,\;yz^{-1}yz=1,\;xy^{-1}zy^{-1}x=1,\;xz^{-1}x^{-1}z=1,\;x^2y^2=1$&T\\\hline
$92$&$yx^{-1}zx^{-1}y=1,\;yz^{-1}yz=1,\;xy^{-1}zy^{-1}x=1,\;xz^{-1}xz=1x^2y^2=1$&T\\\hline
\end{longtable}
$\mathbf{K_2}$ \textbf{of Figure \ref{125}:} Taking into account the relations from Table \ref{tc4} which are not rejected, it can be see that  there are  $634$ different cases for the relations of $5$ cycles of length $4$ on vertices of degree $ 4 $ in $ \Gamma $ with the structure of the graph $ K_2 $ as Figure \ref{125}. Using GAP \cite{a9}, a free group with generators $ x,y,z $ and the  relations of these $ 5 $ cycles which are between $621$ cases of these $634$ cases  is finite or abelian, that is a
contradiction. Hence, there are $ 13 $ cases which may lead
to the existence of $ K_2$  in $ \Gamma $. 
We checked these such cases. In the following, we show that all of such cases lead to contradictions and therefore $ \Gamma $ does not contain any subgraph isomorphic to $K_{2}$ as Figure \ref{125}.
\begin{itemize}
\item[(1)] $r_1:xzx^{-1}zy^{-1}x=1,\;\;r_2:xzy^{-1}xz^{-1}y=1,\;\;r_3:xzyz=1,\;\;r_4:(x^{-1}y)^2(z^{-1}y)^2=1,\;\;r_5:y(zx^{-1})^2y=1$\\
In this case, it can be seen that $ \left\langle x,y,z\right\rangle  $ has a non-trivial torsion element, a contradiction.
\item[(2)]$r_1:xz^{-1}x^{-1}zy^{-1}x=1,\;\;r_2:xz^{-1}y^{-1}xy^{-1}z=1,\;\;r_3:xz^{-2}yx^{-1}y=1,\;\;r_4:(x^{-1}y)^2x^{-1}zy^{-1}z=1,\;\;r_5:yz^{-1}x^{-1}zx^{-1}y=1$ \\
In this case, it can be seen that  $x=z$, a contradiction.
\item[(3)]$r_1:xz^{-1}y^{-1}zy^{-1}x=1,\;\;r_2:xz^{-2}x^{-1}z=1,\;\;r_3:xz^{-1}x^{-1}zx^{-1}y=1,\;\;r_4:(x^{-1}y)^2z^{-1}yx^{-1}z=1,\;\;r_5:yz^{-1}y^{-1}zx^{-1}y=1$\\
In this case, it can be seen that $ \left\langle x,y,z\right\rangle  $ is isomorphic to a  quotient  group  of  BS$(1,2)$.
\item[(4)]$r_1:yz^{-1}x^{-1}zy^{-1}x=1,\;\;r_2:yz^{-1}y^{-1}xy^{-1}z=1,\;\;r_3:yz^{-2}yx^{-1}y=1,\;\;r_4:(x^{-1}y)^2x^{-1}zy^{-1}z=1\;\;r_5:xz^{-1}x^{-1}zx^{-1}y=1$ \\
In this case, it can be seen that  $y=z$, a contradiction.
\item[(5)]$r_1:yz^{-1}x^{-1}zy^{-1}x=1,\;\;r_2:yz^{-1}y^{-1}zx^{-1}z=1,\;\;r_3:yz^{-2}xz^{-1}y=1,\;\;r_4:(x^{-1}yx^{-1}z)^2=1,\;\;r_5:xz^{-1}(x^{-1}z)^2=1$\\   
In this case, it can be seen that  $y=z$, a contradiction.
\item[(6)]$r_1:yz^{-1}x^{-1}zy^{-1}x=1,\;\;r_2:yz^{-2}y^{-1}z=1,\;\;r_3:yz^{-1}y^{-1}zx^{-1}y=1,\;\;r_4:(x^{-1}y)^2z^{-1}xy^{-1}z=1,\;\;r_5:xz^{-1}x^{-1}zx^{-1}y=1$\\
In this case, it can be seen that $ \left\langle x,y,z\right\rangle  $ is isomorphic to a  quotient  group  of  BS$(1,2)$, a contradiction.
\item[(7)]$r_1:x(zy^{-1})^2x=1,\;\;r_2:xzx^{-1}yz^{-1}y=1,\;\;r_3:xzyz=1,\;\;r_4:(x^{-1}y)^2z^{-1}xz^{-1}y=1,\;\;r_5:yzy^{-1}zx^{-1}y=1$ \\
In this case, it can be seen that $ \left\langle x,y,z\right\rangle  $ is Abelian.
\item[(8)]$r_1:(xz^{-1})^2y^{-1}x=1,\;\;r_2:xz^{-1}x^2z^{-1}y=1,\;\;r_3:xz^{-1}xy^{-1}xz=1,\;\;r_4:(x^{-1}y)^2(z^{-1}y)^2=1,\;\;r_5:yz^{-1}xz^{-1}x^{-1}y=1$ \\
In this case, it can be seen that  $x=y$, a contradiction.
\item[(9)]$r_1:(xz^{-1})^2y^{-1}x=1,\;\;r_2:xz^{-1}xyz^{-1}y=1,\;\;r_3:xz^{-1}xy^{-1}xz=1,\;\;r_4:(x^{-1}y)^2z^{-1}xz^{-1}y=1,\;\;r_5:yz^{-1}xz^{-1}x^{-1}y=1$ \\
In this case, it can be seen that  $x=y$, a contradiction.
\item[(10)]$r_1:xz^{-1}yz^{-1}y^{-1}x=1,\;\;r_2:xz^{-1}y^2z^{-1}y=1,\;\;r_3:xz^{-1}yx^{-1}yz=1,\;\;r_4:(x^{-1}y)^2z^{-1}xz^{-1}y=1,\;\;r_5:(yz^{-1})^2x^{-1}y=1$ \\
In this case, it can be seen that  $x=yz$, a contradiction.
\item[(11)]$r_1:yzx^{-1}zy^{-1}x=1,\;\;r_2:yz^2y^{-1}z=1,\;\;r_3:yzy^{-1}zx^{-1}y=1,\;\;r_4:(x^{-1}y)^2z^{-1}xy^{-1}z=1,\;\;r_5:x(zx^{-1})^2y=1$.
In this case, it can be seen that $ \left\langle x,y,z\right\rangle  $ is isomorphic to a  quotient  group  of  BS$(1,-2)$, a contradiction.
\item[(12)]$r_1:yz^{-2}xy^{-1}x=1,\;\;r_2:yz^{-1}(y^{-1}z)^2=1,\;\;r_3:yz^{-1}x^{-1}zx^{-1}y=1,\;\;r_4:x^{-1}(yz^{-1})^2xy^{-1}z=1,\;\;r_5:xz^{-2}xz^{-1}y=1$ $y=x$\\
In this case, it can be seen that  $x=y$, a contradiction.
\item[(13)]$r_1:yz^{-1}xz^{-1}y^{-1}x=1,\;\;r_2:yz^{-1}xyz^{-1}y=1,\;\;r_3:yz^{-1}xy^{-1}xz=1,\;\;r_4:(x^{-1}y)^2z^{-1}xz^{-1}y=1,\;\;r_5:(xz^{-1})^2x^{-1}y=1$\\
In this case, it can be seen that  $z=y$, a contradiction.
\end{itemize}
\textbf{H of Figure \ref{k114}:} Suppose that graph H  is a subgraph of $ Z(\alpha ,\beta) $. Since $ deg_{Z(\alpha ,\beta)}(g_1)=5$ and $ deg_{Z(\alpha ,\beta)}(g_2)=5$, Remark \ref{deg12} implies $ \theta_4(g_1)=0,\;\theta_3(g_1)=1,\;\theta_4(g_2)=0 $ and $ \theta_3(g_2)=1 $. Hence, there exist $ s,s'\in \Delta_{(\alpha,\beta)}^{3} $ such that $ g_1\in \mathcal{V}_{(\alpha,\beta)}(s) $ and $ g_2\in \mathcal{V}_{(\alpha,\beta)}(s') $.
In view of the graph H,  a  cycle of length $ 3 $ in $ Z(\alpha ,\beta) $  contains $ g_1 $ if and only if it  contains $ g_2 $. Hence, $ s= s' $ since otherwise  contradicting to the part $ (3) $ of Lemma \ref{3part}.  Let $ \mathcal{V}_{(\alpha,\beta)}(s)=\{g_{s1},g_1,g_2\} $ and $ C,C'$ and $C'' $ be $ 3 $ cycles of length $ 3 $ with vertex sets  $\mathcal{V}_{C}=\{ g_1,g_2,g_3\} $, $ \mathcal{V}_{C'}=\{g_1,g_2,g_4\} $ and $ \mathcal{V}_{C''}=\{g_1,g_2,g_5 \}$, respectively. Suppose that $T_C=[h_1,h'_1,h_2,h'_2,h_3,h'_3] $, $T_{C'}=[h_1,h'_1,r_2,r'_2,r_3,r'_3] $ and $T_{C''}=[h_1,h'_1,t_2,t'_2,t_3,t'_3] $ are  the $ 6 $-tuples of $ C,C'$ and $C'' $, respectively, with the following relations:
 \begin{equation*}
R(C): \left\{
\begin{array}{ll}
h_{1} g_{1}=h'_{1}g_{2} \\  
h_{2}g_{2}=h'_{2}g_{3} \\
h_{3}g_{3}=h'_{3}g_{1} 
\end{array} \right.,
\qquad\qquad
R(C'):\left\{
\begin{array}{ll}
h_{1} g_{1}=h'_{1}g_{2} \\  
r_{2}g_{2}=r'_{2}g_{4} \\
r_{3}g_{4}=r'_{3}g_{1} 
\end{array} \right.,
\qquad\qquad
R(C''):\left\{
\begin{array}{ll}
h_{1} g_{1}=h'_{1}g_{2} \\  
t_{2}g_{2}=t'_{2}g_{5} \\
t_{3}g_{5}=t'_{3}g_{1} 
\end{array} \right. .
\end{equation*}
Clearly, $ h_1=sg_1^{-1} $ and $ h'_1=sg_2^{-1} $. Now, since  $ \Theta_3(g_1)=\{sg_1^{-1}\}$,  $ \Theta_4(g_2)=\{sg_2^{-1}\}$, $  \theta_4(g_1)=0$ and $\theta_4(g_2)=0 $, it is easy to see that 
$ h_1\neq h'_1\neq h_2\neq h'_2\neq h_3\neq h'_3\neq h_1$, $  r_2\neq r'_2\neq r_3\neq r'_3\neq h_1 $, $ t_2\neq t'_2\neq t_3\neq t'_3\neq h_1 $, $ h_2\neq r_2\neq t_2\neq h_2  $ and $ h'_3\neq r'_3\neq t'_3\neq h'_3   $. On the other hand, replacing $ \alpha $ by $ (sg_1^{-1})^{-1}\alpha $, we may assume that $ h_1=1 $. Then the relations corresponding to these 3 cycles are the same as the relations  corresponding to $ 3 $ cycles of length $ 3 $ in $ K_{1,1,3} $ as Figure \ref{707} and therefore H is a forbidden subgraph of $ Z(\alpha ,\beta) $.   This completes the proof of Theorem \ref{55}.
\begin{figure}
\begin{tikzpicture}[scale=.75]
\draw [fill] (0,0) circle
[radius=0.1] node  [left]  {$ g_1 $};
\draw [fill] (1.5,0) circle
[radius=0.1] node  [right]  {$ g_2 $};
\draw [fill] (-.75,-1.5) circle
[radius=0.1] node  [below]  {$ g_{s1} $};
\draw [fill] (.25,-1.5) circle
[radius=0.1] node  [below]  {$ g_3 $};
\draw [fill] (1.25,-1.5) circle
[radius=0.1] node  [below]  {$ g_4 $};
\draw [fill] (2.25,-1.5) circle
[radius=0.1] node  [below]  {$ g_5 $};
\draw (0,0) -- (1.5,0) -- (-.75,-1.5) -- (0,0);
\draw  (1.5,0) -- (.25,-1.5) -- (0,0);
\draw  (1.5,0) -- (1.25,-1.5) -- (0,0);
\draw  (1.5,0) -- (2.25,-1.5) -- (0,0);
\end{tikzpicture}
\subfloat{H}
\caption{A forbidden subgraph of $ Z(\alpha ,\beta) $, where the degree of vertices  $g_1$ and $ g_2$ of the subgraph in $ Z(\alpha ,\beta) $ must be 5. }\label{k114}
\end{figure}

\section{\textbf{Appendix}}\label{app10}
In  this section we give more details on the proof of  Theorems  \ref{app}  and \ref{appunit}. Before we  do these,  we need some explanation which are almost similar to Appendix \ref{app12} as follows:\\
In this section, let    $ \alpha$ be   a  zero divisor or a unit in $\mathbb{F}[G]$ for a possible torsion free group $G$ and arbitrary field $\mathbb{F}$ with  $ supp(\alpha)=\{1,x,x^{-1},y\} $, where $ x $ and $y$ are distinct non-trivial elements of $G$, and $ \beta $ be a mate of $ \alpha $. By Remarks \ref{iden} and \ref{idenunit}, we may assume that  $ G=\left\langle supp(\alpha)\right\rangle=\left\langle supp(\beta)\right\rangle  $. 
 Suppose  that $ C $ is a cycle of length $ n $ as Figure \ref{cn10} in $ Z(\alpha,\beta) $ or $ U(\alpha,\beta) $. Suppose further that $ \{g_1,\ldots,g_n\}\subseteq supp(\beta) $ is the  vertex set of $ C $ such that $ g_i\sim g_{i+1} $ for all $ i\in \{1,\ldots,n-1\} $ and $ g_1\sim g_n $. By an arrangement $ l $ of the vertex set $C$, we mean a sequence of all vertices as $x_1,\ldots,x_k$ such that $ x_i\sim x_{i+1} $ for all $ i\in \{1,\ldots,n-1\} $ and $ x_1\sim x_n $.
\begin{figure}[htp]
\begin{tikzpicture}
\draw (0,-1.25) .. controls (-0.700,-1.25) and (-1.25,-0.7) .. (-1.25,0)
.. controls (-1.25,0.7) and (-0.7,1.25) .. (0,1.25)
.. controls (0.7,1.25) and (1.25,0.7) .. (1.25,0);
\draw (1.25,0)[dashed] .. controls (1.25,-0.7) and (0.7,-1.25) .. (0,-1.25);
\draw [fill] (0,-1.25) circle
[radius=0.09] node  [below]  {$g_{n-2}$};
\draw [fill] (-.95,-.8) circle
[radius=0.09] node  [left]  {$g_{n-1}$};
\draw [fill] (-1.20,.4) circle
[radius=0.09] node  [left]  {$g_{n}$};
\draw [fill] (-.5,1.15) circle
[radius=0.09] node  [above]  {$g_{1}$};
\draw [fill] (.8,1) circle
[radius=0.09] node  [above]  {$g_{2}$};
\draw [fill] (1.25,0) circle
[radius=0.09] node  [right]  {$g_{3}$};
\end{tikzpicture}
\caption{A cycle of length $ n $ in $ Z(\alpha,\beta) $ or $ U(\alpha,\beta) $.}\label{cn10}
\end{figure}
Since $ g_i\sim g_{i+1} $ for all $ i\in \{1,\ldots,n-1\} $ and $ g_1\sim g_n $ for each $ i\in \{1,2,\ldots,n\} $, there exist distinct elements $ h_i,h'_i\in supp(\alpha) $ satisfying the following relations:
\begin{equation}\label{cycle10}
C: \left\{
\begin{array}{ll}
 h_{1} g_{1}=h'_{1}g_{2} \rightarrow g_{1}g_{2}^{-1}= h^{-1}_{1}h'_{1}\\  
h_{2}g_{2}=h'_{2}g_{3} \rightarrow g_{2} g_{3}^{-1}={h}_{2}^{-1}h'_{2}\\
\vdots\\
h_{n}g_{n}=h'_{n}g_{1} \rightarrow g_{n}g_{1}^{-1}=h_{n}^{-1}h'_{n}
\end{array} \right.. 
\end{equation} 
 We assign a $ 2n $-tuple $ T_C^l=[h_1,h'_1,h_2,h'_2,\ldots ,h_n,h'_n] $  to the cycle $ C $ corresponding to the above arrangement $ l $ of the vertex set of $C$. We denote by $ R( T_C^l) $ the above set $R$ of relations.
It can be derived from the relations \ref{cycle10} that $ r( T_C^l):=(h^{-1}_{1}h'_{1})({h}_{2}^{-1}h'_{2})\ldots (h_{n}^{-1}h'_{n}) $ is equal to 1. According to Remarks \ref{multi} and \ref{multiunit} and the explanations before them in Section \ref{s5}, if  $ [h_1,h'_1,h_2,h'_2,\ldots ,h_n,h'_n] $ is a $ 2n $-tuple of $ C $ according to arrangement $ l $ of the vertex set of $C$ and there is $ i\in \{1,\ldots,n\} $ such that  $ (h_i,h'_i)=(1,x), (x,1),(1,x^{-1})$ or $ (x^{-1},1) $, then  $ [h_1,h'_1,\ldots,h_{i-1},h'_{i-1},x^{-1},1,\ldots ,h_n,h'_n] $, $ [h_1,h'_1,\ldots,h_{i-1},h'_{i-1},1,$  $x^{-1},\ldots ,h_n,h'_n] $, $ [h_1,h'_1,\ldots,h_{i-1},h'_{i-1},x,1,\ldots ,\\h_n,h'_n] $ or $ [h_1,h'_1,\ldots,h_{i-1},h'_{i-1},1,x,\ldots ,h_n,h'_n] $, respectively, is also $ 2n $-tuple of $ C $ according to arrangement $ l $ of the vertex set of $C$. For each arrangement $ l $ of the vertex set of $C$, the set of all such $ 2n $-tuples will be denoted by $ \mathcal{T}^{l}(C) $.  Then, if $ T_C^{l'} $ is  a $ 2n $-tuple of $ C $ corresponding to another arrangement $ l' $ of the vertex set of $C$, then there exists $2n $-tuple $ T=[h_1,h'_1,h_2,h'_2,\ldots ,h_n,h'_n]$,  $T\in \mathcal{T}^{l}(C)  $ such that $ T_C^{l'} $ is one of the following $ 2n $-tuples:
 \begin{center}
  \begin{tabular}{c}
 $ [h_2,h'_2,h_3,h'_3,\ldots ,h_n,h'_n,h_1,h'_1] $,\\
 $ [h_3,h'_3,h_4,h'_4,\ldots ,h_1,h'_1,h_2,h'_2] $,\\
 $ \vdots $\\
 $ [h_n,h'_n,h_1,h'_1,\ldots ,h_{n-2},h'_{n-2},h_{n-1},h'_{n-1}] $,\\
$ [h'_n,h_n,h'_{n-1},h_{n-1},\ldots ,h'_2,h_2,h'_1,h_1] $,\\
$ \vdots $\\
$ [h'_1,h_1,h'_{n},h_{n},\ldots ,h_3,h'_3,h_2,h'_2] $,
 \end{tabular}
 \end{center}
Suppose   that $ \mathcal{T}(C)=\bigcup_{l\in L}\mathcal{T}^{l}(C)$, where $ L $ is the set of all such arrangement of the vertex set of $ C $, and also $ \mathcal{R}(T)=\{R(C)|T\in  \mathcal{T}(C) \}$.
\begin{defn}
Let $ \Gamma $ be a zero divisor graph or a unit graph on a pair of elements $ (\alpha,\beta) $ in $\mathbb{F}[G]$ for a possible torsion free group $G$ and arbitrary field $\mathbb{F}$ with  $ supp(\alpha)=\{1,x,x^{-1},y\} $, where $ x $ and $y$ are distinct non-trivial elements of $G$, and $ \beta $ be a mate of $ \alpha $. Let $ C $ be a cycle of length $ n $ in $ \Gamma $. Since $ r(T_1)=1 $ if and only if $ r(T_2)=1 $, for all $ T_1,T_2\in  \mathcal{T}(C) $, a member of $ \{r(T)|T \in  \mathcal{T}(C) \} $ is given as a representative and denoted by $ r(C) $. Also, $ r(C)=1 $  is called the relation of $ C $.
\end{defn}
 \begin{defn}\label{eqquidef10}
Let $ \Gamma $ be a zero divisor graph or a unit graph on a pair of elements $ (\alpha,\beta) $ in $\mathbb{F}[G]$ for a possible torsion free group $G$ and arbitrary field $\mathbb{F}$ with $ supp(\alpha)=\{1,x,x^{-1},y\} $, where $ x $ and $y$ are distinct non-trivial elements of $G$, and $ \beta $ be a mate of $ \alpha $. Let $ C $ and $ C' $ be two cycles of length $ n $ in $  \Gamma$. We say that these two
cycles are equivalent, if $ \mathcal{T}(C)\cap \mathcal{T}(C')\neq\varnothing $ and otherwise are non-equivalent.
\end{defn}
 \begin{rem}\label{eqquirem10}
 Suppose  that $ C $ and $ C' $ are two cycles of length $ n $ in $ Z(\alpha,\beta) $ or $ U(\alpha,\beta) $. Then clearly  if $ C $ and $ C' $ are equivalent cycles, then $ \mathcal{T}(C)= \mathcal{T}(C') $. 
 \end{rem}
 \begin{rem}\label{3p}
 Let  $ \alpha $ be a zero divisor in $\mathbb{F}[G]$ for a possible torsion free group $G$ and arbitrary field $\mathbb{F}$ with $ supp(\alpha)=\{1,x,x^{-1},y\} $, where $ x $ and $y$ are distinct non-trivial elements of $G$, and $ \beta $ be a mate of $ \alpha $. Suppose that $ C $ is a cycle of length $ n $ in $ Z(\alpha,\beta) $ with the vertex set $ \{g_1,g_2,\ldots,g_n\}\subseteq supp(\beta) $ and $ T_1=[h_1,h'_1,h_2,h'_2,\ldots ,h_n,h'_n] $ is  a $ 2n $-tuple of $ C $ according to arrangement $ g_1,g_2,\ldots,g_n $. By Remark {\rm\ref{deg12}}, if $ g $ is a vertex of degree $ 4 $ in $ Z(\alpha,\beta) $, then $ \theta_{3}(g)=0 $ and  $ \theta_{4}(g)=0 $. By the latter and since if $ g $ is a vertex of degree $ 4 $ of types $ (i)$ and $ (ii) $, then $ |M_{Z(\alpha,\beta)}(g)|=0 $ and $ |M_{Z(\alpha,\beta)}(g)|=1 $, respectively,  the following statements hold:
 \begin{itemize}
  \item If $ g_i $ is a vertex of degree $ 4 $  in $Z(\alpha,\beta)$, $ i\in \{1,\ldots,n\} $, then:
  \begin{itemize}
  \item if $ i\in \{3,\ldots,n-1\} $, then $ h_i\neq h'_{i-1} $, $ h'_i\neq h_{i+1} $ and $ h_{i-1}\neq h'_{i-2} $, if  $i=1 $, then $ h_1\neq h'_{n} $, $ h'_1\neq h_{2} $ and $ h_{n}\neq h'_{n-1} $, if  $i=2 $, then $ h_2\neq h'_{1} $, $ h'_2\neq h_{3} $ and $ h_{1}\neq h'_{n} $ and if  $i=n$, then $ h_n\neq h'_{n-1} $, $ h'_n\neq h_{1} $ and $ h_{n-1}\neq h'_{n-2} $. 
   \item Suppose that $ i\in \{3,\ldots,n-1\} $. If  $ h_k^{-1}h'_k = x$, where $ k\in\{i,i-1\} $, then  $ h_{k+1}\notin\{ 1,x\}  $ and $ h'_{k-1}\notin\{ 1,x^{-1}\} $, if  $ h_{i+1}^{-1}h'_{i+1}=x$, then  $ h'_{i}\notin\{ 1,x^{-1}\} $ and if $ h_{i-2}^{-1}h'_{i-2}=x$, then  $ h_{i-1}\notin\{ 1,x\} $. 
   If   $ h_k^{-1}h'_k = x^{-1}$, where $ k\in\{i,i-1\} $, then  $ h_{k+1}\notin\{ 1,x^{-1}\}  $ and $ h'_{k-1}\notin\{ 1,x\} $, if   $ h_{i+1}^{-1}h'_{i+1}=x^{-1}$, then  $ h'_{i}\notin\{ 1,x\} $ and if $ h_{i-2}^{-1}h'_{i-2}=x^{-1}$, then  $ h_{i-1}\notin\{ 1,x^{-1}\} $.
  \item Suppose that $ i=1 $. If  $ h_1^{-1}h'_1 = x$, then  $ h_{2}\notin\{ 1,x\}  $ and $ h'_{n}\notin\{ 1,x^{-1}\} $, if  $ h_n^{-1}h'_n = x$, then  $ h_{1}\notin\{ 1,x\}  $ and $ h'_{n-1}\notin\{ 1,x^{-1}\} $, if  $ h_{2}^{-1}h'_{2}=x$, then  $ h'_{1}\notin\{ 1,x^{-1}\} $ and if $ h_{n-1}^{-1}h'_{n-1}=x$, then  $ h_{n}\notin\{ 1,x\} $. 
   If  $ h_1^{-1}h'_1 = x^{-1}$, then  $ h_{2}\notin\{ 1,x^{-1}\}  $ and $ h'_{n}\notin\{ 1,x\} $, if  $ h_n^{-1}h'_n = x^{-1}$, then  $ h_{1}\notin\{ 1,x^{-1}\}  $ and $ h'_{n-1}\notin\{ 1,x\} $, if  $ h_{2}^{-1}h'_{2}=x^{-1}$, then  $ h'_{1}\notin\{ 1,x\} $ and if $ h_{n-1}^{-1}h'_{n-1}=x^{-1}$, then  $ h_{n}\notin\{ 1,x^{-1}\} $. 
  \item Suppose that $ i=2 $. If  $ h_2^{-1}h'_2 = x$, then  $ h_{3}\notin\{ 1,x\}  $ and $ h'_{1}\notin\{ 1,x^{-1}\} $, if  $ h_1^{-1}h'_1 = x$, then  $ h_{2}\notin\{ 1,x\}  $ and $ h'_{n}\notin\{ 1,x^{-1}\} $, if  $ h_{3}^{-1}h'_{3}=x$, then  $ h'_{2}\notin\{ 1,x^{-1}\} $ and if $ h_{n}^{-1}h'_{n}=x$, then  $ h_{1}\notin\{ 1,x\} $. 
   If  $ h_2^{-1}h'_2 = x^{-1}$, then  $ h_{3}\notin\{ 1,x^{-1}\}  $ and $ h'_{1}\notin\{ 1,x\} $, if  $ h_1^{-1}h'_1 = x^{-1}$, then  $ h_{2}\notin\{ 1,x^{-1}\}  $ and $ h'_{n}\notin\{ 1,x\} $, if  $ h_{3}^{-1}h'_{3}=x^{-1}$, then  $ h'_{2}\notin\{ 1,x\} $ and if $ h_{n}^{-1}h'_{n}=x^{-1}$, then  $ h_{1}\notin\{ 1,x^{-1}\} $. 
  \item   Suppose that $ i=n $. If  $ h_n^{-1}h'_n = x$, then  $ h_{1}\notin\{ 1,x\}  $ and $ h'_{n-1}\notin\{ 1,x^{-1}\} $, if  $ h_{n-1}^{-1}h'_{n-1} = x$, then  $ h_{n}\notin\{ 1,x\}  $ and $ h'_{n-2}\notin\{ 1,x^{-1}\} $, if  $ h_{1}^{-1}h'_{1}=x$, then  $ h'_{n}\notin\{ 1,x^{-1}\} $ and if $ h_{n-2}^{-1}h'_{n-2}=x$, then  $ h_{n-1}\notin\{ 1,x\} $. 
   If  $ h_n^{-1}h'_n = x^{-1}$, then  $ h_{1}\notin\{ 1,x^{-1}\}  $ and $ h'_{n-1}\notin\{ 1,x\} $, if  $ h_{n-1}^{-1}h'_{n-1} = x^{-1}$, then  $ h_{n}\notin\{ 1,x^{-1}\}  $ and $ h'_{n-2}\notin\{ 1,x\} $, if  $ h_{1}^{-1}h'_{1}=x^{-1}$, then  $ h'_{n}\notin\{ 1,x\} $ and if $ h_{n-2}^{-1}h'_{n-2}=x^{-1}$, then  $ h_{n-1}\notin\{ 1,x^{-1}\} $. 
   \end{itemize}
  \item If $ g_i $ is a vertex of degree $ 4 $ of type $ (ii) $ in $Z(\alpha,\beta)$, $ i\in \{1,\ldots,n\} $, then:
  \begin{itemize}
   \item If $ \{h_i,h'_{i-1}\}=\{x,x^{-1}\} $, then $ h'_i\neq h_{i-1} $ and $ 1\in \{h'_i,h_{i-1}\} $.
  \item If $ h_i=1 $, then $ h'_i\in \{x,x^{-1}\} $ and if $ h'_{i-1}=1 $, then $ h_{i-1}\in \{x,x^{-1}\} $.
\end{itemize}   
  \item If $ g_i $ is a vertex of degree $ 4 $ of type $ (i) $ in $Z(\alpha,\beta)$, $ i\in \{1,\ldots,n\} $, then:
  \begin{itemize}
  \item If $ h_i\in \{x,x^{-1}\} $, then  $ h'_i\neq 1$  and if $ h'_{i-1}\in \{x,x^{-1}\} $, then $h_{i-1}\neq 1 $.
  \item If $ h_i=1 $, then $ h'_i=y $ and if $ h'_{i-1}=1 $, then $ h_{i-1}=y $.
   \end{itemize}
  \end{itemize}
 \end{rem}
 \begin{rem}\label{3punit}
 Let  $ \alpha $ be a unit in $\mathbb{F}[G]$ for a possible torsion free group $G$ and arbitrary field $\mathbb{F}$ with $ supp(\alpha)=\{1,x,x^{-1},y\} $, where $ x $ and $y$ are distinct non-trivial elements of $G$, and $ \beta $ be a mate of $ \alpha $. Let $ F=\{g|g\in \mathcal{V}_{U(\alpha,\beta)}, deg_{U(\alpha,\beta)}(g)=4\} $. In view of the proof of Lemma {\rm \ref{degreeconnunit}}, if $ g $ is a vertex of degree $ 3 $ in $ U(\alpha,\beta) $, then  $ \theta_{3}(g)=0 $ and $ \theta_{4}(g)=0 $ and  either for each  $ g\in F $, $ \theta_{3}(g)=0 $ and $ \theta_{4}(g)=0 $ or there exists  $ g'\in F $ such that  $ \theta_{3}(g')=1 $ and $ \theta_{4}(g')=0 $ and for each  $ g\in F\setminus \{g'\} $, $ \theta_{3}(g)=0 $ and $ \theta_{4}(g)=0 $. Hence, if $ C $ is a cycle of length $ 3 $ in $ U(\alpha,\beta) $ such that the vertex set of it contains at least a vertex of degree $ 3 $ in $ U(\alpha,\beta) $ or  two  vertices of degree $ 4 $  in $ U(\alpha,\beta) $  and $ T_1=[h_1,h'_1,h_2,h'_2,h_3,h'_3] $ is  a $ 2n $-tuple of $ C $, then $ h_1\neq h'_1\neq h_2\neq h'_2\neq h_3\neq h'_3 \neq h_1 $. Also, there exists at most one integer number $ j\in\{1,2,3\} $ such that  $ h_{j}^{-1}h'_{j}\in \{x,x^{-1}\}$. Moreover, if $ h_{1}^{-1}h'_{1}= x$  \rm{(}resp. $h_{1}^{-1}h'_{1}=x^{-1})$, then $ h_2\notin\{ 1,x\}$   \rm{(}resp. $h_2\notin\{ 1,x^{-1}\}) $ and $ h'_3\notin\{ 1,x^{-1}\} $  $(h_3\notin\{ 1,x\}) $, if $ h_{2}^{-1}h'_{2}= x$ \rm{(}resp. $h_{2}^{-1}h'_{2}=x^{-1})$, then $ h_3\notin\{ 1,x\}$  $(h_3\notin\{ 1,x^{-1}\}) $ and $ h'_1\notin\{ 1,x^{-1}\}$ \rm{(}resp. $h'_1\notin\{ 1,x\}) $ and if $ h_{3}^{-1}h'_{3}= x$  $(h_{3}^{-1}h'_{3}=x^{-1})$, then $ h_1\notin\{ 1,x\}$ \rm{(}resp. $h_1\notin\{ 1,x^{-1}\}) $ and $ h'_2\notin\{ 1,x^{-1}\}$  \rm{(}resp. $h'_2\notin\{ 1,x\}) $.
 \end{rem}
 \begin{rem}
Let  $ \Gamma $ be a graph with vertices of degree less than or equal to  $ 12 $ and   $ B $  be the set of all cycles in $ \Gamma $. Let $\mathcal{G}(\Gamma):=\left\langle supp(\alpha)\, |\,r(C),\,C\in B\right\rangle.$
If at least one of the following cases happens, then clearly  $ \Gamma $ is a forbidden subgraph of $ Z(\alpha,\beta) $ and $ U(\alpha,\beta) $:
\begin{itemize}
\item[1.]$ \mathcal{G}(\Gamma) $ is an abelian group.
\item[2.]$ \mathcal{G}(\Gamma)$ is a quotient  group  of
 $ BS(1,n)$.
\item[3.]$ \mathcal{G}(\Gamma) $ has a non-trivial torsion element.
\end{itemize} 
 \end{rem}
\textbf{Proof of Theorems \ref{a} and \ref{appmatcal}} Let  $ \alpha $ be a zero divisor in $\mathbb{F}[G]$ for a possible torsion free group $G$ and arbitrary field $\mathbb{F}$ with $ supp(\alpha)=\{1,x,x^{-1},y\} $, where $ x $ and $y$ are distinct non-trivial elements of $G$, and $ \beta $ be a mate of $ \alpha $.\\
$ \mathbf{C_3} $\textbf{:}
  Suppose that $ C $ is a cycle of length $ 3 $ in $ Z:Z(\alpha,\beta) $ such that  its  vertex set  contains at least a vertex of degree $ 4 $ in $ Z(\alpha,\beta) $.  Suppose further that $ T\in \mathcal{T}(C) $ and $ T=[h_1,h'_1,h_2,h'_2,h_3,h'_3] $. 
By Remark \ref{3p}, $ h_1\neq h'_1\neq h_2\neq h'_2\neq h_3\neq h'_3\neq h_1 $ and  there exists at most one integer number $ j\in\{1,2,3\} $ such that  $ h_{j}^{-1}h'_{j}\in \{x,x^{-1}\}$.  Moreover, if $ h_{1}^{-1}h'_{1}= x$    \rm{(}respectively, $h_{1}^{-1}h'_{1}=x^{-1})$, then $ h_2\notin\{ 1,x\}$    \rm{(}respectively, $h_2\notin\{ 1,x^{-1}\}) $ and $ h'_3\notin\{ 1,x^{-1}\} $   \rm{(}respectively, $h_3\notin\{ 1,x\}) $, if $ h_{2}^{-1}h'_{2}= x$   \rm{(}respectively, $h_{2}^{-1}h'_{2}=x^{-1})$, then $ h_3\notin\{ 1,x\}$   \rm{(}respectively, $h_3\notin\{ 1,x^{-1}\}) $ and $ h'_1\notin\{ 1,x^{-1}\}$   \rm{(}respectively, $h'_1\notin\{ 1,x\}) $ and if $ h_{3}^{-1}h'_{3}= x$   \rm{(}respectively, $h_{3}^{-1}h'_{3}=x^{-1})$, then $ h_1\notin\{ 1,x\}$   \rm{(}respectively, $h_1\notin\{ 1,x^{-1}\}) $ and $ h'_2\notin\{ 1,x^{-1}\}$   \rm{(}respectively, $h'_2\notin\{ 1,x\}) $. Also by   \cite[Lemma 2.4]{AT}, we may assume that $ h_1\in \{1,y\} $. Hence,  by using GAP \cite{a9}, it can be seen that there exist $71$ non-equivalent cases for $ T $. We checked the relations of such  non-equivalent cases. In  Table \ref{tt10}, it can be seen that $ 57 $ cases among $71$ cases  lead to contradictions. 
Hence, there are just $ 14 $ cases  which may lead
to the existence of the cycle $ C $  in $ Z(\alpha,\beta) $ which are listed in Table \ref{c3e}.\\
\begin{table}[!hbtp]
\begin{tabular}{|l|l|l|l|}
\hline
$ 1.x^2yx^{-1}y=1 $&$2. x^{-1}y^{-1}xy^{-1}x^{-1}=1 $&$3.yx^{-1}yx^2 =1 $&$ 4.x^{2}y^{-1}x^{-1}y^{-1}=1$ \\\hline
$5. yxyx^{-2}=1 $&$ 6. yx^{-3}y=1  $&$ 7.yx^3y =1$&$ 8.x^{-2}yxy=1 $\\\hline
$ 9.y^{-1}x^3y^{-1}x^{-1}=1  $&$ 10.y^{-1}x^2yx^{-2}=1 $&$ 11.y^{-1}x^2yx^2=1$&$12. (y^{-1}x)^2y^{-1}x^{-1}=1  $\\\hline
$ 13.y^{-1}xy^{-1}x^{-3}=1  $&$14. y^{-1}x(y^{-1}x^{-1})^2=1  $\\\cline{1-2}
\end{tabular}
\caption{Possible relations of a cycle of length $ 3 $ in $ Z(\alpha,\beta) $ which its vertex set contains at least a vertex of degree $ 4 $ in $ Z(\alpha,\beta) $. }\label{c3e}
\end{table}

\begin{longtable}{|l|l||l|l||l|l|}
\caption{The relations of  a cycle of length $ 3 $ which its vertex set contains at least a vertex of degree  $ 4 $  in $ Z(\alpha,\beta) $ that lead to  contradictions. }\label{tt10}\\
\hline
\endfirsthead
\multicolumn{6}{c}%
{{\bfseries \tablename\ \thetable{} -- continued from previous page}} \\\hline
\hline 
 \multicolumn{1}{|l|}{\textbf{$ r(C_3) $}} & \multicolumn{1}{l|}{\textbf{$E$}}&  \multicolumn{1}{|l|}{\textbf{$ r(C_3) $}} & \multicolumn{1}{l|}{\textbf{$E$}}&\multicolumn{1}{|l|}{\textbf{$ r(C_3) $}} & \multicolumn{1}{l|}{\textbf{$E$}}\\ \hline
\endhead

\hline \multicolumn{6}{|r|}{{Continued on next page}} \\ \hline
\endfoot

\hline 
\endlastfoot
\hline
\textbf{$ r(C_3) $}& \textbf{$E$} & \textbf{$ r(C_3)$}& \textbf{$E$} &\textbf{$ r(C_3) $}& \textbf{$E$} \\
\hline
$1.x^3y=1$&Abelian&
$2.x^{5}=1$&T&
$3.x^{4}y=1$&Abelian\\
\hline
$4.x^3y^{-1}x=1$&Abelian&
$5.x^2y^2=1$&BS$ (1,-1) $&
$6.x^2yx^2=1$&Abelian\\
\hline
$7.x(xy)^2=1$&Abelian&
$8.xy^{-1}x^{-1}y=1$&Abelian&
$9.xy^{-1}x^2=1$&Abelian\\
\hline
$10.xy^{-1}xy=1$&BS$ (1,-1) $&
$11.xy^{-2}x=1$&BS$ (1,-1) $&
$12.xy^{-1}xy=1$&BS$ (1,-1) $\\
\hline
$13.xy^{-1}x^3=1$&Abelian&
$14.xy^{-1}x^2y=1$&BS$ (2,-1) $&
$15.(xy^{-1})^2x=1$&Abelian\\
\hline
$16.(yx)^2x=1$&Abelian &
$17.xy^{-1}x^{-1}y=1$&Abelian&
$18.xy^{-1}x^{-2}y=1$&BS$ (2,1) $\\
\hline
$19.x^{-3}y=1$&Abelian&
$20.x^{-5}=1$&T&
$21.x^{-4}y=1$&Abelian\\
\hline
$22.x^{-3}y^{-1}x^{-1}=1$&Abelian&
$23.x^{-2}y^2=1$&BS$ (1,-1) $&
$24.x^{-2}yx^{-2}=1$&Abelian\\
\hline
$25.x^{-1}(x^{-1}y)^2=1$&Abelian&
$26.x^{-1}y^{-1}x^{-2}=1$&Abelian&
$27.x^{-1}y^{-1}x^{-1}y=1$&BS$ (-1,1) $\\
\hline
$28.x^{-1}y^{-1}xy=1$&Abelian&
$29.x^{-1}y^{-2}x^{-1}=1$&BS$ (1,-1) $&
$30.x^{-1}y^{-1}xy=1$&Abelian\\
\hline
$31.x^{-1}y^{-1}x^2y$&BS$ (2,1) $&
$32.yx^{4}=1$&Abelian&
$33.yx^2y^{-1}x=1$&BS$ (1,-2) $\\
\hline
$34.x^{-1}y^{-1}x^{-1}y=1$&BS$ (-1,1) $&
$35.x^{-1}y^{-1}x^{-3}=1$&Abelian&
$36.x^{-1}y^{-1}x^{-2}y=1$&BS$ (-2,1) $\\
\hline
$37.(x^{-1}y^{-1})^2x^{-1}=1$&Abelian&
$38.y^3=1$&T&
$39.y^2x^{-2}=1$&BS$ (1,-1) $\\
\hline
$40.y^2x^{-1}y=1$&Abelian&
$41.y^2x^2=1$&BS$ (1,-1) $&
$42.y^2xy=1$&Abelian\\
\hline
$43.yx^{-4}=1$&Abelian&
$44.yx^{-2}y^{-1}x=1$&BS$ (1,2) $&
$45.yx^{-2}y^{-1}x^{-1}=1$&BS$ (-1,2) $\\
\hline
$46.(yx^{-1})^2x^{-1}=1$&Abelian&
$47.(yx^{-1})^2y=1$&Abelian&
$48.yx^{-1}yxy=1$&BS$ (1,-2) $\\
\hline
$49.yx^2y^{-1}x^{-1}=1$&BS$ (1,2) $&
$50.(yx)^2y=1$&Abelian&
$ 51.yxyx^{-1}y =1$ &BS$ (2,-1)$\\
\hline
$52. y^{-1}x^5=1  $&Abelian &$53. y^{-1}x^3y^{-1}x =1 $&BS$ (1,-3) $&$ 54.(y^{-1}x)^3 =1 $&T \\
\hline
$55. y^{-1}x^{-5} =1 $& Abelian&$56. y^{-1}x^{-3}y^{-1}x^{-1} =1 $&BS$ (-1,1) $&$ 57.(y^{-1}x^{-1})^3 =1 $& T\\
\hline
\end{longtable}

\textbf{Graph} $ \mathbf{H_1} $ \textbf{of Figure \ref{10k}:}
Suppose that $ C $ and $ C' $ are two cycles of  length $  3$  in $Z(\alpha,\beta)$ with vertex sets  $\mathcal{V}_{C} =\{ g_1,g_2,g_3 \}\subset supp(\beta)$ and $\mathcal{V}_{C'} =\{ g_1,g_2,g_4 \}\subset supp(\beta)$, respectively, (i.e.,  $ C $ and $ C' $ have exactly an edge in common )  such that  $deg_Z(g_1)=4$.  Further suppose that $ T'\in \mathcal{T}(C') $, $ T\in \mathcal{T}(C) $ and $ T'=[h_1,h'_1,t_2,t'_2,t_3,t'_3] $ and $ T=[h_1,h'_1,h_2,h'_2,h_3,h'_3] $, where the first two components
are related to the common edge between these cycles. We note that by Remark \ref{3p}, $ h'_3\neq t'_3 $, if $ t_3^{-1}t'_3=x $, then $ h'_3\notin\{1,x\} $ and if $ t_3^{-1}t'_3=x^{-1} $, then $ h'_3\notin\{1,x^{-1}\} $ and  if $ h_2=t_2 $, then $ h'_2\neq t'_2 $ since otherwise $ g_3=g_4 $,  a contradiction.\\ 
With consideration of Table \ref{c3e} and by using GAP \cite{a9}, it can be seen that there are $9$ different cases for the relations of two cycles of length $3$ with this structure. Using GAP \cite{a9}, a free group with generators $ x,y $ and the relations of these cycles which are between $5$ cases of such $9$ cases  is finite, that is a
contradiction.   In the following it can be seen that $ 4 $ other cases lead to  contradictions and so, the graph $Z(\alpha,\beta)$ contains no subgraph isomorphic to the graph $ H_1 $ of Figure \ref{10k}.
\begin{itemize} 
\item[(1)]$ r_1:y^{-1}x^3y^{-1}x^{-1}=1  $, $ r_2:y^{-1}x^2yx^2=1  $.\\
 Using  Tietze transformation where $ y\mapsto xy$ and $ x\mapsto x $, we have:\\
 $ r_1\mapsto y^{-1}x^2y^{-1}x^{-2}=1  $ and $ r_2\mapsto r_2 $. By $ r_1 $ and $ r_2 $, $ x^{4}=y^{-2} $ and therefore $ x^{4}\in Z(\left\langle x,y\right\rangle ) $. On the other hand, $ r_2 $ implies $x^{4}=(x^{-4})^{y}  $. Hence, $ x^{8}=1 $, a contradiction.
\item[(2)]$ r_1:y^{-1}x^2yx^2=1  $, $ r_2:(y^{-1}x)^2y^{-1}x^{-1}=1  $.\\
Using  Tietze transformation where $ y\mapsto xy$ and $ x\mapsto x $, we have:\\
$ r_1\mapsto r_1 $ and $ r_2\mapsto  x^{2}=y^{-3} $. Thus, $ x^{2}\in Z(\left\langle x,y\right\rangle ) $ and therefore $ x^4=1 $, a contradiction.
\item[(3)]$ r_1:y^{-1}x^2yx^2=1  $, $ r_2:y^{-1}xy^{-1}x^{-3}=1   $.\\
Using  Tietze transformation where $ y\mapsto xy$ and $ x\mapsto x $, it can be seen that $ x^8=1 $, a contradiction.
\item[(4)]$ r_1:y^{-1}x^2yx^2=1  $, $ r_2:y^{-1}x(y^{-1}x^{-1})^2 =1   $.\\
Using  Tietze transformation where $ y\mapsto x^{-1}y$ and $ x\mapsto x $, it can be seen that $ x^4=1 $, a contradiction.
\end{itemize}

 \begin{figure}[ht]
\begin{tikzpicture}[scale=.75]
\draw [fill] (0,0) circle
[radius=0.1] node  [left]  {$ g_3 $};
\draw [fill] (2,0) circle
[radius=0.1] node  [right]  {$ g_4 $};
\draw [] (1,1) circle
[radius=0.1] node  [left]  {$ g_1 $};
\draw [fill] (1,-1) circle
[radius=0.1] node  [left]  {$ g_2 $};
\draw  (0,0) -- (.95,.95);
\draw  (1,-1)-- (.03,-.05);
\draw  (1,.9) --  (1,-.94) ;
\draw   (2,.06)--(1.05,.95) ; 
\draw   (1,-.96)--(2,-.05);
\end{tikzpicture}
\subfloat{$ H_1$}
\begin{tikzpicture}[scale=.75,shorten >=2pt]
\draw [] (0,0) circle
[radius=0.1] node  [left]  {$ g_3 $};
\draw [] (2,0) circle
[radius=0.1] node  [right]  {$ g_4 $};
\draw [fill] (1,1) circle
[radius=0.1] node  [left]  {$ g_1 $};
\draw [fill] (1,-1) circle
[radius=0.1] node  [left]  {$ g_2$};
\draw  (.05,.05) -- (1,1);
\draw  (1.98,.08)--(1,1);
\draw  (1,-1)-- (0,-.03);
\draw   (1,1) --  (1,-1) ;
\draw   (1,-1)--(2,-.03);
\end{tikzpicture}
\subfloat{$ H_2$}
 \begin{tikzpicture}[scale=.75]
\draw [] (0,0) circle
[radius=0.1] node  [left]  {};
\draw [fill] (1,0) circle
[radius=0.1] node  [right]  {};
\draw [] (1,-1) circle
[radius=0.1] node  [right]  {};
\draw [fill] (0,-1) circle
[radius=0.1] node  [left]  {};
\draw  (.1,0) -- (.95,0);
\draw  (1,-.1)-- (1,-.9);
\draw  (.9,-1) --  (.1,-1) ;
\draw   (0,-.9)--(0,-.1) ; 
\end{tikzpicture}
\subfloat{$(a)$}
 \begin{tikzpicture}[scale=.75]
\draw [] (0,0) circle
[radius=0.1] node  [left]  {};
\draw [] (1,0) circle
[radius=0.1] node  [right]  {};
\draw [fill] (1,-1) circle
[radius=0.1] node  [right]  {};
\draw [fill] (0,-1) circle
[radius=0.1] node  [left]  {};
\draw  (.1,0) -- (.95,0);
\draw  (1,-.1)-- (1,-.9);
\draw  (.9,-1) --  (.1,-1) ;
\draw   (0,-.9)--(0,-.1) ; 
\end{tikzpicture}
\subfloat{$ (b)$}
\begin{tikzpicture}[scale=.75]
\draw [fill] (0,0) circle
[radius=0.1] node  [left]  {};
\draw [fill] (2,0) circle
[radius=0.1] node  [right]  {};
\draw [fill][black!20!white] (1,1) circle
[radius=0.1] node  [left]  {};
\draw [fill][black!20!white] (1,-1) circle
[radius=0.1] node  [left]  {};
\draw [fill][black!20!white] (1,0) circle
[radius=0.1] node  [left]  {};
\draw  (.05,.05) -- (.95,.95);
\draw  (1.95,.05)--(1.1,.95) ;
\draw  (1.1,-.9)--(1.95,-.05);
\draw(.91,-.9)-- (.05,-.05);
\draw  (1,.9) --  (1,.1) ;
\draw  (1,-.1) --  (1,-.9) ;
\end{tikzpicture}
\subfloat{$ H_3$}\\
\begin{tikzpicture}[scale=.75]
\draw [fill] (0,0) circle
[radius=0.1] node  [left]  {};
\draw [fill] (2,0) circle
[radius=0.1] node  [right]  {};
\draw [] (1,1) circle
[radius=0.1] node  [left]  {};
\draw [] (1,-1) circle
[radius=0.1] node  [left]  {};
\draw [fill] (1,0) circle
[radius=0.1] node  [left]  {};
\draw [fill] (1.5,0) circle
[radius=0.1] node  [right]  {};
\draw   (1.47,.05) -- (1,.95);
\draw   (1,-1)--(1.47,-.05) ;
\draw   (.05,.05) -- (1,.95);
\draw  (1.95,.05)--(1,.95)  ;
\draw   (1,-.95)-- (.05,-.05);
\draw   (1,.95) --  (1,.13) ;
\draw   (1,-.1) --  (1,-.9) ;
\draw   (1,-.95)--(1.95,-.05);
\end{tikzpicture}
\subfloat{$ H_4$}
\begin{tikzpicture}[scale=.75]
\draw [] (0,0) circle
[radius=0.1] node  [left]  {};
\draw [] (2,0) circle
[radius=0.1] node  [right]  {};
\draw [fill] (1,1) circle
[radius=0.1] node  [left]  {};
\draw [fill] (1,-1) circle
[radius=0.1] node  [left]  {};
\draw [] (1,0) circle
[radius=0.1] node  [left]  {};
\draw [] (1.5,0) circle
[radius=0.1] node  [right]  {};
\draw   (1.47,.05) -- (1,.95);
\draw   (1,-1)--(1.47,-.05) ;
\draw   (.05,.05) -- (1,.95);
\draw  (1.95,.05)--(1,.95)  ;
\draw   (1,-1)-- (.05,-.05);
\draw   (1,1) --  (1,.13) ;
\draw   (1,-.1) --  (1,-.9) ;
\draw   (1,-1)--(1.95,-.05);
\end{tikzpicture}
\subfloat{$ H_5$}
\begin{tikzpicture}[scale=.75]
\draw [] (0,0) circle
[radius=0.1] node  [left]  {};
\draw [] (2,0) circle
[radius=0.1] node  [right]  {};
\draw [fill] (1,1) circle
[radius=0.1] node  [left]  {};
\draw [fill] (1,-1) circle
[radius=0.1] node  [left]  {};
\draw [] (1,0) circle
[radius=0.1] node  [left]  {};
\draw [fill] (2,1) circle
[radius=0.1] node  [left]  {};
\draw   (1.1,.05) -- (2,1);
\draw   (2,.05) -- (2,1);
\draw   (.05,.05) -- (1,.95);
\draw  (1.95,.05)--(1,.95)  ;
\draw   (1,-1)-- (.05,-.05);
\draw   (1,1) --  (1,.13) ;
\draw   (1,-.1) --  (1,-.9) ;
\draw   (1,-1)--(1.95,-.05);
\end{tikzpicture}
\subfloat{$ H_6$}
\begin{tikzpicture}[scale=.75]
\draw [] (0,0) circle
[radius=0.1] node  [left]  {};
\draw [fill] (1,0) circle
[radius=0.1] node  [right]  {};
\draw [fill] (0,-1) circle
[radius=0.1] node  [left]  {};
\draw [] (1,-1) circle
[radius=0.1] node  [right]  {};
\draw [fill] (.5,1) circle
[radius=0.1] node  [left]  {};
\draw   (.5,1) --(0,.1);
\draw    (1,0)--(.5,.98)  ;
\draw   (.1,0) --  (.98,0) ;
\draw     (1,0)--(1,-.9) ;
\draw   (.9,-1) --  (0,-1) ;
\draw   (0,-1) --(0,-.1) ;
\end{tikzpicture}
\subfloat{$ H_7$}
\begin{tikzpicture}[scale=.75]
\draw [fill] (0,0) circle
[radius=0.1] node  [left]  {};
\draw [fill] (1,0) circle
[radius=0.1] node  [right]  {};
\draw [] (0,-1) circle
[radius=0.1] node  [left]  {};
\draw [] (1,-1) circle
[radius=0.1] node  [right]  {};
\draw [] (.5,1) circle
[radius=0.1] node  [left]  {};
\draw   (.45,.92) --(0,.1);
\draw    (1,0)--(.55,.92)  ;
\draw   (.1,0) --  (.9,0) ;
\draw     (1,0)--(1,-.9) ;
\draw   (.9,-1) --  (.1,-1) ;
\draw   (0,-.9) --(0,-.1) ;
\end{tikzpicture}
\subfloat{$ H_8$}
\begin{tikzpicture}[scale=.75]
\draw [] (0,0) circle
[radius=0.1] node  [left]  {};
\draw [] (1,0) circle
[radius=0.1] node  [right]  {};
\draw [fill] (0,-1) circle
[radius=0.1] node  [left]  {};
\draw [fill] (1,-1) circle
[radius=0.1] node  [right]  {};
\draw [fill] (.5,1) circle
[radius=0.1] node  [left]  {};
\draw   (.5,1) --(0,.1);
\draw    (1,.1)--(.5,.98)  ;
\draw   (.1,0) --  (.9,0) ;
\draw     (1,-.1)--(1,-.9) ;
\draw   (.9,-1) --  (0,-1) ;
\draw   (0,-1) --(0,-.1) ;
\end{tikzpicture}
\subfloat{$ H_9$}
\caption{ The degrees of white and gray vertices in $ Z(\alpha,\beta) $ must be  $ 4 $  and   $ 4 $ of type $ (ii) $, respectively. }\label{10k}
\end{figure}
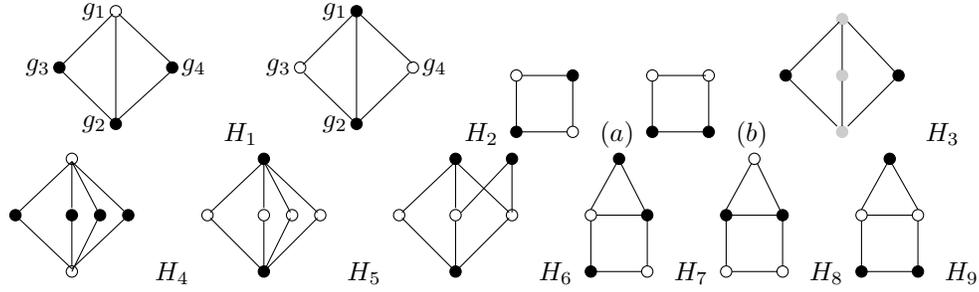
\textbf{Graph} $ \mathbf{H_2} $ \textbf{of Figure \ref{10k}:}
Suppose that $ C $ and $ C' $ are two cycles of  length $  3$  in $Z(\alpha,\beta)$ with vertex sets  $\mathcal{V}_{C} =\{ g_1,g_2,g_3 \}\subset supp(\beta)$ and $\mathcal{V}_{C'} =\{ g_1,g_2,g_4 \}\subset supp(\beta)$, respectively, (i.e.,  $ C $ and $ C' $ have exactly an edge in common )  such that  $deg_Z(g_3)=4$ and $deg_Z(g_4)=4$.  Further suppose that $ T'\in \mathcal{T}(C') $, $ T\in \mathcal{T}(C) $, $ T'=[h_1,h'_1,t_2,t'_2,t_3,t'_3] $ and $ T=[h_1,h'_1,h_2,h'_2,h_3,h'_3] $, where the first two components
are related to the common edge between these cycles. We note that in view of  Remark \ref{3p}, $ h'_3\neq t'_3 $, if $ t_3^{-1}t'_3=x $, then $ h'_3\notin\{1,x\} $, if $ t_3^{-1}t'_3=x^{-1} $, then $ h'_3\notin\{1,x^{-1}\} $ and also  $ h_2\neq t_2 $ and if $ t_2^{-1}t'_2=x $, then $ h_2\notin\{1,x^{-1}\} $, if $ t_2^{-1}t'_2=x^{-1} $, then $ h_2\notin\{1,x\} $. Thus, clearly  the relations  corresponding to the graph $ H_2 $ are a subset of the relations corresponding to the graph $ H_1 $ as Figure \ref{10k}. Hence, the graph $Z(\alpha,\beta)$ contains no subgraph isomorphic to the graph $ H_2 $ of Figure \ref{10k}.\\
$ \mathbf{C_4:} $ Suppose that $ C $ is a cycle of length $ 4 $ in $ Z(\alpha,\beta) $ such that  its  vertex set  contains at least two vertices of degree $ 4 $ in $ Z(\alpha,\beta) $ (see graphs $ (a),(b) $ of Figure \ref{10k}).  Suppose further that $ T\in \mathcal{T}(C) $ and $ T=[h_1,h'_1,h_2,h'_2,h_3,h'_3,h_4,h'_4] $. 
By Remark \ref{3p}, $ h_1\neq h'_1\neq h_2\neq h'_2\neq h_3\neq h'_3\neq h_4\neq h'_4\neq h_1 $,  if $ h_{1}^{-1}h'_{1}= x$   $(h_{1}^{-1}h'_{1}=x^{-1})$, then $ h_2\notin\{ 1,x\}$   $(h_2\notin\{ 1,x^{-1}\}) $ and $ h'_4\notin\{ 1,x^{-1}\} $  $(h_4\notin\{ 1,x\}) $, if $ h_{2}^{-1}h'_{2}= x$  $(h_{2}^{-1}h'_{2}=x^{-1})$, then $ h_3\notin\{ 1,x\}$  $(h_3\notin\{ 1,x^{-1}\}) $ and $ h'_1\notin\{ 1,x^{-1}\}$  $(h'_1\notin\{ 1,x\}) $, if $ h_{3}^{-1}h'_{3}= x$  $(h_{3}^{-1}h'_{3}=x^{-1})$, then $ h_4\notin\{ 1,x\}$  $(h_4\notin\{ 1,x^{-1}\}) $ and $ h'_2\notin\{ 1,x^{-1}\}$  $(h'_2\notin\{ 1,x\}) $ and if $h_{4}^{-1}h'_{4}= x$  $(h_{4}^{-1}h'_{4}=x^{-1})$, then $ h_1\notin\{ 1,x\}$  $(h_1\notin\{ 1,x^{-1}\}) $ and $ h'_3\notin\{ 1,x^{-1}\}$  $(h'_3\notin\{ 1,x\}) $. Also as for Lemma 2.4 of \cite{Dr. abdollahi T}, we may assume that $ h_1\in \{1,y\} $. Hence,  by using GAP \cite{a9}, it can be seen that there exist $351$ non-equivalent cases for $ T $. We checked the relations of such  non-equivalent cases. In  Table \ref{tt1}, it can be seen that $ 120 $ cases among $351$ cases  lead to contradictions. Hence, there are just $ 231 $ cases which may lead
to the existence of $ C$ in $ Z(\alpha,\beta) $ which listed in Table \ref{tt011}.\\
\begin{longtable}{|l|l||l|l||l|l|}
\caption{The relations of  a cycle of length $ 4 $ which its vertex set contains  at least $ 2 $ vertices of degree  $ 4 $  in $ Z(\alpha,\beta) $ that lead to  contradictions. }\label{tt1}\\
\hline
\endfirsthead
\multicolumn{6}{c}%
{{\bfseries \tablename\ \thetable{} -- continued from previous page}} \\\hline
\hline 
\multicolumn{1}{|l|}{\textbf{$ n $}} & \multicolumn{1}{|l|}{\textbf{$ r(C_4) $}} & \multicolumn{1}{l|}{\textbf{$E$}}& \multicolumn{1}{|l|}{\textbf{$ n $}} & \multicolumn{1}{|l|}{\textbf{$ r(C_4) $}} & \multicolumn{1}{l|}{\textbf{$E$}}\\ \hline
\endhead

\hline \multicolumn{6}{|r|}{{Continued on next page}} \\ \hline
\endfoot

\hline 
\endlastfoot
\hline
\textbf{$n $} & \textbf{$ r(C_4) $}& \textbf{$E$} & \textbf{$ n $} & \textbf{$ r(C_4)$}& \textbf{$E$}  \\
\hline
 $1$&$x^6=1$&T&
$2$&$x^5y=1$&Abelian\\\hline
$3$&$x^4y^{-1}x=1$&Abelian&
$4$&$x^3yx^2=1$&Abelian\\\hline
$5$&$x^2(xy)^2=1$&BS$(1,-1)$&
$6$&$x^3y=1$&Abelian\\\hline
$7$&$x^7=1$&T&
$8$&$x^6y=1$&Abelian\\\hline
$9$&$x^5y^{-1}x=1$&Abelian&
$10$&$x^4yx^2=1$&Abelian\\\hline
$11$&$x^3y^{-1}x^{-1}y=1$&BS$(1,3)$&
$12$&$x^3y^{-1}x^2=1$&Abelian\\\hline
$13$&$x^3y^{-1}xy=1$&BS$(1,-3)$&
$14$&$x^3y^{-1}x^3=1$&Abelian\\\hline
$15$&$(x^2y)^2=1$&Abelian&
$16$&$x^2yxy^{-1}x=1$&BS$(3,-1)$\\\hline
$17$&$x^2yx^{-1}y^{-1}x=1$&BS$(3,1)$&
$18$&$x^2yx^3=1$&Abelian\\\hline
$19$&$x^2yx^4=1$&Abelian&
$20$&$x^2yx^3y=1$&Abelian\\\hline
$21$&$x(xy)^2y=1$&BS$(2,-1)$&
$22$&$x(xy)^3=1$&Abelian\\\hline
$23$&$xy^{-1}x^{-2}y=1$&BS$(2,1)$&
$24$&$xy^{-1}x^{-3}y=1$&BS$(3,1)$\\\hline
$25$&$xy^{-1}x^{-1}y^2=1$&BS$(2,1)$&
$26$&$xy^{-1}(x^{-1}y)^2=1$&BS$(2,1)$\\\hline
$27$&$xy^{-1}x^{-1}yx^2=1$&BS$(1,3)$&
$28$&$xy^{-1}x^{-1}yxy=1$&BS$(2,1)$\\\hline
$29$&$xy^{-1}x^2y=1$&BS$(2,-1)$&
$30$&$xy^{-1}x^4=1$&Abelian\\\hline
$31$&$xy^{-1}x^3y=1$&BS$(3,-1)$&
$32$&$(xy^{-1}x)^2=1$&Abelian\\\hline
$33$&$xy^{-1}xyx^2=1$&BS$(1,-3)$&
$34$&$xy^{-2}x^{-1}y=1$&BS$(1,2)$\\\hline
$35$&$xy^{-1}(y^{-1}x)^2=1$&BS$(2,-1)$&
$36$&$xy^{-1}x^5=1$&Abelian\\\hline
$37$&$xy^{-1}x^4y=1$&BS$(4,-1)$&
$38$&$xy^{-1}x^3y^{-1}x=1$&Abelian \\\hline
$39$&$(xy^{-1})^2x^{-1}y=1$&BS$(1,2)$&
$40$&$(xy^{-1})^2x^2=1$&BS$(1,-1)$\\\hline
$41$&$(xy^{-1})^2y^{-1}x=1$&BS$(1,-2)$&
$42$&$(xy^{-1})^3x=1$&Abelian\\\hline
$43$&$xy^{-1}x^{-3}y=1$&BS$(3,1)$&
$44$&$xy^{-1}x^{-4}y=1$&BS$(4,1)$\\\hline
$45$&$x(y^{-1}x^{-1})^2y=1$&BS$(1,2)$&
$46$&$x^{-6}=1$&T\\\hline
$47$&$x^{-5}y=1$&Abelian&
$48$&$x^{-4}y^{-1}x^{-1}=1$&Abelian\\\hline
$49$&$x^{-3}yx^{-2}=1$&Abelian&
$50$&$x^{-2}(x^{-1}y)^2=1$&BS$(1,-1)$\\\hline
$51$&$x^{-7}=1$&T&
$52$&$x^{-6}y=1$&Abelian\\\hline
$53$&$x^{-5}y^{-1}x^{-1}=1$&Abelian&
$54$&$x^{-4}yx^{-2}=1$&Abelian\\\hline
$55$&$x^{-3}y^{-1}x^{-2}=1$&Abelian&
$56$&$x^{-3}y^{-1}x^{-1}y=1$&BS$(1,-3)$\\\hline
$57$&$x^{-3}y^{-1}xy=1$&BS$(1,3)$&
$58$&$x^{-3}y^{-1}x^{-3}=1$&Abelian\\\hline
$59$&$x^{-2}yx^{-1}y^{-1}x^{-1}=1$&BS$(-3,1)$&
$60$&$(x^{-2}y)^2=1$&Abelian\\\hline
$61$&$x^{-2}yx^{-4}=1$&Abelian&
$62$&$x^{-2}yx^{-3}y=1$&Abelian\\\hline
$63$&$x^{-1}(x^{-1}y)^2y=1$&BS$(2,-1)$&
$64$&$x^{-1}(x^{-1}y)^3=1$&Abelian\\\hline
$65$&$x^{-2}yxy^{-1}x^{-1}=1$&BS$(3,1)$&
$66$&$x^{-1}y^{-1}x^{-2}y=1$&BS$(-2,1)$\\\hline
$67$&$x^{-1}y^{-1}x^{-4}=1$&Abelian&
$68$&$x^{-1}y^{-1}x^{-3}y=1$&BS$(-3,1)$\\\hline
$69$&$(x^{-1}y^{-1}x^{-1})^2=1$&Abelian&
$70$&$x^{-1}y^{-1}x^{-1}yx^{-2}=1$&BS$(-1,3)$\\\hline
$71$&$x^{-1}y^{-1}x^3y=1$&BS$(3,1)$&
$72$&$x^{-1}y^{-1}xy^2=1$&BS$(1,2)$\\\hline
$73$&$x^{-1}y^{-1}xyx^{-2}=1$&BS$(1,3)$&
$74$&$x^{-1}y^{-1}xyx^{-1}y=1$&BS$(1,2)$\\\hline
$75$&$x^{-1}y^{-1}(xy)^2=1$&BS$(2,1)$&
$76$&$x^{-1}y^{-1}(y^{-1}x^{-1})^2=1$&BS$(2,-1)$\\\hline
$77$&$x^{-1}y^{-1}x^4y=1$&BS$(4,1)$&
$78$&$x^{-1}(y^{-1}x)^2y=1$&BS$(2,1)$\\\hline
$79$&$x^{-1}y^{-1}x^{-3}y=1$&BS$(-3,1)$&
$80$&$x^{-1}y^{-1}x^{-5}=1$&Abelian\\\hline
$81$&$x^{-1}y^{-1}x^{-4}y=1$&BS$(-4,1)$&
$82$&$x^{-1}y^{-1}x^{-3}y^{-1}x^{-1}=1$&Abelian\\\hline
$83$&$(x^{-1}y^{-1})^2x^{-2}$&BS$(1,-1)$&
$84$&$(x^{-1}y^{-1})^2xy=1$&BS$(1,2)$\\\hline
$85$&$(x^{-1}y^{-1})^2y^{-1}x^{-1}=1$&BS$(2,-1)$&
$86$&$(x^{-1}y^{-1})^3x^{-1}=1$&Abelian\\\hline
$87$&$y^4=1$&T&
$88$&$y^3x^{-1}y=1$&Abelian\\\hline
$89$&$y^3xy=1$&Abelian&
$90$&$y(yx^{-1})^2x^{-1}=1$&BS$(1,-2)$\\\hline
$91$&$y(yx^{-1})^2y=1$&BS$(1,-1)$&
$92$&$y^2x^{-1}yxy=1$&BS$(1,-3)$\\\hline
$93$&$y^2xyx^{-1}y=1$&BS$(3,-1)$&
$94$&$y(yx)^2x=1$&BS$(2,-1)$\\\hline
$95$&$y(yx)^2y=1$&BS$(1,-1)$&
$96$&$(yx^{-2})^2=1$&Abelian\\\hline
$97$&$yx^{-1}(x^{-1}y)^2=1$&BS$(2,-1)$&
$98$&$yx^{-6}=1$&Abelian\\\hline
$99$&$yx^{-4}y^{-1}x=1$&BS$(1,4)$&
$100$&$yx^{-4}y^{-1}x^{-1}=1$&BS$(1,-4)$\\\hline
$101$&$yx^{-3}yx^{-2}=1$&Abelian&
$102$&$(yx^{-1}y)^2=1$&Abelian\\\hline
$103$&$(yx^{-1})^3x^{-1}=1$&Abelian&
$104$&$(yx^{-1})^3y=1$&Abelian\\\hline
$105$&$(yx^2)^2=1$&Abelian&
$106$&$yx(xy)^2=1$&BS$(2,-1)$\\\hline
$107$&$yx^6=1$&Abelian&
$108$&$yx^4y^{-1}x=1$&BS$(1,-4)$\\\hline
$109$&$yx^4y^{-1}x^{-1}=1$&BS$(1,4)$&
$110$&$(yxy)^2=1$&Abelian\\\hline
$111$&$(yx)^3x=1$&Abelian&
$112$&$(yx)^3y=1$&Abelian \\\hline
$113$&$yx^3yx^2=1$&Abelian&$ 114 $&$ y^{-1}x^7 =1 $&Abelian\\
\hline
$115$&$(y^{-1}x^3)^2 =1$&Abelian&$ 116 $&$ (y^{-1}x)^4 =1$&Abelian\\\hline
$117$&$(y^{-1}xy^{-1}x^{-1})^2 =1$&BS$(1,-1)$&$ 118 $&$ y^{-1}x^{-7} =1$&Abelian\\\hline
$119$&$(y^{-1}x^{-1})^4 =1$&Abelian&120&$ (y^{-1}x^{-3})^2 =1$ & Abelian\\\hline
 \end{longtable}
\begin{rem}
In view of  Remark {\rm\ref{3p}}, it can be seen that  there exist  $ 31 $ cases  for the  existence of  a cycle of length $ 4 $ on vertices of degree  $4$ of type $ (ii) $ in $ Z(\alpha,\beta) $ which are marked by  *s in Table {\rm\ref{tt011}}. 
\end{rem}  

 \begin{longtable}{|l|l|l|l|l|l|}
\caption{The relations for the existence of a cycle of length $ 4 $ which its vertex set contains  at least $ 2 $ vertices of degree  $ 4 $  in $ Z(\alpha,\beta) $. }\label{tt011}\\
\hline
\endfirsthead
\multicolumn{6}{c}%
{{\bfseries \tablename\ \thetable{} -- continued from previous page}} \\\hline
\hline  
 \multicolumn{1}{|l|}{\textbf{$ n $}} & \multicolumn{1}{l|}{\textbf{$r(C_4)$}}&  \multicolumn{1}{|l|}{\textbf{$ n $}} & \multicolumn{1}{l|}{\textbf{$r(C_4)$}}&\multicolumn{1}{|l|}{\textbf{$ n $}} & \multicolumn{1}{l|}{\textbf{$r(C_4)$}}\\ \hline
\endhead

\hline \multicolumn{6}{|r|}{{Continued on next page}} \\ \hline
\endfoot

\hline 
\endlastfoot
\hline
\textbf{$ n $}& \textbf{$r(C_4)$} & \textbf{$ n $}& \textbf{$r(C_4)$} &\textbf{$ n $}& \textbf{$r(C_4)$} \\
\hline
 $1$&$x^3y^2=1$&
$2$&$x^3yx^{-1}y=1$&
$3$&$x^4y^2=1$\\
$4$&$x^4yx^{-1}y=1$&
$5$&$x^3(xy)^2=1$&
$6$&$x^3y^{-2}x=1$\\
$7$&$x^3y^{-1}x^2y=1$&
$8^{*}$&$x^2(xy^{-1})^2x=1$&
$9$&$x^3y^{-1}x^{-2}y=1$\\
$10^{*}$&$x^3y^{-1}x^{-1}y^{-1}x=1$&
$11^{*}$&$x^2yx^{-2}y=1$&
$12$&$x^2y^3=1$\\
$13$&$x^2y^2x^{-1}y=1$&
$14$&$x^2y^2x^2=1$&
$15$&$x^2y^2xy=1$\\
$16$&$x^2yx^{-2}y=1$&
$17$&$x^2yx^{-3}y=1$&
$18^{*}$&$x^2yx^{-2}y^{-1}x=1$\\
$19$&$x^2yx^{-1}y^2=1$&
$20^{*}$&$x^2(yx^{-1})^2y=1$&
$21^{*}$&$x^2yx^{-1}yx^2=1$\\
$22^{*}$&$x^2yx^{-1}yxy=1$&
$23^{*}$&$x^2yx^2y^{-1}x=1$&
$24^{*}$&$x(xy)^2x^{-1}y=1$\\
$25^{*}$&$x(xy)^2x^2=1$&
$26$&$xy^{-1}x^{-2}y^{-1}x=1$&
$27$&$xy^{-1}xy^2=1$\\
$28$&$xy^{-1}xyx^{-1}y=1$&
$29$&$xy^{-1}(xy)^2=1$&
$30$&$xy^{-2}x^2=1$\\
$31$&$xy^{-2}xy=1$&
$32$&$xy^{-3}x=1$&
$33$&$xy^{-2}xy=1$\\
$34$&$xy^{-2}x^3=1$&
$35$&$xy^{-2}x^2y=1$&
$36$&$xy^{-2}x^{-2}y=1$\\
$37$&$xy^{-2}x^{-1}y^{-1}x=1$&
$38$&$xy^{-1}xy^2=1$&
$39$&$xy^{-1}xyx^{-1}y=1$\\
$40$&$xy^{-1}(xy)^2=1$&
$41$&$xy^{-1}x^2y^2=1$&
$42$&$xy^{-1}x^2yx^{-1}y=1$\\
$43$&$xy^{-1}x^2yx^2=1$&
$44$&$xy^{-1}x(xy)^2=1$&
$45$&$(xy^{-1})^2xy=1$\\
$46$&$(xy^{-1})^2xy=1$&
$47$&$(xy^{-1})^2x^3=1$&
$48$&$(xy^{-1})^2x^2y=1$\\
$49$&$(xy^{-1})^2x^{-2}y=1$&
$50^{*}$&$(xy^{-1})^2x^{-1}y^{-1}x=1$&
$51^{*}$&$xy^{-1}x^{-2}y^{-1}x=1$\\
$52$&$xy^{-1}x^{-3}y^{-1}x=1$&
$53$&$xy^{-1}x^{-2}y^2=1$&
$54$&$xy^{-1}x^{-1}(x^{-1}y)^2=1$\\
$55$&$xy^{-1}x^{-2}yx^2=1$&
$56$&$xy^{-1}x^{-2}yxy=1$&
$57$&$xy^{-1}x^{-1}y^{-1}x^2=1$\\
$58$&$xy^{-1}x^{-1}y^{-1}xy=1$&
$59$&$xy^{-1}x^{-1}y^{-2}x=1$&
$60$&$xy^{-1}x^{-1}y^{-1}xy=1$\\
$61$&$xy^{-1}x^{-1}y^{-1}x^3=1$&
$62$&$xy^{-1}x^{-1}y^{-1}x^2y=1$&
$63^{*}$&$xy^{-1}x^{-1}(y^{-1}x)^2=1$\\
$64$&$x(y^{-1}x^{-1})^2x^{-1}y=1$&
$65^{*}$&$x(y^{-1}x^{-1})^2y^{-1}x=1$&
$66$&$x^{-3}y^2=1$\\
$67$&$x^{-3}yxy=1$&
$68$&$x^{-4}y^2=1$&
$69$&$x^{-3}(x^{-1}y)^2=1$\\
$70$&$x^{-4}yxy=1$&
$71$&$x^{-3}y^{-2}x^{-1}=1$&
$72$&$x^{-3}y^{-1}x^2y=1$\\
$73^{*}$&$x^{-3}y^{-1}xy^{-1}x^{-1}=1$&
$74$&$x^{-3}y^{-1}x^{-2}y=1$&
$75^{*}$&$x^{-2}(x^{-1}y^{-1})^2x^{-1}=1$\\
$76$&$x^{-2}y^3=1$&
$77$&$x^{-2}y^2x^{-2}=1$&
$78$&$x^{-2}y^2x^{-1}y=1$\\
$79$&$x^{-2}y^2xy=1$&
$80^{*}$&$x^{-2}yx^{-2}y^{-1}x^{-1}=1$&
$81^{*}$&$x^{-1}(x^{-1}y)^2x^{-2}=1$\\
$82^{*}$&$x^{-1}(x^{-1}y)^2xy=1$&
$83^{*}$&$x^{-2}yx^2y=1$&
$84^{*}$&$x^{-2}yx^2y=1$\\
$85$&$x^{-2}yx^3y=1$&
$86^{*}$&$x^{-2}yx^2y^{-1}x^{-1}=1$&
$87$&$x^{-2}yxy^2=1$\\
$88^{*}$&$x^{-2}yxyx^{-2}=1$&
$89^{*}$&$x^{-2}yxyx^{-1}y=1$&
$90^{*}$&$x^{-2}(yx)^2y=1$\\
$91$&$x^{-1}y^{-1}x^{-1}y^2=1$&
$92$&$x^{-1}y^{-1}(x^{-1}y)^2=1$&
$93$&$x^{-1}y^{-1}x^{-1}yxy=1$\\
$94$&$x^{-1}y^{-1}x^2y^{-1}x^{-1}=1$&
$95$&$x^{-1}y^{-2}x^{-2}=1$&
$96$&$x^{-1}y^{-2}x^{-1}y=1$\\
$97$&$x^{-1}y^{-3}x^{-1}=1$&
$98$&$x^{-1}y^{-2}x^2y=1$&
$99$&$x^{-1}y^{-2}xy^{-1}x^{-1}=1$\\
$100$&$x^{-1}y^{-2}x^{-1}y=1$&
$101$&$x^{-1}y^{-2}x^{-3}=1$&
$102$&$x^{-1}y^{-2}x^{-2}y=1$\\
$103$&$x^{-1}y^{-1}x^3y^{-1}x^{-1}=1$&
$104$&$x^{-1}y^{-1}x^2y^2=1$&
$105$&$x^{-1}y^{-1}x^2yx^{-2}=1$\\
$106$&$x^{-1}y^{-1}x^2yx^{-1}y=1$&
$107$&$x^{-1}y^{-1}x(xy)^2=1$&
$108$&$x^{-1}y^{-1}xy^{-1}x^{-2}=1$\\
$109$&$x^{-1}y^{-1}xy^{-1}x^{-1}y=1$&
$110$&$x^{-1}y^{-1}xy^{-2}x^{-1}=1$&
$111$&$x^{-1}(y^{-1}x)^2xy=1$\\
$112^{*}$&$x^{-1}(y^{-1}x)^2y^{-1}x^{-1}=1$&
$113$&$x^{-1}y^{-1}xy^{-1}x^{-1}y=1$&
$114$&$x^{-1}y^{-1}xy^{-1}x^{-3}=1$\\
$115$&$x^{-1}y^{-1}xy^{-1}x^{-2}y=1$&
$116^{*}$&$x^{-1}y^{-1}x(y^{-1}x^{-1})^2=1$&
$117$&$x^{-1}y^{-1}x^{-1}y^2=1$\\
$118$&$x^{-1}y^{-1}(x^{-1}y)^2=1$&
$119$&$x^{-1}y^{-1}x^{-1}yxy=1$&
$120$&$x^{-1}y^{-1}x^{-2}y^2=1$\\
$121$&$x^{-1}y^{-1}x^{-2}yx^{-2}=1$&
$122$&$x^{-1}y^{-1}x^{-1}(x^{-1}y)^2=1$&
$123$&$x^{-1}y^{-1}x^{-2}yxy=1$\\
$124$&$(x^{-1}y^{-1})^2x^{-1}y=1$&
$125$&$(x^{-1}y^{-1})^2x^2y=1$&
$126^{*}$&$(x^{-1}y^{-1})^2xy^{-1}x^{-1}=1$\\
$127$&$(x^{-1}y^{-1})^2x^{-1}y=1$&
$128$&$(x^{-1}y^{-1})^2x^{-3}=1$&
$129$&$(x^{-1}y^{-1})^2x^{-2}y=1$\\
$130$&$y^3x^{-2}=1$&
$131$&$y^3x^2=1$&
$132$&$y^2x^{-4}=1$\\
$133$&$y^2x^{-3}y=1$&
$134$&$y^2x^{-2}y^{-1}x=1$&
$135$&$y^2x^{-2}y^{-1}x^{-1}=1$\\
$136$&$y^2x^{-1}yx^2=1$&
$137$&$y^2x^4=1$&
$138$&$y^2x^3y=1$\\
$139$&$y^2x^2y^{-1}x=1$&
$140$&$y^2x^2y^{-1}x^{-1}=1$&
$141$&$y^2xyx^{-2}=1$\\
$142$&$yx^{-2}yx^2=1$&
$143$&$yx^{-2}yxy=1$&
$144$&$yx^{-5}y=1$\\
$145$&$yx^{-2}(x^{-1}y)^2=1$&
$146$&$yx^{-3}yx^2=1$&
$147$&$yx^{-3}yxy=1$\\
$148$&$yx^{-2}y^{-1}x^{-2}=1$&
$149$&$yx^{-2}y^{-1}x^{-1}y=1$&
$150$&$yx^{-2}y^{-1}x^2=1$\\
$151$&$yx^{-2}y^{-1}xy=1$&
$152$&$yx^{-2}y^{-2}x=1$&
$153$&$yx^{-2}y^{-2}x^{-1}=1$\\
$154$&$yx^{-2}y^{-1}x^3=1$&
$155$&$yx^{-2}y^{-1}x^2y=1$&
$156$&$yx^{-2}(y^{-1}x)^2=1$\\
$157$&$yx^{-2}y^{-1}xy^{-1}x^{-1}=1$&
$158$&$yx^{-2}y^{-1}x^{-3}=1$&
$159$&$yx^{-2}y^{-1}x^{-2}y=1$\\
$160$&$yx^{-1}(x^{-1}y^{-1})^2x=1$&
$161$&$yx^{-1}(x^{-1}y^{-1})^2x^{-1}=1$&
$162$&$yx^{-1}y^2x^2=1$\\
$163$&$yx^{-1}y^2xy=1$&
$164$&$(yx^{-1})^2x^{-3}=1$&
$165$&$(yx^{-1})^2x^{-2}y=1$\\
$166$&$(yx^{-1})^2x^{-1}y^{-1}x=1$&
$167$&$(yx^{-1})^2x^{-1}y^{-1}x^{-1}=1$&
$168$&$(yx^{-1})^2yx^2=1$\\
$169$&$(yx^{-1})^2yxy=1$&
$170$&$yx^{-1}yx^4=1$&
$171$&$yx^{-1}yx^3y=1$\\
$172$&$yx^{-1}yx^2y^{-1}x=1$&
$173$&$yx^{-1}yx^2y^{-1}x^{-1}=1$&
$174$&$yx^{-1}yxyx^{-2}=1$\\
$175$&$yx^{-1}yxyx^{-1}y=1$&
$176$&$yx^{-1}(yx)^2x=1$&
$177$&$yx^{-1}(yx)^2y=1$\\
$178$&$yx^5y=1$&
$179$&$yx^3yx^{-2}=1$&
$180$&$yx^3yx^{-1}y=1$\\
$181$&$yx^2(xy)^2=1$&
$182$&$yx^2y^{-1}x^3=1$&
$183$&$yx^2y^{-1}x^2y=1$\\
$184$&$yx(xy^{-1})^2x=1$&
$185$&$yx(xy^{-1})^2x^{-1}=1$&
$186$&$yx^2y^{-1}x^{-3}=1$\\
$187$&$yx^2y^{-1}x^{-2}y=1$&
$188$&$yx^2y^{-1}x^{-1}y^{-1}x=1$&
$189$&$yx^2(y^{-1}x^{-1})^2=1$\\
$190$&$yxyx^{-4}=1$&
$191$&$yxyx^{-3}y=1$&
$192$&$yxyx^{-2}y^{-1}x=1$\\
$193$&$yxyx^{-2}y^{-1}x^{-1}=1$&
$194$&$yx(yx^{-1})^2x^{-1}=1$&
$195$&$yx(yx^{-1})^2y=1$\\
$196$&$yxyx^{-1}yx^2=1$&
$197$&$yxyx^{-1}yxy=1$&
$198$&$(yx)^2x^3=1$\\
$199$&$(yx)^2x^2y=1$&
$200$&$(yx)^2xy^{-1}x=1$&
$201$&$(yx)^2xy^{-1}x^{-1}=1$\\
$202$&$(yx)^2yx^{-2}=1$&
$203$&$(yx)^2yx^{-1}y=1$&204& $ y^{-1}x^5y^{-1}x =1$\\
205&$ y^{-1}x^5y^{-1}x^{-1} =1$&206& $ y^{-1}x^4yx^{-2} =1$&207&$ y^{-1}x(xy)^2x^2 =1$\\ 
208& $ y^{-1}x^4yx^2 =1$&209&  $ y^{-1}x^2(xy^{-1})^2x =1$&210& $ y^{-1}x^2yx^2y^{-1}x^{-1} =1$\\
211&$ y^{-1}x^2(xy^{-1})^2x^{-1} =1$&212&$ y^{-1}x^3y^{-1}x^{-3} =1$&213&$ y^{-1}x(xy)^2x^{-2} =1$\\
214&$ y^{-1}x^3y^{-1}x^{-1}y^{-1}x =1$&215& $ y^{-1}x^3(y^{-1}x^{-1})^2 =1$&216&$ y^{-1}x^2yx^2y^{-1}x =1$\\ 
217&$ y^{-1}x^2yx^{-4} =1$&218& $ y^{-1}x^2yx^{-2}y^{-1}x =1$&219&$ y^{-1}x^2yx^4 =1$\\
220& $ y^{-1}x^2yx^{-2}y^{-1}x^{-1} =1$&221&$ y^{-1}x^2(yx^{-1})^2x^{-1} =1$&222&$ y^{-1}x^2yx^{-1}yx^2 =1$\\
$ 223^* $& $ (y^{-1}x)^3y^{-1}x^{-1} =1$&224&$ (y^{-1}x)^2y^{-1}x^{-3} =1$&$225^{*}$&$ (y^{-1}x)^2(y^{-1}x^{-1})^2 =1$\\
226&$ y^{-1}xy^{-1}x^{-5} =1$&227&$ y^{-1}xy^{-1}x^{-3}y^{-1}x^{-1} =1$&228& $ y^{-1}x(y^{-1}x^{-1})^2x^{-2} =1$\\
$ 229^* $&$ y^{-1}x(y^{-1}x^{-1})^3 =1$&230&$ y^{-1}x^{-5}y^{-1}x^{-1} =1$&231&$ y^{-1}x^{-2}(x^{-1}y^{-1})^2x^{-1} =1$ \\
\end{longtable}

\textbf{Graphs} $ \mathbf{H_3, H_4, H_5}$   \textbf{and}  $\mathbf{H_6}$ \textbf{of Figure \ref{10k}:} Suppose that $ C $ and $ C' $ are two cycles of  length $  4$  in $Z(\alpha,\beta)$ with vertex sets  $\mathcal{V}_{C} =\{ g_1,g_2,g_3,g_4 \}\subset supp(\beta)$ and $\mathcal{V}_{C'} =\{ g_1,g_2,g_3,g_5 \}\subset supp(\beta)$, respectively, (i.e.,  $ C $ and $ C' $ have exactly two consecutive edges in common )  such that  each of the sets $ \mathcal{V}_{C'} $ and $ \mathcal{V}_{C} $ and $ \{g_1,g_4,g_3,g_5\} $ contains at least two vertices of degree $ 4 $ in $Z(\alpha,\beta)$.  Further suppose that $ T\in \mathcal{T}(C) $, $ T'\in \mathcal{T}(C') $, $ T=[h_1,h'_1,h_2,h'_2,h_3,h'_3,,h_4,h'_4] $ and   $T'=[h_1,h'_1,h_2,h'_2,t_3,t'_3,t_4,t'_4] $, where the first four components are related to the common edges between these cycles. Taking into account Table \ref{tt011}  and by a same discussion as about the graph $  H_1$ as Figure \ref{10k} and using GAP \cite{a9}, it can be seen that  there are  $442$ different cases for  the cycles of length $ 4 $ with above structure. Using GAP \cite{a9},   a free group with generators $ x,y $ and the relations of these two cycles which are between $410$ cases of these $442$ cases  is finite or abelian, that is a
contradiction. Hence, there are $ 32 $ cases which may lead
to the existence of two cycles with this structure in $ Z(\alpha,\beta) $. We checked these such cases. In Table \ref{tk23}, it can be seen that $ 28 $ cases among these $ 35 $ cases lead to contradictions and so there are just $ 4 $ cases which may lead
to the existence of these cycles  in $Z(\alpha,\beta) $ as follows:
 \begin{itemize}
 \item[1.]$ T=[ 1, y, x, y, 1, y, x^{-1}, y ] \rightarrow r(C):yx^{-1}y^2xy=1 $ and $ T'=[ 1, y, x, y, x^{-1}, x, y, x^{-1} ]\rightarrow r(C'):yx^{-1}yx^2y^{-1}x^{-1}=1 $.
 \item[2.]$ T=[ 1, y, x, y, 1, y, x^{-1}, y ] \rightarrow r(C):yx^{-1}y^2xy=1 $ and $ T'=[ 1, y, x, y, x, y,  x^{-1},x ]\rightarrow r(C'):(yx^{-1})^2yx^2=1 $.
  \item[3.]$ T=[ y, x, x^{-1},x, x^{-1}, x, y, x ] \rightarrow r(C):y^{-1}x^{5}y^{-1}x=1 $ and $ T'=[ y, x, x^{-1},x, y, x^{-1},  y, x^{-1} ]\rightarrow r(C'):y^{-1}x^3(y^{-1}x^{-1})^2=1 $.
  \item[4.]$ T=[ y, x, x^{-1},y, x, y, x^{-1}, x ] \rightarrow r(C):y^{-1}x^{2}yx^{-1}xyx^{2}=1 $ and $ T'=[ y, x, x^{-1},y, x^{-1},x,  y, x^{-1} ]$  $\rightarrow r(C'):y^{-1}x^2yx^{2}y^{-1}x^{-1}=1 $.
 \end{itemize}
As for Remark  \ref{3p}, none of the above cases can not be related to the graph $H_3$ of Figure \ref{10k} and therefore $ Z(\alpha,\beta) $ does not contain any subgraph isomorphic to $H_3$.\\ 
With consideration of  above relations and  Table \ref{tt011}, it can be seen that  there are  $3$ different cases for the relations of three cycles of length $ 4 $ with the structure of the graphs $ H_4 $ and $ H_5 $ in $ Z(\alpha,\beta) $. Using GAP \cite{a9},    a free group with generators $ x,y $ and the relations of such cases  is finite. Another case is the following:\\
 $r_1:yx^{-1}y^2xy =1,\;\; r_2:(yx^{-1})^2yx^2 =1,\;\;r_3:yx^{-1}yx^2y^{-1}x^{-1}=1$.\\
$ r_2 $ and $ r_3 $ imply that $ y $ is a torsion element of $ G $, a contradiction. Thus, $ Z(\alpha,\beta) $ does not contain any subgraph isomorphic to $H_4$ and $ H_5 $ as Figure \ref{10k}.
Similarly, with consideration of  above relations and  Table \ref{tt011}, it can be seen that  there are  $25$ different cases for the relations of three cycles of length $ 4 $ with the structure of the graph  $ H_6 $ in $ Z(\alpha,\beta) $. Using GAP \cite{a9},    a free group with generators $ x,y $ and the relations of such cases  is finite and therefore $ Z(\alpha,\beta) $ does not contain any subgraph isomorphic to   $H_6$.\\ 
  \begin{longtable}{|l|l|l|l|}
\caption{The relations of two cycles of length $ 4 $ in $ Z(\alpha,\beta) $ which have exactly two consecutive edges in common that lead to contradictions.}\label{tk23}\\
\hline
\endfirsthead
\multicolumn{4}{c}%
{{\bfseries \tablename\ \thetable{} -- continued from previous page}} \\\hline
\hline 
 \multicolumn{1}{|l|}{\textbf{$ n $}} & \multicolumn{1}{l|}{\textbf{$r_1$}}&  \multicolumn{1}{|l|}{\textbf{$r_2 $}} & \multicolumn{1}{l|}{\textbf{$E$}}\\ \hline
\endhead

\hline \multicolumn{4}{|r|}{{Continued on next page}} \\ \hline
\endfoot

\hline 
\endlastfoot
\hline
\textbf{$ n $}& \textbf{$r_1$} &  \textbf{$r_2$} &\textbf{$ E $} \\
\hline
1 &$x^3y^2=1$&$x^3y^{-2}x=1$&T\\
2&$x^4y^2=1$&$x^3y^{-2}x=1$&T\\
3&$x^2y^2x^{-1}y=1$&$x(xy)^2x^2=1$&Abelian\\
4&$x^{-4}y^2=1$&$x^{-3}y^{-2}x^{-1}=1$&T\\
5&$x^{-2}y^2x^{-2}=1$&$x^{-2}yx^2y=1$&T\\
6&$y^2x^{-4}=1$&$y^2x^4=1$&T\\
7&$yx^{-2}yx^2=1$&$yx^{-2}y^{-1}x^{-2}=1$&T\\
8&$yx^{-2}yx^2=1$&$yx^{-2}y^{-1}x^{-2}y=1$&T\\
9&$(yx^{-1})^2x^{-3}=1$&$yx^{-1}yxyx^{-1}y=1$&T\\
10&$yxyx^{-1}yxy=1$&$(yx)^2x^3=1$&T\\
11&$x^2y^2x^2=1$&$x^2yx^{-2}y=1$&T\\
12&$xy^{-1}x^{-2}y^{-1}x=1$&$xy^{-1}x^{-2}yxy=1$&BS$ (2,-1) $\\
13&$xy^{-1}x^{-2}y^2=1$&$xy^{-1}x^{-1}y^{-1}x^2=1$&BS$ (2,-1) $\\
14&$xy^{-1}x^{-2}yxy=1$&$x(y^{-1}x^{-1})^2y^{-1}x=1$&T\\
15&$x^{-1}y^{-1}x^2y^2=1$&$x^{-1}y^{-1}xy^{-1}x^{-2}=1$&BS$ (1,2) $\\
16&$x^{-1}y^{-1}x^2yx^{-1}y=1$&$x^{-1}(y^{-1}x)^2y^{-1}x^{-1}=1$&T\\
17&$yx^{-2}yxy=1$&$yx^{-2}y^{-1}x^2=1$&Abelian\\
18&$(yx^{-1})^2yx^2=1$&$yx^{-1}yx^2y^{-1}x^{-1}=1$&T\\
19&$yxyx^{-2}y^{-1}x=1$&$(yx)^2yx^{-2}=1$&T\\
20&$ y^{-1}x^5y^{-1}x=1 $&$ y^{-1}x^3y^{-1}x^{-3}=1 $&T \\
21&$ y^{-1}x^4yx^{-2}=1 $&$ y^{-1}x^2(xy^{-1})^2x=1 $&T \\
22&$ y^{-1}x^4yx^{-2}=1 $&$ y^{-1}x^3y^{-1}x^{-1}y^{-1}x=1 $& T\\
23&$ y^{-1}x^4yx^2=1 $&$  y^{-1}x^2(xy^{-1})^2x^{-1}=1 $&T \\
24&$ y^{-1}x^4yx^2=1 $&$ y^{-1}x^3y^{-1}x^{-3}=1 $& T\\
25&$ y^{-1}x^4yx^2 =1 $&$  y^{-1}x^3(y^{-1}x^{-1})^2=1 $&T \\
26&$ y^{-1}x^2yx^{-4} =1 $&$  y^{-1}x^2yx^4=1$&T \\
27&$ y^{-1}x^2yx^{-2}y^{-1}x^{-1}=1 $&$ y^{-1}x(xy)^2x^2=1 $& Abelian\\
28&$ y^{-1}x^2(yx^{-1})^2x^{-1}=1 $&$ y^{-1}x^2yx^2y^{-1}x=1 $& Abelian \\
\hline
 \end{longtable}
$\mathbf{ H_7, H_8}$   \textbf{and}  $\mathbf{H_9}$ \textbf{of Figure \ref{10k}:} Taking into account Tables \ref{tt011} and \ref{c3e},  it can be seen that  there are  $217$ different cases for the relations of  a cycle of length $ 3 $ and a cycle of length $ 4 $ with the structure of each of the graphs $  H_7, H_8 $ and $ H_9 $. Using GAP \cite{a9},   a free group with generators $ x,y$ and the relations of these $ 2$ cycles which are between $194$ cases of these $217$ cases  is finite or abelian, that is a
contradiction.   Hence, there are $ 23 $ cases which may lead
to the existence of each of the graphs $  H_7, H_8 $ and $ H_9 $ in $ Z(\alpha,\beta) $. We checked these such cases. In Table \ref{h7}, it can be seen that all such cases  lead to contradictions and so $ Z(\alpha,\beta) $ does not contain any subgraph isomorphic to the graphs $  H_7, H_8 $ and $ H_9 $.
 \begin{longtable}{|l|l|l|l|}
\caption{The relations of a cycle of length $ 3$ and a cycle of length $ 4$  which have exactly an edge in common.}\label{h7}\\
\hline
\endfirsthead
\multicolumn{4}{c}%
{{\bfseries \tablename\ \thetable{} -- continued from previous page}} \\\hline
\hline 
 \multicolumn{1}{|l|}{\textbf{$ n $}} & \multicolumn{1}{l|}{\textbf{$r_1$}}&  \multicolumn{1}{|l|}{\textbf{$r_2 $}} & \multicolumn{1}{l|}{\textbf{$E$}}\\ \hline
\endhead

\hline \multicolumn{4}{|r|}{{Continued on next page}} \\ \hline
\endfoot

\hline 
\endlastfoot
\hline
\textbf{$ n $}& \textbf{$r_1$} &  \textbf{$r_2$} &\textbf{$ E $} \\
\hline
1&$ x^2yx^{-1}y =1 $&$ (xy^{-1})^2x^3 =1 $& T\\
2&$ xy^{-1}x^{-1}y^{-1}x =1 $&$ x^3(xy)^2=1 $&T\\
3&$ x^{-2}yxy=1 $&$ (x^{-1}y^{-1})^2x^{-3}=1 $&T\\
4&$ x^{-1}y^{-1}xy^{-1}x^{-1}=1 $&$ x^{-3}(x^{-1}y)^2=1 $&T\\
5&$y^{-1}x^2yx^2=1$&$(y^{-1}x)^3y^{-1}x^{-1}=1$&T\\
6&$y^{-1}x^2yx^2=1$&$(y^{-1}x)^2y^{-1}x^{-3}=1$&BS$ (5,1) $\\
7&$y^{-1}x^2yx^2=1$&$(y^{-1}x)^2(y^{-1}x^{-1})^2=1$&T\\
8&$y^{-1}x^2yx^2=1$&$y^{-1}xy^{-1}x^{-5}=1$&T\\
9&$y^{-1}x^2yx^2=1$&$y^{-1}xy^{-1}x^{-3}y^{-1}x^{-1}=1$&T\\
10&$y^{-1}x^2yx^2=1$&$y^{-1}x(y^{-1}x^{-1})^2x^{-2}=1$&T\\
11&$y^{-1}x^2yx^2=1$&$y^{-1}x(y^{-1}x^{-1})^3=1$&T\\
12&$(y^{-1}x)^2y^{-1}x^{-1}=1$&$y^{-1}x^4yx^2=1$&T\\
13&$(y^{-1}x)^2y^{-1}x^{-1}=1$&$y^{-1}x^2(xy^{-1})^2x=1$&T\\
14&$(y^{-1}x)^2y^{-1}x^{-1}=1$&$y^{-1}x^3y^{-1}x^{-1}y^{-1}x=1$&T\\
15&$(y^{-1}x)^2y^{-1}x^{-1}=1$&$y^{-1}x^2yx^4=1$&BS$ (5,-1) $\\
16&$y^{-1}xy^{-1}x^{-3}=1$&$y^{-1}x^5y^{-1}x=1$&T\\
17&$y^{-1}xy^{-1}x^{-3}=1$&$y^{-1}x^4yx^2=1$&T\\
18&$y^{-1}xy^{-1}x^{-3}=1$&$y^{-1}x^2yx^4=1$&T\\
19&$y^{-1}xy^{-1}x^{-3}=1$&$y^{-1}x(xy)^2x^2=1$&T\\
20&$y^{-1}x(y^{-1}x^{-1})^2=1$&$y^{-1}x^4yx^2=1$&T\\
21&$y^{-1}x(y^{-1}x^{-1})^2=1$&$y^{-1}x^2(xy^{-1})^2x=1$&T\\
22&$y^{-1}x(y^{-1}x^{-1})^2=1$&$y^{-1}x^3y^{-1}x^{-1}y^{-1}x=1$&T\\
23&$y^{-1}x(y^{-1}x^{-1})^2=1$&$y^{-1}x^2yx^4=1$&T\\
\hline
 \end{longtable}
  $ \mathbf{\Gamma_1}$ and $ \mathbf{\Gamma_2}$ \textbf{of Figure \ref{10b}:} Taking into account the relations from Table \ref{tt011} which are marked by *s, it can be seen that  there are  $31$ different cases for the relations of $2$ cycles of length $3$ which have exactly an edge in common on vertices of degree $ 3 $ in $  Z(\alpha,\beta)  $. Using GAP \cite{a9},    a free group with generators $ x,y $ and the relations of these $ 2$ cycles which are between $  24$ cases of these $31$  cases, is finite. Hence there are $ 7 $ cases which may lead to the existence of these $2$ cycles.   We checked $ 7 $ such  cases,  $ 4 $ cases which are between these $ 7 $ cases imply  that $ G $ has a non-trivial torsion element that is  a contradiction. Four such cases are as follows:
 \begin{itemize}
 \item[(1)]$r_1:x^2yx^{-2}y =1,\qquad r_2: x(y^{-1}x^{-1})^2y^{-1}x $.
 \item[(2)]$r_1:x^2(yx^{-1})^2y =1,\;\; r_2: xy^{-1}x^{-2}y^{-1}x=1$.
 \item[(3)]$r_1:x^{-2}yx^2y =1 \qquad r_2:x^{-1}(y^{-1}x)^2y^{-1}x^{-1}=1    $.
 \item[(4)]$r_1:x^{-2}yx^2y =1 \qquad r_2:x^{-1}(y^{-1}x)^2y^{-1}x^{-1}=1    $.
\end{itemize}  
Hence, there are $ 3 $ cases which may lead to  the existence of $2$ cycles of length $3$ which have exactly an edge in common on vertices of degree $ 3 $ in $  Z(\alpha,\beta)  $. Three such cases are as follows: 
\begin{itemize}
 \item[(1)]$r_1:x^2yx^{-2}y=1,\qquad r_2:xy^{-1}x^{-2}y^{-1}x=1 $.
 \item[(2)]$r_1:x^2(yx^{-1})^2y =1,\qquad r_2:x(y^{-1}x^{-1})^2y^{-1}x =1$.
 \item[(3)]$r_1: x^{-2}(yx)^2y=1, \qquad r_2: x^{-1}(y^{-1}x)^2y^{-1}x^{-1}=1 $.
\end{itemize} 
 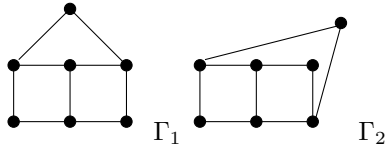
\begin{figure}
\begin{tikzpicture}[scale=.75]
\draw [fill] (0,0) circle
[radius=0.1] node  [left]  {};
\draw [fill] (1,0) circle
[radius=0.1] node  [below left]  {};
\draw [fill] (0,-1) circle
[radius=0.1] node  [left]  {};
\draw [fill](1,-1) circle
[radius=0.1] node  [below]  {};
\draw [fill](2,0) circle
[radius=0.1] node  [right]  {};
\draw [fill] (2,-1) circle
[radius=0.1] node  [right]  {};
\draw [fill] (1,1) circle
[radius=0.1] node  [left]  {};
\draw    (.95,.95) --(.09,.1);
\draw    (1.95,.05)--(1.05,.95)  ;
\draw  (.1,0)--(.9,0) ;
\draw  (1,-.1) --(1,-.9);
\draw  (.9,-1)-- (.1,-1)  ;
\draw  (0,-.9)--(0,-.1)  ;
\draw  (1.9,0)--(1.1,0)  ;
\draw  (2,-.9)--(2,-.1) ;
\draw  (1.1,-1)--(1.9,-1)  ;
\end{tikzpicture}
\subfloat{$\Gamma_1$}
\begin{tikzpicture}[scale=.75]
\draw [fill] (0,0) circle
[radius=0.1] node  [left]  {};
\draw [fill] (1,0) circle
[radius=0.1] node  [below left]  {};
\draw [fill] (0,-1) circle
[radius=0.1] node  [left]  {};
\draw [fill](1,-1) circle
[radius=0.1] node  [below]  {};
\draw [fill](2,0) circle
[radius=0.1] node  [below left]  {};
\draw [fill] (2,-1) circle
[radius=0.1] node  [right]  {};
\draw [fill](2.5,.75) circle
[radius=0.1] node  [above]  {};
\draw  (.1,0)--(.9,0) ;
\draw  (1,-.1) --(1,-.9);
\draw  (.9,-1)-- (.1,-1)  ;
\draw  (0,-.9)--(0,-.1)  ;
\draw  (1.9,0)--(1.1,0)  ;
\draw  (2,-.9)--(2,-.1) ;
\draw  (1.1,-1)--(1.9,-1)  ;
\draw   (2.45,.7)--(.09,.1) ;
\draw   (2.05,-.95) --(2.5,.7) ;
\end{tikzpicture}
\subfloat{$ \Gamma_2$}
\caption{ Two forbidden subgraphs of $ Z(\alpha ,\beta) $, where the degree of all vertices of any subgraph in $ Z(\alpha ,\beta) $ must be 4 of type $ (ii) $. }\label{10b}
\end{figure}
By considering  the relations from Table \ref{tt011} which are marked by *s and the relations of $ 3 $ above cases, it can be seen that  there are  $2$ different cases for the relations of $3$ cycles of length $3$ with the structure of the graph $ \Gamma_1 $ as Figure \ref{10b} on vertices of degree $ 3 $ in $  Z(\alpha,\beta)  $. Using GAP \cite{a9},    a free group with generators $ x,y $ and  the relations of each of these two cases is finite. Hence, $ Z(\alpha ,\beta) $ does not contain any subgraph isomorphic to the graph $ \Gamma_1 $ as Figure \ref{10b}. Also by considering  the relations of $ 3 $ above cases, it is easy to see that there are $ 2 $ different cases for two cycles of length $ 4 $ and a cycle of length $ 5 $ with the structure of the graph  $ \Gamma_2$ as Figure \ref{10b}. Using GAP \cite{a9},    a free group with generators $ x,y $ and  the relations of each of these two cases is finite. Hence, $ Z(\alpha ,\beta) $ does not contain any subgraph isomorphic to the graph $ \Gamma_2 $ as Figure \ref{10b}.\\
\textbf{Graphs $ \mathbf{\Gamma_{12},\; \Gamma_{13}} $ and $ \mathbf{\Gamma_{14}} $ of Figure \ref{36}:} Suppose that $ g $ is a vertex of degree $ 5 $ in $  Z(\alpha,\beta)  $. Hence, by Remark \ref{deg12}, $ \theta_3(g)=1 $ and $ \theta_4(g)=0 $. Thus, there exist exactly an element  $ s\in \Delta^3_{(\alpha,\beta)} $ such that $ g\in \mathcal{V}_{(\alpha,\beta)}(s) $. If  one of the  graphs $ \Gamma_{12} ,\; \Gamma_{13} $ and $ \Gamma_{14}$ as Figure \ref{36}  is a subgraph of $ Z(\alpha ,\beta) $, then clearly there are at least two cycles of length $ 3 $ in  $ Z(\alpha ,\beta) $ such as $ C $ and $ C' $ such that $ |\mathcal{V}_{C'}\cap \mathcal{V}_C|=2 $ (i.e., $ C $ and $ C' $ have exactly an edge in common), $ g\in \mathcal{V}_{C'}\cap \mathcal{V}_C\cap \mathcal{V}_{(\alpha,\beta)}(s) $,  $ |\mathcal{V}_C\cap \mathcal{V}_{(\alpha,\beta)}(s)|\leq 2 $ and  $ |\mathcal{V}_{C'}\cap \mathcal{V}_{(\alpha,\beta)}(s)|\leq 2 $. Now, suppose that  $ T_C=[h_1,h'_1,h_2,h'_2,h_3,h'_3] $ and $ T_{C'}=[h_1,h'_1,t_2,t'_2,t_3,t'_3] $ are $ 6- $tuples corresponding  to $ C $ and $ C' $, respectively, where the first two components are related to the common edge between  these cycles. With the above explanation and as for part (3) of Lemma \ref{3part}, it is easy to see that the conditions for $ T_C $ and $ T_{C'} $ are exactly the same as the conditions relating to the $ 6- $tuples of two cycles of length $ 3 $ with the structure of the graph $ H_1 $ as Figure \ref{10k} which  was explained in the early part of this section. Hence, the relations of all possible cases for $ T_C $ and $ T_{C'} $ lead to contradictions and therefore $ Z(\alpha ,\beta) $ does not contain any subgraph isomorphic to one of the graphs $ \Gamma_{12} ,\; \Gamma_{13} $ and $ \Gamma_{14}$ as Figure \ref{36}. 
 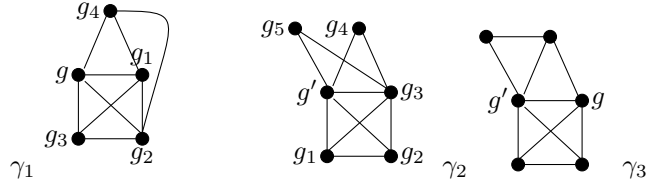
\begin{figure}
 \subfloat{$\gamma_1$}
 \begin{tikzpicture}[scale=.85]
\draw [fill] (0,0) circle
[radius=0.1] node  [left]  {$ g $};
\draw [fill] (1,0) circle
[radius=0.1] node  [above]  {$ g_1 $};
\draw [fill] (0,-1) circle
[radius=0.1] node  [left]  {$ g_3 $};
\draw [fill] (1,-1) circle
[radius=0.1] node  [below]  {$ g_2 $};
\draw [fill] (.5,1) circle
[radius=0.1] node  [left]  {$ g_4 $};
\draw  (.09,.1) --  (.45,.95) ;
\draw  (.55,.95) --  (.95,.05) ;
\draw  (.1,0) --  (.9,0) ;
\draw  (1,-.9) --  (1,-.1) ;
\draw  (.9,-1) --  (.1,-1) ;
\draw  (0,-.1) --  (0,-.9) ;
\draw (1,-.9) .. controls (1.5,1) and (1.75,1)  .. (.5,1) ;
\draw  (1,-.93) --  (.1,-.1) ;
\draw   (.99,-.05)  --  (0,-.93) ;
\end{tikzpicture}
\qquad
\begin{tikzpicture}[scale=.85]
\draw [fill] (0,0) circle
[radius=0.1] node  [left]  {$g'$};
\draw [fill] (1,0) circle
[radius=0.1] node  [right]  {$ g_3 $};
\draw [fill] (0,-1) circle
[radius=0.1] node  [left]  {$ g_1 $};
\draw [fill] (1,-1) circle
[radius=0.1] node  [right]  {$ g_2 $};
\draw [fill] (.5,1) circle
[radius=0.1] node  [left]  {$ g_4 $};
\draw [fill] (-.5,1) circle
[radius=0.1] node  [left]  {$ g_5 $};
\draw  (.09,.1) --  (.45,.95) ;
\draw  (.55,.95) --  (.95,.05) ;
\draw  (.1,0) --  (.9,0) ;
\draw  (1,-.9) --  (1,-.1) ;
\draw  (.9,-1) --  (.1,-1) ;
\draw  (0,-.1) --  (0,-.9) ;
\draw  (1,-.93) --  (.1,-.1) ;
\draw   (.99,-.05)  --  (0,-.93) ;
\draw  (-.5,1) --  (.95,.05) ;
\draw  (-.5,1) --  (0,.1) ;
\end{tikzpicture}
\subfloat{$\gamma_2$}
\begin{tikzpicture}[scale=.85]
\draw [fill] (0,0) circle
[radius=0.1] node  [left]  {$g'$};
\draw [fill] (1,0) circle
[radius=0.1] node  [right]  {$g$};
\draw [fill] (0,-1) circle
[radius=0.1] node  [left]  {};
\draw [fill] (1,-1) circle
[radius=0.1] node  [left]  {};
\draw [fill] (.5,1) circle
[radius=0.1] node  [left]  {};
\draw [fill] (-.5,1) circle
[radius=0.1] node  [left]  {};
\draw  (.09,.1) --  (.45,.95) ;
\draw  (.55,.95) --  (.95,.05) ;
\draw  (.1,0) --  (.9,0) ;
\draw  (1,-.9) --  (1,-.1) ;
\draw  (.9,-1) --  (.1,-1) ;
\draw  (0,-.1) --  (0,-.9) ;
\draw  (1,-.93) --  (.1,-.1) ;
\draw   (.99,-.05)  --  (0,-.93) ;
\draw  (-.5,1) --  (0,.1) ;
\draw  (-.5,1) --  (.5,1) ;
\end{tikzpicture}
\subfloat{$\gamma_3$}
\caption{  Three forbidden subgraphs of $ \mathcal{Z}(\alpha,\beta) $, where the degrees of vertices  $ g $ and $ g' $ of any subgraph in $ \mathcal{Z}(\alpha,\beta) $ must be $ 4 $ of type $ (ii) $ and $ 5 $. }\label{37}
\end{figure}
\\ \textbf{$ \mathbf{\gamma_1,\;\gamma_2 } $ and $ \mathbf{\gamma_3} $ of Figure \ref{37}:} Suppose that $ g $ is a vertex of degree $ 4 $ of type $ (ii) $ in $ \mathcal{Z}(\alpha,\beta) $ and $ g_1,g_2,g_3 $ and $ g_4 $ are all vertices in $ \mathcal{Z}(\alpha,\beta) $ which are adjacent to $ g  $. According to our discussion in Section 8, there exists exactly an element $ s\in \Delta_{(\alpha,\beta)}^{4} $ such that $ g\in \mathcal{V}_{(\alpha,\beta)}(s) $ and $ \theta_3(g)=0 $. Suppose that $\mathcal{V}_{(\alpha,\beta)}(s)=\{g,g_1,g_2,g_3\} $ and $ g_4\sim g_1 $ and $ g_4\sim g_2 $ (see Figure \ref{37}). Hence, there are two cycles of length 3 which have an edge in common in $ Z(\alpha,\beta) $ such as  $ C $ and $ C' $ with the vertex sets  $\mathcal{V}_{C}=\{ g,g_4,g_1\} $ and $\mathcal{V}_{C}=\{ g,g_4,g_2\} $. Let  $ T_C=[h_1,h'_1,h_2,h'_2,h_3,h'_3] $ and $ T_{C'}=[h_1,h'_1,t_2,t'_2,t_3,t'_3] $ be the $ 6- $tuples corresponding  to $ C $ and $ C' $, respectively, where the first two component are related to common edge of these cycles. With the above explanation, it is easy to see that the conditions for $ T_C $ and $ T_{C'} $ are exactly the same as the conditions relating to the $ 6- $tuples of two cycles of length $ 3 $ with the structure of graph $ H_1 $ as Figure \ref{10k} which  was explained in the early part of this section. Hence, the relations of all possible cases for $ T_C $ and $ T_{C'} $ lead to contradictions and therefore $ \mathcal{Z}(\alpha ,\beta) $ does not contain any subgraph isomorphic to  the graph $ \gamma_{1}$ as Figure \ref{37}. \\
Now, suppose that $ g' $ is a vertex of degree $ 5 $ in $ \mathcal{Z}(\alpha,\beta) $ and $ g_1,g_2,g_3, g_4 $ and $ g_5 $ are all vertices in $ \mathcal{Z}(\alpha,\beta) $ which are adjacent to $ g'  $. According to our discussion in Section 8, there exists exactly an element $ s\in \Delta_{(\alpha,\beta)}^{4} $ such that $ g'\in \mathcal{V}_{(\alpha,\beta)}(s) $ and $ \theta_3(g')=0 $. Suppose that $ \mathcal{V}_{(\alpha,\beta)}(s)=\{g',g_1,g_2,g_3\} $ (see Figure \ref{37}). With the same description as about the graph $\gamma_{1}$ as Figure \ref{37}, it can be seen that  it is impossible that $ g_4\sim g_3 $, $ g_4\sim g_2 $ and $ g_4\sim g_1 $ and also it is impossible that $ g_4\sim g_3 $, $ g_5\sim g_3 $.  Also it can be seen that if $ g_3 $ is a vertex of degree $ 4 $ of type $ (ii) $ in $ \mathcal{Z}(\alpha,\beta) $, then it is impossible that $ g_4\sim g_3 $ and $ g_4\sim g_5 $. Therefore, $ \mathcal{Z}(\alpha ,\beta) $ does not contain any subgraph isomorphic to one of the graphs $ \gamma_{2}$ and $ \gamma_{3}$ as Figure \ref{37}. This completes the proof of Theorems \ref{app} and \ref{appmatcal}. \\
In view of Remark \ref{3punit}, the argument to prove Theorem \ref{appunit} is exactly the same as the argument to prove that the graph $  H_1$ as Figure \ref{10k} is a forbidden subgraph of  $ Z(\alpha,\beta) $ which  was explained in the early part of this section.
\end{document}